\documentclass[reqno]{amsart}
\usepackage[utf8]{inputenc}
\usepackage{times}

\usepackage{times}
\usepackage[margin=1.3in]{geometry}

\usepackage{amsmath, amsfonts, amssymb, amsthm, mathtools}
\usepackage{wasysym}
\usepackage{mathrsfs}
\usepackage{bbm}
\usepackage{bm}
\usepackage{stmaryrd}
\usepackage{tensor}
\usepackage[h]{esvect}
\usepackage{calligra}
\usepackage{xcolor} 
\definecolor{darkergreen}{RGB}{0,110,0}

\usepackage{tikz-cd}
\usetikzlibrary{matrix}
\usepackage[all]{xy}

\usepackage{graphicx}
\usepackage{epstopdf}
\usepackage[export]{adjustbox}
\usepackage{caption}
\usepackage{subcaption}
\usepackage{pict2e}

\usepackage{pifont}
\usepackage{csquotes}
\usepackage{halloweenmath}

\usepackage[linktocpage]{hyperref}
\hypersetup{
    colorlinks=true,
    linkcolor=blue,
    citecolor=blue,      
    urlcolor=blue,
}

\newtheorem{theorem}{Theorem}
\newtheorem{proposition}[theorem]{Proposition}
\newtheorem{lemma}[theorem]{Lemma}
\newtheorem{claim}[theorem]{Claim}
\newtheorem{corollary}[theorem]{Corollary}
\newtheorem{conjecture}[theorem]{Conjecture}

\newtheorem{remark}[theorem]{Remark}

\theoremstyle{definition}
\newtheorem{definition}[theorem]{Definition}
\newtheorem{example}[theorem]{Example}

\numberwithin{equation}{section}
\numberwithin{theorem}{section}

\usepackage[english]{babel}
\usepackage{enumitem}

\usepackage{hyphenat}

\usepackage[backend=bibtex,style=alphabetic,maxalphanames=4,maxnames=4]{biblatex}

\renewbibmacro{in:}{}

\newcommand{\op}{\operatorname}

\DeclareMathOperator{\Hom}{\mathscr{H}\text{\kern -3pt {\calligra\large om}}\,}
\DeclareDelimFormat[bib,biblist]{nametitledelim}{\addcomma\space}

\DeclareFieldFormat*{title}{\mkbibitalic{#1}\addcomma}
\DeclareFieldFormat*{journaltitle}{#1}
\DeclareFieldFormat*{volume}{\mkbibbold{#1}}
\DeclareFieldFormat{pages}{#1}
\DeclareFieldFormat[misc]{date}{preprint {#1}}
\DeclareFieldFormat{mr}{%
  MR\addcolon\space
  \ifhyperref
    {\href{http://www.ams.org/mathscinet-getitem?mr=MR#1}{\nolinkurl{#1}}}
    {\nolinkurl{#1}}}
    
\AtEveryBibitem{
  \clearfield{url}
  \clearfield{number}
  \clearfield{doi}
  \clearfield{issn}
  \clearfield{isbn}
  \clearfield{eprintclass}
}
\AtEveryBibitem{\ifentrytype{book}{\clearfield{pages}}{}}

\bibliography{hecke}

\usepackage{fancyhdr}
\renewcommand{\H}{\mathbb{H}}
\newcommand{\C}{\mathbb{C}}

\newcommand{\R}{\mathbb{R}}
\newcommand{\Z}{\mathbb{Z}}
\newcommand{\Q}{\mathbb{Q}}

\renewcommand{\P}{\mathbb{P}}
\newcommand{\bdry}{\partial}
\newcommand{\s}{\vskip.1in}
\newcommand{\n}{\noindent}
\newcommand{\J}{\mathfrak{J}}

\newcommand{\be}{\begin{enumerate}}
\newcommand{\ee}{\end{enumerate}}

\usepackage{todonotes}

\title{Obstructions to smoothing orbicurves and Orbifold Hecke algebras}

\author{Ko Honda}
\address{University of California, Los Angeles, Los Angeles, CA 90095}
\email{honda@math.ucla.edu} \urladdr{http://www.math.ucla.edu/\char126 honda}

\author{Roman Krutowski}
\address{University of California, Los Angeles, Los Angeles, CA 90095}
\email{romankrut@ucla.edu} \urladdr{http://romakrut.com}

\author{Yin Tian}
\address{School of Mathematical Sciences, Beijing Normal University; 
Laboratory of Mathematics and Complex Systems, Ministry of Education, Beijing 100875, China}
\email{yintian@bnu.edu.cn} \urladdr{}

\author{Tianyu Yuan}
\address{School of Mathematical Sciences, Eastern Institute of Technology, Ningbo, Zhejiang, 315200, China}
\email{tyyuan@eitech.edu.cn} \urladdr{}

\date{\today}

\keywords{orbifold Floer theory, Hecke algebra}

\subjclass[2020]{Primary 53D37; Secondary 20C08, 57R18.}

\thanks{KH supported by NSF Grants DMS-2003483 and DMS-2505876. YT supported by NSFC Grant No. 12471064. TY supported by NSFC 12501080.}

\begin{document}

\begin{abstract}
In this paper, we study obstructions to smoothing nodal orbicurves with orbighosts mapped into cyclic quotient singularities. As an application, we show, under some assumptions, that the orbifold Hecke algebra of a complex global quotient orbifold $[X/G]$ is isomorphic to the degree $0$ cohomology of the $A_\infty$-algebra of endomorphisms of a regular cotangent fiber $T_{[x]}^*[X/G]$ regarded as an object of the bulk-deformed wrapped Fukaya category of the orbifold $T^*[X/G]$ for a compact complex manifold $X$ and finite group $G \subset \operatorname{Aut}(X)$.
\end{abstract}

\maketitle

\tableofcontents

\section{Introduction}

The aim of this paper is to study the orbifold Floer theory of a global quotient $[M/G]$
for a symplectic manifold $(M,\omega)$ together with an action of a finite subgroup $G$ of its symplectomorphism group $\operatorname{Symp}(M,\omega)$. Specifically, we build tools for defining \emph{refined} invariants such as \emph{reduced} (open) Gromov-Witten invariants and Fukaya categories. This usually requires a careful analysis of the moduli space of stable curves, and, in particular, of the compactification of the main stratum. One way to analyze an essential part of the compactification is to determine which orbicurves are not \emph{smoothable}. In this work, we focus on twisted nodal orbicurves with ghosts on cyclic quotient singularities and obtain obstructions to smoothing such curves. We then use these obstructions to define and obtain some computations of wrapped Fukaya categories of particular global quotient orbifolds.

In the non-orbifold Calabi-Yau setting, recent advances of Doan-Walpuski~\cite{doan2023} and Ekholm-Shende \cite{ekholmshende2025ghost} enable us to define reduced (open) Gromov-Witten invariants of arbitrary genus. In both \cite{doan2023} and \cite{ekholmshende2025ghost}, the authors obtain obstruction results to smoothings of closed curves with a ghost bubble of positive genus attached at a single point. These results build on a long line of work \cite{zinger2009sharp, zinger2009reduced, niu2016thesis} under certain assumptions on the genus, starting with the work of Ionel \cite{ionel1998genus1}. On the algebro-geometric side, these ideas were developed in low genus settings in \cite{vakil2000enumerative, vakil2001tool}. The algebro-geometric analogs of the obstruction results of \cite{doan2023, ekholmshende2025ghost} were sharpened in the paper of Rezaee-Swaminathan \cite{rezaeeswaminathan2024obstruction} to the case of ghosts attached at multiple points, and \cite{rezaeeswaminathan2025construction} provided a criterion for smoothing nodal curves with ghost bubbles. (Their symplectic analogs are known only in the lower-genus cases from the works of \cite{zinger2009sharp, niu2016thesis}.)  We also note that recently Xu \cite{xu2026reduced} obtained a definition of reduced Gromov-Witten invariants using \emph{soft} topological methods, avoiding the need to establish obstructions to smoothings.

Recall that an \emph{orbicurve}
$$[u] \colon (\mathscr{C},q_1, \ldots, q_b) \to [M/G]$$
is a map from a \emph{stacky curve} $\mathscr{C}$, possibly with boundary,\footnote{We will often refer to both $[u]$ and the domain $\mathscr{C}$ as ``orbicurves'', following common usage in the non-orbifold case.} whose smooth stacky points $q_i$ are required to be mapped to the singular locus of $[M/G]$. A sequence of smooth orbicurves may converge to a nodal orbicurve with an \emph{orbighost}, corresponding to a tuple of stacky points colliding on the domain. Ruling out such orbighosts is highly nontrivial even in the case of genus $0$ orbicurves.

In this paper, we initiate the study of obstructions to smoothings of orbicurves and obtain criteria for ruling out orbighosts landing on \emph{cyclic quotient singularities}; see Section~\ref{section: obstruction-to-smoothings} for more precise statements.

We then apply the smoothing obstructions developed in this paper to one important class of examples:
Let $X$ be a complex manifold, equipped with a finite subgroup $G$ of automorphisms. There is a naturally induced action of $G$ on $M=T^*X$, and our aim is to define and obtain computations of the \emph{regular wrapped Fukaya category of $[T^*X/G]$ bulk-deformed by} {\em inertia components} of the inertia orbifold $\mathcal{I}[T^*X/G]$, denoted $\mathcal{W}^{\operatorname{reg}}_{\bm\hbar}([T^*X/G])$, where $\bm \hbar$ is the associated tuple of bulk parameters. The inertia components of $[T^*X/G]$ are naturally in bijection with the inertia components of $[X/G]$. A connected (complex) hypersurface $Y \subset X$ fixed by the $G$-action corresponds to an embedded copy of $T^*Y \subset T^*X$, which is fixed by the induced action. This component is an example of a \emph{(compound) $A_{m-1}$-singularity}, i.e., locally $[T^*X/G]$ near a generic point of $[G \cdot T^*Y/G]$ is given by
$$[\C^2 /\mu_m] \times \C^{n-2},$$ 
with the cyclic group $\mu_m$ of order $m$ acting on $\C^2$ with opposite characters $\zeta_m^{\pm 1}=e^{\pm 2\pi i/m}$. Using our obstruction results 
we define $\mathcal{W}^{\operatorname{reg}}_{\bm\hbar}([T^*X/G])$ bulk-deformed by these $A_{m-1}$-singularities; see Theorem~\ref{theorem: well-defined-ainfty-cat}, which proves a proposal from \cite[Section 3.2]{mak2025orbifoldhamiltonianfloertheory} in this particular case.  

More significantly, we obtain a partial computation of the $A_\infty$-algebra of endomorphisms of a cotangent fiber $T^*_{[x]}[X/G]$ over a point $[x]$ associated with a regular point $x \in X^{\operatorname{reg}}$. This computation is an extension of the arguments of Abouzaid \cite{abouzaid2012wrapped} and Abbondandolo-Schwarz \cite{abbondandolo2006floer, abbondandolo2010floer} to the orbifold setting. It involves counting orbidisks with boundary on the $0$-section $[X/G] \subset [T^*X/G]$, which is a singular Lagrangian (by this we mean a Lagrangian with a nontrivial singular locus). To our knowledge, this is among the few available computations in orbifold Lagrangian Floer theory, involving singular Lagrangians. We refer the reader to \cite{COW2024twisted} for a proposal of a framework of such computations for arbitrary orbifolds.

Specifically, we show that there is a surjective algebra morphism 
\begin{equation}
    \mathcal{EV} \colon HW^0_{\bm \hbar}(T^*_{[x]}[X/G]) \to \mathcal{H}_{\bm \hbar}([X/G])
\end{equation}
from the degree-zero cohomology algebra of the $A_\infty$-algebra $CW^*_{\bm \hbar}(T^*_{[x]}[X/G])$ to Etingof's {\em orbifold Hecke algebra} $\mathcal{H}_{\bm \hbar}([X/G])$ \cite{EtingofPavel} and show that $\mathcal{EV}$ is an isomorphism when $\pi_2(X) \otimes \mathbb{Q}=0$. See Section~\ref{section: fukaya and Hecke intro} below for the precise statement and discussion of the result.

\subsection{Obstructions to smoothing nodal orbicurves}\label{section: obstruction-to-smoothings}

Let $(M, \omega)$ be a symplectic manifold equipped with an $\omega$-tame almost complex structure $J_0$ and an effective action of a finite group $G$ of symplectomorphisms preserving $J_0$. Consider a \emph{cyclic quotient singularity} $[G \cdot \mathcal{Y}/G] \subset [M/G]$ arising from a connected component $\mathcal{Y}$ of the fixed point set $M^{H}$, where $H\subset G$ is a cyclic subgroup isomorphic to the group $\mu_m=\{\zeta_m^j \mid j=0, \dots, m-1\}$ of $m$th roots of unity generated by $\zeta_m=e^{2\pi i/m}$, for some $m \in \Z_{\ge 2}$. The component $\mathcal{Y}$ is smooth and the normal bundle $\nu_{\mathcal{Y}}$ admits an \emph{isotypic decomposition}
\begin{equation}
    \nu_{\mathcal{Y}}= \oplus_{\chi \in \operatorname{Irr}_{\C}(H)} \nu_{\chi},
\end{equation}
where the action of $H$ on the fibers of the $J_0$-complex subbundle $\nu_\chi$ corresponds to a direct sum of copies of the nontrivial irreducible $\C$-representations of $H$ corresponding to the character $\chi$. 

We will assume that $J_0$ is \emph{adapted} to the cyclic quotient singularity  $[G \cdot \mathcal{Y}/G]$ in the following sense: Let $\mathcal{Y}^{\circ} \subset \mathcal{Y}$ denote the immersed submanifold of points $y \in \mathcal{Y}$ with the minimal possible stabilizer, i.e., with $G_y=H$.
Then:
\be
\item an auxiliary choice of an open $G$-invariant neighborhood $U$ of $G \cdot \mathcal{Y}^{\circ}$, together with a choice of almost complex $G$-invariant immersed \emph{isotypic submanifolds} $V_\chi \subset U$ for $\chi \in \operatorname{Irr}_{\C}(H)$ containing $\mathcal{Y}^{\circ}$, such that
$$TV_{\chi}|_{\mathcal{Y}^{\circ}}=T\mathcal{Y}^{\circ} \oplus \nu_\chi;$$
\item a choice of $G$-equivariant $J_0$-holomorphic projections
$$\pi_\chi \colon U \to V_{\chi}, \quad \pi_\chi|_{V_\chi}=\operatorname{id}.$$
\ee
In Section \ref{neighborhoods-of-fixed-sets} we show that such adapted almost complex structures $J_0$ exist in abundance.

A pseudoholomorphic orbicurve $[u] \colon (\mathscr{C},q_1, \ldots, q_b) \to [M/G]$, given by a map
from a (potentially nodal) twisted orbidisk $\mathscr{C}$ which is \emph{injective on isotropy groups}, can be represented by a $G$-equivariant $J_0$-holomorphic curve
$$u \colon \Sigma \to M,$$
with $[\Sigma/G] =\mathscr{C}$ (see Lemma~\ref{lemma: orbifold=equivariant}), and moreover the projection $\pi \colon \Sigma \to C$ has $(q_1, \ldots, q_b)$ as its branching locus, where $C$ is the coarse curve of $\mathscr{C}$. Specifying the inertia components to which these stacky points are mapped, together with generators of the local isotropy, specifies the \emph{Hurwitz datum} $\bm h=((h_1), \dots, (h_b))$ of the cover $\pi \colon \Sigma \to C$, where $(h_i) \in \operatorname{Conj}(G)$ is the conjugacy class of an element of $G$ fixing a ramification point in $\pi^{-1}(q_i)$. Hence, domains of the orbicurves can be viewed as elements of the Hurwitz moduli space $\overline{\mathscr{H}}_{0, \bm h, \emptyset}^G$ of admissible $G$-covers over orbidisks with prescribed Hurwitz datum $\bm h$. In the case of admissible $G$-covers over closed genus $0$ orbicurves, these spaces are known to be orbifolds by \cite{abramovichvistoli2002, ACV2003, olsson2007log}, and in the case of orbidisks one can deduce that there is a structure of orbifolds with corners on these spaces using doubling.

Given a stacky point $q_i$ and $(h_i)$ with $h_i \in H$, in Section~\ref{section: transversality-for-jets} we show that the Taylor expansion of $\pi_{\chi} \circ u$ at a point $p \in \pi^{-1}(q_i)$ fixed by $h_i$ has an \emph{expected leading term} of order $k(\chi, q_i, (h_i))$, i.e., the $J_0$-holomorphic jet $j^{k(\chi, q_i, (h_i))-1}_p(\pi_\chi \circ u)$ necessarily vanishes, while $j^{k(\chi, q_i, (h_i))}_p(\pi_\chi \circ u)$ is generically nonzero. This number $k(\chi, q_i, (h_i))$ is given by an element of $\{1, \dots, d_i\}$, where $d_i=\operatorname{ord}(h_i)$, such that 
\begin{equation*}
    \chi(h_i)=\zeta_{d_i}^{k(\chi, q_i, (h_i))}. 
\end{equation*}
Since this number depends on $\chi$, the desire to pick the auxiliary submanifolds $V_\chi$ should appear natural.

Now consider a sequence of smooth domains $(\pi_k \colon \Sigma_k \to \mathscr{C}_k) \in \mathscr{H}_{0, \bm h, \emptyset}^G$ converging to a twisted nodal orbicurve 
$$(\pi_\infty \colon \Sigma_\infty \to \mathscr{C}_\infty= \mathscr{C}_{\bullet} \cup_{\ddot{n}}  \mathscr{C}_{\mathghost}) \in \overline{\mathscr{H}}_{0, \bm h, \emptyset}^G,$$ 
where $\mathscr{C}_{\mathghost}$ is a connected orbidisk or a genus $0$ orbicurve attached to the main component along a unique (potentially stacky) node $\ddot{n}$, and there is a decomposition $\Sigma_\infty=\Sigma_{\bullet} \cup_{\pi^{-1}(\ddot{n})} \Sigma_{\mathghost}$
into the preimages of $\mathscr{C}_\bullet$ and $\mathscr{C}_{\mathghost}$ under $\pi_\infty$. We endow curves in this sequence with $G$-equivariant domain-dependent almost complex structures $J_k$ on $M$ converging in the $C^\infty$-topology to $J_\infty$, which are assumed to coincide with $J_0$ near the preimages of the stacky points and on $\Sigma_{\mathghost}$.

Furthermore, we assume that there is a sequence of orbicurves $[u_k] \colon \mathscr{C}_k \to [M/G]$, represented by $G$-equivariant $J_k$-holomorphic maps $u_k \colon \Sigma_k \to M$, 
which converges in the Gromov topology to an orbicurve
$$[u_\infty] \colon \mathscr{C}_\infty=\mathscr{C}_{\bullet} \cup_{\ddot{n}}  \mathscr{C}_{\mathghost} \to [M/G]$$
with $[u_\infty]|_{\mathscr{C}_{\mathghost}}$ mapping $\mathscr{C}_{\mathghost}$ to a point $[y] \in [G \cdot \mathcal{Y}^{\circ}/G]$ with $y \in \mathcal{Y}^{\circ}$, and this map is represented by
\begin{equation}\label{eq: equivariant-limit-cruve}
    u_\infty=u_\bullet \cup u_{\mathghost} \colon \Sigma_{\bullet} \cup \Sigma_{\mathghost} \to M.
\end{equation}
We call $[u_{\mathghost}]$ an \emph{interior ghost} if $\mathscr{C}_{\mathghost}$  is a closed genus $0$ orbicurve, and a \emph{boundary ghost} if $\mathscr{C}_{\mathghost}$ is an orbidisk. 
We denote by $\Sigma_{\mathcal{Y}} \subset \Sigma_{\mathghost}$ the union of components mapped to $\mathcal{Y}^{\circ}$ under $u_{\mathghost}$. The space $H^{1,0}(\Sigma_{\mathcal{Y}})$ of holomorphic $1$-forms on $\Sigma_{\mathcal{Y}}$ naturally admits an $H$-action, and hence admits a character decomposition 
$$H^{1,0}(\Sigma_{\mathcal{Y}})=\textstyle\oplus_{\chi \in \operatorname{Irr}_\C(H)}H^{1,0}(\Sigma_{\mathcal{Y}})_\chi.$$

We also require $\partial\Sigma_k$ to be mapped to a $G$-orbit $\mathcal{L}=G \cdot L$ of an embedded Lagrangian $L \subset M$ which is preserved by the $H$-action. 

The following theorem is our main result on obstructing smoothings of orbicurves and is a combination of Theorems~\ref{ref: theorem-isotypic-direction-jet-vanishing} and \ref{ref: theorem-real-derivatives-vanishing}.

\begin{theorem}\label{theorem: main-obstruction-result}
    Let $[u_k]$ be a sequence of orbicurves represented by $J_k$-holomorphic equivariant curves $u_k \colon \Sigma_k \to M$, with domain-dependent almost complex structures $J_k$ as above, which converges to a $J_\infty$-holomorphic curve as in~\eqref{eq: equivariant-limit-cruve}. Then one of the following holds:
    \begin{enumerate}
        \item[(1)] The curve $[u_{\mathghost}]$ is an interior ghost and $\ddot{n}$ is a regular node.
        \item[(2)] The curve $[u_{\mathghost}]$ is an interior ghost and $\ddot{n}$ is a stacky node. If there exists a $1$-form $\eta \in H^{1,0}(\Sigma_{\mathcal{Y}})_\chi$ with a zero of order exactly $k(\chi, \ddot{n}, (h_{\ddot{n}}))-1$ at any node $p \in \pi^{-1}(\ddot{n})$ of $\Sigma_{\mathcal{Y}}$, then the $J_0$-holomorphic jet
        \begin{equation}
            j^{k(\chi, q_i, (h_i))}_p(\pi_\chi \circ u_\bullet)
        \end{equation}
        vanishes at the node $p \in \pi^{-1}(\ddot{n})$.
        \item[(3)] The curve $[u_{\mathghost}]$ is a boundary ghost, in which case $\ddot{n}$ is automatically regular. If there exists a nonzero $1$-form
        $$\overline{\eta}\otimes v \in( H^{0,1}_{\R}(\Sigma_{\mathghost})\otimes T_y L)^{H},$$
        such that $\eta(p)\neq 0$, then 
    \begin{equation}
        \langle du_\infty(p), v \rangle =0
    \end{equation}
    for each node $p$ of $\Sigma_{m_{\mathcal{Y}}}$ in $\pi^{-1}(\ddot{n})$ and any $v \in T_yL$ as above.
    \end{enumerate}
\end{theorem}

By Theorem~\ref{theorem: main-obstruction-result}, once the obstruction space of $1$-forms on the ghost curve $[u_{\mathghost}]$ has been determined to be nonzero, the existence of a smoothing imposes additional constraints on the Taylor series of the main curve $u_\bullet$ at the places where the nodes are attached. The computation of the obstruction space for an arbitrary cyclic quotient singularity reduces to an index calculation, which is somewhat tedious and depends on the combinatorics of the stacky points on the ghost curve $\mathscr{C}_{\mathghost}$. 

In this paper, we are mainly interested in the case of (compound) $A_{m-1}$-singularities, and for those, Theorem~\ref{theorem: main-obstruction-result} admits the following strengthening, using explicit computations from Lemmas~\ref{lemma: 1-forms-on-cyclic-covers} and \ref{lemma: index-computation-for-An-singularity}.

Notice that with respect to some trivialization of $T_yM=\C^2_{A_{m-1}} \oplus \C^{n-2}_{\operatorname{triv}}$ as an $H$-representation, the tangent space $T_yL$ is given by the real $H$-subrepresentation
\begin{equation}\label{eq: local-form-for-L}
    \R^2_{A_{m-1}} \oplus \R^{n-2}_{\operatorname{triv}}.
\end{equation}

\begin{theorem}\label{theorem: obstruction-for-Am-main}
    In the situation of Theorem~\ref{theorem: main-obstruction-result}, assuming that $[G \cdot \mathcal{Y}/G] \subset [M/G]$ is a (compound) $A_{m-1}$-singularity and the tangent space $T_yL$ to the Lagrangian $L$ is given by~\eqref{eq: local-form-for-L}, one of the following holds:
    \begin{enumerate}
    \item[(1)] The curve $[u_{\mathghost}]$ is an interior ghost and $\ddot{n}$ is a regular node.
    \item[(2)]  The curve $[u_{\mathghost}]$ is an interior ghost and $\ddot{n}$ is a stacky node. The $J_0$-holomorphic jet
        \begin{equation}
            j^{k(\chi, q_i, (h_i))}_p(\pi_\chi \circ u_\bullet)
        \end{equation}
    vanishes at any node $p \in \pi^{-1}(\ddot{n})$ for at least \emph{one of the two} characters $\chi$ associated with this singularity.
    \item[(3)] The curve $[u_{\mathghost}]$ is a boundary ghost, in which case $\ddot{n}$ is automatically regular. If $\mathscr{C}_{\mathghost}$ has at least $2$ stacky points, then
\begin{equation}
    \langle du_\infty(p), v \rangle =0
\end{equation}
for any node $p$ of $\Sigma_{m_{\mathcal{Y}}}$ in $\pi^{-1}(\ddot{n})$ and any $v \in T_yL \cap (\nu_{\mathcal{Y}})_y$.
\end{enumerate}
\end{theorem}

\begin{remark}
    We make no assumptions on the genus of orbicurves in the proof of Theorem~\ref{theorem: main-obstruction-result}. On the other hand, we use assumptions on the genus of the orbighost in our proof of Theorem~\ref{theorem: obstruction-for-Am-main}.
\end{remark}

\begin{remark}
    The results are not immediate from the obstruction results of \cite{doan2023} and \cite{ekholmshende2025ghost}, even in the simplest case $H \cong \mu_2$. However, in the case of $\mu_2$-stabilizers, one does not need to require $J$ to be adapted to the singularity, since the expected leading term in the normal direction is linear. Our proof is an upgraded version of the proof \cite{doan2023} to the equivariant setup. It would be interesting to see whether the techniques of \cite{ekholmshende2025ghost} are sufficient to arrive at a similar conclusion.
\end{remark}

\begin{remark}
    We comment here on the applicability of Theorem~\ref{theorem: main-obstruction-result} to defining reduced invariants of general cyclic quotient singularities. A particularly advantageous property of $A_{m-1}$-singularities is that the index of $[u_\bullet]$ coincides with the index of any curve in the sequence $[u_k]$, or in other words, $[u_{\mathghost}]$ gives a dimensionless contribution. For general cyclic quotient singularities, though, the index of $u_{\mathghost}$ may be negative, and we need to introduce virtual perturbations. 
    If both dimensionless and ghosts of negative dimensions are expected to appear in the Gromov compactification, one would have to choose virtual perturbations carefully in order to apply our Theorem~\ref{theorem: main-obstruction-result}. This is a promising future venue to pursue, and the only work that explores this kind of idea that we are aware of is \cite{ekholmshende2024counting}.
\end{remark}

\subsection{Wrapped Fukaya category of global quotient orbifolds and Hecke algebras}\label{section: fukaya and Hecke intro} 

Let $X$ be a complex manifold and $G \subset \operatorname{Aut}(X)$ be a finite subgroup. We denote by $X^{\operatorname{reg}} \subset X$ the \emph{regular} locus of points (i.e., those with trivial stabilizers), and by $X^{\operatorname{sing}}$ its complement. 
Let $Y \subset X^{\operatorname{sing}}$ be a connected (complex) hypersurface with stabilizer $H_Y\subset G$, which is necessarily isomorphic to $\mu_{m_Y}$ for some $m_Y>1$. To each $G$-orbit of $Y$, we associate formal variables $\hbar_{0}^Y, \dots, \hbar_{m_Y-1}^Y$. Let $C_Y$ be the conjugacy class of a small counterclockwise loop around $[G \cdot Y/G]$ in $\pi_1(X^{\operatorname{reg}}/G)$.

The \emph{orbifold Hecke algebra} $\mathcal{H}_{\bm \hbar}([X/G])$ is the quotient of the $\bm \hbar$-adic completion $$\C[\pi_1(X^{\operatorname{reg}}/G)]\llbracket \bm \hbar \rrbracket$$ of the group algebra of the \emph{braid group} $\pi_1(X^{\operatorname{reg}}/G)$ by the {\em Hecke relations}
\begin{equation}\label{eq: hbar-Hecke-relation-intro}
    \gamma^{m_Y}-\hbar_{m_Y-1}^{Y}\gamma^{m_Y-1}-\dots-\hbar_{1}^{Y}\gamma-(\hbar_0^Y+1)=0,
\end{equation}
for any connected hypersurface $Y$ in $X^{\operatorname{sing}}$ and any $\gamma \in C_Y$. We denote by $\mathcal{H}_{\bm \hbar}|_{\bm \hbar_0=0}([X/G])$ the specialization of the above algebra obtained by setting $\hbar_0^Y=0$ for all fixed hypersurfaces $Y$.

After a change of variables, the above algebra is isomorphic to the more familiar form of the orbifold Hecke algebra defined by Etingof \cite{EtingofPavel}, with variables $\tau_0^Y, \dots, \tau_{m_Y-1}^Y$ and Hecke relations
\begin{equation}\label{eq: Hecke-tau-relation-intro}
        \prod_{j=0}^{m_Y-1}(\gamma-\zeta_{m_Y}^je^{\tau_j^Y})=0.
\end{equation}
The insight of generalizing the quadratic Hecke relation to the higher-degree relation of the form~\eqref{eq: Hecke-tau-relation-intro} is due to Brou\'{e}-Malle-Rouquier \cite{BMR}, where they used it in the case of defining representations of \emph{complex reflection groups}. Many examples of Hecke algebras of orbifolds, arising from actions with non-real reflections, are generalizations of the Hecke algebras of complex reflection groups. 

Etingof's definition of the orbifold Hecke algebras can be viewed as a far-reaching generalization of Cherednik's original topological interpretation of the \emph{double affine Hecke algebra} of type $A_1$ \cite{cherednik2005}. It provides a unified way of defining Hecke algebras which appear in representation theory. 
Among the examples of orbifold Hecke algebras are \emph{affine Hecke algebras} and \emph{double affine Hecke algebras} of connected simple Lie groups, Hecke algebras of complex reflection groups of Brou\'e-Malle-Rouquier \cite{BMR}, and also generalized double affine Hecke algebras considered in \cite{ego2006GDAHA} and \cite{eor2007GDAHA}. We review these and other examples in Section~\ref{section: examples}, and we also refer the reader to \cite[Section 3.7]{EtingofPavel}.

\s
In this paper, we define the \emph{regular wrapped Fukaya category} $\mathcal{W}_{\bm \hbar}^{\operatorname{reg}}(T^*[X/G])$ of the cotangent bundle orbifold $[T^*X/G]$ with the induced action of $G$, bulk-deformed by components of the inertia orbifold $\mathcal{I}[T^*X/G]$ associated with $A_{m_Y-1}$-singularities $T^*Y \subset T^*X$, where $Y$ is a connected hypersurface. An object $\mathcal{L}$ of this category is a disjoint $G$-orbit of an embedded, cylindrical-at-infinity Lagrangian $L \subset T^*X$, which can also be treated as a Lagrangian in $[T^*X/G]$, avoiding the singular locus. We denote by $T^*_{[x]}[X/G]$ the $G$-orbit of a fiber $T^*_xX$ at $x \in X^{\operatorname{reg}}$. We denote by $HW^*_{\bm \hbar}(T^*_{[x]}[X/G])$ the graded algebra associated with the $A_\infty$-algebra of endomorphisms of $T^*_{[x]}[X/G]$ in  $\mathcal{W}_{\bm \hbar}^{\operatorname{reg}}(T^*[X/G])$.
We show:

\begin{theorem}[= Theorem~\ref{theorem: Hecke-isomorphism} and Corollary~\ref{corollary: derived-Nakayama}]\label{theorem: comparison result}
    There exists a surjective morphism 
    $$\mathcal{EV} \colon HW^0_{\bm \hbar}(T^*_{[x]}[X/G]) \to \mathcal{H}_{\bm \hbar}|_{\bm \hbar_0=0}([X/G]).$$
     If $\pi_2(X) \otimes \mathbb{Q}=0$, the morphism $\mathcal{EV}$ is an isomorphism. In particular, if $X$ is a $K(\pi,1)$, the graded algebra $HW^*_{\bm \hbar}(T^*_{[x]}[X/G])$ is concentrated in degree $0$, and therefore the orbifold Hecke algebra gives the full computation of $HW^*_{\bm \hbar}(T^*_{[x]}[X/G])$.
\end{theorem}

    The proof of Theorem~\ref{theorem: comparison result} is based on an analysis of the boundary degenerations in the moduli spaces of stacky half-strips, with the requirement that the outgoing end of such a half-strip is mapped to the $0$-section (here we are using Abouzaid's terminology \cite{abouzaid2011cotangent}). Using Theorem~\ref{theorem: obstruction-for-Am-main} we rule out all boundary degenerations except for the one corresponding to the formation of a nodal orbidisk with a unique stacky point. These degenerations give rise to the Hecke relation~\eqref{eq: hbar-Hecke-relation-intro}.

    Let us conclude the introduction with some remarks on Theorem~\ref{theorem: comparison result}:

\begin{remark}
    Theorem~\ref{theorem: comparison result} significantly extends the scope of 
    the main theorem of \cite{HTY2026}, which also follows from the results of \cite{hkty2025morse}. In \cite{CHT, hkty2025morse}, the authors defined an analog of our bulk-deformed orbifold wrapped Fukaya category in the special case of symmetric products. The approach used --- which goes back to the work of Lipshitz \cite{lipshitz2006cylindrical} --- is the \emph{cylindrical reformulation} of the orbifold wrapped Fukaya category for symmetric products.
    This is based on the classical observation \cite[Proposition 4.2.2, Theorem 4.3.2]{ACV2003} that the stack of twisted $S_k$-covers is equivalent to the \emph{properly defined} stack of admissible $k$-fold branched covers. The same idea is used in \cite{KY2023} to define a version of symplectic cohomology for symmetric products.
\end{remark}

\begin{remark}
    The classical references for orbifold Gromov-Witten invariants and orbifold Floer theory in the symplectic literature include \cite{Chen2002, Cho2014}. Recently, several groups of authors have developed the foundational aspects further, among which we note \cite{gironellazhou2021exact, mak2025orbifoldhamiltonianfloertheory}. There is also ongoing work by McLean and Ritter aimed at developing a framework for orbifold Gromov-Witten theory based on global Kuranishi charts.
    There have been several recent developments and applications in orbifold Gromov-Witten and Floer theories, which include \cite{LCHL2024, polterovichshelukhin2023, FLYZ, maksmith2021}.
\end{remark}

\begin{remark}
    It would be interesting to see whether Theorem~\ref{theorem: comparison result} can be proved using virtual methods.  It is asserted in \cite{mss2025orbifoldspectral} that one can define the regular (wrapped) Fukaya category of orbifolds by combining their ideas together with \cite{hirschihugtenburg2025} and \cite{rabah2024}. However, performing the computation provided by Theorem~\ref{theorem: Hecke-isomorphism}  requires working with a $0$-section, which is a singular Lagrangian, and the methods of these papers alone do not suffice. Recently, Xu \cite{xu2026reduced} obtained a definition of reduced closed Gromov-Witten invariants by using a new type of stratified transversality result, which is similar in spirit to FOP transversality developed in \cite{baixu2022transversality}. This at least suggests that one may attempt to reprove Theorem~\ref{theorem: comparison result} using virtual methods combined with stratified transversality methods, which may be helpful for generalizations of such computations to other quotient singularities.
\end{remark}

\begin{remark}
    Theorem~\ref{theorem: comparison result} only gives a computation in degree $0$. The work \cite{hkty2025morse} suggests how one could build a Morse chain model computing such an $A_\infty$-algebra by using Morse theory on the space of $G$-paths. It is still unclear whether such a Morse-theoretic approach provides a way to compute this $A_\infty$-algebra in many cases of interest. Therefore, a robust definition of such a derived orbifold Hecke algebra is warranted, and one may anticipate methods of string topology to provide an efficient way of computing such an algebra. 
\end{remark}

\s\n 

{\em Organization of the paper.} In Section~\ref{section: review of orbifolds} we introduce some standard notions from the general theory of orbifolds and orbifold Floer theory, and in Section~\ref{section: orbifold Floer theory} we define the wrapped Fukaya category of $T^*[X/G]$. In Section~\ref{section: transversality-for-jets} we introduce the notion of equivariant $J$-holomorphic jets, and establish the necessary transversality results for maps evaluating jets of equivariant curves at the preimages of the stacky points. Theorem~\ref{theorem: main-obstruction-result} is proved in Section~\ref{section: obstructions to smoothings} and Theorem~\ref{theorem: comparison result} is proved in Section~\ref{section: A_infty and orbifold Hecke}. In Section~\ref{section: examples}, we provide several examples for which the isomorphism of Theorem~\ref{theorem: Hecke-isomorphism} recovers the orbifold Hecke algebras in terms of the wrapped Fukaya category.

\s\n
{\em Acknowledgments.}  We thank Pavel Etingof and Rapha\"{e}l Rouquier for enlightening conversations about Hecke algebras and helpful comments on an earlier draft. RK is grateful to Mohammed Abouzaid, Shaoyun Bai, Tobias Ekholm,  Cheuk Yu Mak, Ivan Smith, Sobhan Seyfaddini, Semon Rezchikov, and Mohan Swaminathan for useful conversations about orbifold Floer theory and gluing at the early stages of this project, and to the organizers of the Symplectothon 2024 workshop, which served as a great introduction to orbifolds and orbifold Floer theory. RK is especially grateful to John Pardon for suggesting the idea of localization tricks that we implement in Section~\ref{section: localization tricks}. RK is indebted to Gleb Terentiuk and Burt Totaro for patiently explaining various foundational aspects of stacks, and also thanks Ilya Dumanski, Eugene Gorsky, and Monica Vazirani for helpful conversations about representation theory.  TY thanks Cheuk Yu Mak, Kaoru Ono, and Zhengyi Zhou for helpful conversations.

\s\n
{\em AI Disclosure.} Authors used LLMs during preparation of this version of the manuscript for proofreading, literature search, and editing.

\section{Basic notions for global quotient orbifolds} \label{section: review of orbifolds}

In this section, we give a brief review of some topological and geometric aspects of orbifolds, focusing mainly on global quotients. We expect the reader's familiarity with the general theory of orbifolds and orbifolds with corners, and we usually understand an orbifold $\mathcal{X}$ as a stack over the category of smooth manifolds (with corners) which admits an atlas $U \to \mathcal{X}$, with the associated groupoid $U \times_{\mathcal{X}}U \rightrightarrows U$ being a proper \'{e}tale Lie groupoid; see \cite[Remark 4.33]{Lerman2010stacks}. For more details on the general theory of orbifolds, differentiable stacks and orbispaces, we refer the reader to \cite{Lerman2010stacks,  Adem-Leida-Ruan, moerdijkmrcun2003groupoids, solomontukachinsky2024forms, zernik2017moduli, metzler2003smoothstacks, behrendxu2011, joyce2017kuranishi1, henriques2005thesis, pardon2022enough}.  We note, however, that for global quotient orbifolds, most of the paper can be understood from the simpler equivariant perspective.

\subsection{Some definitions and examples}

We use the notation $[X/G]$ for the \emph{global quotient orbifold} associated to an $n$-dimensional manifold $X$ with an action of a finite group $G$. Given $x \in X$, we denote its stabilizer by $G_x \subset G$. We also use the notation $X^g$ to denote the fixed point set of an element $g \in G$. We set
\begin{equation}
    X^{\operatorname{sing}}=\bigcup_{g \in G\setminus \{e\}}X^g, \quad X^{\operatorname{reg}}=X\setminus X^{\operatorname{sing}}.
\end{equation}

\begin{remark}
    We notice that $[X^{\operatorname{sing}}/G] \subset [X/G]$ is the suborbifold, consisting of all points with nontrivial isotropy. As the notation suggests, we treat these points as \emph{singular}. Nevertheless, these points are \emph{smooth}, since they admit orbifold charts in $[X/G]$, and we also refer to them as \emph{smooth stacky} points.
\end{remark}
To any orbifold, Chen-Ruan \cite{Chen2004} assign its \emph{inertia orbifold}. In the case of a global quotient orbifold $[X/G]$, it has a simple description
\begin{equation}\label{eq: inertia-for-global-quotient}
    \mathcal{I}[X/G]=\big[\left(\sqcup_{g\in G} X^g\right)/G\big]\cong\sqcup_{(g) \in \operatorname{Conj}(G)}[X^g/C(g)],
\end{equation}
where $G$ acts on the disjoint union by $h(x,g)=(hx, hgh^{-1})$, $\operatorname{Conj}(G)$ is the set of conjugacy classes of $G$, and $C(g)$ is the centralizer of $g \in G$; see \cite[Example 3.1.3]{Chen2004} for more details. We denote the conjugacy class of an element $g\in G$ by $(g)\in \operatorname{Conj}(G)$. 

    Connected components of the inertia orbifold will be called {\em inertia components}. 
    Let $Z \subset X^g$ be a connected component of $X^g$. We assign to it an inertia component
    \begin{equation}
        (Z, (g)) \coloneqq [Z/C(g)] \subseteq [X^g/C(g)] \subseteq \mathcal{I}[X/G].
    \end{equation}

\begin{remark}
    In the orbifold literature, connected components of the inertia orbifold are usually called \emph{twisted sectors}. On the other hand, in the symplectic literature, there are \emph{Liouville sectors}, which play an important role. To avoid a potential clutter of terminology, the connected components of $\mathcal{I}[X/G]$ will be called \emph{inertia components}.\footnote{In this paper, we mostly work with cotangent bundles of closed global quotient orbifolds, but these issues would occur if we were also to work with cotangent bundles of global quotient orbifolds with boundary.}
\end{remark}

We denote by $\mu_m$ the cyclic group of roots of unity of order $m \in \Z_{\ge 1}$, with preferred generator $\zeta_m=e^{2\pi i/m}$.
The complex irreducible representations $\operatorname{Irr}(\mu_m)$ of $\mu_m$ are $1$-dimensional and are in bijection with the characters of $\mu_m$; they will be denoted $\chi^m_j=\zeta_m^j$, 
for $j\in \{0,1,\dots, m-1\}$.

Let $\operatorname{Irr}^{\R}(\mu_m)$ be the set of \emph{real irreducible representations} of $\mu_m$.  This set consists of the trivial $1$-dimensional representation, the $1$-dimensional sign representation for even $m$, and approximately $ [\frac{m}{2}]$ real $2$-dimensional representations (corresponding to the standard complex $1$-dimensional representations).

\s
Now we recall the definition of the \emph{age grading}; cf.\ \cite[Section 4.2]{Adem-Leida-Ruan}. Let $(M,J)$ be an almost complex manifold of real dimension $2n$ with a $G$-action on $M$ which preserves $J$.

\begin{definition}[Age grading]
    Let $(Z, (g))$ be an inertia component for $[M/G]$ with $\operatorname{ord}(g)=m$. Its \emph{age grading} is given by:
    \begin{equation}
        \iota(Z, (g))\coloneqq\textstyle\sum_{j=1}^n \tfrac{m_j}{m},
    \end{equation}
    where the action of $g$ on $T_xM$ for any $x\in Z$ diagonalizes in some $\C$-basis (with respect to $J$) as:
    \begin{equation}
        g_*(x)=\operatorname{diag}(\zeta_m^{m_1}, \dots, \zeta_m^{m_n}),
    \end{equation}
    with $0 \le m_i<m$.
\end{definition}

For any connected submanifold $Z \subset X^g$, $TX|_Z$ admits an \emph{isotypic decomposition}
\begin{equation}
    TX|_Z= \bigoplus_{\rho \in \operatorname{Irr}^{\R}(\mu_m)} V_{\rho},
\end{equation}
where $\operatorname{ord}(g)=m$ and $V_\rho|_x$, $x\in Z$, is a direct sum of copies of $\rho$ as representations of $\langle g\rangle\cong \mu_m$. Given the choice of a $G$-invariant metric, the isotypic components are orthogonal with respect to such a metric by Schur's lemma.

\begin{example} \label{ex-rotation}
   \normalfont Let $X$ be an $n$-dimensional manifold with a $G$-action. Let $(Z, (g))$ be an inertia component of $\mathcal{I}[X/G]$ for which the summand $V_\sigma$ of the isotypic decomposition of $TX|_Z$ corresponding to the sign representation $\sigma$ has even rank (which is the case if $G$ preserves an orientation on $X$). In particular, $Z$ is a real codimension $2k$ submanifold for some $k>0$. Hence, at each $x \in Z$ there exists an $\R$-basis of $T_xX$ such that
    \begin{equation}
        g_*(x)=\operatorname{diag}(\widetilde{\zeta_m^{m_1}},\dots, \widetilde{\zeta_m^{m_k}}, 1, \dots, 1),
    \end{equation}
    where $\widetilde{\zeta_m^j}$ is the $2\times 2$ orthogonal matrix corresponding to a $\frac{2\pi j}{m}$-rotation for $j \in \{1, \dots, \lfloor\frac{m}{2}\rfloor\}$. 
    
    Next, let $M=T^*X$ be the cotangent bundle of $X$ with the standard symplectic form $\omega_{\operatorname{std}}$. We assume a choice of a $G$-invariant metric on $X$, which induces a $G$-invariant metric $g$ on $T^*X$.
    Let $J$ be an almost complex structure which is compatible with $\omega_{\operatorname{std}}$ and $g$. Then in the complexification of the $\R$-basis above, the action of $g$ on $T_{(x,0)}M$ diagonalizes as
    \begin{equation}
        g_*((x,0))=\operatorname{diag}(\zeta^{m_1}_m, \zeta_m^{m-m_1}, \dots, \zeta_m^{m_k}, \zeta_m^{m-m_k}, 1,\dots, 1).
    \end{equation}
    The age grading for $(T^*Z, (g)) \subset [M/G]$ is therefore given by
    \begin{equation}
        \iota(T^*Z, (g))=k.
    \end{equation}
\end{example}

\begin{example}[Hypersurface inertia and type $A$ inertia components] \label{exmpl: hypersurface-inertia}
    Let $X$ be a manifold with an effective action of a finite group $G$ by orientation-preserving diffeomorphisms.  Let $Y$ be a connected component of $X^g$ for some nontrivial $g \in G$ such that $Y$ is a real codimension $2$ submanifold of $X$.
    We call $(Y, (g))$ a \emph{hypersurface inertia component}. The stabilizer $H_Y \subset G$ is necessarily a cyclic subgroup, isomorphic to some $\mu_m$.
    
    Now consider 
    $$(N_Y, (g))\coloneqq (T^*Y, (g)) \subset \mathcal{I}[T^*X/G].$$
    We call $(N_Y, (g))$ an \emph{inertia component of type} $A$. Its normal bundle $\nu_{N_Y}$ admits an isotypic decomposition $V_{\chi^m_j} \oplus V_{\chi^m_{m-j}}$ for some $j \in \{1, \ldots, m-1\}$ with $\gcd (j,m)=1$, as a complex $H_Y$-bundle with respect to the standard almost complex structure on $T^*X$, induced by a $G$-invariant metric (note that the above depends on a particular choice of generator of $H_Y$).
    The computation from Example~\ref{ex-rotation} gives:
    \begin{equation}\label{eq: age-for-hypersurface}
            \iota(N_Y, g)=1.
    \end{equation}
\end{example}

\begin{remark}\label{remark: choice-of-chart-presentation}
    We notice that the choice of $j \in \{1, \ldots, m-1\}$ in Example~\ref{exmpl: hypersurface-inertia} is irrelevant, as any two such choices differ by an automorphism of the cyclic group generated by $g$, and this automorphism gives rise to an isomorphism of orbifold charts.
\end{remark}

\begin{remark}
    When $X$ is a complex manifold and $G$ acts by biholomorphisms, the fixed point sets of the smallest possible codimension are complex hypersurfaces, hence the name.
\end{remark}

\begin{example}[Cyclic quotient singularity]\label{exmpl: cyclic-quotient}
\normalfont Let $(M, \omega)$ be a symplectic manifold equipped with an action of a finite group $G \subset \operatorname{Symp}(M, \omega)$, and let $\mathcal{Y} \subset M^H$ be a connected component of the fixed point set of the action of a cyclic subgroup $H \cong \mu_m$ of $G$. We call the suborbifold $[G \cdot \mathcal{Y}/G]$ a \emph{cyclic quotient singularity}. 

Near any $y \in \mathcal{Y}$,  by the equivariant Darboux theorem, there exists an $H$-equivariant chart $U$ of $M$ symplectomorphic to an open ball in $V$, where $V$ is a \emph{complex} representation of $\mu_m$. In particular, we have
\begin{equation}
    (U, \omega) \cong (V_{\chi^m_0} \times \dots \times V_{\chi^m_{m-1}}, \omega_{\operatorname{std}}^0 \oplus \dots\oplus \omega_{\operatorname{std}}^{m-1}),
\end{equation}
where $V_{\chi^m_j}$ is the subrepresentation of $V$ corresponding to the character $\chi^m_j$. Observe that $\mathcal{Y} \cap U$ corresponds to $V_{\chi^m_0} \times \{0\}$. Moreover, near a generic point $y \in \mathcal{Y}$, the orbifold $[G \cdot \mathcal{Y}/G]$ is symplectomorphic to $[V/\mu_m]$.

A cyclic quotient singularity of particular interest in this paper is a \emph{compound $A_{m-1}$-singularity}, which corresponds to the case where $V_{\chi^m_j}$ and $V_{\chi^m_{m-j}}$ are $1$-dimensional, while $V_{\chi^m_{j'}}$ vanishes for all other $j\neq j' \in \{1, \dots, m-1\}$. In particular, the submanifolds $N_Y \subset T^*X$ as in Example~\ref{exmpl: hypersurface-inertia} provide $A_{m-1}$-singularities.
\end{example}

\subsection{Maps from orbicurves}

In this paper, we are concerned with smooth orbifold maps
\begin{equation}\label{eq: orbicurve-map}
    [u] \colon \mathscr{C} \to [M/G],
\end{equation}
where $[M/G]$ is an effective global quotient symplectic orbifold (i.e., $G$ is a subgroup of the symplectomorphism group $\operatorname{Symp}(M)$) and $\mathscr{C}$ is a (closed) nodal orbicurve or a nodal orbidisk. We refer the reader to \cite{taams2025} for a detailed introduction to orbicurves in the context of algebraic geometry and to  \cite{Cho2014} for a discussion of orbidisks from the perspective of smooth topology. All orbicurves and orbidisks under consideration are orientable, and hence the main difference between them and nodal curves/disks is the potential presence of \emph{stacky points}, i.e., points with nontrivial isotropy groups. 

The stacky points are either {\em interior stacky points} with a local chart of the form $[\C/ \mu_m]$ for some $m>0$, or \emph{balanced interior stacky nodes} with a local chart of the form
\begin{equation}
    [\{zw=0\}/\mu_m] \subset [\C^2/\mu_m],
\end{equation}
where the action of $\zeta_m \in \mu_m$ on $\C^2$ is of $A_{m-1}$-type and is given by
$$\zeta_m(z,w)=(\zeta_m z, \zeta_m^{-1}w).$$
The balancing condition, in fact, makes a nodal orbidisk into a \emph{twisted} nodal orbidisk, and we only consider twisted orbidisks.
Notice that there is no need to consider stacky points on the boundary of an orbidisk since we are only interested in orientable orbidisks. 

All smooth nodal maps~\eqref{eq: orbicurve-map} that we study in this paper are assumed to be \emph{injective on isotropy groups}, i.e., the associated maps on isotropy groups are assumed to be injective. In particular, any such $u$ maps stacky points of $\mathscr{C}$ to singular points of $[M/G]$.

Let us assume for a moment that a map $[u]$ comes from an honest functor between proper \'{e}tale Lie groupoids representing these orbifolds
$$U \colon [C_1 \rightrightarrows C_0] \to [M \times G \rightrightarrows M],$$
i.e., $U$ is a functor between these two categories. Then injectivity on isotropy groups for $U$ is equivalent to its \emph{faithfulness}. However, orbifold morphisms do not always come from morphisms of Lie groupoids on the nose; see \cite[Section 2.1]{gironellazhou2021exact}. 

As observed in \cite{mak2025orbifoldhamiltonianfloertheory}, we may interpret a smooth nodal curve that is injective on isotropy groups in a global quotient orbifold as an equivariant map:

\begin{lemma}\label{lemma: orbifold=equivariant}
    The data of a smooth nodal orbicurve $[u] \colon \mathscr{C} \to [M/G]$ that is injective on isotropy groups is equivalent to the data of the commutative diagram
    \begin{equation}
    \begin{tikzcd}
	\Sigma & M\\
	\mathscr{C} & {[M/G]}
	\arrow["u",from=1-1, to=1-2]
	\arrow["\tilde{\pi}", from=1-1, to=2-1]
	\arrow[from=1-2, to=2-2]
	\arrow["{[u]}",from=2-1, to=2-2]
    \end{tikzcd}
    \end{equation}
    where $\Sigma$ is a nodal Riemann surface with boundary equipped with a $G$-action, such that $[\Sigma/G]=\mathscr{C}$, and $\tilde{\pi}$ is the atlas for $\mathscr{C}$ associated with a $G$-cover $$\pi \colon \Sigma \to C,$$
    where $C$ is the coarse curve of $\mathscr{C}$, and the branched points of $\pi$ are exactly the stacky points of $\mathscr{C}$. The map $u$ is a $G$-equivariant smooth nodal map, and the $G$-action on $\Sigma$ is balanced at the nodes which project to the stacky nodes. 
\end{lemma}

\begin{proof}
    This is immediate from \cite[Theorem 2.45]{Adem-Leida-Ruan}.
\end{proof}

\subsection{Remark on representability}

We recall that a morphism 
$$f \colon \mathscr{X} \to \mathscr{Y}$$ 
of stacks over a site $\mathscr{S}$ (which in our case is the category $\op{Man}$ of smooth manifolds) is \emph{representable} if for any $Y\in \mathscr{S}$ and any morphism $Y \to \mathscr{Y}$ the fibered product stack $\mathscr{X} \times_{\mathscr{Y}}Y$ is representable by an object of $\mathscr{S}$. Here, by the standard abuse of notation, we denote by $Y$ the stack $\operatorname{Mor}(-, Y)$ over $\mathscr{S}$ associated with $Y$.

When $\mathscr{S}$ is the category $\op{Top}$ of topological spaces, a map $f \colon \mathscr{X} \to \mathscr{Y}$ between two orbispaces is representable if and only if it is injective on isotropy groups by \cite[Corollary 3.6]{pardon2022enough}. It is well-known, however, that $\op{Man}$ is not closed under fiber products, and hence a smooth map between two smooth manifolds $f \colon M \to N$ is representable if and only if $f$ is a submersion; see e.g., \cite[Lemma 71]{metzler2003smoothstacks}. Therefore, one should not expect mere injectivity on isotropy groups to be sufficient for representability.

Nonetheless, there is a notion of \emph{weak representability}, in which $Y \to \mathscr{Y}$ is assumed to be an atlas instead of an arbitrary morphism; see \cite[Definition 2.34]{behrendxu2011}. The generalization of \cite[Theorem 2.45]{Adem-Leida-Ruan} then states that a morphism of orbifolds injective on isotropy groups is weakly representable if and only if it is injective on isotropy groups. This statement, which was observed in \cite[Remark 2.37]{behrendxu2011}, follows almost immediately from \cite[Theorem 11.39]{joyce2017kuranishi2}, since any atlas is a submersion in the sense of \cite[Definition 2.5]{behrendxu2011}. The main observation here is that the fibered product of transversal morphisms of orbifolds is again an orbifold (for a concise discussion of this fact see  \cite[Section 2.1.3]{gironellazhou2021exact}), and hence in our case $\mathscr{X} \times_{\mathscr{Y}}Y$ is an orbifold since $Y \to \mathscr{Y}$ is a submersion. Then the computation of the isotropy groups in \cite[Theorem 11.39]{joyce2017kuranishi2}, which is equivalent in this case to the proof of \cite[Theorem 2.45]{Adem-Leida-Ruan}, shows that the isotropy groups of $\mathscr{X} \times_{\mathscr{Y}}Y$ are trivial, implying that this orbifold is representable by a manifold.

We also note that this lack of fiber products is the reason why one may wish to work with the category of $C^{\infty}$-schemes, and as \cite[Theorem 7.21]{joyce2017kuranishi1} shows, the equivalence between representability and injectivity on isotropy groups holds for stacks in this category.

Historically, in orbifold Floer theory, smooth maps that are injective on isotropy groups are called \emph{representable} after \cite{Chen2002}. As our discussion above indicates, this name is not satisfactory from the stacky perspective. 

\section{Wrapped Fukaya category of a global quotient orbifold} \label{section: orbifold Floer theory}

\subsection{Geometric preliminaries for Fukaya categories of orbifolds}\label{section: orbifold-Fukaya-prelim}

In this paper, we are mainly interested in studying global quotients of Liouville manifolds, but we emphasize that most of the results of this section are applicable in a more general setting. The following geometric discussion is an adaptation of the orientation package of Seidel \cite[Section 11]{seidel2008fukaya} to the equivariant setting. Our discussion is similar in spirit to \cite{bao2021, cho2013finitegroupactionslagrangian, CDW2020}.

Recall that a {\em Liouville manifold} is a pair $(M, \lambda)$ consisting of a manifold $M$ and a {\em Liouville $1$-form} $\lambda$ on $M$ such that $\omega=d\lambda$ is symplectic and,  
outside of a compact region, $(M,\lambda)$ is identified with a cylindrical end $([1, +\infty)_r \times \partial M,r\alpha)$, where $\alpha$ is a contact form on $\bdry M$. 

\s\n
{\em Assumptions on $(M,\lambda)$, $G$ and $L$.} Let $(M,\lambda)$ be a Liouville manifold and $G$ a finite group acting on $M$. We assume the following:
\begin{itemize}
\item[(A1)] the action of $G$ on $M$ preserves $\lambda$ and the action of $G$ on the cylindrical end is induced from the action of $G$ on $\partial M$; and
\item[(A2)] $c_1(M)=0$ and there is a choice of a $G$-invariant complex volume form $\eta$ on $M$ with respect to some compatible $G$-invariant almost complex structure.
\end{itemize}
Then there is an associated phase map $\alpha_M\colon Gr^+(TM) \to S^1$ given by
\begin{equation*}
    v_1\wedge\dots\wedge v_n\mapsto \frac{\eta(v_1 \wedge \dots \wedge v_n)}{|\eta(v_1\wedge\dots\wedge v_n)|},
\end{equation*}
for any choice of oriented basis $(v_1,\dots,v_n)$ for a given oriented Lagrangian in the oriented Lagrangian Grassmannian $Gr^+(TM)$.

We further assume that any Lagrangian submanifold $L \subset M$ under consideration satisfies the following:
\begin{enumerate}
    \item[(A3)] $L$ is \emph{exact}, that is, there exists a function $f_L \colon L \to \mathbb{R}$ such that $\lambda|_L = df_L$;
    \item[(A4)] $L$ is \emph{cylindrical at infinity}, that is, there exists a Legendrian submanifold $\Lambda$ of $\partial M$ such that $L$ takes the form $[1, \infty) \times \Lambda$ on the cylindrical end $[1, \infty) \times \partial M$.
\end{enumerate}

\s
It is well-known that with the above assumptions (A1)--(A4), one can define the wrapped Fukaya category over $\Z$; we refer the reader to \cite{seidel2008fukaya, Seidel2000graded} for more details.

An \emph{equivariant Lagrangian} is a $G$-invariant collection $\mathcal{L}$ of oriented Lagrangian submanifolds of $M$; by abuse of notation, we also write $\mathcal{L}$ for the union of its elements. 

\begin{remark}
    Although $\mathcal{L}$ itself is not necessarily a Lagrangian submanifold of $M$, the suborbifold $[\mathcal{L}/G] \subset [M/G]$ can be given a structure of a Lagrangian suborbifold of the symplectic orbifold $[M/G]$ in the sense of \cite{COW2024twisted}. The simplest example is the orbit of a line passing through the origin under the standard rotation action of $G=\mu_m$ on $M=\R^2$. This orbit is an equivariant Lagrangian, but only immersed in $\R^2$. 
\end{remark}

The tangent space of an equivariant Lagrangian $\mathcal{L}$ is a collection of tangent spaces to Lagrangians that belong to $\mathcal{L}$, and, at a point $p \in \mathcal{L}$ with stabilizer $G_p$, the tangent space to $\mathcal{L}$ is identified with a $|G_p|$-tuple of Lagrangian subspaces of $T_pM$ related by the $G_p$-action.

The main example of interest to us is the following:

\begin{example}[Cotangent bundle] \label{example: cotangent bundle}
    Let $M=T^*X$ with a $G$-action induced from a $G$-action on $X$. The standard Liouville form $\lambda$ and hence the symplectic structure $d\lambda$ are $G$-invariant. Fixing a choice of an equivariant metric on $X$, we have a $G$-invariant Hamiltonian
    \begin{equation}
        H_0(q,p)=\tfrac{1}{2}|p|^2, \quad q\in X, \, p\in T^*_qX.
    \end{equation} 
    The equivariant Lagrangians $\mathcal{L}_0=X \subset T^*X$ and $\mathcal{L}_x=\cup_{g \in G} T^*_{gx}X$ can be viewed as embedded Lagrangian submanifolds of $M$.
\end{example}

Fix a vector bundle $E_b$ on $M$ whose second Stiefel-Whitney class represents a background class $b \in H^2(M; \Z/2)$. An object in the wrapped Fukaya category $\mathcal{W}([M/G])$ is given by an \emph{oriented equivariant Lagrangian brane}, i.e., a tuple $(\mathcal{L}, \alpha^{\#}, \mathfrak{s}^{\#})$ where:
\begin{enumerate}
    \item[(1)] $\mathcal{L}$ is an oriented equivariant Lagrangian;
    \item[(2)] $\alpha^{\#}$ is a $G$-invariant \emph{grading} on $\mathcal{L}$, i.e., a function $\alpha^{\#} \colon \mathcal{L} \to \R$ such that $e^{2\pi i \alpha^{\#}(x)}=\alpha_M(T_x\mathcal{L})$ for any $x \in \mathcal{L}$;
    \item[(3)] $\mathfrak{s}^{\#}$ is a $G$-invariant spin structure on $E_b\oplus T\mathcal{L}|_{\mathcal{L}}$, that we call a {\em relative spin structure}.
\end{enumerate}

As shown in \cite[Section 6]{cho2013finitegroupactionslagrangian}, Condition (2) is easily satisfied. Condition (3), however, does not hold for any $G$-action even if it preserves the isomorphism type of the spin structure; see \cite[Section 5]{cho2013finitegroupactionslagrangian} and \cite[Section 3.6]{bao2021}. 

We now discuss the two main examples of interest:

\begin{example} 
    Let $M=T^*X$ and let $E_b=\pi^*TX$ where $\pi \colon T^*X \to X$ is the natural projection. By \cite[Section 2.1]{abouzaid2011cotangent}, there are relative spin structures on any cotangent fiber and on the $0$-section. Consider an orientation-preserving $G$-action on $X$. If $x\in X$ and $G_x$ is trivial, then the action immediately lifts to an action on the relative spin structure on $\mathcal{L}_x$. If $G_x$ is nontrivial, then we need to lift the action of $G_x$ to a spin structure on $E_b \oplus TT_x^*X=TT^*X|_{T_x^*X}$. Since the cotangent fiber is $G_x$-equivariantly contractible, it suffices to lift the action to the fixed spin structure on $T_{(x,0)}T^*X=V \oplus V^*$, where $V=T_xX$ is a $G_x$-representation.  In other words, we are given $\rho\colon G_x \to SO(V)$, and we need to lift 
    $$(\rho, \rho^*) \colon G_x \to SO(V \oplus V^*).$$ 
    We can lift any generator $g$ of $G_x$ (for some presentation) arbitrarily to some $\widetilde{\rho}(g) \in Spin(V)$.  Then any relation $r$ for these generators lifts to $\pm1 \in Spin(V)$, and therefore $$(\widetilde{\rho}, \widetilde{\rho}^*)(r)=\widetilde{\rho}(r)\widetilde{\rho}^*(r)=1 \in Spin(V \oplus V^*),$$ providing a desired lift. We can also apply an argument using the second equivariant Stiefel-Whitney class (as in the proof of Claim~\ref{claim: zero equivariant Stiefel-Whitney}).
\end{example}
    
\begin{claim} \label{claim: zero equivariant Stiefel-Whitney}
    Any spin structure on $(E_b \oplus TX)|_{\mathcal{L}_0}$ admits a $G$-equivariant lift. 
\end{claim}

\begin{proof}
    The existence of a $G$-equivariant lift is equivalent to the vanishing of the second equivariant Stiefel-Whitney class $w_2^G(TX \oplus TX) \in H^2_G(X; \Z/2)$ by \cite[Proposition 4.1(1)]{mclaughlin1992orient} (notice that the proof of the result provided in loc. cit. applies to any finite group action). The claim then follows, since 
    $$w_2^G(TX \oplus TX)=2w_2^G(TX)+w_1^G(TX)^2=0.$$
    \vskip-.2in
\end{proof}

{\em From now on, we restrict to the case of Example~\ref{example: cotangent bundle}, namely $M=T^*X$ with a $G$-action induced from a $G$-action on $X$ and $ H_0(q,p)=\tfrac{1}{2}|p|^2$ is the $G$-invariant Hamiltonian.}\footnote{Note that much of the presentation is applicable to an arbitrary Liouville domain equipped with an effective action of a finite group.}

\s
In what follows, we assume that all Hamiltonian chords of $H_0$ between Lagrangians are nondegenerate.  Let $\psi^\rho \colon M \to M$ be
the time-$\op{log}(\rho)$ flow of the standard Liouville vector field $|p| \partial_{|p|}$. 
Given two oriented equivariant Lagrangians $\mathcal{L}_0$ and $\mathcal{L}_1$, a \emph{Hamiltonian orbichord} $\hat{x}$ from $\mathcal{L}_0$ to $\mathcal{L}_1$ is an unordered $|G|$-tuple of Hamiltonian chords from Lagrangians in $\mathcal{L}_0$ to Lagrangians in $\mathcal{L}_1$, which are all related by the $G$-action. Alternatively, given a Hamiltonian chord $x$ connecting $L_0 \in \mathcal{L}_0$ to $L_1 \in \mathcal{L}_1$, we may associate to it the Hamiltonian orbichord $\hat{x}:= \{ g\cdot x~|~g\in G\}$. 
Any such tuple $\hat{x}$ provides a representable Hamiltonian chord in $[M/G]$; compare with \cite{mak2025orbifoldhamiltonianfloertheory} for the closed string case. We treat $\hat{x}$ as a map 
$$\hat{x} \colon \hat{I}=\sqcup_{g \in G}I_g \to M,$$
where $I_g \cong [0,1]$ and the action on $\hat{I}$ is induced from the left $G$-action on itself. Two such maps that differ only by a $G$-action represent the same orbichord $\hat{x}$. Let $\chi(\mathcal{L}_0, \mathcal{L}_1)$ be the set of Hamiltonian orbichords from $\mathcal{L}_0$ to $\mathcal{L}_1$.

Given $\hat{x}\in \chi(\mathcal{L}_0, \mathcal{L}_1)$ we choose an equivariant path of graded Lagrangian subspaces $\Lambda_{\hat{x}} \subset \hat{x}^*TM$ connecting $T_{\hat{x}(0)}\mathcal{L}_0$ to $T_{\hat{x}(1)}\mathcal{L}_1$ and a relative equivariant spin structure on $\Lambda_{\hat{x}} \oplus E_b|_{\hat{x}}$. To make these choices, it suffices to make them along one of the chords in the tuple $\hat{x}$ and then apply the $G$-action. We refer the reader to \cite[Section (12b)]{seidel2008fukaya} and \cite[Section 2.1]{abouzaid2010geometric} for details regarding these choices. Recall that there are exactly $2$ relative spin structures along a given path $\Lambda_{\hat{x}}$. 

Each chord in the tuple $\hat{x}$ is equipped with a Maslov index. This index is the same for each chord in $\hat{x}$ by our assumption on the gradings, and we denote it $|\hat{x}|$.  Using the path $\Lambda_{\hat{x}}$ we define an elliptic operator $D_{\hat{x}}$ on a $|G|$-tuple of punctured disks. Each such disk corresponds to a chord in $\hat{x}$, and there is a natural $G$-action on this disjoint union of disks, which matches the action on $\hat{x}$. We denote by $D_{\hat{x}}^G$ the equivariant part of this elliptic operator. We denote by $o_{\hat{x}}$ the determinant line of $D_{\hat{x}}^G$, and point out that it is isomorphic to the standard orientation line $o_x$ for any chord $x$ in the tuple $\hat{x}$.

We set 
\begin{equation}
    CW^i(\mathcal{L}_0, \mathcal{L}_1)=\bigoplus_{\hat{x}\in \chi(\mathcal{L}_0, \mathcal{L}_1), \,|\hat{x}|=i}\Bbbk \otimes_{\Z}|o_{\hat{x}}|,
\end{equation}
where $|\cdot|$ denotes the rank $1$ abelian group associated with the orientation line as usual, and $\Bbbk$ is a field with $\operatorname{char} \Bbbk=0$.

\subsection{Bulk parameters for type $A$ inertia components}

Let $Y \subset X^{\operatorname{sing}}$ be a connected submanifold of real codimension $2$ in $X$, whose stabilizer $H_Y$ (i.e., the stabilizer of a generic point) is isomorphic to the cyclic group $\mu_{m_Y}$ for some $m_Y\in \Z_{\geq 2}$. 
By Example~\ref{exmpl: hypersurface-inertia}, for any $g \in H_Y\setminus\{e\}$ the age grading of $(N_Y, (g)):=(T^*Y, (g))$ is
\begin{equation}
    \iota([G\cdot N_Y/G], g)=1.
\end{equation}

To any $G$-orbit of a hypersurface $Y$ as above (i.e., to any $A_{m-1}$-singularity in $[T^*X/G]$, for some $m \in \Z_{\ge 2}$), we fix a preferred generator $s_Y$ in the stabilizer $H_Y$ and associate a tuple of \emph{bulk parameters} 
$$\hbar^Y_{s_Y^1}, \dots, \hbar^Y_{s_Y^{m_Y-1}};$$
we also write $\hbar^Y_{s_Y^j}=\hbar^Y_j$.

One can view $\hbar_g^Y$ as the variable associated with the connected component $(Y, (g))$ of the inertia $\mathcal{I}[X/G]$. Let $\bm \hbar$ be the collection of all such variables.

We introduce the \emph{bulk-deformed} morphisms by $A_{m-1}$-singularities as the module 
\begin{equation}
    CW^*_{\bm \hbar}(\mathcal{L}_0, \mathcal{L}_1)\coloneqq CW^*(\mathcal{L}_0, \mathcal{L}_1)\llbracket \bm \hbar \rrbracket,
\end{equation}
i.e., the $\mathfrak{m}_{\bm \hbar}$-adic completion of the graded module $CW^*(\mathcal{L}_0, \mathcal{L}_1)$ over the ring $\Bbbk\llbracket \bm \hbar \rrbracket$ with respect to the maximal ideal $\mathfrak{m}_{\bm \hbar}$, which is generated by the variables $\bm \hbar$.

\subsection{Moduli spaces with \texorpdfstring{$G$}{G}-equivariant Floer data}

In this subsection, we set up the moduli spaces used to define the regular wrapped Fukaya category of the orbifold $[M/G]$, using the global equivariant perspective in view of Lemma~\ref{lemma: orbifold=equivariant}.

\subsubsection{Hurwitz moduli spaces}\label{subsection: hurwitz moduli spaces}

A \emph{Hurwitz datum} is an ordered tuple $\bm h=((h_1), \ldots, (h_b))$ of $b$ elements $(h_i) \in \operatorname{Conj}(G)$, $ (h_i)\not=(e)$, for some $b\in\Z_{>0}$.
Let $\mathscr{H}^G_{0, \bm h, d}$ be the \emph{rigidified} stack of smooth \emph{admissible} $G$-covers
$$(\pi \colon \Sigma \to C; q_1, \dots, q_b)$$ 
of a disk $C$ with $d+1$ boundary punctures, ordered counterclockwise --- the $0$th puncture is negative, and the rest are positive --- whose Hurwitz datum is $\bm h$.
By this we mean $(h_i)$ is the conjugacy class of an element of $G_p$ for $p\in \pi^{-1}(q_i)$, $q_i\in \op{int}(C)$, acting via $\zeta_{|G_p|}$ on $T_p\Sigma$. This is also given by the conjugacy class of a deck transformation of $\Sigma$ corresponding to a small counterclockwise loop around $q_i$ and hence is sometimes called the {\em local monodromy around $q_i$.}

We assume that $d+2b \ge 2$, which implies that the marked curves $C$ are stable. 
Since we consider the rigidified stack, the stabilizers are trivial, and this stack is representable by a smooth manifold. The forgetful map 
$$\mathscr{H}_{0, \bm h, d}^G \to \mathcal{R}_{d,b}$$ 
is a smooth finite cover, where $\mathcal{R}_{d,b}$ is the moduli space of $(d+1)$-punctured disks with $b$ interior marked points.
 
We always denote by $\mathscr{C}$ the orbidisk associated to $(\pi \colon \Sigma \to C; q_1, \dots, q_b)$, which has $q_1, \dots, q_b$ as its \emph{smooth stacky points}, and we treat $C$ as the coarse curve of $\mathscr{C}$. 

\begin{remark}
    The non-rigidified stack decomposes into components corresponding to elements of the set of conjugacy classes of subgroups of $G$, corresponding to the image $H$ in $G$ of the orbifold fundamental group $\pi_1(\mathscr{C})$ under the monodromy action. Such an image in particular determines the number of connected components $|G|/|H|$ in the arbitrary cover in the given component. The group of global automorphisms of this component is isomorphic to the centralizer $C_G(H)$ of $H$ in $G$ (see \cite[Proposition 2.12]{Galeotti2022}) and is quotiented out upon rigidification. In particular, for the substack of connected $ G$-covers, one quotients out by the action of the centralizer $Z(G)$. For more details on rigidification, we refer the reader to \cite[Section 5]{ACV2003}.
\end{remark}

We denote by $\mathscr{H}^G_{0, \bm h, \emptyset}= \mathscr{H}^G_{0, \bm h, -1}$ the Hurwitz moduli space of $G$-covers over a disk without punctures.

The stack $\mathscr{H}^G_{0, \bm h, d}$ is an open substack of the stack of admissible covers of a nodal $(d+1)$-punctured disk
$$\overline{\mathscr{H}}^G_{0, \bm h, d}=\operatorname{Adm}^G_{0, \bm h, d} \subset \operatorname{Adm}^G_{0,b,d}.$$
In the algebro-geometric setting, the stack $\overline{\mathscr{H}}^G_{0, \bm h}$ of admissible covers over a closed genus $0$ curve was introduced by Abramovich-Corti-Vistoli \cite[Section 4.3]{ACV2003} (see also \cite[Section 2.3]{Galeotti2022}).  

\begin{claim}\label{claim: orbifold-with-corners}
The stack $\overline{\mathscr{H}}^G_{0, \bm h, d}$ is an orbifold with corners.
\end{claim}

\begin{proof}[Sketch of proof] 
The stack $\overline{\mathscr{H}}^G_{0, \bm h}$ is a smooth Deligne-Mumford stack \cite[Theorem 3.0.2]{ACV2003} and the corresponding stack over the smooth category $\op{Man}$ is therefore an orbifold.  The claim is proved by doubling:  Indeed, each element of $\overline{\mathscr{H}}^G_{0, \bm h, d}$ is a $G$-cover over a tree of disks and $\mathbb{P}^1$'s attached to this tree, i.e., an element of $\overline{\mathcal{R}}_{d,b}$ defined in \cite{siegel2021}. Any such curve doubles in the obvious way to an element of the Deligne-Mumford stack $\overline{\mathcal{M}}_{0, 2b+(d+1)}$ contained in the locus fixed by the anticomplex involution. This doubling extends to the doubling of the whole $G$-cover, and one obtains an element of $\overline{\mathscr{H}}^G_{(0, d+1), \bm h\cup \overline{\bm h}}$ of the stack of admissible $G$-covers with $(d+1)$ marked (non-stacky) points downstairs and Hurwitz datum $\bm h \cup \overline{\bm h}$, where local monodromies in $\overline{\bm h}$ are obtained by inverting those of $\bm h$. 
\end{proof}
\begin{figure}[ht]
    \centering
\includegraphics[width=0.8\linewidth]{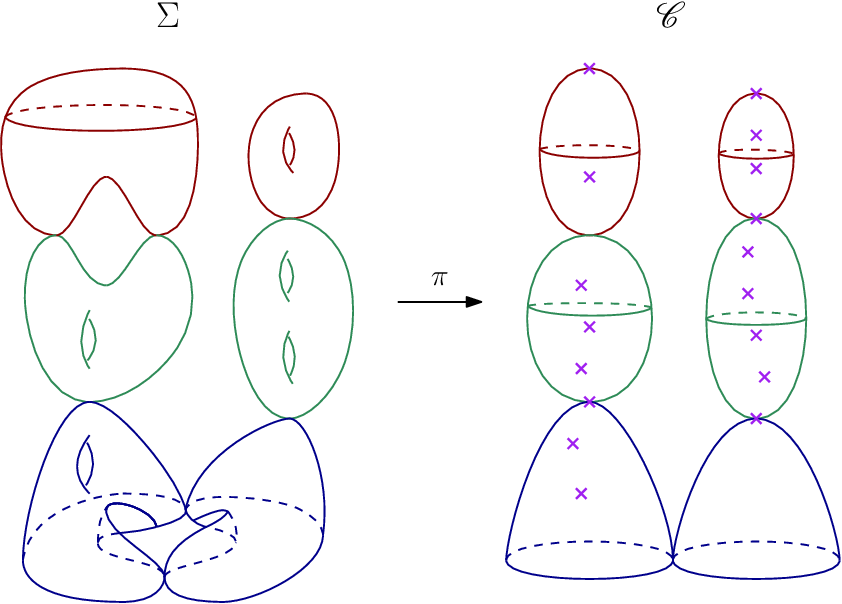}
    \caption{An example of an element of $\overline{\mathscr{H}}_{0, \bm h, \emptyset}^{\mu_2}$ with $b=14$ (the stacky nodes do not contribute to $b$) and $\mathscr{C}$ having $1$ boundary node and $4$ interior nodes, $3$ of which are stacky. Looking at the normalization $\Sigma$, the generators of the stabilizer $\Gamma \cong \mu_2 \times \mu_2 \times \mu_2$ correspond to $\pi$-rotations of the red component on the right, the green component on the right, and the union of the green and red components on the left.}
    \label{fig: mu_2_cover}
\end{figure}
The above is a lift of the strategy outlined in \cite[Lemma 9.2]{seidel2008fukaya}, but one can take a more direct route and use the deformation theory of curves with boundary of \cite[Section 3.3]{Liu2020moduli} to extend the original argument of \cite[Theorem 3.0.2]{ACV2003}. In either approach, near an element $(\pi \colon \Sigma \to C; q_1, \dots, q_b)$ of $\overline{\mathscr{H}}^G_{0, \bm h, d}$, one obtains a local orbifold chart, associated with local gluing parameters for (potentially stacky) nodes, of the form 
\begin{equation}\label{eq: gluing-coordinates}
    [U/\Gamma] \subset [\C^k \times \R_{\ge 0}^{\ell}/\Gamma]\times \R^{2(b-k)+d-2-\ell},
\end{equation}
where $k$ and $\ell$ are the numbers of interior and boundary nodes of $C$, respectively. 
By \cite[Lemma 7.3.1]{ACV2003}, the group $\Gamma$ can be identified with a subgroup of a product of cyclic groups
\begin{equation}\label{eq: inclusion-automorphism}
\Gamma \subseteq \prod_{\ddot{n} \in \mathscr{C}^{\operatorname{node}}}G_{\ddot{n}}=\hat{\Gamma},
\end{equation}
where $\mathscr{C}^{\operatorname{node}}$ is the set of nodes of $\mathscr{C}$. Notice that only interior nodes have nontrivial stabilizers. Standard nontrivial $1$-dimensional representations of $G_{\ddot{n}}$ for each interior node $\ddot{n}$, give rise to the action of $\hat{\Gamma}$ on $\C^k$ (regarded as a product of these representations), and hence of $\Gamma$. The group $\Gamma$ acts trivially on the coordinates corresponding to boundary gluing, since in our doubling construction from the proof of Claim~\ref{claim: orbifold-with-corners} the boundary nodes double to interior non-stacky nodes.

\begin{remark}
    A sufficient condition that makes the inclusion~\eqref{eq: inclusion-automorphism} into an isomorphism is given by \cite[Lemma 7.3.3]{ACV2003}: for a connected nodal $G$-cover $\pi \colon \Sigma \to \mathscr{C}$ satisfying a technical condition called {\em characteristic} (see \cite[Definition 7.3.2]{ACV2003}), in a smoothing $\mathscr{C}'$ of $\mathscr{C}$ a Dehn twist around a neck in $\mathscr{C}'$ corresponding to a given node $\ddot{n} \in \mathscr{C}$ lifts to an automorphism of the corresponding smoothed $G$-cover $\pi' \colon \Sigma' \to \mathscr{C}'$. One immediately deduces that there is a corresponding automorphism of the initial $G$-cover $\pi$.
    
    In general, however, an automorphism of the orbicurve $\mathscr{C}$ which fixes the coarse curve $C$, need not lift to an automorphism of the $G$-cover $\pi \colon \Sigma \to C$.
    For instance, in the case $G=\mu_2$, the above condition is satisfied for any connected cover, and therefore the stabilizer is the product of $k'$ copies of $\mu_2$, where $k'$ is the number of the stacky interior nodes of a given $G$-cover $\pi \colon \Sigma \to \mathscr{C}$. We illustrate a cover with $\Gamma=\hat{\Gamma}\cong\mu_2 \times \mu_2 \times \mu_2$ in $\overline{\mathscr{H}}_{0, \bm h, \emptyset}^{\mu_2}$ in Figure \ref{fig: mu_2_cover}, with each of the $3$ stacky nodes contributing a copy of $\mu_2$ to the stabilizer.
\end{remark}

There are universal curves $\mathfrak{S}_{0, \bm h, d}^G  \to \mathscr{H}_{0, \bm h, d}^G$ and $\mathfrak{C}_{0, \bm h, d}^G \to \mathscr{H}_{0, \bm h, d}^G$ fibered over the Hurwitz moduli space $\mathscr{H}_{0, \bm h, d}^G$ corresponding to the domains and codomains of admissible covers, i.e., the fibers of $\mathfrak{S}_{0, \bm h, d}^G$ and $\mathfrak{C}_{0, \bm h, d}^G$ above $(\pi \colon \Sigma \to C; q_1, \dots, q_b)$ are given by the curves $\Sigma$ and $C$, respectively.
The space $\mathfrak{S}_{0, \bm h,d}^G$ admits a fiberwise $G$-action and $\mathfrak{C}_{0, \bm h, d}^G$ comes equipped with marked point sections $q_1, \ldots, q_b \colon \mathscr{H}^G_{0, \bm h,d} \to \mathfrak{C}_{0, \bm h,d}^G$. Also let $\partial^v\mathfrak{S}_{0, \bm h, d}^G  \to \mathscr{H}_{0, \bm h, d}^G$ and $\partial^v\mathfrak{C}_{0, \bm h, d}^G \to \mathscr{H}_{0, \bm h, d}^G$ be the fibrations obtained by restricting to the boundary fiberwise. All of these fibrations naturally extend to fibrations over $\overline{\mathscr{H}}^G_{0, \bm h,d}$.

The data of an admissible nodal $G$-cover $(\pi \colon \Sigma \to C; q_1, \ldots, q_b) \in \overline{\mathscr{H}}^G_{0, \bm h, d}$ is equivalent to the data of the underlying \emph{nodal orbidisk} $(\mathscr{C}, q_1, \ldots, q_b)$, where $q_i$ is a stacky point of order $\operatorname{ord}(h_i)$, together with a \emph{representable} orbifold map $\mathscr{C} \to BG$; see e.g.\ \cite[Theorem 4.3.2]{ACV2003} in the closed case. A nodal orbidisk $\mathscr{C}$ is allowed to have necessarily \emph{balanced} stacky interior nodes and standard boundary nodes. Such orbidisks are called \emph{twisted}.

We choose strip-like ends, finite strips, and finite cylinders on the coarse curves $C$ of points of $\overline{\mathscr{H}}^G_{0, \bm h, d}$ varying smoothly over $\mathscr{H}^G_{0, \bm h, d}$ and such that the ends, strips, and cylinders are pairwise disjoint and disjoint from open neighborhoods of stacky points.  Using the coordinates~\eqref{eq: gluing-coordinates} near the boundary, these ends, strips, and cylinders can be made to coincide with those coming from the analogous data picked for lower-dimensional strata as in the standard references \cite{seidel2008fukaya, ganatra2013thesis, siegel2021}. These ends, strips, and cylinders give rise to ends, strips, and cylinders on each $G$-cover $\pi: \Sigma\to C$ by pullback and are therefore equipped with a $G$-action.
We denote by $\epsilon^i$ the strip-like end corresponding to the $i$th boundary puncture on $C$.

\subsubsection{Almost complex structures}

We now describe a specific space of almost complex structures on $M=T^*X$, which we use to set up the counts of pseudoholomorphic curves. First, let us fix a $G$-invariant almost complex structure $J_0$ on $M$, such that $J_0=J_0' \operatorname{exp}(Z_0)$, where $J_0'$ is $\omega$-compatible and $Z_0\in C^{\infty}(\operatorname{End}(TM))$ is such that $\sup_{x \in M}\|Z_0(x)\|<\varepsilon$ for a fixed small constant $\varepsilon>0$, and we are measuring with respect to the metric induced by $J_0'$ and $\omega$. In addition, we require $J_0$ to satisfy the following conditions:
\begin{enumerate}
    \item[(J1)] $J_0$ is \emph{$\omega$-tame}.
    \item[(J2)] $J_0$ is \emph{of contact type} near $\partial_\infty M$, i.e.,  $\lambda \circ J_0=dH_0$.
    \item[(J3)]  $J_0$ is \emph{adapted to type $A$} in the sense of Definition~\ref{defn: adapted J}.
\end{enumerate}

The condition (J3) loosely means that $J_0$ is sufficiently nice near the fixed point sets $T^*Y$ for each fixed hypersurface $Y \subset X$. 
Following the approach of \cite[Section 3.3]{perutzsheridan2023relative} we define
\begin{gather}\label{eq: almost complex structures}
    \mathcal{J}=\{J_0\exp(Z) \mid Z \in \mathcal{Z}\}, \text{ where }\\
    \mathcal{Z}=\{Z \in C^\infty(\operatorname{End}(TM)) ~|~ ZJ_0+J_0Z=0,  \|Z\|_{C^0}< \log\tfrac{3}{2}-\varepsilon, \,  \, dH_0 \circ Z=0 \text{ near }\partial_{\infty}M\}. \nonumber
\end{gather}
By \cite{perutzsheridan2023relative}, the bound on the $C^0$-norm in the definition of $\mathcal{Z}$ implies that all almost complex structures in $\mathcal{J}$ are $\omega$-tame. We write $\mathcal{J}^k$ and $\mathcal{Z}^k$ for the similarly defined spaces with $Z$ of class $C^k$. We have a natural smooth $G$-action on both $\mathcal{J}$ and $\mathcal{Z}$. The conditions on the elements of $\mathcal{Z}$ imply that each almost complex structure in $\mathcal{J}$ is $\omega$-tame and of contact type (one may put more conditions to guarantee that any such $J \in \mathcal{J}$ would be adapted to type $A$ inertia components, but we do not require this).

We also make a choice of a family $J_t, \, t\in[0,1]$, of almost complex structures in $\mathcal{J}$.

\begin{remark}\label{remark: maximum principle}
    The argument of \cite[Lemma B.1]{abouzaid2010geometric}, which is based on the Abouzaid-Seidel maximum principle \cite[Lemma 7.2]{abouzaidseidel2010}, is also valid for $\omega$-tame almost complex structures of contact type.
    This observation and its variants have been made by several authors; see \cite[Section 3.1]{perutzsheridan2023relative} and \cite[Section 13]{oh2021reduced}, and also \cite[Proposition 13.2]{wendl2020sft} for a similar statement in a slightly different context.
\end{remark}

\subsubsection{Conformally consistent $G$-equivariant Floer data}

We briefly explain what we mean by a choice of \emph{conformally consistent $G$-equivariant Floer data} over all $\overline{\mathscr{H}}^G_{0, \bm h, d}$.

Choose a \emph{consistent collection} of strip-like ends, finite strips, and finite cylinders on the coarse curves $C$ of elements $(\pi \colon \Sigma \to C; q_1, \dots, q_b)$ of $\overline{\mathscr{H}}_{0, \bm h, d}^G$. This can be done as in \cite[Section 4.1.4]{siegel2021} using gluing coordinates~\eqref{eq: gluing-coordinates}. We assume that all such ends, strips, and cylinders do not contain any stacky points. The choice of such ends, strips, and cylinders provides a $G$-equivariant choice of such ends, strips, and cylinders on each cover $\Sigma$ by pulling back these data from $C$ via $\pi$.

The following is designed so that the choice of Floer datum for $(\pi \colon \Sigma \to C; q_1, \dots, q_b)$ only depends on the underlying marked curve $(C; q_1, \dots, q_b)$.

\begin{definition}\label{defn: floer data}
    A \emph{Floer data} on $\overline{\mathscr{H}}_{0, \bm h, d}^G$ is a tuple 
    $$(\rho, \alpha, H, J),$$
    consisting of a time-shifting map $\rho$, a basic $1$-form $\alpha$, a domain-dependent Hamiltonian $H$, and a domain-dependent almost complex structure $J$, given as follows:

    \begin{enumerate}
        \item Let $\partial^v \overline{\mathfrak{C}}_{0, \bm h, d}^G$ be the vertical boundary of the fibration $\overline{\mathfrak{C}}_{0, \bm h, d}^G \to \overline{\mathscr{H}}^G_{0, \bm h, d}$, whose fiber over $(\pi \colon \Sigma \to C; q_1, \dots, q_b)$ is $\bdry C$. A \emph{time-shifting map} $\rho \colon \partial^v \overline{\mathfrak{C}}_{0, \bm h, d}^G \to [1, +\infty)$ is smooth, constant near the boundary punctures of each disk $(C, q_1, \ldots, q_b)$ --- we call these values {\em weights} of $(\pi \colon \Sigma \to C; q_1,\dots, q_b)$ --- and constant on each component of $\bdry C$ without any boundary punctures.  Let $w_{k}$ be the function on $\overline{\mathscr{H}}_{0,\bm h, d}^G$ which recovers the weight at the $k$th boundary puncture.
    
        \item A \emph{basic $1$-form} is a choice of \emph{vertical} $\alpha \in \Omega^1(\overline{\mathfrak{C}}_{0, \bm h, d}^G)$.  We assume that $\alpha$ vanishes
        along $\partial^v \overline{\mathfrak{C}}_{0, \bm h, d}^G$, 
        on a neighborhood $\mathcal{N}$ of the sections $q_1, \ldots, q_b$, and  
        on some auxiliary open subset which is disjoint from $\mathcal{N}$ and intersects each fiber $C$ in an open subset $U_C$ which is disjoint from the strip-like ends, finite strips, and cylinders. The restriction $\alpha|_C$ to any fiber $C$ is required to be a sub-closed $1$-form, i.e., $d\alpha|_C \le 0$.
    
      \item An {\em equivariant domain-dependent Hamiltonian} $H\in\mathcal{H}_{\bm h,d}(M)$, where the latter is the space of $C^\infty$-maps
        $$H:\overline{\mathfrak{C}}_{0, \bm h, d}^G \times M \to \R,$$
        such that $H=H_0$ on a neighborhood of 
        $\partial_\infty M$, 
        a neighborhood of the sections $q_1,\dots, q_b$, and 
        each component of a given domain $C$ not containing any punctures, where $C$ is the fiber over a point $(\pi \colon \Sigma \to C; q_1, \dots, q_b) \in \overline{\mathscr{H}}_{0, \bm h, d}^G$. 
        Away from some small interpolation region near the boundary on the $k$th strip-like end we require 
        $$X_H\otimes \alpha= X_{\frac{H_0}{w_k^2}\circ\psi^{w_k}} \otimes (w_kdt).$$
        
        \item A {\em domain-dependent almost complex structure} $J \in \mathcal{J}_{\bm h, d}(M)^G$, where the latter space is defined as follows:
        
        First, let $\mathcal{J}_{\bm h,d}(M)$ be the space of \emph{domain-dependent} almost complex structures, i.e., the space of smooth maps
        $$J \colon \overline{\mathfrak{S}}_{0, \bm h,d}^G \to \mathcal{J},$$
        whose restrictions to the $G$-orbit of any finite strip or a strip-like end of the fiber curve $\Sigma$ over a point $(\pi \colon \Sigma \to C; q_1, \dots, q_b) \in \overline{\mathscr{H}}^G_{0, \bm h, d}$ coincide with $(\psi^{w_k})^*J_t,$ for the fixed family $J_t$, $t\in [0,1]$ (specified above). In addition, any such almost complex structure $J$ is required to coincide with $J_0$ (fixed above) on a neighborhood of the preimages of the marked points, on finite cylinders, and on all components not containing punctures of $\Sigma$.
        
        Then let $\mathcal{J}_{\bm h,d}(M)^G \subset \mathcal{J}_{\bm h,d}(M)$ be the subspace of domain-dependent almost complex structures that are $G$-invariant with respect to the diagonal action of $G$ on $\overline{\mathfrak{S}}_{0, \bm h, d}^G \times M$, i.e., satisfy
        \begin{equation}\label{eq: J-equivariance}
            J(z, x)(V)=dg\circ J_{(g^{-1}z, g^{-1}x)}\circ dg^{-1}(V),
        \end{equation}
        for any $(z, x) \in \overline{\mathfrak{S}}_{0, \bm h, d}^G \times M, \, V \in T_xM, \, g \in G$.
    \end{enumerate}

\end{definition}

\begin{remark}
    A triple $( \rho, \alpha,H)$ as above naturally gives rise to a $G$-equivariant triple $(\tilde{\rho}, \tilde{\alpha} , \tilde{H})$ on  $\overline{\mathfrak{S}}_{0, \bm h, d}^G \times M$ via pullback. We note that we could also define a $G$-equivariant triple $(\tilde{\rho}, \tilde{\alpha} , \tilde{H})$ first, and then define $( \rho, \alpha,H)$ via descent, since for each $(\pi \colon \Sigma \to C; q_1, \dots, q_b) \in \overline{\mathscr{H}}_{0, \bm h, d}^G$ these functions and $1$-forms are locally constant near the ramification points.
\end{remark}
    
\begin{definition}
A choice of Floer data for all $\overline{\mathscr{H}}_{0, \bm h, d}^G$ over all $d \ge 2$ and all Hurwitz data $\bm h$ is \emph{universal and conformally consistent} if its restriction to a boundary stratum is \emph{conformally equivalent} in the sense of \cite[Definition 4.2]{abouzaid2010geometric} to those coming from lower-dimensional moduli spaces. In particular, along a boundary stratum, the Floer data agrees to infinite order, with respect to the gluing coordinates, with the Floer data obtained by gluing Floer data from lower-dimensional moduli spaces.  
\end{definition}

We note that the equivariance of the data $(\rho, \alpha, H, J)$ guarantees that glued data near a boundary stratum is invariant under symmetries of the local symmetry group $\Gamma$ as in~\eqref{eq: gluing-coordinates}.

Let us fix a choice of universal and conformally consistent Floer data for all $\overline{\mathscr{H}}^G_{0, \bm h, d}$ with $d \ge 1$, $b \ge 0$ and $\bm h \in \operatorname{Conj}(G)^b$. This can be done via an inductive procedure as in \cite[Lemma 4.3]{abouzaid2010geometric}.

\subsubsection{Moduli spaces}

Consider a tuple of equivariant Lagrangians $\mathcal{L}_0, \dots, \mathcal{L}_d $ in $\mathcal{W}^{\operatorname{reg}}([M/G])$, a tuple of orbichords $\overrightarrow{\hat{x}}=(\hat{x}_d, \ldots, \hat{x}_1)$ with $\hat{x}_i \in \chi(\mathcal{L}_{i-1}, \mathcal{L}_{i})$ for $i=1,\dots,d$, and an orbichord $\hat{x}_0 \in \chi(\mathcal{L}_0, \mathcal{L}_d)$. Further, we pick a connected component $\bm m$ of $\mathcal{I}[M/G]^b$ contained in 
$$\prod_{i=1}^b[M^{h_i}/C(h_i)],$$
using the decomposition given by \eqref{eq: inertia-for-global-quotient}. Assume $(\pi \colon \Sigma \to C; q_1, \dots, q_b) \in \overline{\mathscr{H}}_{0, \bm h, d}^G$ and $u \colon \Sigma \to M$ is a $G$-equivariant map representing an orbidisk $[u] \colon \mathscr{C} \to [M/G]$. By abuse of notation, we denote by $(h_i)$ the generator of the stabilizer $G_{q_i}$ of the stacky point $q_i$, and we point out that in the global quotient case $\mathscr{C}=[\Sigma/G]$, the stabilizer is defined as the stabilizer of any point in $\pi^{-1}(q_i)$, which is a well-defined group, and its generator corresponds to the conjugacy class $(h_i)$. Let us also note that in the global quotient case, a component $\bm m$ as above clearly specifies the corresponding Hurwitz datum.

We define the \emph{inertia evaluation map} $\mathcal{E}^{\mathcal{I}}$ at $q_i$ via
\begin{equation}
\mathcal{E}^\mathcal{I}([u], q_i)=[u]^{\mathcal{I}}(q_i, (h_i)) \in 
\mathcal{I}[M/G],
\end{equation}
where $(q_i, (h_i))$ is a point of inertia $\mathcal{I}\mathscr{C}$, and $$[u]^\mathcal{I} \colon \mathcal{I}\mathscr{C} \to \mathcal{I}[M/G]$$
is the map on the inertia orbifolds induced by $[u]$.
For the whole curve $[u]$, we also use the notation
$$\mathcal{E}^{\mathcal{I}}([u]) \coloneqq(\mathcal{E}^\mathcal{I}([u], q_1), \dots, \mathcal{E}^\mathcal{I}([u], q_b)) \in \mathcal{I}[M/G]^b.$$

\begin{definition}\label{defn: moduli-of-orbidisks}
    The moduli space $\mathcal{M}(\overrightarrow{\hat{x}}, \hat{x}_0; \bm m)$ is the space of $G$-equivariant maps
    $$u \colon \Sigma \to M$$
    for $(\pi \colon \Sigma \to C; q_1, \dots, q_b) \in \mathscr{H}_{0, \bm h, d}^G$ satisfying
    \begin{gather}\label{eq: Floer-equation}
        (du-\pi^*(X_H \otimes \alpha)|_C)^{0,1}_{J|_C}=0,\\
        \mathcal{E}^{\mathcal{I}}([u]) \in \bm m, \nonumber
    \end{gather}
    and
    \begin{gather}
        u(z) \in \psi^{\rho(z)|_C}\mathcal{L}_i \quad \text{ for }z\in \pi^{-1}(w) \text{ with } w\in \partial C \text{ between }\xi^i \text{ and } \xi^{i+1}, \\
        \lim_{s\to\pm \infty}u|_ {\pi^{-1}(\epsilon^i(s, \cdot))}=\psi^{w_k}\circ\hat{x}_i.
    \end{gather}
\end{definition}
In the special case $\overrightarrow{\hat{x}}=\hat{x}_1$ (i.e., $d=1$) and $\bm h =\emptyset$, we define $\mathcal{M}(\hat x_1, \hat x_0)$ to be the quotient of the similarly defined space (in this case the Floer data is just $(J_t, H_0)$). 
\begin{remark}
 Two such maps $u$ and $u'$ from the same domain in the Hurwitz moduli space are identified if they only differ by a $G$-action on the domain.
\end{remark}

\begin{proposition}\label{prop: dimension-transverse}
    For a comeager set of Floer data and any choice of Lagrangians $\mathcal{L}_0, \dots, \mathcal{L}_d$, orbichords $\overrightarrow{\hat{x}}$, $\hat{x}_0$, and components of the inertia orbifold $\bm m$, the moduli space $\mathcal{M}(\overrightarrow{\hat{x}}, \hat{x}_0; \bm m)$ is a smooth manifold of dimension
    \begin{equation} \label{equation: dim}
        \operatorname{dim} \mathcal{M}(\overrightarrow{\hat{x}}, \hat{x}_0; \bm m)=|\hat{x}_0|-|\overrightarrow{\hat{x}}|-2\iota(\bm m)+2b+(d-2).
    \end{equation}
\end{proposition}

\begin{proof}
    The dimension count \eqref{equation: dim} follows from the index formula proved in Proposition~\ref{proposition: index}. The regularity of the moduli space follows from Lemmas~\ref{lemma: transversality-for-strips} and \ref{lemma: basic-surjectivity-on-tangents}, proved in the next section.
\end{proof}

\begin{remark}
    If all the components in $\bm m$ correspond to some $A_{m-1}$-singularities, then
    $$2\iota(\bm m)=2b,$$
    by Example~\ref{exmpl: hypersurface-inertia} and the corresponding orbifold bulk parameters are \emph{``dimensionless''}.
\end{remark}

Given a curve $u \in \mathcal{M}(\overrightarrow{\hat{x}}, \hat{x}_0; \bm m)$ representing an orbicurve $[u]$ such that $\operatorname{dim}\mathcal{M}(\overrightarrow{\hat{x}}, \hat{x}_0; \bm m)=0$, by Proposition~\ref{proposition: orientations} there is an induced map
\begin{equation}
    \partial_{[u]} \colon o_{\overrightarrow{\hat{x}}} \to o_{\hat{x}_0}.
\end{equation}

\subsection{Definition of the regular, bulk-deformed, wrapped Fukaya category}

Let 
$$(\bm Y, \bm h)\coloneqq((Y_1, (h_1)), \dots, (Y_b, (h_b))) \subset \mathcal{I}[X/G]^b$$ 
be a tuple of connected components in the inertia orbifold corresponding to a collection of connected hypersurfaces $Y_i \subset X^{\operatorname{sing}}$. We denote the set of such inertia components of type $A$  by $\mathcal{H}^{\mathcal{I}}([X/G])$. In $M=T^*X$ we associate 
\begin{equation}
    m_{\bm Y, \bm h}\coloneqq(T^*\bm Y, \bm h)=((T^*Y_1, (h_1)), \dots, (T^*Y_b, (h_b)))
\end{equation}
to a given $(\bm Y, \bm h) \in \mathcal{H}^{\mathcal{I}}([X/G])^b$. We will also use notation $\mathcal{H}([X/G])$ for connected hypersurfaces in $X^{\operatorname{sing}}$.

\begin{definition}
    The {\em regular wrapped Fukaya category of $[T^*X/G]$ bulk-deformed by hypersurfaces} is the $A_\infty$-category $\mathcal{W}^{\operatorname{reg}}_{\bm \hbar} ([T^*X/G])$ with regular equivariant Lagrangians as objects, $CW^*_{\bm \hbar}(\mathcal{L}, \mathcal{L}')$ as morphisms between  any equivariant regular $\mathcal{L}$ and $\mathcal{L}'$, and $A_\infty$-operations
    \begin{equation}
        \mu^d \colon CW_{\bm \hbar}^*(\mathcal{L}_{d-1}, \mathcal{L}_d) \otimes \dots \otimes CW_{\bm \hbar}^*(\mathcal{L}_0, \mathcal{L}_1) \to CW^*_{\bm \hbar}(\mathcal{L}_0, \mathcal{L}_d)
    \end{equation}
\begin{gather}
    [\hat{x}_d]\otimes \dots\otimes [\hat{x}_1] \mapsto (-1)^{\dagger_d}\sum_{\substack{\hat{x}_0 \in \chi(\mathcal{L}_0, \mathcal{L}_d), \, b \in \Z_{\ge 0}\\ (\bm Y, \bm h) \in \mathcal{H}^{\mathcal{I}}([X/G])^b,\\ |\hat{x}_0|=|\overrightarrow{\hat{x}}|+2\iota(\bm m)-2b-(d-2),
    [u] \in \mathcal{M}(\overrightarrow{\hat{x}}, \hat{x}_0; \bm m_{\bm Y, \bm h})}}\frac{1}{b!}\prod_{i=1}^b\hbar^{Y_i}_{h_i}\partial_{[u]}([\hat{x}_d], \dots, [\hat{x}_1]), \nonumber
\end{gather} 
extended linearly to all elements of $CW^*_{\bm \hbar}(\mathcal{L}_{i-1}, \mathcal{L}_i)$. Here $[\hat{x}_i] \in o_{\hat{x}_i}$ is some generator. 
    
\end{definition}

\begin{remark}
    The notion of a bulk-deformed $A_\infty$-category by such orbifold hypersurfaces is closely related in spirit to the \emph{relative Fukaya category} that has been studied extensively in the works of Perutz and Sheridan \cite{sheridan2015CYhypersurfaces, sheridan2020versality, perutzsheridan2023relative, sheridan20225bigrelative}. In fact, Sheridan envisioned a potential definition of the relative orbifold Fukaya category in a different setting \cite[Remark 5.1.6]{sheridan2015CYhypersurfaces}.
\end{remark}

It remains to verify that the above indeed gives rise to an $A_\infty$-category, modulo some results that are proven in later sections.

\begin{theorem}\label{theorem: well-defined-ainfty-cat}
    The operations $\{\mu^d\}_{d \ge 1}$ are well-defined and satisfy the $A_\infty$-relations.
\end{theorem}

\begin{proof}
    The proof follows the standard analysis of the compactifications of $\mathcal{M}(\overrightarrow{\hat{x}}, \hat{x}_0; \bm m_{\bm Y, \bm h})$ with $$\operatorname{dim}\mathcal{M}(\overrightarrow{\hat{x}}, \hat{x}_0; \bm m_{\bm Y, \bm h})=1.$$ 
    By Gromov compactness and the maximum principle of \cite[Lemma B.1]{abouzaid2010geometric} (see  Remark~\ref{remark: maximum principle}), there is a compactification $\overline{\mathcal{M}}(\overrightarrow{\hat{x}}, \hat{x}_0; \bm m_{\bm Y, \bm h})$ which consists of curves $u$ as in Definition~\ref{defn: moduli-of-orbidisks} with domains $(\pi \colon \Sigma \to C; q_1, \dots, q_b)$ representing elements of $\overline{\mathscr{H}}_{0, \bm h, d}^G$ and all the moduli spaces involved in computing $\mu^d$ consist of finitely many points. We notice that the maximum principle works for tame contact type almost complex structures, as was observed by several authors; see, for instance, \cite[Lemma 3.2]{perutzsheridan2023relative}.

    A curve $u$ representing a boundary point has its domain in the image of the natural embeddings 
    \begin{gather}
        \overline{\mathscr{H}}_{0, \bm h_1, d_1}^G \times \overline{\mathscr{H}}^G_{0, \bm h_2, d_2} \hookrightarrow \overline{\mathscr{H}}^G_{0, \bm h, d}, \quad d=d_1+d_2-1, \, \bm h=(\bm h_1, \bm h_2);\label{eq: breaking} \\ 
        \overline{\mathscr{H}}_{0, \bm h_1, d}^G \times \overline{\mathscr{H}}^G_{0, \bm h_2} \hookrightarrow \overline{\mathscr{H}}^G_{0, \bm h, d}, \quad \bm h=\bm h_1 \#_{(h_{\ddot{n}})} \bm h_2. \label{eq: interior-bubbling}
    \end{gather}
    Here $\bm h=(\bm h_1, \bm h_2)$ is shorthand for equality up to a shuffle permutation; $\bm h_1 \#_{(h_{\ddot{n}})} \bm h_2$ is the result of the operation which takes two vectors $\bm h_1, \bm h_2$, with $\bm h_1$ containing a conjugacy class $(h_{\ddot{n}})$ and $\bm h_2$ containing a class $(h_{\ddot{n}}^{-1})$,
    concatenates them, and removes both $(h_{\ddot{n}})$ and $(h_{\ddot{n}}^{-1})$;  notice that although $\bm h$ is assumed to contain only nontrivial conjugacy classes, we are allowing the possibility that $(h_{\ddot{n}})=(h_{\ddot{n}}^{-1})=(e)$ which corresponds to the usual nonstacky node; and $\overline{\mathscr{H}}^G_{0, \bm h_2}$ is the moduli space of $G$-covers over a closed genus $0$ curve with Hurwitz datum $\bm h_2$. The first case \eqref{eq: breaking} corresponds to the splitting of the underlying orbidisk $\mathscr{C}$ into two orbidisks $\mathscr{C}_1$ and $\mathscr{C}_2$ glued at matching punctures (notice that we allow the unstable case, in which for some $i=1,2$ one has $(d_i,b_i)=(1,0)$), and the second case \eqref{eq: interior-bubbling} corresponds to the formation of a closed genus $0$ orbicurve bubble attached at the (stacky) node $(\ddot{n}, (h_{\ddot{n}}))$. We notice that the formation of orbidisk bubbles (i.e., those attached along boundary nodes) is ruled out since the Lagrangians are regular.

    Let us consider the second case~\eqref{eq: interior-bubbling}, which corresponds to the collision of several stacky points in the interior of $\mathscr{C}$. Since $T^*X$ is exact, the preimage of the genus $0$ bubble under $\pi$ must be a ghost curve in $T^*X$. For simplicity, we assume that the limit curve $$u_\infty \colon \Sigma=\Sigma_{\bullet}\cup \Sigma_{\mathghost} \to T^*X$$
    has domain $\Sigma$ representing $[\Sigma/G]=\mathscr{C}=\mathscr{C}_{\bullet}\cup\mathscr{C}_{\mathghost}$ with only one node $\ddot{n} \in \mathscr{C}_{\bullet} \cap \mathscr{C}_{\mathghost}$ and with branching profiles $\bm h_1$ and $\bm h_2$ for $G$-covers $\Sigma_{\bullet}\to C_{\bullet}$ and $\Sigma_{\mathghost}\to C_{\mathghost}$, respectively. Let $(q_{i_1}, \ldots, q_{i_{b_2}})$ be the stacky points of $\mathscr{C}_{\mathghost}$ with corresponding hypersurface inertia components 
    $$((Y_{i_1}, (h_{i_1})), \dots, (Y_{i_{b_2}}, (h_{i_{b_2}}))).$$ 
    Notice that at the node $\ddot{n}$ the corresponding datum is given by $(h_{\ddot{n}})=(h_{i_{b_2}}^{-1}\dots h_{i_1}^{-1})$. Depending on the evaluation maps at the stacky points on $\mathscr{C}_{\mathghost}$, there are several cases that we treat separately.

    \s\n
    {\em Case 1 (Ghost is mapped to a lower-dimensional stratum):} Suppose two of the $G$-orbits of the hypersurfaces $Y_{i_j}$ are different.  Then the whole ghost and $\ddot{n}$ in particular are mapped to the component of the inertia $\mathcal{I}[T^*X/G]$ containing $(T^*Z, (h_{\ddot{n}}))$ with $Z$ of real codimension $>2$ in $X$. Without loss of generality, we can assume that $Z = Y_{i_1} \cap \dots \cap Y_{i_{b_2}}$. 
    
    First assume that $T^*Z$ is the entire connected component of $(T^*X)^{h_{\ddot{n}}}$.
    Let $D_{h_{\ddot{n}}}$ be the diagonalization of the action of $h_{\ddot{n}}$ on $TX|_Z$ at some point of $Z$. As in Example~\ref{ex-rotation}, the diagonalization of the action of $h_{\ddot{n}}$ on $TT^*X|_{T^*Z}$ is
    $$\begin{pmatrix}
        D_{h_{\ddot{n}}} & 0 \\
        0 & D_{h_{\ddot{n}}}^{-1}
    \end{pmatrix}.
    $$
    Since $D_{h_{\ddot{n}}}$ has at least two entries which are not equal to $1$, it follows that 
    \begin{equation}\label{eq: index-several-inertia-collide}
        \iota(T^*Z, (h_{\ddot{n}})) \ge 2.
    \end{equation}
    Hence $u_{\bullet}$ is a curve in $\mathcal{M}(\overrightarrow{\hat{x}},\hat{x}_0; \bm m_{\bm Y_1, \bm h_1})$, where $\bm Y_1$ consists of type $A$ inertia components $(T^*Y_j, (h_j))$, with $j \notin \{i_1, \ldots, i_{b_2}\} \subseteq \{1, \dots, b\}$, in addition to $(T^*Z, (h_{\ddot{n}}))$.
    By Proposition~\ref{prop: dimension-transverse} and \eqref{eq: index-several-inertia-collide}, the dimension of this moduli space is negative, so such a curve $u_{\bullet}$ does not exist.

    Now if we drop the assumption that $T^*Z$ forms a connected component of $(T^*X)^{(h_{\ddot{n}})}$ (this can happen for instance if $h_{\ddot{n}}=e \in G$, which is the case if the node $\ddot{n}$ is non-stacky), and is a subset of some such component $(W, (h_{\ddot{n}}))$, it may so happen that $u_{\bullet}$ belongs to $\mathcal{M}(\overrightarrow{\hat{x}},\hat{x}_0; \bm m_{\bm Y_1, \bm h_1})$ of nonnegative dimension $\le1$, where $(\bm Y_1, \bm h_1)$ is a tuple defined as before with $(W, (h_{\ddot{n}}))$ instead. Nevertheless, $u_{\bullet}(\pi^{-1}(\ddot{n}))$ must be mapped to $T^*Z$, which is a symplectic submanifold of $W$ of real codimension $\ge 2$. This imposes an additional condition on $u_{\bullet}$, and by Proposition~\ref{prop: proposition-jets-transversality-with-submanifolds} it must belong to a manifold of negative dimension.

    \s\n
    {\em Case 2 (Ghost is mapped to a generic point of type $A$ inertia):} Suppose all the $G$-orbits of the hypersurfaces $Y_{i_j}$ are identical.  This is the more difficult case. Without loss of generality, we assume that $Y=Y_{i_1}=\dots=Y_{i_{b_2}}$ and that $h_{i_1}, \dots, h_{i_{b_2}} \in H_Y$, where $H_Y$ is the stabilizer of $Y$. If the node $\ddot{n}$ is not stacky, we can argue as before. In the case of a stacky node, we may assume that $u_{\bullet}(p)$ for some $p \in \pi^{-1}(\ddot{n})$ has $H_Y$ as its stabilizer, as otherwise it would belong to a submanifold of $T^*Y$ of codimension at least $2$, and as before there are additional dimension constraints on the moduli space of such curves $u_\bullet$. Hence, we may assume that $$\Sigma_{\mathghost}=\bigsqcup_{[g] \in G/H_Y}g\Sigma_{m_Y},$$
    where $p \in \Sigma_{m_Y}$ and $\pi \colon \Sigma_{m_Y} \to C_{\mathghost}$ is a cyclic $\mu_{m_Y}$-cover.
    
    By Lemma~\ref{lemma: 1-forms-on-cyclic-covers} below, there is a holomorphic $1$-form on $\Sigma_{m_Y}$ corresponding to the character $\chi^{m_Y}_1$ or $\chi^{m_Y}_{m_Y-1}=\chi^{m_Y}_{-1}$, which has a zero of order $k(\chi^{m_Y}_{\pm 1}, \Sigma_{\mathghost}, \ddot{n})-1$. Here $k(\chi, \Sigma_{\mathghost}, \ddot{n})\in \{1, \dots, m_{\ddot{n}}-1\}$ satisfies
    $$\chi(h_{\ddot{n}})=\zeta_{m_{\ddot{n}}}^{k(\chi, \Sigma_{\mathghost}, \ddot{n})};$$
    see~\eqref{eq: local character} for more details. Applying Lemma~\ref{lemma: localization trick} (the {\em localization trick}) and Theorem~\ref{ref: theorem-isotypic-direction-jet-vanishing}, which gives an {\em obstruction to smoothing interior ghosts,} we conclude that $u_{\bullet}$ has a vanishing condition on its equivariant jet in the isotypic direction corresponding to either $\chi_1^{m_Y}$ or $\chi^{m_Y}_{m_Y-1}$ (i.e., a map of the form~\eqref{eq: isotypic-jet-evaluation} evaluates to zero on $u_{\bullet}$). By Corollary~\ref{cor: moduli-spaces-with-condtions}, such $u_{\bullet}$ does not exist.

    \s
    The above analysis eliminates the possibility of \eqref{eq: interior-bubbling} and the only curves that appear in the limit are those coming from~\eqref{eq: breaking}.  Moreover, any such limit curve $u_\infty=(u_\infty^1, u_\infty^2)$ consists of two smooth curves. They contribute to the counts in $\mu^{d_1}$ and $\mu^{d_2}$; the sign verification is analogous to that of \cite[Proposition 12.3]{seidel2008fukaya} and will be omitted. The $\frac{1}{b!}$ coefficient comes from shuffle permutation ambiguity in the equality $\bm h=(\bm h_1, \bm h_2)$; cf.\ \cite[Section 4.1.8]{siegel2021} or \cite[Section 5.1]{sheridan2015CYhypersurfaces}.
\end{proof}

\begin{remark}
    In the proof of the preceding theorem, we used the fact that one can choose sufficiently \emph{regular} universal and conformally consistent Floer data. Once the claimed regularity (i.e., the transversality of the moduli spaces together with submersive evaluation and jet evaluation maps) is established, showing the existence of such Floer data is routine and is omitted here. Section~\ref{section: transversality-for-jets} is devoted to showing the necessary regularity for generic data.
\end{remark}

\begin{lemma}\label{lemma: 1-forms-on-cyclic-covers}
    Let $\pi \colon \Sigma_m \to C$ be a cyclic $\mu_m$-cover with $b \ge 3$ branch points over a closed genus $0$ curve $C$.  Then:
    \be
    \item $H^{1,0}(\Sigma_m)_{\chi_1^m} \neq 0$ or $H^{1,0}(\Sigma_m)_{\chi_{m-1}^m} \neq 0$; and
    \item if $H^{1,0}(\Sigma_m)_\chi \neq0$, then for any branch point $n\in C$ and $p\in \pi^{-1}(n)$, there exists a holomorphic $1$-form $\eta \in H^{1,0}(\Sigma_m)_\chi$ with a zero of order exactly $k(\chi, \Sigma_m, n)-1$ at $p$.
    \ee
\end{lemma}

 Recall that by the remark at the end of Example~\ref{exmpl: hypersurface-inertia}, we may normalize $V_{\chi^m_j} \oplus V_{\chi^m_{m-j}}$ to $V_{\chi^m_1} \oplus V_{\chi^m_{m-1}}$. 

\begin{proof}
    (1) To simplify the exposition, let us assume that $C$ is smooth. Let $\zeta^{a_1}_m, \dots, \zeta^{a_b}_m$ be the local monodromies at $q_1, \dots, q_b$, respectively. By the Chevalley-Weil formula \cite{CW1934formula} (see e.g. \cite[Lemma 3.1]{tamborini2022shimura} for the statement in the cyclic case, and also \cite{arapura2022residues} for a modern exposition), the dimension of $H^{1,0}(\Sigma_m)_{\chi_j^m}$ is given by the formula
    \begin{equation}
        h^1(\Sigma_m, \chi_j^m)=\sum_{i=1}^b\big\{\tfrac{ja_i}{m}\big\}-1,
    \end{equation}
    where $\{a\}$ denotes the fractional part of $a\in \Q$.
    We make a simple observation that for any $x \in \Q \setminus \Z$ one has $\{x\}+\{-x\}=1$. Hence
    $$h^1(\Sigma_m, \chi_1^m)+h^1(\Sigma_m, \chi_{m-1}^m)=\sum_{i=1}^b(\big\{\tfrac{a_i}{m}\big\}+\big\{\tfrac{-a_i}{m}\big\})-2 \ge 1,$$
    which implies (1).
    
    (2) Observe that $\pi_*K_{\Sigma_m}$ is a locally free sheaf with a $\mu_m$-action.  Hence, it decomposes into isotypic components 
    $$\pi_*K_{\Sigma_m}= \oplus_{j=0}^{m-1}(\pi_*K_{\Sigma_m})_{\chi_j^m},$$
    each of which is also locally free as the image of the standard projection on the $\chi_j^m$-component. Now, we claim that each isotypic component is a rank $1$ sheaf. Indeed, over any open subset $U \subset C$ which is disjoint from the branch points we have $$\pi_*K_{\Sigma_m}|_{U} \cong \oplus_{j=0}^{m-1}K_C|_{U}$$
    with the $G$-action on the right-hand side coinciding with the left $G$-action on $G$ itself. This follows since $\pi$ is a cover over $U$. From this decomposition, the claim follows since the right-hand side is just the tensor product of $K_C$ with the regular representation of $\mu_m$, which contains each character once.

    Finally, we observe that there is an isomorphism
    \begin{equation}\label{eq: iso-for-global-section}
        H^0(\Sigma_m, K_{\Sigma_m})_\chi \cong H^0(C, (\pi_*K_{\Sigma_m})_\chi).
    \end{equation}
    Notice that a nonzero holomorphic $1$-form in $H^0(\Sigma_m, K_{\Sigma_m})_\chi$ with a zero of degree at least $k(\chi, \Sigma_m, n)$ at $p \in \pi^{-1}(n)$ corresponds to a nontrivial section $\tilde{\eta}$ of $(\pi_*K_{\Sigma_m})_\chi$ that vanishes at $n$. The isomorphism~\eqref{eq: iso-for-global-section} and our assumption imply that $H^{0}(C, (\pi_*K_{\Sigma_m})_\chi)\neq0.$  By the existence of the section $\tilde{\eta}$,  $(\pi_*K_{\Sigma_m})_\chi$ is a sheaf of nonnegative degree on $C$ of genus $0$, and hence is basepoint-free. The claim then follows.
\end{proof}

\section{Transversality for equivariant jet evaluation}\label{section: transversality-for-jets}

Let $M$ be a smooth manifold equipped with a $G$-equivariant almost complex structure $J$.
For a $G$-equivariant $J$-holomorphic curve $u \colon \Sigma \to M$ at a point $p \in \Sigma$ with a nontrivial stabilizer, the equivariance itself usually imposes a strong vanishing condition on the coefficients of its Taylor series. For instance, a holomorphic curve representing a morphism of orbifolds
$$[\C/ \mu_m] \to [\C/{\mu_m}],$$
with actions by characters $\chi_1^m$ and $\chi_k^m$ on the domain and codomain, has a Taylor expansion of the form
$$u(z)=\sum_{j=0}^{\infty}c_{k+jm}z^{k+jm}.$$
For this reason, the mere vanishing of the first or even second derivative does not prevent a certain curve $u$ from appearing in the moduli space of $G$-equivariant curves of a fixed dimension. 

The goal of this section is to systematically study equivariant $J$-holomorphic jets, extending the ideas of \cite{Zehmisch2015}, and show that the subspace of curves with vanishing of the first nontrivial equivariant $J$-holomorphic jet is a submanifold of positive codimension.

\subsection{Equivariant $J$-holomorphic jets}\label{section: equivariant-j-jets} 

Let $G$ be a finite group and let $(M,J)$ be an almost complex manifold with a $J$-holomorphic $G$-action. Consider a subgroup $H \subset G$, its fixed point submanifold $M^H$, and a connected component $\mathcal{Y} \subset M^H$. This submanifold is almost complex, since $G$ preserves the almost complex structure, and, moreover, we have an \emph{isotypic splitting}
\begin{gather}\label{eq: isotypic}
    TM|_{\mathcal{Y}} \cong \bigoplus_{\rho \in \operatorname{Irr}_{\C}(H)} \nu_{\rho}, \\
    \nu_\rho \cong \Hom_H(\tilde{\rho}, TM|_{\mathcal{Y}}) \otimes \tilde{\rho}, \nonumber
\end{gather}
such that $\tilde{\rho}$ is a trivial bundle over $\mathcal{Y}$ with action of $H$ given by the irreducible representation $\rho$. In particular, each fiber of $\nu_{\rho}$ as an $H$-representation is a sum of finitely many copies of $\rho$. Observe that $T\mathcal{Y}$ corresponds to the trivial representation. We mainly focus on the case of a cyclic subgroup $H \cong \mu_m$, and we assume that $H$ is an essential subgroup, i.e., the stabilizer of a generic point of $\mathcal{Y}$ is exactly $H=\langle g \rangle$ for some generator $g \in H$.

We point out that the existence of the splitting~\eqref{eq: isotypic} can be shown by using $H$-linear charts (e.g., the ones given by exponential maps with respect to some $G$-invariant metric). 
Explicitly, given $H$-equivariant charts
$$ \varphi_x \colon U_x \to V_x \subset T_xM, \quad \varphi_y \colon U_y \to V_y \subset T_yM,$$
where $U_x$ and $U_y$ are open neighborhoods of points $x, y \in M^H$, and $V_x, V_y$ are neighborhoods of $0$, we note that $D(\varphi_y \circ \varphi_x^{-1})_p$ (whenever it is well-defined) is $H$-equivariant, where $p \in V_x \cap T_xM^H$ is an arbitrary such point. 
By Schur's lemma, $D(\varphi_y \circ \varphi_x^{-1})_p$ splits into isotypic components and is an isomorphism. This is exactly the transition function for $TM|_{M^H}$ for these charts.

We assume that $J$ is \emph{adapted to the inertia component $(\mathcal{Y}, (g))$} in the sense of Definition~\ref{defn: adapted J}. That is, there is an auxiliary choice of an $H$-invariant tubular neighborhood $U$ of $\mathcal{Y}^\circ$ together with a choice of almost complex immersed $H$-invariant submanifolds 
\begin{equation}
    V_{\chi_1}, \dots, V_{\chi_{m-1}},
\end{equation}
containing $\mathcal{Y}^\circ$, which come equipped with $J$-holomorphic projections
\begin{equation}
    \pi_{\chi_j}\colon U \to V_{\chi_j}, \quad \pi_{\chi_j}|_{V_{\chi_j}}=\operatorname{id},
\end{equation}
and these submanifolds are tangent to the corresponding isotypic components, i.e.,
\begin{equation}
    TV_{\chi_j}|_{\mathcal{Y}^\circ}=T\mathcal{Y}^\circ \oplus\nu_{\chi_j}.
\end{equation}

\s 
We consider a two-dimensional disk $(D,0)$ with the standard complex structure $j$ and the standard action of the cyclic group $\mu_d$ with generator $\zeta_d$ acting by multiplication by $\zeta_d$, and an injective homomorphism 
\begin{equation}
    \rho_0 \colon \mu_d \to H.
\end{equation}
Notice that $d \mid m$.
Our goal is to describe the notion of $J$-holomorphic jets of orbifold morphisms 
$$[D/\mu_d] \to [M/G]$$ 
whose underlying map on topological spaces sends the origin to the class of a point $x \in \mathcal{Y}$, with associated homomorphism on isotropy groups given by $\rho_0$.

Let us recall the notion of \emph{$J$-holomorphic jets} $\J^{k,J}(D,M)$ of smooth curves.

\begin{definition}
    We say that two $J$-holomorphic maps $u_1, u_2 \colon D \to M$ such that $u_1(z_0)=u_2(z_0)=x$  are \emph{$k$-jet equivalent} at $z_0 \in D$ if for any smooth curve $\gamma \colon \R \to D$ with $\gamma(0)=z_0$ the curves $u_1 \circ \gamma$ and $u_2 \circ \gamma$ have \emph{$k$th order of contact} at $0$, i.e., for any smooth function $\varphi \colon M \to \R$ the first $k$ derivatives of the difference 
    $$\varphi\circ u_1 \circ \gamma- \varphi \circ u_2 \circ \gamma$$ 
    vanish at $0 \in \R$. 
    
    A \emph{$J$-holomorphic $k$-jet} at $(z_0, x) \in D \times M$ is an equivalence class of $J$-holomorphic maps $u \colon D \to M$ with $u(z_0)=x$ under the $k$-jet equivalence relation;  such a jet is denoted $j^k_{z_0}u$ by slight abuse of notation.
\end{definition} 

We refer the reader to \cite[Section 12]{KMS1993} for a detailed introduction to jets of smooth maps. Let $\J^k(D,M)$ be the set of standard $k$-jets, which is known to be a smooth fiber bundle over $D \times M$, and let $\J^k_z(D,M)_x$ denote the fiber above $(z,x) \in D \times M$.

In \cite{Zehmisch2015} it is shown that the set of $J$-holomorphic $k$-jets $\J^{k,J}(D,M)$ is a smooth submanifold of $\J^k(D,M)$.
In particular, in \cite[Section 2.1]{Zehmisch2015}  it is explained that in local coordinates $z=x+iy$ on $D$ the $J$-holomorphic jet $j^k_zu$ at $z$ is determined by the first $k$ partial derivatives with respect to the $x$-coordinate, i.e., all the remaining mixed partial derivatives $$\partial_x^a\partial_y^bu(z) \text{ with } a+b \le k$$
can be uniquely recovered. In other words, $\J\tensor*[^{k,J}_{z}]{(}{}D,M)_x$ is a $(k+1)\cdot \operatorname{dim}_{\R} M$-dimensional linear subspace of $\J^{k}_z(D,M)_x$ given by a system of linear equations, whose coefficients are determined by derivatives of $J$ at $x$ up to order $k$. Moreover, it is shown in \cite[Proposition 2.1]{Zehmisch2015} that any element of this subspace locally gives rise to a $J$-holomorphic map $u$.

We use the notation 
\begin{equation}
    \pi_{k-1}^k \colon \J^k(D,M) \to \J^{k-1}(D,M)
\end{equation}
for the standard restriction map, which is known to be an affine fiber bundle modeled on the pullback of $$TM \otimes \operatorname{Sym}^k(T^*D)$$ to $\J^{k-1}(D,M)$
via $\pi^{k-1}_0$; see \cite[Proposition 12.11]{KMS1993}. We also have a \emph{constant $k$-jet section} 
\begin{equation}
    s_c^k \colon D \times M \to \J^k(D,M),
\end{equation}
 where $s_c^{k}(z,x)$ is the $k$-jet of the constant map sending $D$ to $x \in M$.

Recall that $TM \otimes \operatorname{Sym}^k(T^*D)$ can be realized as a submanifold of $\J^k(D,M)$ as the preimage under $\pi_{k-1}^k$ of the image  $s_c^{k-1}(D \times M)$ of the constant $(k-1)$-jet section.
Clearly, the projection $\pi_{k-1}^k$ restricts to $J$-holomorphic jets providing an affine bundle 
\begin{equation}
    \pi_{k-1}^k \colon \J^{k, J}(D,M) \to \J^{k-1, J}(D,M),
\end{equation}
modeled on a subbundle $(TM \otimes \operatorname{Sym}^k(T^*D))^J$ of $TM \otimes \operatorname{Sym}^k(T^*D)$ given by the fixed points of the endomorphism
$$-J \otimes\operatorname{Sym}^k(\operatorname{id} \otimes \dots \otimes \operatorname{id} \otimes j^*) \in \operatorname{End}(TM \otimes \operatorname{Sym}^k(T^*D)).$$ 
To see this, we use the identification $$TM \otimes \operatorname{Sym}^k(T^*D)=(\pi_{k-1}^k)^{-1}(s_c^{k-1}) \subset \J^{k, J}(D,M)$$ and then differentiate $(k-1)$ times the $J$-holomorphic equation $du=-J \circ du \circ j$ for $u$ with $j^{k-1}u \in s_c^{k-1}(D \times M)$ to obtain
\begin{equation}
    D^ku(X_1, \dots, X_k)=-J(u)D^ku(X_1, \dots, X_{k-1}, jX_k),
\end{equation}
which is equivalent to the condition stated above. (When $k=1$ the right-hand side is $-J(u)D u(jX_1)$.) Notice that there are no terms depending on the derivatives of $J$ precisely because of our assumption on $j^{k-1}u$. In local coordinates $(dx^{\odot a}  \odot dy^{\odot b})_{a+b=k}$ any vector of $(TM \otimes \operatorname{Sym}^k(T^*D))^J$ is explicitly given by
\begin{equation}\label{eq: J-holomorphic k-jet}
    (v, Jv, J^2v, \dots, J^{k}v)\in TM\otimes \operatorname{Sym}^k(T^*D), \quad v \in TM.
\end{equation}

\begin{definition}
    The set of \emph{equivariant $J$-holomorphic jets} $\J^{k, J, \rho_0}(D, M)$ is the smooth submanifold of $\J^{k, J}(D, M)$ given by the fixed point set of the following action of $\mu_d$ on $\J^{k}(D,M)$
\begin{equation}\label{eq: jet-action}
    j^ku \mapsto (j^k\rho_0(\gamma))\circ j^ku\circ (j^k\gamma^{-1}),
\end{equation}
which preserves $J$-holomorphic jets, where $\gamma$ and $\rho_0(\gamma)$ are treated as diffeomorphisms of $D$ and $M$, respectively.
\end{definition} 

We refer the reader to \cite[Section 12.4]{KMS1993} for more details on actions on jets. 

\begin{remark} \label{rmk: fixed point set}
    By definition,  $\J^{k, J, \rho_0}(D, M) \subset \cup_{x\in M^H} \J^{k,J}_{z_0}{(}{}D,M)_x$, where $z_0=0$, and, moreover, it is fibered over the components of $M^H$.
\end{remark}

By Remark~\ref{rmk: fixed point set}, $\J^{k, J, \rho_0}(D, M)=\J^{k, J, \rho_0}(D, U)$, where $U$ is the $H$-invariant tubular neighborhood of $\mathcal{Y}$ from earlier, and 
$$\iota_{i}^k \colon \J^{k, J, \rho_0}(D, V_{\chi_i}) \hookrightarrow \J^{k, J, \rho_0}(D, U)$$ is a submanifold of the latter, where we use the restriction of $J$ to $V_{\chi_i}$; see \cite{CieliebakMohnke2007} for more details in the non-equivariant case. Since we have a $J$-holomorphic projection $\pi_{\chi_i} \colon U \to V_{\chi_i}$, there is a natural induced projection
$$\pi^{k, \chi_i} \colon \J^{k, J, \rho_0}(D, U) \to \J^{k, J, \rho_0}(D, V_{\chi_i}),$$
satisfying $\pi^{k,\chi_i} \circ \iota_{i}^k= \operatorname{id}$. Hence $\pi^{k, \chi_i}$ is a submersion.

The following is the equivariant version of \cite[Proposition 2.1]{Zehmisch2015}:

\begin{lemma}\label{lemma: equivariant jet lift}
    Any equivariant jet in $\J^{k, J, \rho_0}(D, M)_x$ for $x \in \mathcal{Y}$ can be represented as a $k$-jet of an equivariant $J$-holomorphic map $u \colon D \to M$ with $u(z_0=0)=x$.
\end{lemma}

\begin{proof}
    Since the statement is local, we use local coordinates
    $$\C^n=\C^{n_0} \times \C^{n_1} \times \dots \times \C^{n_{m-1}}, \quad n={n_0+ \dots+ n_{m-1}},$$ 
    near $x=0$, on which $H$ acts linearly (here the $\C^{n_0}$ component is the local coordinate for $\mathcal{Y}$, and $\C^{n_0} \times \C^{n_i}$ is the coordinate for $V_{\chi_i}$). We consider the totally real subbundle $F$ of the trivial bundle along $\partial D$ given by $F_{e^{i\theta}}=\R^ne^{ik\theta/2}$, and denote by $D_F$ its linearized $\overline{\partial}$-operator. We also consider the map
    \begin{gather*}
       \mathcal{F} \colon W^{k+1, p}_F(D, \C^n) \to W^{k,p}(D, \C^n) \times  \J^{k,J}_{z_0}(D, \C^n),\\ 
       u \mapsto (\overline{\partial}_Ju, j^k_{z_0}u),
    \end{gather*}
    where $z_0=0$.
    As in the proof of \cite[Proposition 2.1]{Zehmisch2015}, combining Steps 2 and 4 of the proof of \cite[Theorem C.4.1]{mcduffsalamon2004} one can show that the differential 
    $$D\mathcal{F}(0)[\xi]=(D_F\xi, j_0^k\xi)$$ 
    of $\mathcal{F}$ at $0$ is invertible.
    
    Notice that since $\mathcal{F}$ is $\rho_0$-equivariant with respect to conjugation actions on the domain and codomain (see~\eqref{eq: jet-action} for the action on the second coordinate, and see the proof of Lemma~\ref{lemma: basic-surjectivity-on-tangents} below, where equivariance of the linearized operator is verified), we immediately obtain that $D\mathcal{F}(0)$ is also invertible when restricted to the $\rho_0$-invariant parts. The lemma now follows from the implicit function theorem.
\end{proof}

\begin{remark}\label{remark: tangent-space-to-jets}
    The proof of the above lemma also implies that any tangent vector in $T_{j^k_{0}u}\J^{k, J, \rho_0}(D,M)$ can be given as a $k$-jet $j^k_0 \xi$ of a $\rho_0$-equivariant section of $W^{k+1}_F(D, \C^n)$ such that $D_F\xi=0$. 
\end{remark} 

At this point let us view $\rho_0$ as a homomorphism $\mu_d\to \mu_m$ by postcomposing with the preferred isomorphism $H\stackrel\sim\to \mu_m$.

\begin{lemma} \label{lemma: dimension}
    Suppose that $\chi_i(\rho_0(\zeta_d))=(\rho_0(\zeta_d))^i=\zeta_d^{r_i}$ for some $r_i \in \{ 1, \ldots, d\}$. Then the space $\J^{r_i, J, \rho_0}(D, V_{\chi_i})$ is a smooth manifold of dimension
    $$\operatorname{dim}\J^{r_i, J, \rho_0}(D, V_{\chi_i})= \operatorname{dim}V_{\chi_i}= \operatorname{rk}\nu_{\chi_i} + \operatorname{dim}\mathcal{Y},$$
    i.e., $\operatorname{dim}\J^{r_i, J, \rho_0}(D, V_{\chi_i})_x= \operatorname{dim}V_{\chi_i}$ for any $x \in \mathcal{Y}$, while $\J^{k, J, \rho_0}(D, V_{\chi_i}) \cong \mathcal{Y}$ for $k< r_i$.
\end{lemma}

\begin{proof}
    Let us first show that 
    $$\operatorname{dim}\J^{r_i, J, \rho_0}(D, V_{\chi_i})_x \le \operatorname{dim}V_{\chi_i}=n_i.$$ 
    Applying Lemma~\ref{lemma: equivariant jet lift}, we pick a $\rho_0$-equivariant $J$-holomorphic map $u \colon D \to V_{\chi_i}$ with $u(0)=x$ representing a given jet in $\J^{r_i, J, \rho_0}(D, V_{\chi_i})_x$. We pick a local coordinate chart $\C^{n_i} \times \C^{n_0}$ near $x \in V_{\chi_i}$, with generator $\zeta_m$ of $\mu_m$ acting trivially on $\C^{n_0}$ and by multiplication by $\chi_i(\zeta_m)=\zeta_m^i$ on $\C^{n_i}$, and $J(0)$ the standard complex structure; here $n_0=\operatorname{dim} \mathcal{Y}$. By the proof of \cite[Lemma 2.4.1]{mcduffsalamon2004}, $u$ has the form
    \begin{equation}
        u(z)=(a_1, a_2)z^\ell+O(|z|^{\ell+1}), \quad (a_1,a_2) \in \C^{n_i} \times \C^{n_0},
    \end{equation}
    for some $\ell \in \Z_{\ge 1}$. By the $\rho_0$-equivariance,
    \begin{equation}\label{eq: equivariance-jets-condition}
        (a_1, a_2)\zeta_d^\ell z^\ell+O(|z|^{\ell+1})=(\zeta_d^{r_i}a_1, a_2)z^\ell+O(|z|^{\ell+1}).
    \end{equation}
    Hence, if $u$ is a nonconstant equivariant $J$-holomorphic map, its leading term must be of degree $r_i$, and with its only nontrivial component in the $\C^{n_i}$-direction, as $a_2$ in~\eqref{eq: equivariance-jets-condition} must vanish. We also see immediately that  $\J^{k, J, \rho_0}(D, V_{\chi_i}) \cong \mathcal{Y}$ for $k< r_i$, and it coincides with the image of $\mathcal{Y}$ under the section 
    \begin{equation}\label{eq: constant-map-section}
        s_c^k \colon V_{\chi_i} \to \J^k_{z_0}(D, V_{\chi_i})
    \end{equation} 
    given by sending a point $x \in U_{\chi}$ into the $k$th jet of the constant map with the image being the point $x$. 

     \n In other words, $$\J^{r_i, J, \rho_0}(D, V_{\chi_i}) \subset (\pi_{r_i-1}^{r_i})^{-1}(s^{r_i-1}_c(\mathcal{Y}))=TV_{\chi_i}|_{\mathcal{Y}} \otimes \operatorname{Sym}^{r_i}(T^*D)_0.$$
     
     We note that the action~\eqref{eq: jet-action} of $\mu_d$ on this subspace is given by the linear map $$D\rho_0(\gamma) \otimes \operatorname{Sym}^{r_i}(D^*\zeta_d).$$ For this computation, it is crucial that $\mu_d$ acts linearly on the domain. Since $D\zeta_d \colon T_{z_0}D \to T_{z_0}D$ is a rotation by the angle $\frac{2\pi}{d}$, we have that its dual is a rotation by the angle $-\frac{2\pi}{d}$, and $v= (1, i)$ is its eigenvector in the complexification. We notice that $\operatorname{Sym}^{r_i}(D^*\gamma)$ then has
    $$v^{\odot r_i}=(1, i, -1, \dots, i^{r_i})\in \C \otimes \operatorname{Sym}^{r_i}(T_{z_0}^*D)$$ as its eigenvector of eigenvalue $\zeta^{-r_i}_d=e^{\frac{-2\pi r_i}{d}i}$. Using $J(x)=i$ we may regard $T_xV_{\chi_i} \otimes \operatorname{Sym}^{r_i}(T^*D)$ as a direct sum of $\frac{\operatorname{dim}V_{\chi_i}}{2}$ copies of $\C \otimes \operatorname{Sym}^{r_i}(T^*D)$, and hence, since $D\rho_0(\zeta_d)$ acts via multiplication by $\zeta^{r_i}_d$, any vector of the form~\eqref{eq: J-holomorphic k-jet} is invariant under $\mu_d$, which provides the identification
    $$\J^{r_i, J, \rho_0}(D, V_{\chi_i})= (TV_{\chi_i}|_{\mathcal{Y}}\otimes \operatorname{Sym}^{r_i}(T_{z_0}^*D)_0)^J.$$
    \vskip-.25in
\end{proof}

\begin{remark}\label{remark: tangent-jet-space-first}
    We claim that Remark~\ref{remark: tangent-space-to-jets} and Lemma~\ref{lemma: dimension} imply that any tangent vector at $T_{j_0^{r_i}u}\J^{r_i, J, \rho_i}(D, V_{\chi_i})$ is not only given by the $r_i$-jet of some $\xi$ with $D_F\xi=0$, but is moreover determined only by its $r_i$th partial derivative $\partial_x^{r_i}\xi$ with respect to $x$-coordinate. Indeed, since we may assume $\xi$ to have a vanishing $(r_i-1)$-jet, and since in local coordinates we have the expression
    $$D_u\xi=\partial_x\xi+J(z)\partial_y\xi+A(z)\xi,$$
    with $A(z)$ a linear operator depending on $z$; see \cite[Proposition 3.1.1]{mcduffsalamon2004}. The claim then follows using the same argument as in the case of $J$-holomorphic jets, and therefore $T_{j_0^{r_i}u}\J^{r_i, J, \rho_i}(D, V_{\chi_i})$ can be identified with
    $$(T_{u(z_0)}V_{\chi_i}\otimes \operatorname{Sym}^{r_i}(T_{z_0}^*D)_0)^J.$$
\end{remark}

\subsection{Transversality} 

\subsubsection{Transversality for strips}
In this section, we assume the setup of Section~\ref{section: orbifold-Fukaya-prelim}, and work under the assumption $M=T^*X$, while keeping the discussion sufficiently general.

We first show that there is a choice of a family $J_t, t\in[0,1]$, with each $J_t \in \mathcal{J}$ such that the moduli spaces $\mathcal{M}(\hat{x}_1, \hat{x}_0; \emptyset)$ are transversely cut out for all choices of $\hat{x}_1$ and $\hat{x}_0$. Observe that $\mathscr{H}_{0, \emptyset, 1}$ consists of just one point corresponding to the trivial $G$-cover over a strip.
 
\begin{lemma}\label{lemma: transversality-for-strips}
    There exists a family $J_t, t \in [0,1]$, such that all moduli spaces $\mathcal{M}(\hat{x}_1, \hat{x}_0; \emptyset)$ are transversely cut out.
\end{lemma}

\begin{proof}[Sketch of proof]
  First, observe that we only consider pairs $(\mathcal{L}, \mathcal{L}')$ of regular equivariant Lagrangians, hence, curves in $\mathcal{M}(\hat{x}_1, \hat{x}_0; \emptyset)$ for any pair $\hat{x}_1, \hat{x}_0 \in \chi(\mathcal{L}, \mathcal{L}')$ are not contained in the fixed point set of any subgroup of $G$. Any such curve corresponds to a $G$-orbit of a $J_t$-holomorphic strip connecting a fixed Hamiltonian chord $x \in \hat{x}_0$ to an arbitrary chord $y \in \hat{x}_1$.
  The statement of the lemma in the case $G=\mu_2$ then essentially follows from \cite[Proposition 5.13]{seidelkhovanov2002}. The same proof can be repeated for an arbitrary finite group $G$, and we omit the details. We also refer the reader to \cite[Theorem 5.2.7]{MMMR2019}, which treats arbitrary $G$ for Hamiltonian Floer homology. 
\end{proof}

\subsubsection{Transversality for general $G$-equivariant curves}\label{section: transversality for G-equivariant curves}

Consider the case $\bm h \neq \emptyset$ or $d>2$. 

By the usual deformation theory, a neighborhood $U_\pi$ of $(\pi \colon \Sigma \to C; q_1, \dots, q_b)$ in $\mathscr{H}^G_{0, \bm h, d}$ is diffeomorphic to $(\mathbb{C}^{b} \times \R^{d-2})_y$ and is given by varying the almost complex structure on $C$ away from the boundary and the marked points $(q_1, \ldots, q_b)$, where $j_{y=0}$ is the standard complex structure on the $(d+1)$ boundary punctured unit disk. These almost complex structures naturally induce almost complex structures $j_y$ on $\Sigma$ for $y \in \mathbb{C}^{b} \times \R^{d-2}$ within a small neighborhood of the origin, and we regard $U_\pi$ as the set of such almost complex structures. Clearly, this is identified with a variation of $\pi^*j_0$ in the set of $G$-equivariant almost complex structures on $\Sigma$.  The universal curve $\mathfrak{S}_{0, \bm h, d}^G$ admits a trivialization over $U_\pi$, i.e., 
$$\mathfrak{S}_{0, \bm h,d}^G|_{U_\pi} \cong \Sigma \times U_\pi.$$

Let us fix a collection of orbichords $\overrightarrow{\hat{x}}=(\hat{x}_d, \ldots, \hat{x}_1), \hat{x}_0$ connecting a $(d+1)$-tuple of regular $G$-equivariant Lagrangians $\mathcal{L}_0, \dots, \mathcal{L}_d$. We also fix a connected component $\bm m$ of $\mathcal{I}[M/G]^b$, consisting of some components of fixed point sets of the elements of $\bm h$.
We then consider the following Banach manifold to model the solutions of the Floer equations locally
\begin{equation}
    \mathscr{X}_{k,p}^G=  W^{k,p}(\Sigma, M; \overrightarrow{\hat{x}},\hat{x}_0; \bm m)^G \times U_\pi \times \mathcal{J}_{\bm h,d}(M)^G|_{U_\pi},
\end{equation}
which is the fixed point set of the diagonal action on
\begin{equation}
    \mathscr{X}=   W^{k,p}(\Sigma, M; \overrightarrow{\hat{x}},\hat{x}_0; \bm m) \times U_\pi  \times \mathcal{J}_{\bm h,d}(M)|_{U_\pi}
\end{equation}
for $kp>2$. Here $W^{k,p}(\Sigma, M; \overrightarrow{\hat{x}},\hat{x}_0)$ is the Banach manifold of $W^{k,p}$-curves with boundary components mapped to $\mathcal{L}_0, \dots, \mathcal{L}_d$ and convergence near the punctures to chords specified by $\overrightarrow{\hat{x}}$ and $\hat{x}_0$ in the $W^{k,p}$-sense. The subspace $$W^{k,p}(\Sigma, M; \overrightarrow{\hat{x}},\hat{x}_0; \bm m)^G \subset W^{k,p}(\Sigma, M; \overrightarrow{\hat{x}},\hat{x}_0)^G$$
of the space of $G$-equivariant maps in $$W^{k,p}(\Sigma, M; \overrightarrow{\hat{x}},\hat{x}_0)$$ is given by those maps whose associated evaluation $\mathcal{E}^{\mathcal{I}}$ at the stacky points of the underlying orbicurve takes values in $\bm m$.

We use the notation 
$$\mathscr{B}=\mathscr{B}^{k,p}=W^{k,p}(\Sigma, M; \overrightarrow{\hat{x}},\hat{x}_0).$$ 
The tangent space $T_u\mathscr{B}$ is given by $W^{k,p}_F(\Sigma, u^*TM)$ with real boundary condition $F$ imposed by the equivariant Lagrangians $\mathcal{L}_0, \dots, \mathcal{L}_d$.
We have a Banach vector bundle $\mathscr{E} \to \mathscr{X}$ with fiber at $(u, y,  J) \in \mathscr{X}$ given by
$$\mathscr{E}_{u,y, J}=W^{k-1,p}(\Sigma, \Omega^{0,1}(u^*TM)).$$
We also have a Banach vector bundle $\mathscr{E}^G \to \mathscr{X}^G$ over the fixed points, with fiber $$\mathscr{E}_{u,y, J}^G=W^{k-1,p}(\Sigma, \Omega^{0,1}(u^*TM))^G$$ at $(u,y,J) \in \mathscr{X}^G$, consisting of $G$-invariant sections of $\Omega^{0,1}(u^*TM)$. The vector bundle $\mathscr{E}^G$ is obtained from $\mathscr{E}$ by restricting to $\mathscr{X}^G$ and taking the $G$-invariant subbundle. The action on a $1$-form $\eta$ is given by
\begin{equation}\label{eq: action-on-1-forms}
    g^*\eta_z(V)=dg_{u(g^{-1}z)}\circ\eta_{g^{-1}z}\circ dg_z^{-1}[V],
\end{equation}
where we rely on the presence of the $G$-action on $u^*TM$ for an equivariant map $u \colon \Sigma \to M$.
\begin{remark}
    Strictly speaking, one has to work with a Banach space of almost complex structures $\mathcal{J}_{\bm h, d}(M)$. To remedy this, one can replace $\mathcal{J}$ from~\eqref{eq: almost complex structures} with $\mathcal{J}^{\varepsilon}$, defined by using $Z \in C^{\varepsilon}(\operatorname{End}TM)$ in the space of Floer's $C^{\varepsilon}$-sections for some sequence $\varepsilon=(\varepsilon_i)_{i \in \Z_{\ge1}}$. We do not go into further details and omit these technicalities. For more details, we refer the reader to \cite[Appendix B]{wendl2020sft}.
\end{remark}
 The moduli space of solutions to the Floer equation is locally given by
\begin{equation}
    \mathscr{M}=\mathfrak{s}^{-1}(0) \subset \mathscr{X},
\end{equation}
where $$\mathfrak{s}(u, y, J)=\overline{\partial}_{J, H, j_y}u=(du-X_{\tilde{H}} \otimes \tilde{\alpha})^{0,1}_{(J, j_y)}.$$
Here we denote by $\tilde{H}$ and $\tilde{\alpha}$ the pullbacks 
along the projection $$\mathfrak{S}^G_{0, \bm h, d} \to \mathfrak{C}^G_{0, \bm h, d}$$
of fixed data of a domain-dependent Hamiltonian $H \in \mathcal{H}_{\bm h,d}(M)$ and a basic $1$-form $\alpha$ as in Definition~\ref{defn: floer data}. The domain-dependent Hamiltonian $\tilde{H}$ and the 1-form $\tilde{\alpha}$ are $G$-invariant when restricted to $(\Sigma, j_y)$.
The space of $G$-invariant solutions is denoted $\mathscr{M}^G=\mathscr{M} \cap \mathscr{X}^G$. We recall that the linearization of $\mathfrak{s}$ along $\mathscr{M}$ is given by the formula
\begin{equation}\label{eq: full-linearized-d-bar}
     D^{\mathscr{X}}_{u, y, J}(\xi, \upsilon, Z) \coloneqq D_{u, y}\xi+\tfrac{1}{2}Z(u,z, y) \circ (du-X_{\tilde{H}} \otimes \tilde{\alpha}) \circ j_y+\tfrac{1}{2}J\circ (du-X_{\tilde{H}} \otimes \tilde{\alpha})\circ\upsilon,
\end{equation}
where $D_{u, y}$ is the linearization of  the operator $\overline{\partial}_{J, H, j_y}$. 
\begin{lemma}\label{lemma: basic-surjectivity-on-tangents} $\mbox{}$
   \begin{enumerate}
       \item  The spaces $\mathscr{M}$ and $\mathscr{M}^G$ are smooth Banach manifolds, whose tangent spaces are given by
    \begin{align}
        T_{(u, y, J)}\mathscr{M}&=\{(\xi, \upsilon, Z) \in  T_u \mathscr{B} \oplus T_yU_\pi \oplus T_{J} \mathcal{J}_{\bm h,d}(M) \mid D^{\mathscr{X}}_{u, y, J}(\xi, \upsilon, Z)=0  \},\\ 
        T_{(u, y, J)}\mathscr{M}^G&=\{(\xi, \upsilon, Z) \in T_u \mathscr{B}^G \oplus T_yU_\pi \oplus T_{J} \mathcal{J}_{\bm h,d}(M)^G \mid D^{\mathscr{X}}_{u, y, J}(\xi, \upsilon, Z)=0\}. 
    \end{align}
\item The following restrictions of $D^{\mathscr{X}}$
\begin{gather}
    D^{\mathscr{X}}_{u,y,J} \colon T_u \mathscr{B} \oplus T_{J} \mathcal{J}_{\bm h,d}(M) \to \mathscr{E}_{u,y, J}, \quad (u, y, J) \in \mathscr{M},
\\
D^{\mathscr{X}}_{u,y,J} \colon T_u \mathscr{B}^G \oplus T_{J} \mathcal{J}_{\bm h,d}(M)^G \to \mathscr{E}_{u,y, J}^G, \quad (u, y, J) \in \mathscr{M}^G,
\end{gather}
are surjective, and $D^\mathscr{X}$ is equivariant for $(u, y, J) \in \mathscr{X}^G$.
     \end{enumerate}
\end{lemma}

\begin{proof}
    Let us denote by $(\pi \colon \Sigma \to C; q_1, \dots, q_b)$ the element of the Hurwitz moduli space associated with any curve $u$ under discussion.
    The fact that $\mathscr{M}$ is smooth is standard, and is a combination of the proof of \cite[Proposition 4.3]{CieliebakMohnke2007} and the proof presented in \cite[Section (9k)]{seidel2008fukaya}. Let us briefly spell out the main difference: since we only use the perturbation of an almost complex structure, in order to apply a unique continuation type argument, we need to be able to show that the term of the form
    $$Z \circ (du-X_{\tilde{H}} \otimes \tilde{\alpha}) \circ j_y$$
    can be picked to be sufficiently arbitrary on some open subset of the domain $\Sigma$ in order to use it to approximate a Dirac $\delta$-type section. Since $\bm h \neq \emptyset$ or $d>2$, the curve $u$ is nonconstant on an open dense subset, and by intersecting this open set with the preimage of the auxiliary open subset $U \subset C$ where $\alpha$ vanishes (see Definition~\ref{defn: floer data}), we obtain the open subset where the above form takes shape $Z \circ du \circ j_y,$
    with $Z$ arbitrary, hence can be used for the above-mentioned purposes.
    
    Since the action of $G$ on $\mathscr{X}$, and hence on $\mathscr{M}$, is smooth, the fixed point set $\mathscr{M}^G$ of the restriction of this action on $\mathscr{M}$ is a smooth submanifold. The description of the tangent spaces follows immediately since $D^\mathscr{X}$ is the linearization of the section $\mathfrak{s}$ along $\mathscr{M}$. The proof of the first surjectivity claim in (2) is also given in the proof of \cite[Proposition 4.3]{CieliebakMohnke2007}, and the surjectivity of the restriction to the $G$-invariant parts follows by averaging once we establish the equivariance of $D^\mathscr{X}$. 

    We show the equivariance of $D^{\mathscr{X}}_{u,y,J}$ for $(u,y,J) \in \mathscr{X}^G$ by verifying that each of the terms in~\eqref{eq: full-linearized-d-bar} is equivariant.
First, using~\eqref{eq: J-equivariance} and~\eqref{eq: action-on-1-forms} we see
    \begin{gather}
        g^*(Z(u,z,y))\circ (du-X_{\tilde{H}} \otimes \tilde{\alpha}) \circ j_y)_z[V]=(g_* \circ Z(u, g^{-1}z, y) \circ (du-X_{\tilde{H}} \otimes \tilde{\alpha}) \circ j_y \circ g_*^{-1})[V] \nonumber\\
        \stackrel{(\text{\ding{106}})}{=}(g_* \circ Z(u, g^{-1}z, y) \circ g_{*}^{-1} \circ (du-X_{\tilde{H}} \otimes \tilde{\alpha}) \circ j_y )[V]=
        (g_*(Z)(u, z, y) \circ (du-X_{\tilde{H}} \otimes \tilde{\alpha}) \circ j_y)[V],
    \end{gather}
    where $(\text{\ding{106}})$ follows from the equivariance of $u, {\tilde{H}}, \tilde{\alpha}$ and $j_y$. The $G$-invariance of the last term in~\eqref{eq: full-linearized-d-bar} follows immediately from $G$-invariance of $(J, \tilde{H}, \tilde{\alpha})$, $u$, and $\upsilon$.

By extending the computation in \cite[Equation (3.1.2)]{mcduffsalamon2004}, one obtains the following formula for $D_{u,y}\xi$ in local coordinates
\begin{equation}\label{eq: linear-d-bar-local-formula}
    D_{u,y} \xi=\overline{\partial}_{J, j_y}\xi-\tfrac{1}{2}(J \partial_\xi J)(u)\partial_{J,j_y, H}(u)-(\partial_\xi X_{\tilde{H}} \otimes \tilde{\alpha})^{0,1}.
\end{equation}
It suffices to consider the local situation where
$u$ is a map from a disjoint union of finitely many disks, regarded as a subdomain of $\Sigma$, to several copies of $\R^{2n}$, with the subgroup of $G$ fixing a given copy acting linearly.

Let us show that the first term $\overline{\partial}_{J, j_y}\xi=\tfrac{1}{2}(d\xi-J\circ d\xi \circ j_y)$ is $G$-equivariant. Indeed,
\begin{equation}
    d(g_*\xi)[V]=d(g\circ \xi \circ g^{-1})[V]\stackrel{(\text{\ding{105}})}{=}g_*d\xi[g^{-1}_*V]=g^*d\xi[V],
\end{equation}
where $(\text{\ding{105}})$ follows from the chain rule and the linearity of the $G$-action on $u^*TM$, and
\begin{align}\label{eq: equivariance-of J-j}
    g^*(J \circ d\xi \circ j_y)[V]=&dg \circ J\circ d\xi \circ j_y \circ dg^{-1}[V]\stackrel{(\text{\ding{109}})}{=}J\circ dg \circ d\xi \circ dg^{-1} \circ j_y[V]\\ \nonumber
    =&J\circ g^*(d\xi) \circ j_y[V], 
\end{align}
where $(\text{\ding{109}})$ follows from the equivariance of $J$ and $j_y$. 

To show the equivariance of the second term in~\eqref{eq: linear-d-bar-local-formula} it suffices to verify the equivariance of $\partial_\xi J=dJ(\xi)$, i.e., to check that
\begin{equation}
    dJ[dg(\xi)]_{gz}=dg\circ(dJ[\xi]_z) \circ dg^{-1} 
\end{equation}
for all $z\in \Sigma$. It follows by differentiating the condition~\eqref{eq: J-equivariance} on the equivariance of $J$ itself.

It remains to show the equivariance of the third term in~\eqref{eq: linear-d-bar-local-formula}. 
Using the equivariance of $\tilde{H}$, an analogous argument shows that $\partial_{\xi} X_H$ is equivariant in $\xi$, i.e.,
\begin{equation}
    (\partial_{dg(\xi)}X_{\tilde{H}})_{gz}=dg([\partial_\xi X_{\tilde{H}}]_z)
\end{equation}
for any $z\in \Sigma$. Combining this with the equivariance of $\tilde{\alpha}$ one obtains that the $1$-form $\partial_\xi X_{\tilde{H}} \otimes \tilde{\alpha}$ is $G$-equivariant, and hence its $(0,1)$-part by a computation analogous to the one in~\eqref{eq: equivariance-of J-j}.
\end{proof}

\begin{remark}
    The proof of equivariance can also be carried out using an extension of the global expression for $D_{u,y}$ from \cite[Proposition 3.1.1]{mcduffsalamon2004} by utilizing the Levi-Civita connection induced by an equivariant metric on $M$. One can avoid direct computations entirely, since the above is a variation of an equivariant section $\mathfrak{s}$, but we find this direct computation illustrative.
\end{remark}

\begin{remark}
    Equivariant transversality of this kind has been studied before. Besides the already mentioned \cite{MMMR2019}, the case of closed curves was considered in \cite{rothgang2024}. In particular, \cite[Theorem 4.3]{rothgang2024} shows that equivariant transversality can be achieved by picking a generic $G$-invariant almost complex structure on $M$ which is not domain-dependent.
\end{remark}

\subsubsection{Extending the transversality to jet evaluations}

In what follows, we assume that
\[
    \bm h=((h_1),\ldots,(h_b)), \qquad h_i=g_i^{k_i},
\]
where $g_i\in G$ has order $m_i$ and $M^{g_i}=M^{h_i}$ with the stabilizer of a generic point of $M^{g_i}$ being the cyclic subgroup $\langle g_i \rangle$. We set 
\begin{equation}
    d_i \coloneqq \operatorname{ord}(h_i)=\frac{m_i}{\gcd(k_i, m_i)}.
\end{equation}
A point $\tilde{p} \in \pi^{-1}(q_i)$ determines an element $\tilde{h}_i \in (h_i)$, which generates the stabilizer $G_{\tilde{p}}\cong \mu_{d_i}$ of $\tilde{p}$, and acts via multiplication by $\zeta_{d_i}$ on $T_{\tilde{p}}\Sigma$; see \cite[Remark 2.30]{Galeotti2022}. Any equivariant map $u \colon \Sigma \to M$ takes $\tilde{p}$ to a point in $\mathcal{Y}_i \subset M^{\tilde{h}_i}=M^{\tilde{g}_i}$ for some connected component $\mathcal{Y}_i$ and $\tilde{g}_i \in (g_i)$. 
Hence the inclusion 
\begin{equation}
    \rho_i \colon G_{\tilde{p}} \hookrightarrow G_{u(\tilde{p})}, \quad \tilde{h}_i \mapsto \tilde{g}_i^{k_i}
\end{equation}
is the homomorphism at the point $q_i$ associated with the orbifold map $$[u] \colon \mathscr{C} \to [M/G],$$ where we assume that $u(\tilde{p})$ is generic in $M^{\tilde{g}_i}$ and hence $G_{u(\tilde{p})} \cong \mu_{m_i}$. 
 
The inertia evaluation $\mathcal{E}^{\mathcal{I}}([u])$ therefore belongs to the component
 \begin{equation}
     \bm m=((\mathcal{Y}_1, (g_{1}^{k_1})), \dots, (\mathcal{Y}_b, (g_{b}^{k_b}))).
 \end{equation}
The induced inertia map $[u]^{\mathcal{I}}$ maps a pair $(q_i, (\tilde{h}_i)) \in \mathcal{I}\mathscr{C}$
to $$([u(\tilde{p})], (\rho_i(\tilde{h}_i))) \in \mathcal{I}[M/G].$$ We emphasize that $[u]^{\mathcal{I}}$ does not depend on the choice of $\tilde{p} \in \pi^{-1}(q_i)$.

The choice of $\tilde{p}$ gives rise to a \emph{jet evaluation} map at $q_i$
\begin{gather}\label{eq: jet-evaluation-map-on-moduli}
    \mathcal{J}ev_{q_i} \colon\mathscr{M}^G \to \J^{l, J_0, \rho_i}(\Sigma, M),\\
    (u,y,J) \mapsto j^{l}_{\tilde{p}}u, \nonumber
\end{gather}
and this map is well-defined by elliptic regularity.
The differential of this map is given by
\begin{gather}
    D\mathcal{J}ev_{q_i} \colon T_{(u, y, J)}\mathscr{M}^G\to T_{j^{l}_{\tilde{p}}u}\J^{l, J_0, \rho_i}(\Sigma, M),\\
    (\xi, \upsilon, Z) \mapsto j^{l}_{\tilde{p}}\xi.
\end{gather}

More generally, given a subset $I \subset \{1, \dots, b\}$, a collection of nonnegative integers $(l_i)_{i \in I}$, and a collection of points $(\tilde{p}(q_i) \in \pi^{-1}(q_i))_{i\in I}$, there is a jet evaluation map
\begin{gather}
 \mathcal{J}ev_{I} \colon \mathscr{M}^G \to \prod_{i \in I}\J^{l_i, J_0, \rho_i}(\Sigma, M), \\
    \mathcal{J}ev_{I}(u, y, J)=(j^{l_i}_{\tilde{p}(q_i)}u)_{i \in I} . \nonumber
\end{gather}

\begin{remark}
    There is an ambiguity in the choice of a point $\tilde{p} \in \pi^{-1}(q_i)$ which is slightly unsatisfactory. Instead, one could talk about $G$-equivariant jets over $q_i$ by looking at $G$-invariant subset of
    $$\prod_{g \in G} \J^{l, J_0}(\Sigma,M),$$
    with the $G$-action given by
    $$g((j^lu_h))_{h \in G}=(j^l(g\circ u_h \circ g^{-1}))_{h\in G}.$$
 It is straightforward to verify that the fiber above $q_i \in C$ of the above $G$-invariant set is naturally isomorphic to $\J^{l, J_0, \rho_i(z)}(\Sigma, M)$ for any $\tilde{p} \in \pi^{-1}(q_i)$. 
\end{remark}
 We set $r_{ij}\in \{1, \dots, d_i-1\}$ to satisfy 
\begin{equation} \label{eqn: rho-equivariance}
    \zeta_{d_i}^{r_{ij}}=\chi^{m_i}_j(\rho_i(\zeta_{d_i}))
\end{equation} for $j=1, \dots, {m_i-1}$. 
We denote by $\mathscr{M}^G(q_i, \chi^{m_i}_j)$ the preimage of the section 
\begin{equation}
    s^{r_{ij}}_c \colon \mathcal{Y}_i \to  \J^{r_{ij}, J_0, \rho_i}(\Sigma, V_{\chi^{m_i}_j}^i),
\end{equation}
under the composition of jet evaluation $\mathcal{J}ev_{q_i}$ with the projection
$$\pi^{r_{ij}, \chi_j^{m_i}} \colon \mathfrak{J}^{l, J, \rho_i}(\Sigma, U_i) \to \J^{r_{ij}, J_0, \rho_i}(\Sigma, V_{\chi^{m_i}_j}^i),$$
where we use the notation from Section~\ref{section: equivariant-j-jets}. Here we denote by $V_{\chi^{m_i}_j}^i$ the auxiliary isotypic component submanifold associated with $\mathcal{Y}_i$ and by $U_i$ the tubular neighborhood admitting projections onto these components.

Similarly, given a subset $I \subset \{1, \dots, b\}$, a collection of connected components $(\mathcal{Y}_i \subset M^{\tilde{h}_i})_{i \in I}$, and a collection of characters $\chi_I=(\chi^{m_i}_{j(i)})_{i \in I}$ for each $i \in I$, we denote by $\mathscr{M}^G(I, \chi_I)$ the preimage of the section
\begin{equation} 
s_c^{I, \chi_I} \colon \prod_{i \in I}\mathcal{Y}_i \to \prod_{i \in I}\J^{r_{ij(i)}, J_0, \rho_i}(\Sigma, V_{\chi^{m_i}_{j(i)}}^i)
\end{equation}
under composition of $\mathcal{J}ev_I$ with a tuple of projections $(\pi^{r_{ij(i)}, \chi^{m_i}_{j(i)}})_{i \in I}$.

\begin{proposition}\label{prop: jet-evaluation-submersion}
    The evaluation map \begin{equation}\label{eq: isotypic-jet-evaluation}
    \pi^{r_{ij}, \chi_j^{m_i}} \circ\mathcal{J}ev_i \colon \mathscr{M}^G \to \J^{r_{ij}, J_0, \rho_i}(\Sigma, V_{\chi^{m_i}_{j}}^i)
    \end{equation}
    is a submersion. In particular, $\mathscr{M}^G(q_i, \chi^{m_i}_j)$ is a smooth Banach submanifold of $\mathscr{M}^G$ of codimension 
    \begin{equation}
        \operatorname{codim}\mathscr{M}^G(q_i, \chi^{m_i}_j)=\operatorname{dim}_{\R}V^i_{\chi^{m_i}_j}.
    \end{equation} 
    Moreover, $\mathscr{M}^G(I, \chi_I)$ is a smooth Banach submanifold of codimension 
    \begin{equation}
        \sum_{i \in I} \operatorname{dim}V_{\chi^{m_i}_{j(i)}}^i.
    \end{equation}
\end{proposition}

\begin{proof}
    We adapt the proof of \cite[Lemma 6.5]{CieliebakMohnke2007} to the equivariant setting. 
   
    Let $k-2/p>k'+1$ for some $k' \ge 1$ and $p>2$.  We first claim that for each $(u, y, J) \in \mathscr{M}$ there is a surjective linear operator    \begin{gather*}
    \mathcal{F}^{k'} \colon B^{k,p}_{k'} \oplus T_J \mathcal{J}_{\bm h,d}(M) \to E^{k,p}_{k'},\\
    (\xi, Z) \mapsto D_{u,y}\xi+\tfrac{1}{2}Z(u)\circ (du-X_{\tilde{H}} \otimes \tilde{\alpha}) \circ j_y,
    \end{gather*}
    where $B^{k,p}_{k'} \subset T_{u}\mathscr{B}$ is a subset consisting of those sections $\xi \in W^{k,p}_F(\Sigma, u^*TM)$ for which the $k'$-jet $j^{k'}_{\tilde{p}'}\xi$ vanishes for each $\tilde{p}' \in \pi^{-1}(q_i)$, and $E^{k,p}_{k'} \subset \mathscr{E}_{(u,y,J)}$ consists of $1$-forms $\eta$, whose $(k'-1)$-jet $j^{k'-1}_{\tilde{p}'} \eta$ vanishes for each $\tilde{p}' \in \pi^{-1}(q_i)$.  We closely follow the proof of \cite[Lemma 6.6]{CieliebakMohnke2007}. The first step in the proof \cite[Lemma 6.6]{CieliebakMohnke2007} in our setting goes through as in our proof of Lemma~\ref{lemma: basic-surjectivity-on-tangents}, which allows one to show that any distribution in the orthogonal complement of $\operatorname{Im}(\mathcal{F}^{k'})$ is supported on $\pi^{-1}(q_i)$, and the remainder of the claim proceeds verbatim.

    Since $\mathcal{F}^{r_{ij}}$ is the restriction of $D^{\mathscr{X}}$ to a subrepresentation, Lemma~\ref{lemma: basic-surjectivity-on-tangents} implies that $\mathcal{F}^{r_{ij}}$ is $G$-equivariant and, therefore, we have a surjective linear operator 
    $$
    \mathcal{F}^{r_{ij}} \colon (B^{k,p}_{r_{ij}})^G \oplus T_J \mathcal{J}_{\bm h,d}(M)^G \to (E^{k,p}_{r_{ij}})^G
    $$
    for any $(k,p)$ such that $k-2/p>r_{ij}+1$, $k \ge 1$ and $p>2$.

    Now, given $(u, y, J) \in \mathscr{M}^G$, we consider $G_{u(\tilde{p})}$-equivariant local coordinates $\C^{n_0} \times \dots \times\C^{n_{m_i-1}}$ near $u(z) \in M$ with $u(\tilde{p}(q_i))=0$, where $\C^{n_0}$ is a local coordinate of $\mathcal{Y}_i$ and $\C^{n_s}$ trivializes the fibers of $V^i_{\chi^{m_i}_s}$. Notice that in these coordinates $J_0$ is standard at the origin, which corresponds to $u(\tilde{p})$. Let $v \in \C^{n_j}$ be an arbitrary vector, and first consider 
    $$\tilde{\xi}(z)=z^{r_{ij}}v$$
    on a small neighborhood $D$ of $\tilde{p}(q_i)$. Observe that $\tilde{\xi} \in W^{k,p}(D, u^*TM)$ is $\rho_i$-equivariant since
    $$\tilde{\xi}(\tilde{h}_iz)=\zeta_{d_i}^{r_{ij}}z^{r_{ij}}v=z^{r_{ij}}\chi^{m_i}_j(\rho_i(\zeta_{d_i}))v=\chi^{m_i}_j(\tilde{g}_i^{k_i})\tilde{\xi}(z)$$
    by \eqref{eqn: rho-equivariance}.
    We then extend $\tilde{\xi}$ to the $G$-orbit $G \cdot D \subset \Sigma$ in a natural way to guarantee its $G$-equivariance on $G \cdot D$, and further extend it to all of $\Sigma$, while preserving its $G$-equivariance. As in the proof of \cite[Lemma 6.6]{CieliebakMohnke2007}, the section $\tilde{\eta}=D_u\tilde{\xi}$ has a vanishing $(r_{ij}-1)$-jet, and, moreover, it is $G$-invariant. Hence, there is a pair $$(\xi', Z) \in (B^{k,p}_{r_{ij}})^G \oplus T_J \mathcal{J}_{\bm h,d}(M)^G$$ such that $\mathcal{F}^{r_{ij}}(\xi', Z)= -\tilde{\eta}$. We conclude that $(\tilde{\xi}+\xi', 0, Z) \in T_{(u, y, J)} \mathscr{M}^G$ and $$D \mathcal{J}ev_i(\tilde{\xi}+\xi', 0, Z)=(v, Jv, \dots,  J^{r_{ij}}v),$$ 
    where we use the coordinates explained in Remark~\ref{remark: tangent-jet-space-first}. This proves surjectivity of $D(\pi^{r_{ij}, \chi_j^{m_i}} \circ\mathcal{J}ev_i)$, which implies the first and second statements.

    The proof of the last statement is just the adaptation of the proof above to the case of multiple points, and we leave the details to the reader.
\end{proof}

\begin{remark}
    The choice of the auxiliary open set where the form $\alpha$ vanishes allowed us to avoid using Hamiltonian perturbations as in \cite[Section (9k)]{seidel2008fukaya}. The only reason we do this is to simplify the exposition; an analogous argument works for Hamiltonian perturbations.
\end{remark}

 \begin{corollary}\label{cor: moduli-spaces-with-condtions}
     There exists a comeager subset $\mathcal{J}^{\operatorname{reg}}_{\bm h,d}(M) \subset \mathcal{J}_{\bm h,d}^G(M)$ such that for any $J \in \mathcal{J}^{\operatorname{reg}}_{\bm h,d}(M)$ the space $\mathcal{M}(\overrightarrow{\hat{x}}, \hat{x}_0 ;\bm m)$ is a smooth manifold of dimension 
     \begin{equation}
         |\hat{x}_0|-|\overrightarrow{\hat{x}}|-2\iota(\bm m)+2b+(d-2).
     \end{equation}
     Moreover, for any $I \subset [b]$ the moduli space
     $$\mathcal{M}(\overrightarrow{\hat{x}}, \hat{x}_0 ; \bm m, (I, \chi_I))$$ 
     is its smooth submanifold of codimension 
     \begin{equation}
         \sum_{i \in I}\operatorname{dim}V^i_{\chi^{m_i}_{j(i)}}.
     \end{equation}
 \end{corollary}
 
\begin{proof}
    This is an immediate consequence of Proposition~\ref{prop: jet-evaluation-submersion} and the Sard-Smale theorem. 
\end{proof}

Consider the Hurwitz moduli space $\mathscr{H}_{0, (\bm h, \ell), d}^G$ of $G$-covers $(\pi \colon \Sigma \to C; q_1, \dots, q_b)$ with $\ell$ additional marked points on $C$. In other words, the underlying orbicurve $\mathscr{C}$ comes equipped with $\ell$ regular markers. This moduli space, among other things, comes equipped with the space of domain-dependent almost complex structures $\mathcal{J}^{\operatorname{reg}}_{\bm h, d, \ell}(M)$ defined in a similar way as $\mathcal{J}^{\operatorname{reg}}_{\bm h,d }(M)$.
Let us fix a collection of almost complex $G$-invariant submanifolds (not necessarily connected) $\bm Z=(Z_1, \ldots , Z_\ell)$. In addition, with $\bm m$ as above, we consider some submanifolds $Z'_1 \subset \mathcal{Y}_1, \dots, Z_b' \subset \mathcal{Y}_b$, and denote their tuple by $\bm Z'$.
We consider the moduli space 
\begin{equation}
    \mathcal{M}_{\ell}(\overrightarrow{\hat{x}}, \hat{x}_0 ; \bm m, (I, \chi_I), \bm Z', \bm Z)
\end{equation}
defined similarly as $\mathcal{M}(\overrightarrow{\hat{x}}, \hat{x}_0 ; \bm m, (I, \chi_I))$ using  $\mathscr{H}_{0, (\bm h, \ell), d}^G$ as the space of domains and requiring the condition of passing through $Z_i$ at the preimage of the $i$th marker on the orbicurve, and also requiring passing through the $G$-orbit of $Z'_i$ at the points in the preimage of $q_i$. Using a simpler version of the proof above, one obtains the following:

\begin{proposition}\label{prop: proposition-jets-transversality-with-submanifolds}
    There is a comeager subset $\mathcal{J}^{\operatorname{reg}}_{\bm h, d, \ell}(M) \subset \mathcal{J}_{\bm h, d, \ell}^G(M)$
    such that for any $J \in \mathcal{J}^{\operatorname{reg}}_{\bm h, d, \ell}(M)$ the space $\mathcal{M}_{ \ell}(\overrightarrow{\hat{x}}, \hat{x}_0 ; \bm m, (I, \chi_I), \bm Z)$ is a smooth submanifold of the manifold $\mathcal{M}(\overrightarrow{\hat{x}}, \hat{x}_0 ;\bm m)$ of codimension
    \begin{equation}
        \sum_{i \in I} \operatorname{dim}V^i_{\chi^{m_i}_{j(i)}}+\left(\sum_{i=1}^{\ell}\operatorname{codim}_MZ_i-2\ell\right)+\sum_{i=1}^{b}\operatorname{codim}_{\mathcal{Y}_i}(Z_i').
    \end{equation}
\end{proposition}

\section{Obstructions to smoothing curves with ghosts on singularities}\label{section: obstructions to smoothings}

\subsection{Regularization near the fixed points}\label{neighborhoods-of-fixed-sets}

We first introduce the notion of an almost complex structure that is \emph{adapted to the singular locus}. 

\begin{definition}\label{defn: adapted J}
    A $G$-equivariant almost complex structure $J$ on the Liouville manifold $(M,\lambda)$ is \emph{adapted} to a collection $\{(\mathcal{Y}_i, (g_i))\}_{i\in \mathcal{I}}$ of inertia components of $\mathcal{I}[M/G]$ if the following hold:
    \begin{enumerate}
        \item $J$ is $\omega=d\lambda$-tame on $M$ and of contact type on a neighborhood of the boundary at infinity $\partial_{\infty}M$;
        \item for each $(\mathcal{Y}_i, (g_i))$ with isotypic decomposition of $TM|_{\mathcal{Y}_i}$ (as a $J$-complex bundle) with respect to the subgroup $\langle g_i\rangle$ of order $m_{\mathcal{Y}_i}$, given by
        \begin{equation}
            TM|_{\mathcal{Y}_i}=T\mathcal{Y}_i \oplus \bigoplus_{j=1}^{m_{\mathcal{Y}_i}-1}\nu_{\chi_j^{m_{\mathcal{Y}_i}}},
        \end{equation}
        there exist (potentially empty) almost complex $G$-invariant immersed submanifolds (see Figure~\ref{fig:directions})
        \begin{equation}
            V_{\chi_1^{m_{\mathcal{Y}_i}}}, \ldots, V_{\chi_{m_{\mathcal{Y}_i}-1}^{m_{\mathcal{Y}_i}}} \subset U_{\mathcal{Y}_i},
        \end{equation}
        where $U_{\mathcal{Y}_i}$ is a $G$-invariant tubular neighborhood of $G\cdot \mathcal{Y}_i^{\circ}$ \footnote{Recall that $\mathcal{Y}_i^{\circ}$ is the top stratum of the fixed point stratification restricted to $\mathcal{Y}_i$.}, such that each $V_{\chi_j^{m_{\mathcal{Y}_i}}}$ contains $G\cdot \mathcal{Y}_i^{\circ}$, and for each $j=1, \dots, m_{\mathcal{Y}_i}-1$ 
        \begin{equation}
(TV_{\chi_j^{m_{\mathcal{Y}_i}}})|_{\mathcal{Y}^{\circ}_i}=T\mathcal{Y}^{\circ}_i \oplus\nu_{\chi_j^{m_{\mathcal{Y}_i}}};
        \end{equation}
        \item there is a collection of $G$-equivariant $J$-holomorphic projections
        \begin{equation}
            \pi_{j}^{\mathcal{Y}_i} \colon  U_{\mathcal{Y}_i}\to V_{\chi_j^{m_{\mathcal{Y}_i}}}.
        \end{equation}
        \end{enumerate}
\end{definition}

\begin{figure}[ht]
    \centering
    \includegraphics[width=.7\linewidth]{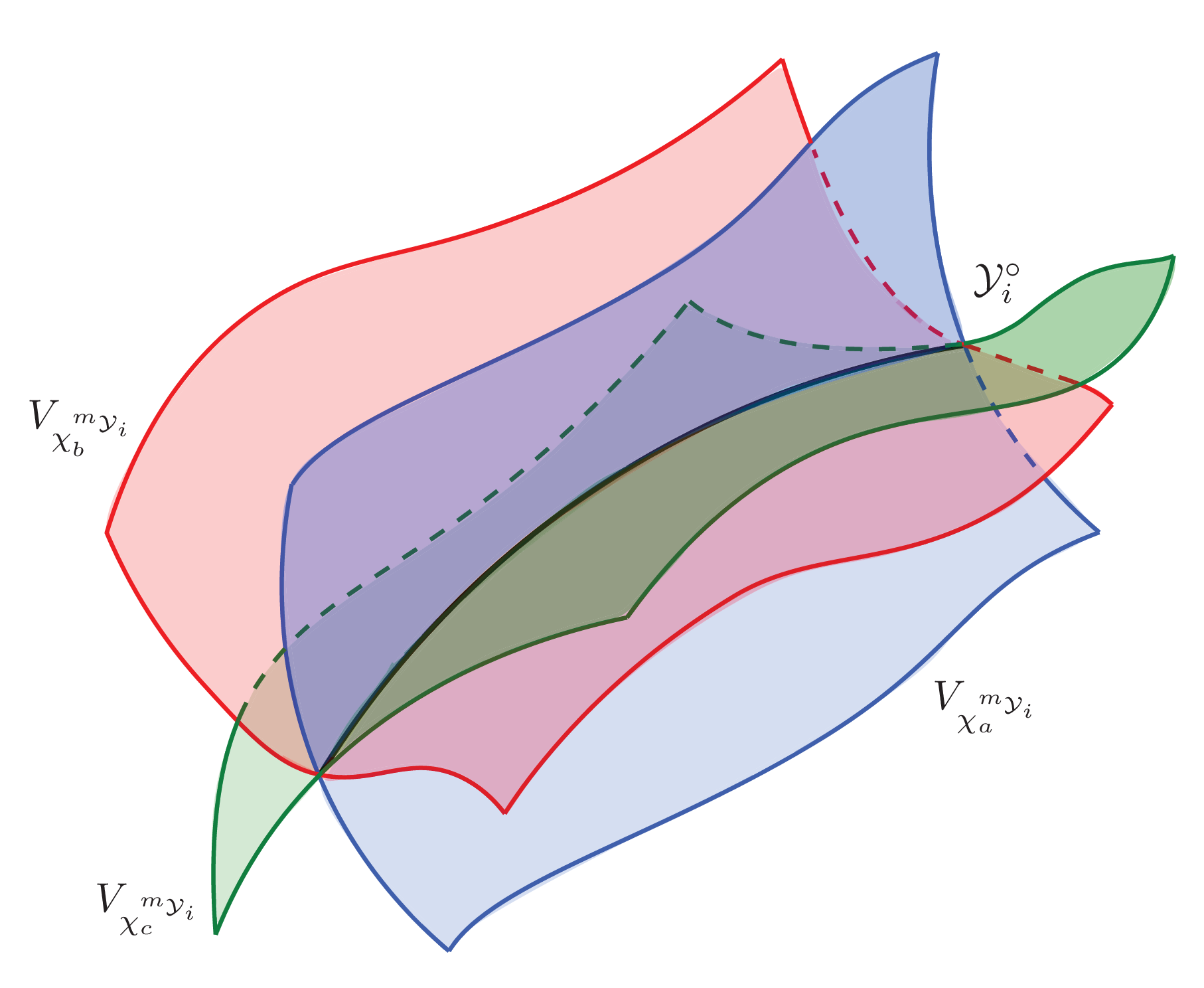}
    \caption{A schematic illustration of the collection of the auxiliary almost complex submanifolds $V_{\chi_j^{m_{\mathcal{Y}_i}}}$ near a point of $\mathcal{Y}_i^{\circ}$ away from deeper strata.}
    \label{fig:directions}
\end{figure}

\begin{figure}[ht]
    \centering
    \includegraphics[width=.7\linewidth]{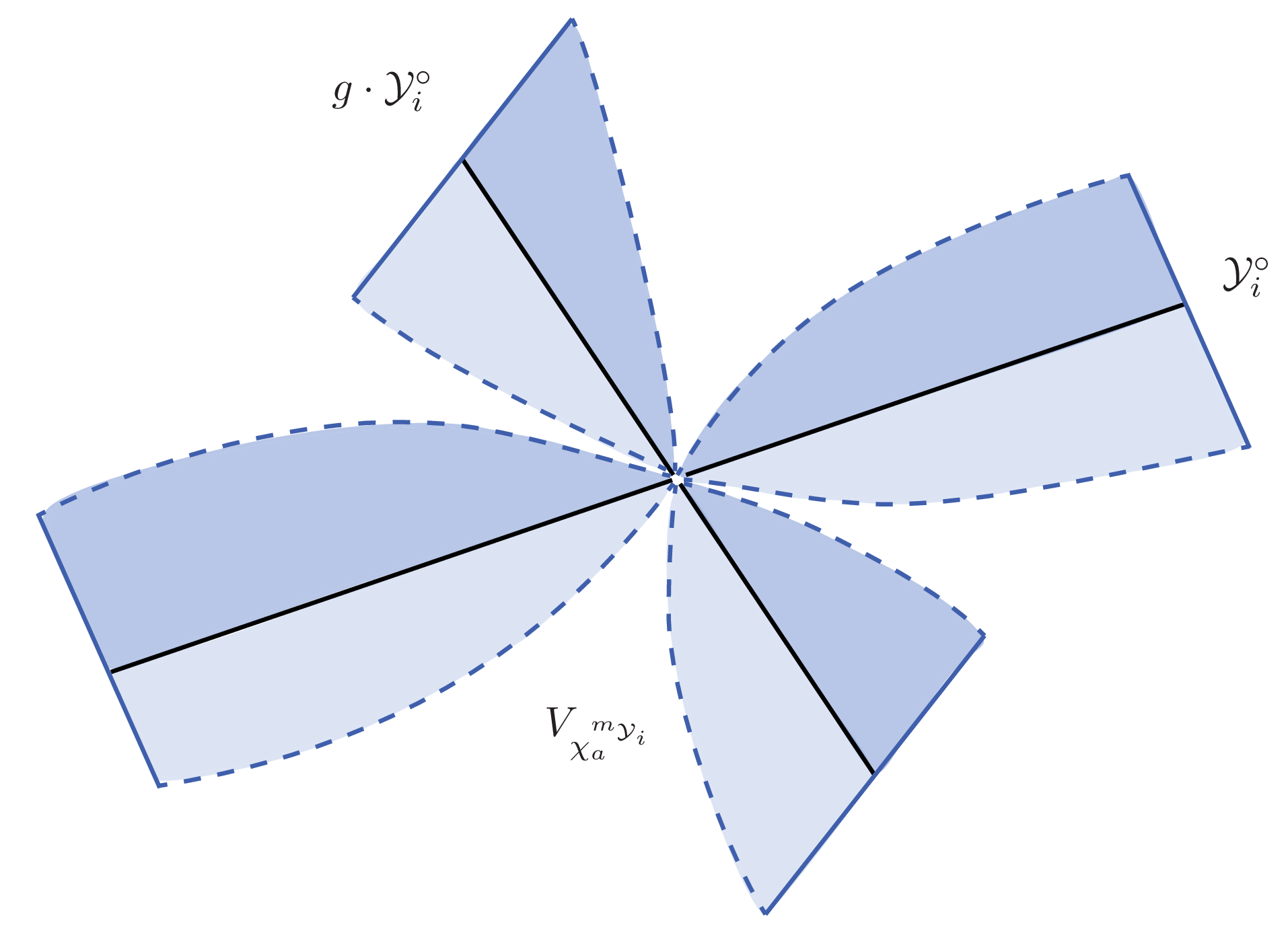}
    \caption{Near a stratum where two (or more) hypersurfaces intersect, one shrinks $V_{\chi_j^{m_{\mathcal{Y}_i}}}$ in order to guarantee that these are injectively immersed submanifolds.}
    \label{fig:components}
\end{figure}
The goal of this subsection is to prove that such a $J$ exists in the case of primary interest, i.e., for type $A$ inertia components associated with $T^*Y \subset M=T^*X$, where $X$ is a complex manifold equipped with an action of a finite group $G \subset \operatorname{Aut}(X)$ and $Y$ is an element of the collection of the fixed hypersurfaces of this action.

Fix a $G$-invariant Hermitian metric $g_X$ on $X$. For each connected hypersurface $Y \subset X^{\operatorname{sing}}$ with stabilizer $H=H(Y)$ we denote by $\nu_Y \subset TX|_Y$ the subbundle 
$$\pi_Y \colon \nu_Y \to Y$$
orthogonal to $TY$, which is also the only nontrivial isotypic component of $TX|_Y$. {\em In what follows, we write $\nu=\nu_Y$ to simplify notation.}  The subbundle $\nu$ provides a natural embedding of $T^*Y$ as a subset 
$$\mathcal{Y}= \{(y,p)\in T^*X \mid y\in Y, \, p(x)=0 \text{ for } x\in \nu\}.$$
Clearly, $\mathcal{Y} \subset T^* X^H$.
The fiberwise exponential map
\begin{equation}\label{eqn: j sub Y}
    j_Y \colon|\nu| \to X
\end{equation}
with respect to $g_X$ provides a diffeomorphism from a neighborhood of the $0$-section in the total space $|\nu|$ onto a tubular neighborhood of $Y$. 
Since $\nu$ is orthogonal to $TY$, the restriction $g_{\nu}$ of $g_X$ to $\nu$ is a Hermitian metric. 
Choose a complex linear connection $\nabla^{\nu}$ which preserves $g_\nu$. 
The metrics $g_X|_Y$, $g_{\nu}$, and the connection $\nabla^{\nu}$ naturally induce a Hermitian metric $g_{|\nu|}$ on $|\nu|$. We also have a splitting provided by $\nabla^{\nu}$
\begin{equation}\label{eq: splitting-for-nu}
    T|\nu|= \mathcal{H} \oplus \mathcal{V}
\end{equation}
into vertical and horizontal parts, where the vertical part is the pullback of $\nu$. Observe that $g_{|\nu|}$ is automatically $H$-equivariant, since $H$ is a finite group and acts via unitary transformations on $\nu$. 

The tangent space to the cotangent bundle $T^*|\nu|$  admits a splitting
\begin{equation}\label{eq: splitting-of-cotangent-to-total-space}
    TT^*|\nu|=T^*\mathcal{H} \oplus T^*\mathcal{V}=(\mathcal{H} \oplus \mathcal{V})^H \oplus (\mathcal{H}^* \oplus \mathcal{V}^*)^V.
\end{equation}
corresponding to the connection on $T^*|\nu| \to |\nu|$ given by the direct sum
\begin{equation}\label{eq: connection-for-cotangent}
    \pi_Y^*(\nabla^{\nu})^{\vee} \oplus \pi_Y^*(\nabla_Y)^{\vee}
\end{equation}
of pullbacks of connections dual to $\nabla^{\nu}$ and the Levi-Civita connection on $Y$. Naturally, the first two components correspond to the horizontal lift of~\eqref{eq: splitting-for-nu}, while the remaining components represent the vertical lift of their duals, which explains the superscript notation. For later use, we notice that the total space of $T^*|\nu|$ is naturally identified with the total space of $\pi_{T^*Y}^*\nu \oplus \pi_{T^*Y}^*\nu^{\vee} \to T^*Y$, and the splitting~\eqref{eq: splitting-of-cotangent-to-total-space} is also induced by the connection
\begin{equation}\label{eq: connection-splitting-copy}
    \pi^*_{T^*Y}\nabla^\nu \oplus \pi^*_{T^*Y}(\nabla^\nu)^{\vee}
\end{equation}
once the splitting induced by $\nabla_Y^{\vee}$ on the total space of $T^*Y$ is fixed.

We warn our reader that this is not the splitting into the horizontal and vertical parts of $TT^*|\nu|$ with respect to the metric that we put on $|\nu|$ and its Levi-Civita connection.  However, it has the following advantage: for the metric on the total space of $T^*|\nu|$ induced from the metric on $|\nu|$ and the connection~\eqref{eq: connection-for-cotangent}, a neighborhood of the $0$-section of the total space of any subbundle $\nu'$ of $$\pi_{T^*\nu} \colon T^*\nu=\pi_{T^*Y}^*\nu \oplus \pi_{T^*Y}^*\nu^{\vee} \to T^*Y$$ 
naturally provides a submanifold of $T^*|\nu|$ and its tangent space is given by 
$$T^*\mathcal{H} \oplus \pi^*_{T^*\nu}\nu'.$$
There is a split almost complex structure on $T^*|\nu|$ 
\begin{equation}\label{eq: split-acs}
    J_{\mathcal{Y}}=\begin{pmatrix}
        J_{T^*\nu} & 0 \\
        0 & J_{T^*Y}
    \end{pmatrix},
\end{equation}
which is $H$-equivariant with respect to the induced action on $T^*|\nu|$.
The bundle $T^*\nu$ splits into the direct sum
\begin{equation}\label{eq: splitting-of-normal-bundle}
    \nu^{(1,0)} \oplus \nu^{(0,1)},
\end{equation}
where $\nu^{(1,0)}$ is the graph of the operator
$$J_{T^*\nu}(-I) \colon \pi^*_{T^*Y}\nu \to \pi^*_{T^*Y}\nu^{\vee},$$
and $\nu^{(0,1)}$ is the graph of $J_{T^*\nu}I$, where $I$ is the complex structure on $\nu$.
The above discussion suggests that $\nu^{(1,0)}$ and $\nu^{(0,1)}$ naturally give rise to submanifolds of $T^*X$, both of which are diffeomorphic to the pullback $\pi_{T^*Y}^*\nu$ for the standard projection $\pi_{T^*Y} \colon T^*Y \to Y$. We denote them by $V_{\chi}$ and $V_{-\chi}$, where $\chi$ is the character of the action of $H$ on $\nu$.

We notice that with respect to $J_{\mathcal{Y}}$, these submanifolds are almost complex, and, moreover, we have $J_{\mathcal{Y}}$-holomorphic projections
\begin{equation}
    \pi_{\pm \chi} \colon T^*|\nu| \to V_{\pm \chi}
\end{equation}
induced from the splitting~\eqref{eq: splitting-of-normal-bundle}. The latter follows since for the connection~\eqref{eq: connection-splitting-copy} on $T^*\nu$ the subbundles $\nu^{(1,0)}$ and $\nu^{(0,1)}$ are parallel (as $\nabla^\nu$ is Hermitian), hence this connection splits with respect to the decomposition~\eqref{eq: splitting-of-normal-bundle}.

Unless $\nabla^{\nu}$ is flat, $J_{\mathcal{Y}}$ does not coincide with $J_{T^*|\nu|}$, and therefore, one should suspect that $J_{\mathcal{Y}}$ may fail to be compatible with the standard symplectic structure $\omega$. The following lemma shows that it satisfies the properties that we require.

\begin{lemma}
    The split almost complex structure $J_{\mathcal{Y}}$ is of contact type. Moreover, this almost complex structure is $\omega$-tame on a sufficiently small neighborhood of $T^*Y \subset T^*|\nu|$. 
\end{lemma}

\begin{proof}
    Consider a point $(x, v, p, w) \in T^*|\nu|$ and a tangent vector $$(X^H, U^H, \alpha^V, \xi^V) \in TT^*_{(x, v, p, w)}|\nu|$$
    with respect to our splitting~\eqref{eq: splitting-of-cotangent-to-total-space}. 
    First, we observe that the standard Hamiltonian function $h=\frac{1}{2}|\cdot|^2_{g_{|\nu|}}$ computed with respect to $g_{|\nu|}$ in these coordinates is given by 
    $$h(x,v,p,w)=\tfrac{1}{2}(|p|^2_{ g_{Y}}+|w|^2_{g_\nu}).$$
    Therefore,
    \begin{equation}
        dh_{(x,v,p,w)}(X^H, U^H, \alpha^V, \xi^V)=\langle p, \alpha\rangle+\langle w, \xi \rangle .
    \end{equation}
    On the other hand, the standard Liouville $1$-form $\lambda$ in these coordinates is given by
    \begin{equation}
        \lambda_{(x,v,p,w)}(X^H, U^H, \alpha^V, \xi^V)=p(X)+w(U).
    \end{equation}
    These two formulas imply the identity $\lambda \circ J_{\mathcal{Y}}=dh$.

    As for the tameness, one computes that $\omega=d\lambda$ in these coordinates for the pair of vector fields 
    $$\Phi=(X^H, U^H, \alpha^V, \xi^V) ~~\text{ and } ~~\Theta=(Y^H, Z^H, \beta^V, \eta^V)$$ is given by
    \begin{equation}
        \omega(\Phi, \Theta)=\beta(X)+\eta(U)-\alpha(Y)-\xi(Z)-w(R^{\nabla^{\nu}}(X,Y)v),
    \end{equation}
    where the last term comes from the formula for commuting the horizontal vector fields
    $$[X^{\mathcal{H}},Y^{\mathcal{H}}]=[X, Y]^{\mathcal{H}}-(R^{\nabla^\nu}(X,Y)v)^{\mathcal{V}},$$
    with respect to the decomposition~\eqref{eq: splitting-for-nu}.
    Hence, evaluating $\omega(\cdot, J_{\mathcal{Y}}\cdot)$ introduces one potentially negative term which depends on the curvature tensor $R^{\nabla^{\nu}}$, and is sufficiently small whenever the norms $|v|, |w|$ are small.
\end{proof}

\begin{remark}
    This proof shows that compatibility is out of reach since $\omega$ and $R^{\nabla^{\nu}}$ are of completely different natures. In the study of relative invariants with respect to a symplectic divisor, some miracles occur due to which one is able to pick the curvature of type $(1,1)$ with respect to an adapted $J$ as in \cite[Remark 6.2.6]{pomerleanoseidel2024}. Our Liouville submanifold $T^*Y$ is of codimension $4$, so there is no reason for such miracles. This is the main reason for working with tame almost complex structures in the context of orbifold Floer theory.
\end{remark}

\begin{remark}\label{remark: coordinates-for-isotypic-decomposition}
   Let us describe our almost complex submanifolds $V_{\pm \chi}$ using the above coordinates. These submanifolds are given by
   $$V_{\pm\chi}=\{(x, v_{\pm}, p, B_{\pm}v_{\pm}) \mid B_{\pm}=J_{T^*\nu}(\mp I)\}$$
   and the projections $\pi_{\pm \chi}$ are given by the formulas
   $$(x,v,p,w) \mapsto (x, \frac{1}{2}(v+B_{\pm}^{-1}w), p, \frac{1}{2}(B_{\pm}v+w)).$$
   Consider a local chart $\R^{2n-2}$ for $Y$. Then we get a local chart for $T^*|\nu|$ near $T^*Y$ given by  
   $$\C_{+} \times \C_{-} \times T^*\R^{2n-2},$$
   where the complex structures on the first two coordinates coincide with the pullback of $J_{\mathcal{Y}}$, and are given by $i_{\pm}=\pm I$ on $\C_{v_{\pm}}$, along $T^*Y$, the almost complex structure $J_{\mathcal{Y}}$ is then given by $i_+\oplus i_- \oplus J_{\R^{2n-2}}$, with $J_{\R^{2n-2}}$ arising from $T^*Y$.
   Let us now denote by $\alpha$ the pullback to $T^*Y$ of the $1$-form expressing the connection $\nabla^{\nu}$ in a local chart $\R^{2n-2}$. A direct computation shows that in these coordinates $J_{\mathcal{Y}}$ is given by the following matrix
   \begin{equation}
   \begin{pmatrix}
       i_+&0&v_+(\alpha+i_+\circ\alpha\circ J_{\R^{2n-2}}) \\
       0&i_-&v_-(\alpha+i_- \circ \alpha \circ J_{\R^{2n-2}}) \\
       0&0&J_{\R^{2n-2}}
   \end{pmatrix},
    \end{equation}
    and from this description it is immediate that projections $\pi_{\pm \chi}$ to $V_{\pm \chi}=\{v_{\pm}=0\}$ are $J_{\mathcal{Y}}$-holomorphic. Treating each $V_{\pm}$ as an almost complex divisor, we also see that $J_{\mathcal{Y}}$ is of the form specified in \cite[Equation (6.1.8)]{pomerleanoseidel2024} with respect to these divisors.
\end{remark}

Choose a smooth function $\delta_{\mathcal{Y}}:T^*Y\to \R_{\ge0}$ (and we feel free to shrink it for future constructions) such that $J_\mathcal{Y}$ is $\omega$-tame on the $\delta_\mathcal{Y}$-neighborhood of $T^*Y$. We define an almost complex structure $\widetilde{J}_{\mathcal{Y}}$ on this neighborhood  by setting it to be equal to $J_{\mathcal{Y}}$ on the $\delta_Y/3$-neighborhood, on the region $|(v,w)| \in [\frac{2}{3}\delta_{\mathcal{Y}}, \delta_{\mathcal{Y}})$ we set it to be equal to $J_{T^*|\nu_Y|}$, and for $s \in [\frac{1}{3}\delta_{\mathcal{Y}}, \frac{2}{3}\delta_{\mathcal{Y}})$ we define it via the standard polarization interpolation
\begin{equation}
    \widetilde{J}_{\mathcal{Y}}|_{|(v,w)|=s}=J_s(-J_s^2)^{-1/2},
\end{equation}
where 
$$J_s=(1-\tfrac{s-\delta_{\mathcal{Y}}/3}{\delta_{\mathcal{Y}}/3})J_{\mathcal{Y}}+\tfrac{s-\delta_{\mathcal{Y}}/3}{\delta_{\mathcal{Y}}/3}J_{T^*|\nu_Y|}.$$

Tameness guarantees that this interpolation is well-defined; see \cite[Section 2.2]{Wendl2010} for more details. Since $J_{\mathcal{Y}}$ and $J_{T^*\nu}$ coincide along $T^*Y$, we can assume that $\widetilde{J}_{\mathcal{Y}}=J_{T^*\nu}\exp(Z)$ for $Z \in C^\infty(TT^*\nu)$ with $$\|Z\|_{C^0}< \varepsilon'$$ for any small $\varepsilon'>0$ (one needs to shrink $\delta_{\mathcal{Y}}$ in response to $\varepsilon'$).

Now that we understand how to construct an adapted almost complex structure on a $\delta_{\mathcal{Y}}$-neighborhood of $T^*Y$ inside $T^*|\nu_Y|$, it only remains to build a $G$-equivariant almost complex structure on the whole manifold $M$.

To construct such an almost complex structure, for each $G$-orbit $G\cdot Y$ of any such hypersurface $Y$, we pick a function $\varepsilon_Y \colon Y \to \R_{\ge 0}$ satisfying the following:
\be
\item[(i)] $\varepsilon_Y$ is positive on $Y^{\circ}$ and vanishes on $\partial Y^{\circ}$, where $Y^{\circ} \subset Y$ is the maximal stratum of the fixed point stratification, i.e., $Y^{\circ}=\{y \in Y \mid G_y=H\}$ (see also Figure~\ref{fig: components};
\item[(ii)] the images of $\varepsilon_Y$-neighborhoods of the $0$-sections of $\nu_Y$ under the embeddings $j_Y$ (given by \eqref{eqn: j sub Y}) do not intersect away from $X^{\operatorname{sing}}$.
\ee
We define a Hermitian metric $\widetilde{g}_X$ on $X$ via an interpolation between $g_X$ and the pushforward of $g_{|\nu_Y|}$ in the $\varepsilon_Y$-neighborhood away from the $\varepsilon_Y/2$-neighborhood for each hypersurface $Y$. This metric is automatically $G$-equivariant. 

Finally, setting $\delta_{\mathcal{Y}}=\varepsilon_Y/2$ we define $\widetilde{J}$ by interpolating between the almost complex structure induced on $M$ by $\widetilde{g}_X$ and the pushforwards of the $J_{\mathcal{Y}}$ as explained above. 
The construction we just presented proves the following:

\begin{proposition} \label{prop: existence of adapted almost complex structure}
     The cotangent bundle $M=T^*X$ of a complex manifold $X$ admits an almost complex structure $J_0$ adapted to all type $A$ inertia components. Moreover, we can assume that such $J_0$ agrees with the standard almost complex structure $J_{\operatorname{std}}$ induced by the metric, away from the type $A$ inertia, and is given by $J_0=J_{\operatorname{std}}\exp(Z)$, for $\|Z\|_{C^0} < \varepsilon$ for any fixed small $\varepsilon>0$.
\end{proposition}

\begin{remark}
    Proposition~\ref{prop: existence of adapted almost complex structure} generalizes to the case of a fixed component $Y \subset X$ with a cyclic stabilizer of a generic point, i.e., to a cyclic quotient singularity case. In this case, one simply needs to split $\nu_Y$ into isotypic components and repeat our constructions from above.
\end{remark}

\begin{remark}
    For future applications, one may want to generalize Proposition~\ref{prop: existence of adapted almost complex structure} to arbitrary collections of inertia components in arbitrary Liouville $G$-manifolds, and we expect such a result to hold true.
    One may also want to pick the auxiliary isotypic directions in a compatible way, and one should be able to achieve this using the methods developed in \cite[Appendix A]{baixu2022transversality}. 
\end{remark}

\subsection{Kuranishi model}\label{section Kuranishi model}

In this subsection, we describe the gluing model for an orbicurve with a ghost bubble, closely following the approach of \cite{doan2023}. We also describe the setup of Theorem~\ref{theorem: main-obstruction-result} in more detail.
Our main applications concern the case where $(M,\omega)$ is a Liouville domain, but the arguments below are valid for an arbitrary symplectic manifold.

Let $u_0 \colon \Sigma_0 \to M$ be a $G$-equivariant curve equipped with some domain-dependent equivariant almost complex structure $J$ and $(\pi \colon \Sigma_0 \to C_0; q_1, \dots, q_b)$ a $G$-cover in the Hurwitz moduli space $\overline{\mathscr{H}}^G_{0, \bm h, d}$.  Assume $u_0$ is a nodal curve with a decomposition of the domain
\begin{equation}
    \Sigma_0=(\Sigma_{\bullet} \sqcup \Sigma_{\mathghost}, \hat{\nu}),
\end{equation}
such that the restriction of $u_0$ to each component of $\Sigma_{\mathghost}$ is constant. (The notation $\hat\nu$ and $\nu$ for the nodal structure will be explained in the next paragraph.)
The underlying nodal orbidisk, which is injective on isotropy, is denoted
$$
    v_0 \colon \mathscr{C}_0 \to [M/G].
$$
We assume that 
\begin{equation}\label{eq: orbidisk-with-ghost-decomposition}
    \mathscr{C}_0=(\mathscr{C}_{\bullet}\sqcup \mathscr{C}_{\mathghost}, \nu)
\end{equation} is a twisted nodal orbidisk, consisting of two components and a unique node $\ddot{n}=(n_{\bullet}, n_{\mathghost})$ with $n_{*}\in \mathscr{C}_{*} $. The component $\mathscr{C}_{\mathghost}$ is either an orbidisk or an orbisphere. In the orbisphere case, $\ddot{n}$ is an interior node; it may be stacky, and if so, is balanced.  In the orbidisk case, $\ddot{n}$ is a standard boundary node. 

As in \cite[Section 2]{doan2023}, we treat nodal orbicurves as smooth orbicurves equipped with a \emph{nodal structure} $\nu$ (or $\hat{\nu}$), i.e., an involution of the normalization of a nodal orbicurve which identifies the preimages of a given node and is the identity elsewhere. In this paradigm, each node is a pair of points on the curve, and we use the notation $\ddot{n}$ to emphasize this. Sometimes we treat $\ddot{n}$ as a common point of $\mathscr{C}_\bullet$ and $\mathscr{C}_{\mathghost}$, which is in fact a point of the underlying quotient $\mathscr{C}_0/\nu$.

We assume that the image of $\mathscr{C}_{\mathghost}$ is given by a point in $[G\cdot \mathcal{Y}/G]$, where $[G\cdot \mathcal{Y}/G]$ is a cyclic quotient singularity. In particular, the stabilizer $H=H_\mathcal{Y}$ of the symplectic submanifold $\mathcal{Y}$ is a cyclic subgroup of order $m_\mathcal{Y}$.  We require the stabilizer of this point to be exactly $H_\mathcal{Y}$, i.e., this point stays away from deeper strata of the fixed point stratification. This implies that the order of the stabilizer of any stacky point $z \in \mathscr{C}_{\mathghost}$ divides $m_\mathcal{Y}$. We denote by $m_{\ddot{n}}$ the order of the stabilizer of $\ddot{n}$.

We fix a $\mathcal{Y}$-adapted $\omega$-tame almost complex structure $J_0$ in the sense of Definition~\ref{defn: adapted J} (2) and (3) (we note that the contact-type condition is not required in the proof of obstruction results).
Hence, for each character $\chi \in \operatorname{Irr}_{\C}(H)$ for which the isotypic component $\nu_{\chi}\subset \nu_{\mathcal{Y}}$ is nontrivial, we have an auxiliary choice of a $G$-invariant almost complex submanifold $V_\chi$. Here $\nu_{\chi}$ is regarded as a $J_0$-complex $H$-equivariant bundle over $\mathcal{Y}$.

The domain-dependent $J$ on $\Sigma_0$ is assumed to coincide with $J_0$ near the branch points and on all closed components. Moreover, we assume that $J$ is in fact a domain-dependent almost complex structure over the Hurwitz moduli space, which varies continuously in the $C^\infty$-topology, and we require the difference $|J-J_0|_{C^0}$ to be sufficiently small.

We assume that the boundary $\partial\Sigma_0$ is mapped to a prescribed tuple of regular equivariant Lagrangians intersecting transversely, and boundary punctures are mapped to $G$-orbits of intersection points.  We will assume that, for the metric on $M$ that we fix, all Lagrangians under consideration are totally geodesic. We also note that we allow non-regular equivariant Lagrangians as well. 

\begin{remark}
    Although the current setup slightly differs from that of Section~\ref{section: orbifold-Fukaya-prelim}, we can reduce to the current setup using the localization tricks of Section~\ref{section: localization tricks}. One could of course rewrite the following gluing analysis in the presence of full Floer data, but we prefer to work with simplified notation and estimates.
\end{remark}

We assume that the linearized operator $$D_{u_0} \colon W^{1,p}_F(\Sigma_0, u_0^*TM) \to L^p(\Sigma_0, u_0^*TM), \quad p>2,$$
associated with $\overline{\partial}_{J,H}$ (see~\eqref{eq: Floer-equation}) is \emph{Fredholm}. Here $(u_0^*TM, F)$ is the bundle pair with real boundary condition $F$ induced by pulling back tangent spaces to the Lagrangians under consideration. The transversality of intersections of Lagrangians is known to imply the Fredholmness of $D_{u_0}$; see \cite[Section 2.3]{salamon1997} for the classical reference and \cite[Section 6.5]{cant2022thesis} for the proof of this statement in full generality.

\subsubsection{Choosing a lift of the obstruction}\label{subsection: choice-of-the-lift-of-the-obstruction}  

In this subsection, we treat the case of an interior node, leaving the boundary node case to Section~\ref{subsection: boundary-node-gluing}. The interior node case is technically more involved as stacky nodes require careful treatment.

As before, $\pi \colon \Sigma_0 \to C_0$ is a $G$-cover over the coarse curve of $\mathscr{C}_0$, 
where $\Sigma_0=(\Sigma_{\bullet} \sqcup \Sigma_{\mathghost}, \hat{\nu})$, equipped with complex structure $j_0$, with $|G|/m_{\ddot{n}}$ nodes 
(in the preimage of $\ddot{n}$), and $\Sigma_{\mathghost}$ is the disjoint union
\begin{equation}\label{eq: ghost-domain-decomposition}
\Sigma_{\mathghost}=\bigsqcup_{gH \in G/H}g\cdot\Sigma_{m_{\mathcal{Y}}}    
\end{equation}
of $|G|/m_\mathcal{Y}$ copies of a cyclic branched $\mu_{m_\mathcal{Y}}$-cover $\pi_{m_\mathcal{Y}} \colon \Sigma_{m_\mathcal{Y}} \to C_{\mathghost}$ over the coarse curve $C_{\mathghost}$, since by our assumption the point $x=u_0(\Sigma_{m_\mathcal{Y}}) \in \mathcal{Y}$ has $G_x=H$. Let $r$ be the number of components of $\Sigma_{m_\mathcal{Y}}$.

The most challenging case in what follows is that of an interior stacky node $\ddot{n}$, so we mainly treat this case. Let $((h_{n}), (h_{n}^{-1})) \in \operatorname{Conj}(G)^2$ be the conjugacy classes of the local monodromy at $n_{\bullet}$ and $n_{\mathghost}$. We have $m_{\ddot{n}}=|h_n|$, and we introduce the notation $d_{\ddot{n}}\coloneqq \frac{m_\mathcal{Y}}{m_{\ddot{n}}}$.

Fix a node $\ddot{p}=(p_{\bullet}, p_{\mathghost}) \in \pi^{-1}(\ddot{n})$ with $p_{\mathghost}\in \Sigma_{m_\mathcal{Y}}$. There is a natural isomorphism
\begin{equation}
    \operatorname{ker}D_{u_{\mathghost}} \cong \bigoplus_{gH \in G/H}  (T_{u(g\ddot{ p})}M)^{\oplus r},
\end{equation}
since the kernel of the restriction $D_{u_{\mathghost}}=\overline{\partial} \otimes_{\mathbb{C}}\operatorname{id}$ to each connected component of the ghost curve is given by constant sections. There is also a map
\begin{equation}
    \operatorname{diff} \colon \operatorname{ker}D_{u_{\bullet}} \oplus \operatorname{ker}D_{u_{\mathghost}} \to u_0|_{\pi^{-1}(\ddot{n})}^*TM, 
\end{equation}
$$(\kappa_1, \kappa_2)\mapsto \kappa_1|_{\pi^{-1}(\ddot{n})}-\kappa_2|_{\pi^{-1}(\ddot{n})}.$$
Since there are exactly $d_{\ddot{n}}$ nodes above $\ddot{n}$ on each component of decomposition~\eqref{eq: ghost-domain-decomposition}, we have
\begin{equation}
    u_0|_{\pi^{-1}(\ddot{n})}^*TM\cong\bigoplus_{gH \in G/H}(T_{u(g\ddot{p})}M)^{\oplus d_{\ddot{n}}}.
\end{equation}
We note that $\operatorname{diff}$ is a morphism of $G$-representations, and, therefore, we may regard $\operatorname{coker} \,\operatorname{diff}$ as a \emph{subrepresentation} $W$ of $u|_{\pi^{-1}(\ddot{n})}^*TM$.

Let us denote by $D_{\ddot{n}}$ the two-term complex
$$(u_0|_{\pi^{-1}(\ddot{n})}^*TM\to0).$$ 
Denoting the two-term complexes associated with linearized $\overline{\partial}$-operators of a curve by the same notation, we recall (see e.g. \cite[Remark 2.28]{doan2023}) that there is a short exact sequence of chain complexes:
\begin{equation}
    0 \longrightarrow  D_{u_0} \longrightarrow  D_{u_{\bullet}}\oplus D_{u_{\mathghost}} \longrightarrow  D_{\ddot{n}} \longrightarrow  0, 
\end{equation}
which, by the snake lemma, yields the short exact sequence 
\begin{equation}
    0\longrightarrow \operatorname{coker} \,\operatorname{diff} \longrightarrow  \operatorname{coker} D_{u_0} \longrightarrow  \operatorname{coker} D_{u_\bullet} \oplus   \operatorname{coker} D_{u_{\mathghost}} \longrightarrow  0. 
\end{equation}

Since 
$$\operatorname{coker}(D_{u_{\mathghost}}) \cong \bigoplus_{gH \in G/H} H^{0,1}(g\cdot \Sigma_{m_\mathcal{Y}})\otimes_{\C} T_{u(g\ddot{p})}M,$$
we can choose a lift
\begin{equation}\label{eq: closed-ghost-obstruction-lift}
    \mathcal{O}_{\mathghost} \subset \bigoplus_{gH \in G/H} \Omega^{0,1}(g\cdot \Sigma_{m_\mathcal{Y}})\otimes_{\C} T_{u(\ddot{gp})}M \subset L^p(\Sigma_{\mathghost}, \Omega^{0,1}_{\mathghost} \otimes u^*_{\mathghost}TM)
\end{equation}
of $\operatorname{coker}(D_{u_{\mathghost}})$ to consist of complex conjugates of holomorphic $1$-forms on $\Sigma_{\mathghost}$, which is a $G$-repre\-sentation. We also choose a $G$-invariant lift $\mathcal{O}_\bullet$ of $\operatorname{coker} D_{u_\bullet}$ to consist of $1$-forms vanishing on a $G$-invariant neighborhood of $\pi^{-1}(\ddot{n})$.

Next choose a bump function $\psi$ on $C_{\mathghost}$ such that $\psi(n_{\mathghost})=1$ and is supported on a small ball containing $n_{\mathghost}$. Then for any $w\in W \cong \operatorname{coker} \, \operatorname{diff}$
we can define a section $\overline{\partial}(\psi \circ \pi)\otimes w \in L^p(\Sigma_{\mathghost}, \Omega^{0,1}_{\Sigma_{\mathghost}} \otimes u_{\mathghost}^*TM)$.  Note that $\overline{\partial}(\psi \circ \pi)\otimes w$ makes sense since $\psi \circ \pi=0$ on the complement of a small neighborhood of the nodes. We denote the space of such sections by
$$\mathcal{O}_{\operatorname{diff}}\subset L^p(\Sigma_{\mathghost}, \Omega^{0,1}_{\Sigma_{\mathghost}} \otimes u_{\mathghost}^*TM),$$
and it has a natural $G$-action.

Summarizing the above discussion, we have chosen a lift
\begin{equation}\label{eq: closed-ghost-full-obstruction-lift}
    \mathcal{O}:=\mathcal{O}_{\operatorname{diff}} \oplus \mathcal{O}_\bullet \oplus \mathcal{O}_{\mathghost} \subset L^p(\Sigma_0, \Omega^{0,1}_{\Sigma_0} \otimes u_0^*TM)
\end{equation}
of the cokernel of $D_{u_0}$ as a subrepresentation of $L^p(\Sigma, \Omega^{0,1}_{\Sigma} \otimes u_0^*TM)$. One can check that $\mathcal{O}_{\op{diff}}^G=0$ and hence 
\begin{equation}    
\mathcal{O}^G=\mathcal{O}_\bullet^G\oplus\mathcal{O}_{\mathghost}^G.
\end{equation}

We now have a surjective Fredholm operator
\begin{gather}
    \overline{D}_{u_0} \colon W^{1,p}_F(\Sigma_0, u_0^*TM) \oplus \mathcal{O} \to L^p(\Sigma_0, \Omega^{0,1}_{\Sigma_0} \otimes u_0^*TM), \\
    (\xi, o) \mapsto D_{u_0}(\xi)+o, \nonumber
\end{gather}
which is a morphism of representations. Moreover, 
$$K:= \operatorname{ker}(\overline{D}_{u_0})=\operatorname{ker}(D_{u_0})$$ 
is a subrepresentation and there is a choice of $G$-linear right inverse
\begin{equation}
    R_{u_0} \colon  L^p(\Sigma_0, \Omega^{0,1}_{\Sigma_0} \otimes u_0^*TM) \to W^{1,p}_F(\Sigma_0, u_0^*TM) \oplus \mathcal{O},
\end{equation}
We can also arrange so that the restriction
\begin{equation}
    R_{u_0}^{\bullet} \colon L^p(\Sigma_{\bullet}, \Omega^{0,1}_{\Sigma_{\bullet}} \otimes u_{\bullet}^*TM) \to W^{1,p}_F(\Sigma_0, u_0^*TM) \oplus \mathcal{O}
\end{equation}
has its image contained in $W^{1,p}_F(\Sigma_{\bullet}, u_{\bullet}^*TM)\oplus \mathcal{O}_\bullet$ (which is embedded via constant extensions over $\Sigma_{\mathghost}$) and is $G$-linear.

\subsubsection{Gluing}\label{subsection: gluing}
We now review the $2$-step gluing construction from \cite[Section 5]{doan2023}. 

\s\n
{\em Preparation for pregluing.} 
We assume a choice of a $G$-invariant torsion-free connection $\nabla$ on $M$.  Let
\begin{equation}\label{eq: gluing-parameters}
    \Delta=\Delta_1 \times \Delta_2 \subset \R^{2(b-1)+d-2}_\sigma \times \C_\tau
\end{equation}
be a small neighborhood of the origin (compare with~\eqref{eq: gluing-coordinates}). We also assume a choice of $R_0>0$ such that the image of each radius $4R_0$ ball $B_{4R_0}(g\ddot{p})$ (with respect to a metric $g_0$ on $\Sigma_0$) is within the tubular neighborhood of $g\cdot \mathcal{Y}$ that we have chosen in Section~\ref{neighborhoods-of-fixed-sets} (and, in particular, is within the injectivity radius of $u_0(gp)$). 

In what follows we set $\varepsilon=\varepsilon_\tau \coloneqq |\tau|$ and
\begin{equation}    
r_{g\ddot{p}}: \Sigma_0\to \R_{\geq 0}, \quad z\mapsto \min(\operatorname{dist}(g\ddot{p}, z), 4R_0).
\end{equation} 
Given $\tau\in \Delta_2$ we define 
\begin{gather*}
    \Sigma_{\tau}\coloneqq\Sigma_\tau^{\circ}/\sim_\tau \text{, where} \\
    \Sigma_\tau^{\circ}=\{z\in \Sigma_0 \colon \,r_{g\ddot{p}}(z) \ge \varepsilon^{1/2} \text{ for all }g\in G\}, \, 
\end{gather*}
and $\sim_\tau$ is the equivalence relation on $\bdry \Sigma_\tau^{\circ}$ associated with an involution $\iota_\tau$ defined as follows: First, we set
\begin{equation}\label{eq: annuluis-model}
    A_\tau\coloneqq \{z \in \Sigma_0 \colon \, \varepsilon/R_0 < r_{g\ddot{p}}(z)<R_0 \text{ for some }g\in G\},
\end{equation}
which is a disjoint union of two annuli. For $z \in A_\tau$ near $gp_\star$, $p_\star=p_\bullet$ or $p_{\mathghost}$, using local $\mu_{m_{\ddot{n}}}$-equivariant charts (recall that the stabilizers of $gp_\star$ and $\hat{\nu}(gp_\star)$ are isomorphic to $\mu_{m_{\ddot{n}}}$)
\begin{equation}\label{eq: charts-near-node}
    \varphi_{gp_\star} \colon U_{gp_\star} \subset\C \to\Sigma_0, \quad \varphi_{\hat{\nu}(gp_\star)}\colon U_{\hat{\nu}(gp_\star)} \subset \C \to \Sigma_0,
\end{equation} 
near $gp_\star$ and $\hat{\nu}(gp_\star)$, which we assume to be isometries with respect to $g_0$ and the Euclidean metric on $\C$, we set
\begin{gather}
    \iota_\tau \colon {A_\tau \cap \varphi_{gp_\star}(U_{gp_\star})} \to A_\tau \cap \varphi_{\hat{\nu}(gp_\star)}(U_{\hat{\nu}(gp_\star)})\\
    z \mapsto \varphi_{\hat{\nu}(gp_\star)}\big(\tfrac{\tau}{\varphi^{-1}_{gp_\star}(z)}\big). \nonumber
\end{gather}
Observe that $\varepsilon/R_0< \varepsilon^{1/2}<\varepsilon^{1/4}< R_0$ (here $\varepsilon^{1/4}$ is a constant that is used later) and that $\iota_\tau$ switches $r_{g\ddot{p}}=\varepsilon/R_0$ and $r_{g\ddot{p}}=R_0$ and  takes $r_{g\ddot{p}}=\varepsilon^{1/2}$ to itself.

\begin{remark}\label{rmk: pregluing-equivariantly}
    We recall that this involution is associated with the map $\C \to \C$ corresponding to the local model of the nodal family $zw=\tau$. We also notice that we are using only one gluing parameter for all the nodes in $\Sigma_0$ above $\ddot{n}$.
\end{remark}

Given $(\sigma,\tau)\in \Delta$, we define $\Sigma_{\sigma, \tau}=(\Sigma_\tau,j_\sigma)$, where $\{j_\sigma ~|~\sigma \in \Delta_1\}$ is a family of complex structures on a neighborhood of $j_0$ on $\Sigma_0$, which is obtained by pulling back a family of complex structures on the underlying coarse curve $C_0$ and which coincide with $j_0$ near the stacky points and boundary. The neighborhood $\Delta_1$ gives rise to an orbifold chart near $G$-cover $\pi_0$ for the boundary component of $\overline{\mathscr{H}}_{0, \bm h, d}^G$ containing $\pi_0$; see the discussion in Section~\ref{subsection: hurwitz moduli spaces} and Section~\ref{section: transversality for G-equivariant curves}.

\begin{lemma}\label{lemma: G-action-on-preglued-curve}
    The curve $\Sigma_{\sigma, \tau}$ is equipped with a natural $G$-action induced from the $G$-action on $\Sigma_0$. This action is almost complex.
\end{lemma}

\begin{proof}
     The $G$-action on $\Sigma_0$ restricts to the $G$-action on $\Sigma_\tau^\circ$. The action of $G$ near the nodes (i.e., on $\{r_{g\ddot{p}}(z) \le \varepsilon^{1/2}\}$) is standard, i.e., built out of the action of $G$ on $G/G_p$ and the action of $G_p \cong \mu_{m_{\ddot{n}}}$ on $U_{gp_{\bullet}}$ and $U_{gp_{\mathghost}}$ as in~\eqref{eq: charts-near-node}. Hence, this action preserves the following subset of $\Sigma_0$ and its boundary (given by the images of balls near the origin under maps~\eqref{eq: charts-near-node}):
    $$\bigsqcup_{gG_p \in G/G_p}\varphi_{gp_{\bullet}}\Big(B_{\varepsilon^{1/2}}(0) \Big)\sqcup \varphi_{gp_{\mathghost}}\Big(B_{\varepsilon^{1/2}}(0)\Big).$$
    This set is exactly $\Sigma_0 \setminus \Sigma_\tau^{\circ}$, hence the claim. Finally, since the node $\ddot{n}$ is balanced, this action commutes with $\iota_\tau$ on the boundary of $\Sigma_0 \setminus \Sigma_\tau^{\circ}$, which is easily seen in the local coordinates in the light of Remark~\ref{rmk: pregluing-equivariantly}:
    $$\iota_\tau(\zeta_{m_{\ddot{n}}}\cdot z)=\iota_\tau(\zeta_{m_{\ddot{n}}}z)=\frac{\tau}{\zeta_{m_{\ddot{n}}}z}=\zeta_{m_{\ddot{n}}}^{-1}\frac{\tau}{z}=\zeta_{m_{\ddot{n}}} \cdot \frac{\tau}{z}.$$
\vskip-.24in
\end{proof}

\begin{definition}
    We set
    $$\mathcal{X}=\{(z, \sigma, \tau) \in \Sigma_0 \times \Delta \mid z \in \Sigma_{\sigma, \tau}^{\circ}\}/ \sim,$$
    where $(z_1, \sigma_1, \tau_1) \sim (z_2, \sigma_2, \tau_2)$ if and only if $\sigma_1=\sigma_2, \tau_1=\tau_2$ and $z_1 \sim_{\tau_1} z_2$, or $\sigma_1=\sigma_2$, $\tau_1=\tau_2=0$ and $z_1=gp=\hat{\nu}(z_2)$. 
\end{definition}

The following is immediate from Lemma~\ref{lemma: G-action-on-preglued-curve} and \cite[Proposition 4.13]{doan2023}:

\begin{lemma}
    The space $\mathcal{X}$ is a smooth manifold with a complex structure induced from $\Sigma_0 \times M$, and the projection $\mathcal{X} \to \Delta$ is a nodal family. The $G$-action on $\mathcal{X}$ is smooth and $[\mathcal{X}/G] \to \ \Delta$ is a twisted nodal family. The fiber above $(\sigma, \tau)$ is naturally biholomorphic to $\Sigma_{\sigma, \tau}$.
\end{lemma}

\begin{remark}
    One can also show that $\mathcal{X}$ is a versal $G$-equivariant deformation in the appropriate sense.
\end{remark}

\s\n
{\em Pregluing.} We briefly review the standard pregluing setup, referring to \cite[Sections 5.2-5.6]{doan2023} for more details.

Recall that for any curve $u \colon \Sigma \to M$ (smooth and potentially nodal) with almost complex structure $j$ on $\Sigma$, equipped with domain-dependent Floer datum on $\Sigma$, there is a map
\begin{gather}\label{eq: dbar-section-local}
    \mathscr{F}_{u}\colon W^{1,p}_F(\Sigma, u^*TM) \to L^p(\Sigma, \Omega^{0,1}_\Sigma \otimes u^*TM), \\
    \mathscr{F}_{u}(\xi) \coloneqq \operatorname{PT}_{\operatorname{exp}_u(\xi) \to u}[d\operatorname{exp}_u(\xi) -X_H\otimes \alpha_\Sigma]^{0,1}_{(J,j)},\nonumber
\end{gather}
where $\operatorname{PT}_{x \to y}$ stands for the parallel transport with respect to $\nabla$ along the geodesic connecting $x$ and $y$. The linearized operator is given by the formula
\begin{equation}
    D_u \xi \coloneqq D{\mathscr{F}_{u}}|_{\eta=0} (\xi)=(\nabla \xi-\alpha_\Sigma \otimes \nabla_\xi X_H)^{0,1}+\tfrac{1}{2}(\nabla_\xi J)\circ (du-\alpha_\Sigma \otimes X_H) \circ j.
\end{equation}

Fix a bump function $\psi \colon [0, \infty) \to [0,1]$ such that $$\psi|_{[0,1]}=1, \quad \psi|_{[2, \infty)}=0.$$
We also define $\psi_\tau: \Sigma_\tau \to [0, 1]$ by
\begin{equation} \label{eqn: psi tau}
\psi_\tau(z) := \psi\left(\frac{r_{\ddot{n}}(z)}{R_0}\right),
\end{equation}
where $r_{\ddot{n}}(z)=\operatorname{min}_{g \in G}r_{g\ddot{p}}(z)$. We then define the \emph{preglued curve} $\widetilde{u}_\tau: \Sigma_\tau\ \to M$ by
\begin{equation}\label{eq: curve-pregluing}
\widetilde{u}_\tau(z) = 
\begin{cases} 
\exp_{u_0(gp_\star)}(\exp_{u_0(gp_\star)}^{-1} \circ u_0(z) + \psi_\tau(z) \cdot \exp_{u_0(gp_\star)}^{-1} \circ u_0(\iota_\tau(z))), & \text{if } r_{gp_\star}(z) \leqslant 2R_0,\\
u_0(z), & \text{otherwise.}
\end{cases}
\end{equation}
 Here $r_{gp_\star}$, $p_\star=p_{\bullet}$ or $p_{\mathghost}$, measures the distance to $gp_\star$ in the {\em normalized curve}, as opposed to $r_{g\ddot{p}}$ which measures the distance to the node in the nodal curve. In the definition of $\widetilde{u}_\tau$, we are using the extended definition of $\iota_\tau$ obtained by replacing $\varepsilon/R_0< r_{g\ddot p}(z)<R_0$ by $\varepsilon/2R_0< r_{g\ddot p}(z)<2R_0$ in \eqref{eq: annuluis-model}.

Let $\widetilde{u}_{\sigma, \tau}$ be the analogously constructed preglued curve $\Sigma_{\sigma, \tau}\to M$.

\begin{lemma}
    The curve $\widetilde{u}_{\sigma, \tau}$ is smooth and $G$-equivariant.
\end{lemma}

\begin{proof}
    This only needs to be proved near the nodes, and it suffices to show equivariance with respect to the $G_p$-action, since 
    \begin{equation}\label{eq: exponential-equivariance}
        g \cdot \exp_{u_0(p_\star)}(v)=\operatorname{exp}_{g\cdot u_0(p_\star)}(g_*v),
    \end{equation}
    as the metric is $G$-invariant. For $h \in G_p$, a generator, we have 
    \begin{gather}
        h\cdot \exp_{u_0(p_\star)}(\exp_{u_0(p_\star)}^{-1} \circ u_0(z) + \psi_\tau(z) \cdot \exp_{u_0(p_\star)}^{-1} \circ u_0(\iota_\tau(z))) \nonumber \\
        \overset{(1)}{=}\exp_{u_0(p_\star)}(h_*[\exp_{u_0(p_\star)}^{-1} \circ u_0(z) + \psi_\tau(z) \cdot \exp_{u_0(p_\star)}^{-1} \circ u_0(\iota_\tau(z))]) \nonumber \\
        \overset{(2)}{=}\exp_{u_0(p_\star)}(\exp_{u_0(p_\star)}^{-1} \circ h\cdot[u_0(z)] + \psi_\tau(z) \cdot \exp_{u_0(p_\star)}^{-1} \circ h\cdot [u_0(\iota_\tau(z))])\nonumber \\
        \overset{(3)}{=}\exp_{u_0(p_\star)}(\exp_{u_0(p_\star)}^{-1} \circ  u_0(h\cdot z) + \psi_\tau(z) \cdot \exp_{u_0(p_\star)}^{-1} \circ u_0(\iota_\tau(h\cdot z))). \nonumber
    \end{gather}
    Here (1) and (2) follow from the equivariance of the exponential map~\eqref{eq: exponential-equivariance} (and therefore of its inverse), while (3) follows from the equivariance of $u_0$ and the $G_p$-equivariance of $\iota_\tau$, which was shown in the proof of Lemma~\ref{lemma: G-action-on-preglued-curve}.
\end{proof}

\s\n
{\em Construction of a right inverse.}
Let $\Phi_y^x \colon T_xM \to T_yM$ be the differential of the exponential map
\begin{equation}
    \Phi_y^x(v)\coloneqq d\operatorname{exp}_x(w)[v], \, \text{ where }w=\operatorname{exp}^{-1}_x(y),
\end{equation}
which is well-defined if $y$ is within the injectivity radius of $x$. 
We also define another bump function 
$$\rho_\tau(z) \coloneqq \psi\Big(\frac{r_{\ddot{n}}(z)}{\varepsilon^{1/4}}\Big).$$
Then we define the \emph{fusing operator}
\begin{gather}\label{eq: fusing operator}
    \operatorname{fuse}_\tau \colon W^{1,p}_F(\Sigma_0, u_0^*TM) \to W^{1,p}_{F_\tau}(\Sigma_\tau, \widetilde{u}_\tau^*TM),\\
\mathrm{fuse}_{\tau}(\xi)(z):=
\begin{cases}
\Phi_{\widetilde{u}_{\tau}(z)}\Bigl(\Phi_{u_{0}(z)}^{-1}\xi(z)
+\rho_{\tau}(z)\cdot(\Phi_{u_{0}\circ\iota_{\tau}}^{-1}\xi(\iota_{\tau}(z))
-\Phi_{u_{0}(gp_\star)}^{-1}\xi(gp_\star))\Bigr),
&  r_{gp_\star}(z)\le 2\,\varepsilon^{1/4},\\[4pt]
\xi(z), & \text{otherwise.}
\end{cases} \nonumber
\end{gather}
Here $i=1,2$, $g\in G$ is arbitrary, and $\Phi_\Box$ stands for $\Phi^{u_0(gp_\star)}_\Box$.

\begin{remark}
    Here, $F_\tau$ is the real boundary condition corresponding to the same Lagrangian boundary conditions as those of $u_0$ pulled back under the map $u_\tau$.
\end{remark}

We also define the following \emph{push and pull operators}:
\begin{gather}
\mathrm{push}_{\tau}\colon
L^{p}\Omega^{0,1}\!\bigl(\Sigma_{\tau},\,\widetilde{u}_{\tau}^{*}TM\bigr)
\longrightarrow
L^{p}\Omega^{0,1}\!\bigl(\Sigma_{0},\,u_{0}^{*}TM\bigr) \\
\mathrm{push}_{\tau}(\eta)(z):=
\begin{cases}
0, & \text{if } r_{gp_\star}(z)<\varepsilon^{1/2},\\[4pt]
\Phi_{u_{0}(z)}\,\Phi_{\widetilde{u}_{\tau}([z])}^{-1}\,\eta([z]),
& \text{if } \varepsilon^{1/2}\le r_{gp_\star}(z)\le 2R_{0},\\[4pt]
\eta([z]), & \text{otherwise.}
\end{cases}\nonumber
\end{gather}
\begin{gather}\label{eq: pull-operator}
    \mathrm{pull}_{\tau}\colon
L^{p}\Omega^{0,1}\!\bigl(\Sigma_{0},\,u_{0}^{*}TM\bigr)
\longrightarrow
L^{p}\Omega^{0,1}\!\bigl(\Sigma_{\tau},\,\widetilde{u}_{\tau}^{*}TM\bigr) \\
\mathrm{pull}_{\tau}(\eta)([z]):=
\begin{cases}
\Phi_{\widetilde{u}_{\tau}(z)}\,\Phi_{u_{0}(z)}^{-1}\,\eta(z),
& \text{if } \varepsilon^{1/2}\le r_{gp_\star}(z)\le 2R_{0},\\[4pt]
\eta(z), & \text{otherwise.}
\end{cases} \nonumber
\end{gather}
The pull operator is used to extend $\overline{D}_{u_0}$ to the preglued curve $\widetilde{u}_\tau$:
\begin{gather}\label{eq: extended-linearized-operator}
\overline{D}_{\widetilde{u}_\tau} \colon W^{1,p}_{F_\tau}(\Sigma_\tau, \widetilde{u}_\tau^*TM) \oplus \mathcal{O} \to L^p(\Sigma_\tau, \Omega^{0,1}_{\Sigma_\tau} \otimes \widetilde{u}_\tau^*TM)\\
    \overline{D}_{\widetilde{u}_\tau}(\xi, o)\coloneqq  D_{\widetilde{u}_\tau}\xi+\operatorname{pull}_\tau(o). \nonumber
\end{gather}

\begin{definition}
    The {\em approximate right inverse} $\widetilde{R}_{\widetilde{u}_\tau}$ to $\overline{D}_{\widetilde{u}_\tau}$ is given by
    \begin{equation}\label{eq: formula-approximate-right-inverse}
        \widetilde{R}_{\widetilde{u}_\tau} \coloneqq (\operatorname{fuse}_\tau \oplus \operatorname{id}) \circ R_{u_0} \circ \operatorname{push}_\tau.
    \end{equation}
\end{definition}

\begin{definition}\label{defn: right inverse}
    The {\em right inverse} $R_{\widetilde{u}_\tau}$ is the infinite sum
    \begin{equation}\label{eq: right inverse formula}
        R_{\widetilde{u}_\tau}\coloneqq \widetilde{R}_{\widetilde{u}_\tau}\sum_{k=0}^\infty(\operatorname{id}-\overline{D}_{\widetilde{u}_\tau}\widetilde{R}_{\widetilde{u}_\tau})^k.
    \end{equation}
\end{definition}

\begin{lemma}\label{lemma: aproximate-right-inverse-equivariance}
    The operator $\widetilde{R}_{\widetilde{u}_\tau}$ is $G$-linear.
\end{lemma}

\begin{proof}
    Since $R_{u_0}$ is $G$-linear it suffices to show that $\operatorname{fuse}_\tau$ and $\operatorname{push}_\tau$ are $G$-linear. We focus on the first of these maps. We only need to check the equivariance for points $z \in \Sigma_\tau$ with $r_{gp_\star}(z) \le 2\varepsilon^{1/4}$ for some $gp \in \pi^{-1}(\ddot{n})$. For simplicity, we focus on the case of $h \in G$ which belongs to the stabilizer $G_{u_0(gp)}$. First, we observe that, for any $x,y \in M$ and $h \in G_x$, the equivariance of the exponential map and the chain rule imply
    \begin{equation}\label{eq-Jacobi-equivariance}
        h_*\Phi^x_y[v]=\Phi^x_{hy}[h_*v], \quad (\Phi^{x}_{hy})^{-1}[h_*v]=h_*(\Phi_y^x)^{-1}[v], \quad \forall v \in T_xX.
    \end{equation}
    Now we compute
    \begin{gather}
        h\cdot \operatorname{fuse}_\tau(\xi) (z)=h_*\operatorname{fuse}_\tau(\xi)(h^{-1}z) \nonumber\\
        =h_*\Phi_{\widetilde{u}_\tau(h^{-1}z)}\Bigl(\Phi_{u_{0}(h^{-1}z)}^{-1}\xi(h^{-1}z) 
+\rho_{\tau}(z)\cdot(\Phi_{u_{0}\circ\iota_{\tau}(h^{-1}z)}^{-1}\xi(\iota_{\tau}(h^{-1}z))
-\Phi_{u_{0}(gp_\star)}^{-1}\xi(gp_\star))\Bigr)  \nonumber\\
\overset{(1)}{=}\Phi_{\widetilde{u}_\tau(z)}\Bigl( h_*\Phi_{u_{0}(h^{-1}z)}^{-1}\xi(h^{-1}z) 
+\rho_{\tau}(h^{-1}z)\cdot(h_*\Phi_{u_{0}\circ\iota_{\tau}(h^{-1}z)}^{-1}\xi(\iota_{\tau}(h^{-1}z))
-h_*\Phi_{u_{0}(gp_\star)}^{-1}\xi(gp_\star))\Bigr)  \nonumber\\
\overset{(2)}{=}\Phi_{\widetilde{u}_\tau(z)}\Bigl( \Phi_{u_{0}(z)}^{-1}h_*\xi(h^{-1}z) 
+\rho_{\tau}(h^{-1}z)\cdot(\Phi_{u_{0}\circ\iota_{\tau}(z)}^{-1}h_*\xi(h^{-1}\iota_{\tau}(z))
-\Phi_{u_{0}(gp_\star)}^{-1}h_*\xi(gp_\star))\Bigr) \nonumber\\
=\operatorname{fuse}_\tau(h\cdot\xi)(z).
    \end{gather}
    Both steps (1) and (2) above use~\eqref{eq-Jacobi-equivariance}, and for (2) we also use equivariance of $\iota_\tau$, which, as we saw in the proof of Lemma~\ref{lemma: G-action-on-preglued-curve}, follows from the balancing condition. We also use equivariance of $u_0$ and $\widetilde{u}_\tau$. 
    
    The proof of the equivariance of $\operatorname{pull}_\tau$ similarly follows from~\eqref{eq-Jacobi-equivariance}, and we leave the details to the reader.
\end{proof}

\begin{lemma}[{\cite[Proposition 5.26]{doan2023}}]\label{lemma: inequalities-right-invers}
The following properties are satisfied by $\widetilde{R}_{\widetilde{u}_\tau}$ and $R_{\widetilde{u}_\tau}$:
    \begin{gather}
        \|\widetilde{R}_{\widetilde{u}_
       \tau}\|, \|R_{\widetilde{u}_
       \tau}\| \le c\|R_{u_0}\|, \\
       \overline{D}_{\widetilde{u}_
       \tau}R_{\widetilde{u}_
       \tau}=\operatorname{id}, \, \operatorname{Im}(\widetilde{R}_{\widetilde{u}_
       \tau})=\operatorname{Im}(R_{\widetilde{u}_
       \tau}). \nonumber
    \end{gather}
    Moreover, $R_{\widetilde{u}_\tau}$ is $G$-equivariant.
\end{lemma}

\begin{remark}
    By our setup, the domain-dependent form $\alpha$ is set to vanish on the ghost components and on the gluing neck-region of each curve $\Sigma_{\sigma, \tau}$ corresponding to $A_\tau$. This implies that the proof of Lemma~\ref{lemma: inequalities-right-invers} indeed follows from estimates in \cite{doan2023} without any additional work, as all the terms one has to estimate for all the operators involved in pregluing are supported in the neck region.
\end{remark}

\s\n
{\em Two-step gluing.}
The above setup lets us \emph{glue} the curve $u_0$, i.e., realize it as an element of a family of curves parametrized by a neighborhood of the origin in $\Delta \times K$. The insight of \cite{doan2023}, which allows for showing the obstruction result on the vanishing of derivatives, is to proceed with a $2$-step gluing, which we describe now. This two-step gluing will play an important role in the proof of Proposition~\ref{proposition: obstruction-on-higher-jets}.

{\em First step.} We trivialize the bundle with fiber $\Omega^{0,1}_F(\Sigma_{\sigma, 0}, u_0^*TM)$ over $\Delta_1$, and use it to treat $\mathcal{O}$ as a subspace of $L^p(\Sigma_{\sigma, 0}, u_0^*TM)$. We then define
\begin{equation}
    \overline{D}_{u_{\sigma, 0}} \colon W^{1,p}_F(\Sigma_{\sigma, 0}, u_0^*TM) \oplus \mathcal{O} \to L^p(\Sigma_{\sigma, 0}, u_0^*TM)
\end{equation}
by extending $D_{u_{\sigma, 0}}$ as in~\eqref{eq: extended-linearized-operator}, where the linearized operator is defined with respect to $J_{\sigma,0}$ (similarly, we denote via $\mathscr{F}_{u_{\sigma, 0}}$ the map~\eqref{eq: dbar-section-local} associated with $J_{\sigma, 0}$). Notice that $\overline{D}_{u_{\sigma, 0}}$ remains equivariant. The formulas~\eqref{eq: formula-approximate-right-inverse} and~\eqref{eq: right inverse formula} then define the equivariant right inverse $R_{u_{\sigma, 0}}$ for this operator.

\begin{proposition}\label{proposition: 2-step-gluing-interior node}
    For any $(\sigma,\kappa)\in \Delta_1\times K$ near the origin there is a unique map
    \begin{equation}
        \operatorname{glue}_1(u_0, \sigma, \kappa)\coloneqq(\xi, o) \in \operatorname{Im}(R_{u_{\sigma, 0}}) \subseteq W^{1,p}_F(\Sigma_{\sigma, 0},u_0^*TM ) \oplus \mathcal{O} ,
    \end{equation}
    such that
    \begin{equation}
        \mathscr{F}_{\tilde{u}_{\sigma,0}}(\kappa+\xi)+o=0.
        \end{equation}
        Moreover, the following inequality is satisfied
        \begin{equation}
                \|\xi \|_{W^{1,p}} \le c(|\sigma|+\|\kappa\|_{W^{1,p}}+|J_{(0,0)} -J_{(\sigma, \tau)}|_{C^0}).
        \end{equation}
        Finally, the map $\operatorname{glue}_1$ is $G$-equivariant, implying that for $\kappa \in K^G$ the curve $\widetilde{u}_{\sigma, \tau, \kappa}$ is $G$-equivariant.
\end{proposition}     

\begin{proof}
    This follows from \cite[Proposition 5.29]{doan2023}.
    The equivariance of the gluing map follows from the definition of this map via Newton-Picard iteration. 
\end{proof}

{\em Second step.} As in \cite[Definition 5.32]{doan2023} we first set 
$$\tilde{u}_{\sigma, \tau, \kappa}\coloneqq\widetilde{(u_{\sigma, \kappa})}_\tau, \, \text{with } u_{\sigma, \kappa}=\operatorname{exp}_{u_{\sigma,0}}(\kappa+\xi),$$ 
where $(\xi,o)=\operatorname{glue}_1(u_0, \sigma, \kappa)$. Then we define
$$\operatorname{pull}_{\sigma, \tau, \kappa} \colon L^p(\Sigma_0, u_0^*TM) \to L^p(\Sigma_{\sigma,\tau}, \widetilde{u}^*_{\sigma, \tau, \kappa}TM)$$ 
as the composition of the parallel transport along the geodesics from $u_0$ to $u_{\sigma, \kappa}$ which we denote by
$$\operatorname{pull}_{\sigma,0, \kappa} \colon L^p(\Sigma_0, u_0^*TM) \to L^p(\Sigma_{\sigma, 0}, u_{\sigma, \kappa}^*TM),$$
with 
$$\operatorname{pull}_{\tau}\colon L^p(\Sigma_{\sigma, 0}, u_{\sigma, \kappa}^*TM) \to L^p(\Sigma_{\sigma,\tau}, \widetilde{u}^*_{\sigma, \tau, \kappa}TM)$$
defined as in~\eqref{eq: pull-operator}.
Since $\operatorname{pull}_{\sigma,0, \kappa}$ is equivariant, the subspace $\operatorname{pull}_{\sigma,0, \kappa}(\mathcal{O})$ is a complemental subrepresentation of $\operatorname{Im} D_{u_{\sigma, \kappa}}$. Hence, applying the construction from Definition~\ref{defn: right inverse} we obtain the right inverse $$R_{\widetilde{u}_{\sigma, \tau, \kappa}} \colon L^p(\Sigma_{\sigma,\tau}, \widetilde{u}^*_{\sigma, \tau, \kappa}TM) \to W_{F_{\sigma, \tau}}^{1,p}(\Sigma_{\sigma,\tau}, \widetilde{u}^*_{\sigma, \tau, \kappa}TM) \oplus \mathcal{O}.$$

\begin{proposition}\label{prop: 2nd-step-gluing}
    For any $(\sigma, \tau, \kappa) \in \Delta \times K$ near the origin there is a unique map
    \begin{equation}
        \operatorname{glue}_2(u_0, \sigma, \tau, \kappa) \coloneqq(\hat{\xi},\hat{o}) \in \operatorname{Im}(R_{\widetilde{u}_{\sigma, \tau}}) \subseteq W_{F_\tau}^{1,p}(\Sigma_{\sigma, \tau},\widetilde{u}_{\sigma, \tau, \kappa}^*TM ) \oplus \mathcal{O}, 
    \end{equation}
    such that
\begin{gather} \label{eqn: o plus hat o}
    \mathscr{F}_{\widetilde{u}_{\sigma, \tau, \kappa}}(\hat{\xi})+\operatorname{pull}_{\sigma, \tau, \kappa}(o+\hat{o})=0,
\end{gather}
for $(\xi, o)=\operatorname{glue}_1(u_0, \sigma, \kappa)$.
Moreover, the following inequality is satisfied
\begin{gather}
    \|\hat{\xi}\|_{W^{1,p}}+|\hat{o}| \le c \| \overline{\partial}_{J_{\sigma, \tau}}(\widetilde{u}_{\sigma, \tau, \kappa})+\operatorname{pull}_{\sigma, \tau, \kappa}(o)\|_{L^p}. \label{eq: bound-on-obstruction}
\end{gather}
For $\kappa \in K^G$ the map $\operatorname{glue}_2$ is also $G$-equivariant, implying that for $\kappa \in K^G$ it holds that $$o, \hat{o} \in \mathcal{O}^G=\mathcal{O}_\bullet^G\oplus\mathcal{O}_{\mathghost}^G.$$
\end{proposition}

\begin{proof}
    This follows from \cite[Proposition 5.33]{doan2023}. 
\end{proof}

{\em One-step gluing.}
We also note that there is a more standard one-step gluing map in place, where the right inverse $\check{R}_{\widetilde{u}_{\sigma, \tau}}$ is obtained by applying the construction of Definition~\ref{defn: right inverse} to $D_{u_{\sigma, 0}}$ right away.
\begin{proposition}\label{proposition-one-step-gluing}
    For any $(\sigma, \tau, \kappa) \in \Delta \times K$ near the origin there is a unique
    \begin{equation}
        \operatorname{glue}(u_0, \sigma, \tau, \kappa)\coloneqq (\check{\xi}, \check{o}) \in \operatorname{Im}(\check{R}_{\widetilde{u}_{\sigma, \tau}})
    \end{equation}
    such that
    $$\mathscr{F}_{\widetilde{u}_{\sigma, \tau}}(\check{\xi})+ \operatorname{pull}_\tau\check{o}=0.$$
    Moreover, the map
    $$\operatorname{glue}(u_0, \sigma, \tau, \cdot) \colon K \to W^{1,p}_{F_{\tau}}(\Sigma_{\sigma, \tau}, u_{\sigma, \tau}^*TM)$$
    is $G$-equivariant. This, in particular, implies that for $\kappa \in K^G$ the curve $\check{u}_{\sigma, \tau, \kappa}\coloneqq \operatorname{exp}_{\widetilde{u}_{\sigma, \tau}}(\check{\xi})$ is $G$-equivariant.
\end{proposition}

\begin{remark}\label{remark: continuous-one-step-gluing}
    The present discussion, which follows \cite{doan2023}, does not guarantee all the important properties of the gluing map that are needed to show that it provides a local homeomorphism. The construction of \cite[Appendix B]{pardonvfc}, which is similar in spirit but mildly different in execution, does guarantee all those properties. In Pardon's approach, the particular choice of right inverse is crucial; see \cite[Section B.7]{pardonvfc} and also \cite[Lemma 4.10]{swaminathan2021relcinfty}. To claim that Pardon's gluing map is $G$-equivariant (which suffices to show that the moduli space of $G$-equivariant $J$-curves is a topological manifold) we would like to show that one can pick an equivariant right inverse via a similar construction.

    Here is a brief explanation:
    Let $x_1 \in \mathscr{C}_{\bullet}$ be a smooth point away from the nodes. We note that $W_{x_1}\coloneqq u^*TM|_{\pi^{-1}(x_1)}$ is naturally a finite-dimensional $G$-representation, and
    $$\operatorname{ev}_{x_1} \colon K \to W_{x_1},$$
    $$\xi \mapsto \xi|_{\pi^{-1}(x_1)},$$
    is a $G$-linear map. Now, take an arbitrary $\kappa \in K$ and its image ${\operatorname{ev}}_{x_1}(\kappa)$, which by unique continuation can be assumed to be nonzero. Let $V_{x_1}$ be a complementary subrepresentation of the subrepresentation $\operatorname{Im}(\operatorname{ev}_{x_1})$. Then $K'={\operatorname{ev}}_{x_1}^{-1}(V_1)$ is a proper subrepresentation of $K$. If $K'\neq 0$, we can now pick a point $x_2$ and repeat the above process. Eventually, we obtain a collection of smooth points $x_1, \ldots, x_k$ and a collection of subrepresentations $V_{x_i} \subset W_{x_i}$ such that the $G$-linear evaluation map
    \begin{gather}
        L_{u_0} \colon W^{1,p}_F(\Sigma, u_0^*TM) \oplus \mathcal{O} \to \bigoplus_{i=1}^k(W_{x_i}/V_{x_i}) ,\\
        (\xi, o) \mapsto ([\operatorname{ev}_{x_1}(\xi)], \dots, [\operatorname{ev}_{x_k}(\xi)]) \nonumber
    \end{gather}
    restricts to an isomorphism from $K$ to the codomain. Hence $\operatorname{ker}(L_0)$ is naturally a subrepresentation, and $\overline{D}_{u_0}|_{\operatorname{ker}(L_0)}$ is an isomorphism of representations.
\end{remark}

\subsubsection{Gluing for a boundary node}\label{subsection: boundary-node-gluing}

In this subsection we briefly explain gluing in the boundary node case, highlighting only the crucial differences. 

As in~\eqref{eq: orbidisk-with-ghost-decomposition} we assume a decomposition of $\mathscr{C}_0$ into two components, but now $\mathscr{C}_{\mathghost}$ is a disk and $\ddot{n}=(n_{\bullet}, n_{\mathghost})$ is a boundary node. We also assume that the image of the ghost component $\mathscr{C}_{\mathghost}$ is a point in $[G \cdot \mathcal{Y}/G]$ away from deeper strata, implying decomposition~\eqref{eq: ghost-domain-decomposition}. The curve $\Sigma_{m_\mathcal{Y}}$ covering $\mathscr{C}_{\mathghost}$ now has exactly $m_\mathcal{Y}$ preimages of $n_{\mathghost}$ and the space of gluing parameters~\eqref{eq: gluing-parameters} is replaced by
\begin{equation}
    \Delta=\R^{2b+d-3}_\sigma \times (\R_{\le 0})_\tau.
\end{equation}

For any node $\ddot p=(p_{\bullet}, p_{\mathghost}) \in \pi^{-1}(\ddot{n})$, the neighborhoods $U_{gp_{\bullet}}, U_{gp_{\mathghost}}$ are treated as subsets of the upper half-plane $\mathbb{H}_+$.  With respect to these coordinates, $\iota_\tau$ is given by
$$z \mapsto \frac{\tau}{z},$$
and such a map indeed maps the upper half-plane back to itself for $\tau \in \R_{\le 0}$.

The definitions of $A_\tau$ (which is now two bands instead of two annuli) and $\Sigma_{\sigma, \tau}$ remain unchanged. The $G$-equivariant curve 
$$u_0 \colon \Sigma_{\bullet} \cup \Sigma_{\mathghost} \to M$$
satisfies $u_0({\partial \Sigma_{\mathghost}}) \subset G \cdot L$ for some Lagrangian submanifold $L$. The fact that $L$ is totally geodesic guarantees that the preglued curve $u_\tau \colon \Sigma_{\sigma, \tau} \to M$ defined by~\eqref{eq: curve-pregluing} satisfies the Lagrangian boundary condition $G \cdot L$ along the boundary component of $\Sigma_{\sigma, \tau}$ corresponding to boundary components of $\Sigma_{\mathghost}$, and that the fusing operator $\operatorname{fuse}_\tau$ respects the boundary conditions.

We choose the lift $\mathcal{O}$ of the obstruction to $u_0$ to consist of three components as in~\eqref{eq: closed-ghost-full-obstruction-lift}, i.e.,
$$\mathcal{O}=\mathcal{O}_\bullet \oplus \mathcal{O}_{\operatorname{diff}}\oplus \mathcal{O}_{\mathghost},$$
where the first two components are described similarly.  The last component
$$\mathcal{O}_{\mathghost} \subset L^p_{u^*_{\mathghost}TL}(\Sigma_{\mathghost}, \Omega^{0,1}_{\Sigma_{\mathghost}} \otimes u^*_{\mathghost}TM) \subset L^p(\Sigma_{\mathghost}, \Omega^{0,1}_{\Sigma_{\mathghost}} \otimes u^*_{\mathghost}TM)$$
consists of antiholomorphic 1-forms on $\Sigma_{\mathghost}$ with values in $u^*_{\mathghost}TM \cong \Sigma_{\mathghost} \times \C^n$ preserving the real boundary condition given by $u^*_{\mathghost}|_{\partial \Sigma_{\mathghost}}TL \cong \partial \Sigma_{\mathghost} \times \R^n$. This is clearly a $G$-subrepresentation of $L^p(\Sigma_{\mathghost}, \Omega^{0,1}_{\Sigma_{\mathghost}} \otimes u^*_{\mathghost}TM)$, and, moreover, the space of $G$-invariant sections $\mathcal{O}_{\mathghost}^G$ can be identified with the space of lifts of antiholomorphic $1$-forms on $C_{\mathghost}$ with values in the desingularization bundle pair $(|(u^*TM)^{\operatorname{orb}}|, |(u^*TL)^{\operatorname{orb}}|)$ of the induced bundle pair on $\mathscr{C}_{\mathghost}$ by the results of Lemma~\ref{lemma: ker and coker}.

With these choices at hand, we have the $\tau$-translate of the linearized operator $D_{u_0}$:
$$\overline{D}_{\widetilde{u}_\tau} \colon W^{1,p}_{F_\tau}(\Sigma_\tau, \widetilde{u}_\tau^*TM) \oplus \mathcal{O} \to L^p(\Sigma_\tau, \Omega^{0,1}_{\Sigma_\tau} \otimes \widetilde{u}_\tau^*TM)$$
and its right inverse $\widetilde{R}_{\widetilde{u}_\tau}$, given by formula~\eqref{eq: formula-approximate-right-inverse}.  By the same proof as in Lemma~\ref{lemma: aproximate-right-inverse-equivariance}  (an easier one, since there are exactly $m_\mathcal{Y}$ nodes in $\pi^{-1}(n_{\mathghost})$ on $\Sigma_{m_\mathcal{Y}}$), the operator $\widetilde{R}_{\widetilde{u}_\tau}$ is equivariant, and hence admits an equivariant right inverse.

We now state the analog of Proposition~\ref{proposition: 2-step-gluing-interior node} in this setting. The proof is analogous, since the quadratic estimates carry over to this setting, and is omitted.

\begin{proposition}
    For any $(\sigma, \tau, \kappa) \in \Delta \times K$ near the origin, there is a unique map
    \begin{equation}
        \operatorname{glue}_1(u_0, \sigma, \kappa)\coloneqq(\xi, o) \in \operatorname{Im}(R_{u_{\sigma, 0}}) \subseteq W^{1,p}_F(\Sigma,u_0^*TM ) \oplus \mathcal{O} ,
    \end{equation} and there is a unique map
    \begin{equation}
        \operatorname{glue}_2(u_0, \sigma, \tau, \kappa) \coloneqq(\hat{\xi},\hat{o}) \in \operatorname{Im}(R_{\widetilde{u}_{\sigma, \tau}}) \subseteq W_{F_\tau}^{1,p}(\Sigma_{\sigma, \tau},\widetilde{u}_{\sigma, \tau, \kappa}^*TM ) \oplus \mathcal{O}, 
    \end{equation}
    such that
\begin{gather}
    \mathscr{F}_{\tilde{u}_{\sigma,0}}(\kappa+\xi)+o=0;\\
    \mathscr{F}_{\widetilde{u}_{\sigma, \tau, \kappa}}(\hat{\xi})+\operatorname{pull}_{\sigma, \tau, \kappa}(o+\hat{o})=0,
\end{gather}
where $\tilde{u}_{\sigma, \tau, \kappa}=\widetilde{(u_{\sigma, \kappa})}_\tau$ with $u_{\sigma, \kappa}=\operatorname{exp}_{u_{0, \sigma}}(\kappa+\xi)$, and $\operatorname{pull}_{\sigma, \tau, \kappa}$ is the composition of parallel transport from $u_0$ to $u_{\sigma, \kappa}$ with $\operatorname{pull}_\tau$ for the curve $u_{\sigma, \kappa}$. Moreover, the following inequalities are satisfied
\begin{gather}
    \|\xi \|_{W^{1,p}} \le c(|\sigma|+\|\kappa\|_{W^{1,p}}+|J_{(0,0)} -J_{(\sigma, \tau)}|_{C^0}); \\
    \|\hat{\xi}\|_{W^{1,p}}+|\hat{o}| \le c \| \overline{\partial}_{J_{\sigma, \tau}}(\widetilde{u}_{\sigma, \tau, \kappa})+\operatorname{pull}_{\sigma, \tau, \kappa}(o)\|_{L^p}. \label{eq: bound-on-obstruction 2}
\end{gather}
Finally, the map $\operatorname{glue}_1$ is $G$-equivariant, implying that for $\kappa \in K^G$ the curve $\widetilde{u}_{\sigma, \tau, \kappa}$ is $G$-equivariant. For $\kappa \in K^G$ the map $\operatorname{glue}_2$ is also $G$-equivariant, implying that for $\kappa \in K^G$
$$o, \hat{o} \in \mathcal{O}^G=\mathcal{O}_\bullet^G\oplus\mathcal{O}_{\mathghost}^G.$$
\end{proposition}

\subsection{Localization tricks}\label{section: localization tricks}

Our ultimate goal is to extract information on the leading terms of the Taylor expansion in the isotypic directions of the main component of a limit curve that has a ghost bubble. In the case of an \emph{interior ghost} formation, i.e., when a ghost is attached at an interior node, one would like to be able to project the curves in the sequence onto the isotypic submanifold $V_\chi$ in order to ``isolate" the leading order term corresponding to this isotypic direction, which cannot be done globally. To remedy this, we replace a given sequence of curves with ``localized'' curves on the domain without changing the behavior near the nodes where the ghost is attached.

\begin{remark}
    \cite[Theorem 1.5]{ekholmshende2025ghost} requires the curves in the sequence to be pseudoholomorphic only on the portions of domains corresponding to ghost formation. Our localization tricks may be viewed as a \emph{bridge} from the ``global on the domain" pseudoholomorphicity requirement in \cite[Theorem 1.1]{doan2023} to this ``local on the domain'' condition of Ekholm-Shende.
\end{remark}

\subsubsection{Localization trick for interior ghosts}\label{subsection: interior localization}

Let $u_k \colon \Sigma_k \to M$ be a sequence of curves in 
$$\mathcal{M}([M/G], \vv{ \hat{x}}, \hat{x}_0, \bm m)$$ 
converging to a curve $u_\infty \colon \Sigma_0 \to M$ in the Gromov topology of the form described in Section~\ref{section Kuranishi model}, i.e., $\Sigma_0=(\Sigma_{\bullet} \sqcup \Sigma_{\mathghost}, \hat{\nu})$
and the restriction of $u_\infty$ to $\Sigma_{\mathghost}$ is a ghost curve attached at the interior nodes. Let 
$$x=u_\infty(\Sigma_{m_\mathcal{Y}}) \in \mathcal{Y}$$ 
and let $\mathcal{U}_\chi$ be an $H$-equivariant Darboux chart of $\pi_\chi(x)$ in $V_\chi$.

\begin{lemma}[Localization trick for interior ghosts] \label{lemma: localization trick}
    In the situation as above, for each character $\chi$ of $H$ there exists a sequence of subdomains $\Sigma_k' \subset \Sigma_k$ and a sequence of $H$-equivariant curves $\widetilde{u}_k \colon \widetilde{\Sigma}_k \to \mathcal{U}_\chi$ satisfying the following:
    \begin{enumerate}
        \item there is a sequence of subdomains $\widetilde{\Sigma}_k ' \subset \widetilde{\Sigma}_k $, obtained by removing tubular neighborhoods of the boundary, and diffeomorphisms $\varphi_k \colon \widetilde{\Sigma}_k ' \to \Sigma_k'$ such that $\widetilde{u}_k|_{\widetilde{\Sigma}_k'}=(\pi_{\chi} \circ u_k)|_{\Sigma_k'} \circ \varphi_k$;
        \item there is a Hurwitz datum $\bm h' \subset \bm h$ determined by stacky points which belong to the underlying orbicurve of $\Sigma_\infty'$ with $|\bm h'|=b'$, and we have $(\widetilde{\pi}_k \colon \widetilde{\Sigma}_k \to \widetilde{\mathscr{C}_k}) \in \overline{\mathscr{H}}_{0,\bm h', -1}^H$;
        \item there is a Lagrangian $L_\chi \subset \mathcal{U}_\chi$ such that $\widetilde{u}_k|_{\partial \widetilde{\Sigma}_k} \subset L_\chi$ (see \eqref{eqn: defn of L} for the definition of $L_\chi$);
        \item there is a domain-dependent almost complex structure $\widetilde{J}$ on $\mathcal{U}_\chi$ varying continuously over $\overline{\mathscr{H}}_{0, \bm h', -1}^H$ such that the curves $\widetilde{u}_k$ are $\widetilde{J}$-holomorphic, and $\widetilde{J}|_{\widetilde{\Sigma}_k'}\equiv \widetilde{J}_0$, where $\widetilde{J}_0$ coincides with $J_0$ on a small neighborhood of $\pi_\chi(x)$ in $\mathcal{U}_{\chi}$, and $L_\chi$ is totally geodesic with respect to the metric induced by $\widetilde{J}_0$;
        \item the sequence $\widetilde{u}_k$ is convergent in the Gromov topology on $\overline{\mathcal{M}}([\mathcal{U}_\chi/H], [L_\chi/H],  \bm m'; \widetilde{J})$ to $\widetilde{u}_\infty$, where $\bm m'$ is a connected component of $\mathcal{I}[\mathcal{U}_\chi/H]^{b'}$ determined by the restriction of $\bm m$.
    \end{enumerate}
\end{lemma}

\begin{remark}
Before we give a proof, we would like to point out that curves $\widetilde{u}_k$ may have nontrivial obstructions for the linearized Cauchy-Riemann operator, including the curve $\widetilde{u}_\infty$, and the moduli space $\overline{\mathcal{M}}([\mathcal{U}_\chi/H], [L_\chi/H],  \bm m'; \widetilde{J})$ is only guaranteed to be an orbispace near $\widetilde{u}_\infty$. However, our gluing setup allows for such components in the obstruction space; see~\eqref{eq: closed-ghost-full-obstruction-lift}. In Theorem~\ref{ref: theorem-isotypic-direction-jet-vanishing} below, we obtain a result on the vanishing of the jets of $\widetilde{u}_\infty$ at the nodes on the side of the non-ghost components, which clearly implies the same statement for the initial curve $u_\infty$, which may belong to a smooth manifold, i.e., a zero-set of some transverse section. This allows for dimension-counting arguments, despite the lack of transversality for the localized curves. 
\end{remark}

\begin{proof}$\mbox{}$

\s\n
\emph{Description of a Darboux chart.} We claim that the Darboux chart $\mathcal{U}_\chi$ (possibly after shrinking) admits a presentation $(U \times W^{\varepsilon},\omega_{\operatorname{std}}=\omega^U_{\operatorname{std}} \times \omega^W_{\operatorname{std}})$, where $U \subset \C^{\operatorname{dim}\mathcal{Y}}$, $W^\varepsilon$ is an open ball near the origin in the representation $W$ of $H$ with character $\chi^{\oplus k}$, and $\omega_{\operatorname{std}}$ is a standard product symplectic structure. Indeed, picking a neighborhood $U$ of $\pi_\chi(x)$ in $\mathcal{Y}$ where $V_\chi$ trivializes $V_\chi \cong U \times W$ with $J_0$ constant in the vertical direction (which provides $W$ with the structure of a complex representation of $H$), we first make $\omega|_{TU }$  standard via a local diffeomorphism of $U$ using the standard Darboux theorem, and extend it trivially fiberwise. Since $U$ is contractible, there is a smooth map $L \colon U \to GL_\C(W, J_0|_W)$ such that $L(y)^* \omega^W_{\operatorname{std}}=\omega_{y}|_{T(\{y\} \times W)}$, since both $\omega_{\operatorname{std}}$ and $\omega_{y}$ are $H$-equivariant and the complex structure $J_0|_{\{y\} \times W}$ on $W$ is compatible with both of them; notice that each such $L(y)$ is automatically equivariant since $W$ is a direct sum of $1$-dimensional complex representations. Hence, after applying the diffeomorphism $\tilde{L}=(\operatorname{id}, L) \colon U \times W \to U \times W$, we can assume $\omega$ to be standard along $U$. The claim follows from the equivariant version of the relative Darboux theorem.

\s\n
\emph{Construction of the localized curves $\widetilde{u}_k$.} The sequence $u_k \colon \Sigma_k \to M$ converges in the Gromov topology to $u_\infty$, so we may assume that for $k \ge N$ for some $N\gg0$ there are $(\sigma_k, \tau_k) \in \Delta$ such that $\Sigma_k=\Sigma_{\sigma_k, \tau_k}$ and the latter is equipped with a natural $G$-action, as explained in Lemma~\ref{lemma: G-action-on-preglued-curve}. By our assumption $u_\infty(p) \in \mathcal{Y}$, so we can pick a small $R>0$ such that for any $z \in \Sigma_{\bullet}$ with $r_{gp}(z) \le R$ for some $g \in H$ it holds that $u_\infty(z)$ is in $\nu_\mathcal{Y}^{\varepsilon}$, and, moreover $\pi_\chi(u_\infty(z)) \in \mathcal{U}_\chi$. We define $$\Sigma_0' \coloneqq\Sigma_{m_\mathcal{Y}} \sqcup \{z \in \Sigma_{\bullet} \mid r_{gp}(z) \le R, \, \text{for }g \in H\},$$
and set $\Sigma'_k$ to be the portion of $\Sigma_k$ which corresponds to $\Sigma_0' \cap \Sigma_{\tau_k}$. We also notice that $Hp$ is exactly the set of all the nodes where $\Sigma_{m_\mathcal{Y}}$ is attached to $\Sigma_{\bullet}$.

Now, we pick a $\delta >0$, set $$\widetilde{\Sigma}_0=\Sigma_0'\cup\{z \in \Sigma_{\bullet} \mid r_{gp}(z) \le R+\delta, \, \text{for }g \in H\},$$ and also set $\widetilde{\Sigma}_k$ to be defined in the similar fashion as $\Sigma_k'$. Observe that the underlying orbicurves $\widetilde{\mathscr{C}_k}$ all have just $1$ boundary component and the stacky points are determined by the stacky points on $\mathscr{C}_{\mathghost}$. This determines the Hurwitz datum $\bm h' \subset \bm h$ and a connected component $\bm m'$ of $\mathcal{I}[\mathcal{U}_\chi/H]^{|\bm h'|}$.
    
By the proof of Lemma~\ref{lemma: dimension}, $\pi_\chi \circ u_\infty$ has a Taylor expansion near $p$ (notice that the Taylor expansion near any other point in $Hp$ differs from the one below by the action of $H$ on $W$, by equivariance) with respect to coordinates $U \times W^{\varepsilon}$ of the form:
\begin{equation}
    \pi_\chi\circ u_\infty(z)=(a_1z^{l_1}, a_2z^{l_2})+(T_1(z), T_2(z))=u^0(z)+T(z),
\end{equation}
where $l_2 \equiv j \operatorname{ mod} m_{\ddot{n}}, \, l_1 \equiv 0 \operatorname{ mod} m_{\ddot{n}}$ with $\chi=\chi_j$ and $T_1(z)=O(|z|^{l_1+1}), T_2=O(|z|^{l_2+1})$. We set $\psi$ to be a nonincreasing bump function such that $\psi|_{[0, R]}\equiv 1$ and $\psi|_{[R+\delta-\delta^2, +\infty)}\equiv 0$, and define 
$$\widetilde{u}_k(z)\coloneqq u^0(z)+\psi(|z|)(\pi_\chi \circ u_k(z)-u^0(z)),$$
which gives rise to an $H$-equivariant smooth curve $\widetilde{u}_k \colon \widetilde{\Sigma}_k \to \mathcal{U}_\chi$ with $\widetilde{\Sigma}_k$ obtained from $\Sigma_k$ by deleting the regions with $r >R+\delta$.

Since we can assume $a_1$ and $a_2$ to be vectors with all of their coordinates nonzero with respect to some basis on $\mathbb{C}^{\operatorname{dim}\mathcal{Y}} \times W$ with coordinates $(w_1, w_2)=(w_1^1, \dots, w_1^{\operatorname{dim}\mathcal{Y}}, w_2^1, \dots, w_2^{\operatorname{dim}W})$ we have that 
\begin{equation} \label{eqn: defn of L}
\widetilde{u}_\infty|_{r=R+\delta} \subset L_\chi=\{(w_1, w_2)\in \mathcal{U}_\chi \mid |w_i^j|=|a_i^j(R+\delta)^{l_i}|\},
\end{equation}
and $H$ acts on $L_\chi$ (the action is free for characters $\chi=\chi_j$ with $\operatorname{gcd}(j, m_\mathcal{Y})=1$). 

\s
The construction of the curves $\widetilde{u}_k$ implies (1)--(3) immediately. Our next task is to prove (4) (and hence (5)). We do it in several steps.

\s\n
\emph{Verifying that curves are symplectic.}   
We first verify that $\widetilde{u}_\infty$ (and hence each $\widetilde{u}_k$ for $k\gg0$) is symplectic. Using polar coordinates $z=r e^{i \theta}$ we have
$$\partial_r\widetilde{u}_\infty(z)=\partial_ru^0(z)+\psi'(r)T(z)+\psi(r)\partial_rT(z), \, \partial_\theta \widetilde{u}_\infty(z)=\partial_\theta u^0(z)+\psi(r)\partial_\theta T(z).$$
We observe that 
\begin{equation}\label{eq: area-estimate}
    \omega_{\operatorname{std}}(\partial_r \widetilde{u}_\infty, \partial_\theta \widetilde{u}_\infty)=\omega_{\operatorname{std}}(\partial_r u^0, \partial_\theta u^0)+C(\delta)\cdot O(r^{2l})
\end{equation}  
on the interpolation region with $l= \min(l_1, l_2)$, and $$\omega_{\operatorname{std}}(\partial_r u^0, \partial_\theta u^0)=r|\partial_ru^0|^2=Cr^{2l-1}+O(|r^{2l}|).$$ 
The following are the estimates for terms that force us to pick a constant depending on $\delta$ in~\eqref{eq: area-estimate}
$$|\omega_{\operatorname{std}}(\psi'(r)T(z),\partial_\theta u^0(z))| \le C' \frac{r^{2l+1}}{\delta}, \quad |\omega_{\operatorname{std}}(\psi'(r)T(z), \psi(r) \partial_\theta T(z))| \le C'' \frac{r^{2l+2}}{\delta}.$$
To make sure that it is dominated by $r^{2l-1}$ on the interpolation region, we require $\frac{R^2}{\delta}$ to be sufficiently small (notice that we are also free to shrink $R$ prior to this choice). For such a choice of $R, \delta>0$, the curve $\widetilde{u}_\infty$ is symplectic.

\s\n
{\em Definitions of $\widetilde{J}_0$ and $\widetilde{J}_k$.}  We now define $\widetilde{J}_0$ and $\widetilde{J}_k$ satisfying (4).

Let $\widetilde{J}_0$ be an $\omega_{\operatorname{std}}$-tame almost complex structure, obtained via an interpolation, which coincides with $J_0$ when $|w| \le |(a_1^1R^{l_1}, \dots, a_2^{\operatorname{dim}W}R^{l_2})|$ and is equal to the standard complex structure near $L_{\chi}$. 

We may choose an equivariant domain-dependent almost complex structure $\widetilde{J}_k$ on $\mathcal{U}_\chi$ for $\widetilde{\Sigma}_k$ such that it coincides with $\widetilde{J}_0$ on a slightly smaller subdomain $\Sigma_k'' \subset \Sigma_k'$, $\widetilde{u}_k$ is $\widetilde{J}_k$-holomorphic, $\widetilde{J}_k$ is $\omega_{\operatorname{std}}$-tame, $L_{\chi}$ is totally geodesic with respect to the induced metric by each $\widetilde{J}_k$, and $\widetilde{J}_k$ converges to $\widetilde{J}_\infty$. 

For each $z \in \widetilde{\Sigma}_k \setminus \Sigma_k'$, we define $\widetilde{J}_k(z)=\widetilde{J}_0\operatorname{exp}Y_k'$ at the point $\widetilde{u}_k(z)$  with $Y_k' \in \operatorname{End}(T_{\widetilde{u}_k(z)}\mathcal{U}_\chi) $ to satisfy the requirement for $\widetilde{u}_k$ to be pseudoholomorphic at $z$, with small norm $|Y_k'|$, which would guarantee tameness. Indeed, a sufficient condition for such  
$Y_k'$ to have small norm is that
\begin{equation}\label{eq: comparison-vector-norms}
    \frac{|\partial_t\widetilde{u}_k(z)-\widetilde{J}_0\partial_s\widetilde{u}_k(z)|}{|\partial_s\widetilde{u}_k(z)|}
\end{equation}
be small (notice that this is an estimate for the linear term in $e^{Y_k}$, which suffices). This, by an argument similar to the one presented above, is bounded above by $C(R+\frac{R^2}{\delta})$, and hence the choice of small $R, \delta>0$ with $R^2/\delta$ small yields the desired bound. Below we provide a sketch of an argument for the existence of $Y_k'$ under the assumption that the term~\eqref{eq: comparison-vector-norms} is sufficiently small.

The existence of $Y_k'(z)$ of small norm once the quantity in~\eqref{eq: comparison-vector-norms} is small is an application of the inverse function theorem, and we sketch it here.

We consider
$$\mathscr{Y}=\{Y \in C^\infty(\operatorname{End}(TU_{\chi})) ~|~ Y\widetilde{J}_0+\widetilde{J}_0Y=0, \, \|Y\|_{C^0}< \varepsilon'\},$$
for a fixed small $\varepsilon'>0$. Let us note that there is a natural $H$-action on $\mathscr{Y}$. We assume the trivialization $TU_{\chi} \cong U_{\chi} \times V$, where $V \cong \C^{\operatorname{dim}\mathcal{Y}+\operatorname{dim}W}$, and the latter product is equipped with its natural $H$-bundle structure. Denote by $\mathbb{S}^V$ the unit sphere bundle of $U_{\chi} \times V$ associated with a metric $\widetilde{g}_0$, which is induced by $\widetilde{J}_0$. We consider an $H$-equivariant map
\begin{gather*}
    \mathcal{F} \colon \mathscr{Y} \times \mathbb{S}^V \to V \times \mathbb{S}^V, \\
    (Y, x, n) \mapsto (\operatorname{exp}(Y(x))n-n, x, n),
\end{gather*}
where $n \in \mathbb{S}^V_x$ is a unit tangent vector at $x \in \overline{U}_{\chi}$. Notice that $D\mathcal{F}_{(0, x, n)}(Z, \dot{x}, \dot{n})=(Z(x)n, \dot{x}, \dot{n})$.

We choose a right inverse 
\begin{equation*}
    \Psi \colon V \times \mathbb{S}^V \to \mathscr{Y} \times \mathbb{S}^V
\end{equation*}
to $D\mathcal{F}$ as follows.
First, we set
$$\Psi^{\mathscr{Y}}(v, x, n)(x)[w]\coloneqq \widetilde{g}_0(w,n)v-\widetilde{g}_0(w,\widetilde{J}_0n)\widetilde{J}_0v,$$
where we regard $(v,w)$ as a pair of tangent vectors in $TU_{\chi}|_{x}$. Clearly, $$\Psi^{\mathscr{Y}}(v,x,n)(x)\widetilde{J}_0(x)+\widetilde{J}_0(x)\Psi^{\mathscr{Y}}(v,x,n)(x)=0.$$
We further extend $\Psi^{\mathscr{Y}}(v,x,n)$ to all points in $H$-orbit of $x$ by equivariance (and this coincides with similarly defined $\Psi^{\mathscr{Y}}(v,hx,n)(hx)$ for any $h \in H$ by equivariance of $\widetilde{J}_0$ and $\widetilde{g}_0$).

To extend this from $x \in \overline{U}_{\chi}$ to all other points $y \in \overline{U}_{\chi}$ we apply the projection to the $\widetilde{J}_0(y)$-antilinear part, using the choice of trivialization of $T\overline{U}_\chi$, i.e.
$$\Psi^{\mathscr{Y}}(v, x, n)(y)\coloneqq \frac{1}{2}(\Psi^{\mathscr{Y}}(v,x,n)(x)+\widetilde{J}_0(y)\Psi^{\mathscr{Y}}(v,x,n)(x)\widetilde{J}_0(y)).$$
We define $\Psi(v,x,n)\coloneqq (\Psi^{\mathscr{Y}}(v, x, n), x, n).$
It is straightforward to check that $\Psi$ is $H$-equivariant and that $$D (\mathcal{F} \circ \Psi)_{(0, x, n)}=\operatorname{id}.$$

The parametrized inverse function theorem (for the \emph{compact} base $\mathbb{S}^V$) implies that there exists an inverse $\mathcal{G}$ to this composition over $B(\varepsilon'') \times \mathbb{S}^V$, where $B(\varepsilon'') \subset V$ is a ball of some small radius $\varepsilon''>0$ in $V$ with respect to the standard Euclidean norm $|\cdot|$.

For a point $z \in \widetilde{\Sigma}_k \setminus \Sigma_k'$ we consider a point
$$a_k(z)=\left(\frac{-\widetilde{J}_0\partial_t\widetilde{u}_k(z)-\partial_s\widetilde{u}_k(z)}{|\partial_s\widetilde{u}_k(z)|_{\widetilde{g}_0}},\widetilde{u}_k(z), \frac{\partial_s\widetilde{u}_k(z)}{|\partial_s\widetilde{u}_k(z)|_{\widetilde{g}_0}}\right) \in V \times \mathbb{S}^V.$$
Since $\widetilde{g}_0$ is a family of metrics on $V$ over a compact ball $\overline{U}_{\chi}$, there exist positive constants $c_1, c_2 >0$ such that 
$$c_1|\cdot| \le |\cdot|_{\widetilde{g}_0(x)} \le c_2|\cdot| $$
for any $x \in \overline{U}_\chi$, where $|\cdot|$ is the norm induced by the standard Euclidean metric. Hence, by our assumption on the smallness of~\eqref{eq: comparison-vector-norms}, we may assume that for our choice of $R, \delta>0$ the point $a_k(z)$ belongs to $B(\varepsilon'') \times \mathbb{S}^V$. Then, for $Y_k'(z)=\Psi^{\mathcal{Y}}\circ \mathcal{G}(a_k(z))$, a direct computation shows
$$\operatorname{exp}(Y_k')\partial_s\widetilde{u}_k(z)=-\widetilde{J}_0\partial_t\widetilde{u}_k(z),$$
implying that $\widetilde{u}_k(z)$ is $J_k'$-holomorphic for $J_k'=\widetilde{J}_0\operatorname{exp}(Y_k')$. By construction, the family $Y_k'(z)$ is $H$-equivariant.

\s\n
{\em Conclusion of proof.}  
To define $\widetilde{J}_k=\widetilde{J}_0\operatorname{exp}(Y_k)$ we choose a family of $H$-invariant bump-functions $\rho_k(z)(\cdot)$ on $\overline{U}_\chi$ and set $Y_k(z)=\rho_k(z)Y'_k(z)$. For $z \in \widetilde{\Sigma}_k \setminus \Sigma_k'$ with $r \in [R, R+\delta-\delta^2]$ these bump-functions are centered at $H$-orbits of $\widetilde{u}_k(z)$ and vanish away from balls of some radius $\hat{R}(r)$ varying smoothly with $r$, and required to be smaller than $\delta^N$ with $N\gg0$ and to vanish near the boundaries of this segment. We then extend $Y_k(z)$ away from this neck by setting it to be equal to $0$.

From the construction, it is clear that each $\widetilde{J}_k$ is equivariant and smooth, as it only depends on the initial choice of $Y^c_k(\widetilde{u}_k(z))$ with $z \in \widetilde{\Sigma}_k$,  and the sequence $\widetilde{J}_k$ converges to $\widetilde{J}_\infty$ in any $C^m$-norm. We also have that $\widetilde{J}_k \equiv \widetilde{J}_0$ on $\Sigma_k''$ corresponding to $|z| < R$.

At last, since $\widetilde{J}_k$ converge to $\widetilde{J}_\infty$, this sequence extends to a continuously varying family $\widetilde{J}$ over $\overline{\mathscr{H}}_{0,\bm h',-1}^H$. Moreover, we can guarantee that $|\widetilde{J}-\widetilde{J}_\infty|_{C^m}$ is small for any $m\ge0$ for any point in the small neighborhood of the point representing $\widetilde{\mathscr{C}}_\infty$.
\end{proof}

\begin{remark}\label{rmk: non-almost-hermitian}
    The Lagrangian $L_\chi$ we constructed in the proof is not totally geodesic with respect to the standard metric. Nevertheless, in the pregluing one can instead use a metric which coincides with the metric induced by $\widetilde{J}_0$ away from $L_\chi$, and for which $L_\chi$ is totally geodesic. This does not affect any of the gluing lemmas that we use.
\end{remark}

\subsubsection{Localization trick for boundary ghosts}

In this subsection, we state and sketch the localization trick for boundary ghosts. Unlike in the previous case, where the reason for localization was forced on us by the necessity of projecting onto isotypic components, our motivation here is simply to avoid dealing with a slightly more involved gluing setup, which involves Hamiltonian data and time-shifting maps.

As in the interior case, $\mathcal{Y}$ is a connected component of $X^H$, where the cyclic group $H$ is the stabilizer of a generic point of $\mathcal{Y}$. Let $L$ be a connected Lagrangian preserved by the action of $H$ and $x \in \mathcal{Y} \cap L'$, where $L'$ is a connected component of $L^H$. Since $J_0$ is $\omega$-tame, we have
\begin{equation}
    \nu_\mathcal{Y} \cong \nu_{L'} \otimes_{\R} \C,
\end{equation}
where $\nu_{L'}$ is the normal bundle for $L' \subset L$, treated as a real $H$-bundle over $L'$. We assume that $\nu_{L'}$ does not have any vectors of eigenvalue $(-1)$, and therefore can be treated as a complex $H$-bundle.

Let $u_k \colon \Sigma_k \to M$ be a sequence of curves with boundary conditions on some equivariant Lagrangians $\mathcal{L}_0, \dots, \mathcal{L}_d$, and on $\mathcal{L}=G \cdot L$, with punctures mapped to either $G$-orbits of intersection points, or Hamiltonian orbichords. We also assume that $\mathcal{L}$ is totally geodesic with respect to the metric induced by $J_0$. We assume, as in the interior case, that domains correspond to points of $\overline{\mathscr{H}}_{0, \bm h, d}^G$, and they are the solutions to the Floer equation~\eqref{eq: Floer-equation} with respect to some continuously varying Floer data $(\rho, \alpha, H, J)$ on $\overline{\mathscr{H}}_{0, \bm h, d}^G$. For what follows, it is important to require that $\alpha$ vanishes near the stacky points and on the closed components, while $J$ coincides with $J_0$ on these neighborhoods.

A particular example of interest is that of $M=T^*X$ and $X=\mathcal{L}$, with the fixed point set $\mathcal{Y}$ being a cyclic quotient singularity $T^*Y$, arising from some submanifold $Y \subset X$, which is a component of $X^H$. This case is studied in more detail in Section~\ref{section: evaluation-map-iso}. 
We assume that the sequence of curves $u_k \colon \Sigma_k \to M$ 
converges to a curve $u_\infty \colon \Sigma_0 \to M$ in the Gromov topology of the form described in Section~\ref{subsection: boundary-node-gluing}, i.e., $\Sigma_0=(\Sigma_{\bullet} \sqcup \Sigma_{\mathghost}, \hat{\nu})$
and restriction of $u_\infty$ to $\Sigma_{\mathghost}$ is a ghost curve attached at boundary nodes mapped to $X$. Let us assume that $x=u_\infty(\Sigma_{m_\mathcal{Y}}) \in \mathcal{Y} \cap L'$. We denote by $\mathcal{U}$ an $H$-equivariant Darboux chart of $x$.

\begin{lemma}\label{lemma: localization trick for boundary}
    In the situation as above, one can pick a sequence of subdomains $\Sigma_k' \subset \Sigma_k$ and a sequence of $H$-equivariant curves $\widetilde{u}_k \colon \widetilde{\Sigma}_k \to \mathcal{U}$ satisfying the following:
    \begin{enumerate}
        \item there is a sequence of subdomains $\widetilde{\Sigma}_k ' \subset \widetilde{\Sigma}_k $ and diffeomorphisms $\varphi_k \colon \widetilde{\Sigma}_k ' \to \Sigma_k'$ such that $\widetilde{u}_k|_{\widetilde{\Sigma}_k'}=u_k|_{\Sigma_k'} \circ \varphi_k$;
        \item there is a Hurwitz datum $\bm h' \subset \bm h$ determined by stacky points which belong to the underlying orbicurve of $\Sigma_\infty'$ with $|\bm h'|=b'$, and we have $(\widetilde{\pi}_k \colon \widetilde{\Sigma}_k \to \widetilde{\mathscr{C}_k}) \in \overline{\mathscr{H}}_{0, \bm h',1}^H$;
        \item there is a Lagrangian $\widetilde{L} \subset \mathcal{U}$ and two $|H|$-tuples $\hat{e}^+$ and $\hat{e}^-$ of $H$-orbits of intersection points in $L \cap X$ such that curves $\widetilde{u}_k$ have their punctures mapped to $\hat{e}$ and $\hat{e}'$ and their boundaries to $L$ and $\widetilde{L}$;
        \item there is a domain-dependent almost complex structure $\widetilde{J}$ on $\mathcal{U}$ varying continuously over $\overline{\mathscr{H}}_{0, \bm h', 1}^H$ such that $\widetilde{J}|_{\widetilde{\Sigma}_k'}\equiv J_0$ and $\widetilde{u}_k$'s are $\widetilde{J}$-holomorphic;
        \item the sequence $\widetilde{u}_k$ is convergent in Gromov topology on $\overline{\mathcal{M}}([\mathcal{U}/H], [L/H], [\widetilde{L}/H],  \bm m'; \widetilde{J})$ to $\widetilde{u}_\infty$, where $\bm m'$ is a connected component of $\mathcal{I}[\mathcal{U}/H]^{b'}$ determined by restriction of $\bm m$.
    \end{enumerate}
\end{lemma}

\begin{proof}[Sketch of proof]
    Let $W_{\R}$ be the representation of $H$ on $T_xL$, obtained from the complex representation $W$ of $H$ via the forgetful functor. Then $T_{x}M$ is equal to $W \oplus \overline{W}$ as a complex representation of $W$, and notice that under this splitting $T_xL=L_{\R}=\{(w, \overline{w}) \in W \oplus \overline{W}\}$. Using similar arguments as in Lemma~\ref{lemma: localization trick}, we may assume that the Darboux chart $\mathcal{U}$ is given by a neighborhood of the origin in $(W \oplus \overline{W}, \omega_{\operatorname{std}}^W \oplus \omega_{\operatorname{std}}^{\overline{W}})$ with $L \cap \mathcal{U}$ given by $L_{\R}$ and $\mathcal{Y}$ given by $W^H \oplus \overline{W}^H$. Since $H$ is cyclic, we have a decomposition $W=W_1 \oplus \dots \oplus W_{n/2}$ into $1$-dimensional complex subrepresentations, hence we have coordinates $(w^+, w^{-})=(w_1^+, w_1^{-}, \dots, w_{n/2}^+, w_{n/2}^-)$ on $\mathcal{U}$.

    In these coordinates $u_\infty$, when restricted to a neighborhood of $p_{\bullet}$ for some node $\ddot{p} \in \Sigma_{m_{\mathcal{Y}}}$ with coordinate $z \in \mathbb{H}_+$ near $p_{\bullet}$, has Taylor series
    $$u_\infty(z)=(a_1^+z, a_1^-z, \dots, a_{n/2}^+z, a_{n/2}^-z)+T(z)=u^0(z)+T(z),$$
where all $a_i^{\pm}$ can be assumed to be nonzero. The Taylor expansion at any other node in $Hp$ is obtained via $H$-action, and, in particular, leading terms are given by the orbit of the vector $$(a^+, a^-)=(a_1^+, a_1^-, \dots, a_{n/2}^+, a_{n/2}^-).$$ Notice that $\overline{a}^+_i=a_{i}^-$ since $u_\infty$ maps $\operatorname{Im} z=0$ to $L_{\R}$.
In addition to a bump function $\psi$ such that $\psi|_{[0, R]}\equiv 1$ and $\psi|_{[R+\delta, +\infty)}\equiv 0$, we also pick a bump function $\psi'$ such that 
$$\psi'|_{[0, R+\delta]}\equiv 0, \quad \psi'|_{[R+2\delta, +\infty)}=1,$$
and define
$$\widetilde{u}^k(gz)\coloneqq u^0(gz)+\psi(|z|)(u_k(gz)-u^0(gz))+\varepsilon'\psi'(|z|)\cdot \frac{z^2-(R+2\delta)^2}{2i(R+2\delta)}\cdot g(0, a_1^-,\dots, 0, a_{n/2}^-),$$
where $\varepsilon'>0$ is sufficiently small, and $gz$ is treated as a coordinate near $gp_{\bullet}$. Here we use $\psi$ from the previous Lemma.

By inspection, one sees that $\widetilde{u}^k|_{r=R+2\delta}$ is mapped to the $H$-orbit $\widetilde{L}$ of the Lagrangian torus
$$\hat{L}=\{(R+2\delta)\cdot(a_1^+e^{i\theta_1^+}, \overline{a}^+_1(1+\varepsilon'\sin\theta_1^{-})e^{i\theta_1^-}, \dots,a_{n/2}^+e^{i\theta_{n/2}^+}, \overline{a}^+_{n/2}(1+\varepsilon'\sin\theta_{n/2}^{-})e^{i\theta_{n/2}^-}) \mid \theta_i^{\pm} \in [0, 2\pi)] \},$$
or in other words denoting by $\gamma_{a,\epsilon}$ the Lagrangian
$\gamma_{a, \epsilon}=\{a(1+\epsilon\sin \theta)e^{i\theta} \mid \theta\in [0, 2\pi)\} \subset \C$
we have 
$$\hat{L}=\gamma_{a_1^+, 0} \times \gamma_{\overline{a}_1^+, \varepsilon'} 
\times \dots \times \gamma_{a_{n/2}^+, 0} \times \gamma_{\overline{a}_{n/2}^+, \varepsilon'} \subset W \oplus \overline{W}.$$
By construction it is clear that $\hat{L}$ intersects $L_{\R}$ transversely in finitely many points, and for points $$e^{\pm}=(R+2\delta)\cdot(\pm a_1^+, \pm \overline{a}_1^+, \dots, \pm a_{n/2}^+, \pm \overline{a}_{n/2}^+) \in \hat{L} \cap L_{\R}$$
we have $\widetilde{u}_k(\pm(R+2\delta))=e^{\pm}$. It is clear that under our assumptions, the action of $H$ on $\widetilde{L}$ is free, and the tori in $\widetilde{L}$ are disjoint.

The choice of a continuously varying family of equivariant domain-dependent almost complex structures on $\mathcal{U}$ goes through as in Lemma~\ref{lemma: localization trick}, and we leave it as an exercise.  
\end{proof}

\begin{remark}
    As before in Remark~\ref{rmk: non-almost-hermitian}, we point out that $\widetilde{L}$ is not totally geodesic with respect to the standard almost complex structure. For gluing constructions, one can use a metric which coincides with the one induced by $J_0$ near point $x$, making $L$ totally geodesic there, and such that both $L$ and $\widetilde{L}$ are totally geodesic. Such a metric is sufficient for all the estimates performed in \cite[Section 5]{doan2023}.
\end{remark}

\subsection{Obstructing leading order terms for equivariant maps} \label{subsection: obstructing leading order terms}

\subsubsection{Interior ghosts}

In this subsection, we work in the setting obtained by applying the localization trick of Section~\ref{subsection: interior localization}: We consider an $H$-equivariant curve $$u_0 \colon (\Sigma_0, \partial \Sigma_0) \to (\mathcal{U}_\chi, L_\chi),$$ where $\mathcal{U}_\chi \cong\C^{n_\chi} \times \C^{n_0}$
with the standard symplectic structure. The group $H \cong \mu_m$ acts linearly on $\C^{n_\chi}$ by the character $\chi$ such that $\chi(\zeta_m)=e^{\frac{2\pi i j}{m}}=\zeta_m^j$, and trivially on $\C^{n_0}$. Moreover, $\Sigma_0=\Sigma_\bullet \cup \Sigma_{\mathghost}$, where $\Sigma_{\mathghost}$ is a closed curve 
and $u_0(\Sigma_{\mathghost})=x_0=\{0\} \subset \mathcal{U}_\chi.$ The underlying orbicurve $\mathscr{C}_0$ has a unique node $\ddot{n}$ with local monodromy of order $m_{\ddot{n}}$. 

Let $p=(p_\bullet, p_{\mathghost}) \in \pi^{-1}(\ddot{n})$ be one of the $d_{\ddot{n}}=m/m_{\ddot{n}}$ nodes of $\Sigma_0$ connecting the main and ghost components. Let $g(p)=\zeta_m^{d_{\ddot{n}}\cdot a}$ be the generator of the stabilizer $G_p$ prescribed by the Hurwitz datum, where $a \in \{1, \dots, m_{\ddot{n}}\}$ with $\gcd(a, m_{\ddot{n}})=1$.

We define
\begin{equation}\label{eq: local character}
     k(\chi,\Sigma_{\mathghost})=k\in\{1,\dots,m_{\ddot{n}}\}
\end{equation} to be the number satisfying
$$k \equiv ja \pmod{m_{\ddot{n}}}.$$ 
Equivalently, for the generator $$g(p)=\zeta_m^{d_{\ddot{n}}\cdot a}=\zeta_{m_{\ddot{n}}}^a \in G_p \subset H$$ we have $\chi(g(p))=\zeta_{m_{\ddot{n}}}^k$. 

\s\n
{\em Basic calculation.} To aid the reader, we include the following description when $d_{\ddot{n}}=1$ and $m=m_{\ddot{n}}$.  After projecting $u_0$ to $V_{\chi^m_j}$, locally, and up to leading term, we have a $\zeta_m$-equivariant holomorphic map $u_{\op{loc}}:\C\to \C$, $w\mapsto w^k$, where $\zeta_m^a$ acts by $w\mapsto \zeta_m w$ on the domain (this is by the definition of the Hurwitz datum/local monodromy) and $\zeta_m$ acts by $z\mapsto \zeta_m^j z$ on the codomain.
We have a commutative diagram:
\[
\begin{tikzcd}
w \arrow[r, mapsto, "u_{\op{loc}}"] \arrow[d, mapsto]
  & w^{k} \arrow[d, mapsto] \\
\zeta_m w \arrow[r, mapsto, "u_{\op{loc}}"]
  & (\zeta_m^{a})^j w^{k}= (\zeta_m w)^k
\end{tikzcd}
\]
where the vertical arrows represent the actions by $\zeta_m^a$. Hence $k\equiv ja$.

\s
By Lemma~\ref{lemma: dimension} the Taylor expansion of $u_\bullet$ at $p_\bullet$, with coordinate $w$ on $\Sigma_{\bullet}$ centered at $p_{\bullet}$, has the form
$$u_{\bullet}(w) = \frac{1}{k!}\cdot\frac{\partial^{k}u_{\bullet}}{\partial w^{k}}(0)w^{k} +O(|w|^{k+1}).$$ 
For later use, we denote by $z$ the standard coordinate on $\Sigma_{\mathghost}$ centered at $p_{\mathghost}$.
Recall also that $$\frac{\partial^{k}u_{\bullet}}{\partial w^{k}}(0) = D^{k}u_{\bullet}(p)[\partial_{w}^{\otimes k}].$$ 
Together with $x_0=u_{\bullet}(p)$ this partial derivative determines the $k$-jet of $u_\bullet$ at $p$; see Remark~\ref{remark: tangent-jet-space-first}.

The main result of Section~\ref{subsection: obstructing leading order terms} identifies the leading term of the $L^{2}$-pairing of the obstruction section 
$$\operatorname{ob}(\sigma,\tau,\kappa):=o+\hat{o}$$ 
satisfying Equation~\eqref{eqn: o plus hat o}, with relevant elements of the ghost obstructions $$\mathcal{O}_{\mathghost}^H = (H^{0,1}(\Sigma_{\mathghost}) \otimes T_{x_0} X)^H.$$
There is a special case of the previous discussion when $G_p$ is trivial and $m_{\ddot{n}}=1$. In this case, $k=1$, and the discussion regarding the Taylor expansion remains valid.

Together with the vanishing of the obstruction term, this result implies the vanishing of the leading jet $D^k u_\bullet(p)=0$; see Theorem~\ref{ref: theorem-isotypic-direction-jet-vanishing}.

We denote by $H^{1,0}(\Sigma_{\mathghost})_{\chi}$ the subspace of holomorphic $1$-forms on which $H$ acts via the character $\chi$.

\begin{proposition}\label{proposition: obstruction-on-higher-jets}
For each $(\sigma, \tau, \kappa) \in \Delta \times K^H$ of sufficiently small norm and each 
$$\overline{\eta}\otimes v \in H^{0,1}(\Sigma_{\mathghost})_{\chi^{-1}}\otimes(T_{x_0} M)_{\chi},$$ 
the following holds:
\begin{gather}
    \langle \operatorname{pull}(\operatorname{ob}(\sigma,\tau,\kappa)), \operatorname{pull}(\overline{\eta}\otimes v) \rangle_{L^{2}} = C \cdot d_{\ddot{n}} \langle (D^k u_{\sigma, \kappa}(p_{\bullet})\otimes D^{k-1}\eta(p_{\mathghost}))[\tau\partial_w \otimes\partial_z]^{\otimes k}, v \rangle + e , \nonumber
\end{gather}
where $|e| \le C_0 \varepsilon^{k+1/2}|\eta|\cdot|v|$ for some constants $C, C_0 \in \R$, $C\ge 0$. Here $\varepsilon=|\tau|$, where $\tau$ is the gluing parameter from before, and $\op{pull}=\op{pull}_{\sigma,\tau,\kappa}$.
\end{proposition}

\begin{proof}
Our proof follows the same steps as that of \cite[Lemma 5.47]{doan2023}. The main novelty is the choice of $\eta$ in the $\chi$-subrepresentation, which modifies Step 3 of their proof. The remaining steps are essentially unchanged, and we include the details for clarity.

By Equation~\eqref{eqn: o plus hat o} we have:
\begin{equation}\label{eq: J-holomorphic-gluing-identity}
    \langle\mathscr{F}_{\widetilde{u}_{\sigma,\tau,\kappa}}(\hat{\xi})+\operatorname{pull}(o+\hat{o}), \operatorname{pull}(\overline{\eta}\otimes v)\rangle_{L^2}=0,
\end{equation}
and there is an identity
\begin{equation}\label{eq: J-section-quadratic-expression}
    \mathscr{F}_{\widetilde{u}_{\sigma,\tau,\kappa}}(\hat{\xi})=\overline{\partial}_{J_{\sigma, \tau}}\widetilde{u}_{\sigma,\tau,\kappa}+D_{\widetilde{u}_{\sigma,\tau,\kappa}}\hat{\xi}+\mathscr{N}_{\widetilde{u}_{\sigma,\tau,\kappa}}\hat{\xi},
\end{equation}
 where $\mathscr{N}_{\widetilde{u}_{\sigma,\tau,\kappa}}\hat{\xi}$ is the quadratic and higher order term for $\hat\xi$. The proof strategy is to utilize~\eqref{eq: J-holomorphic-gluing-identity} and estimate the pairing of each term in~\eqref{eq: J-section-quadratic-expression} with $\operatorname{pull}(\overline{\eta} \otimes v)$.

\s\n
\emph{Step 1.} Recall that $\operatorname{ob}(\sigma,\tau,\kappa) = (o,\hat{o}) \in \mathcal{O}_\bullet \oplus \mathcal{O}_{\mathghost}$. Since $o$ vanishes near $\pi^{-1}(\ddot{n})$ by the definition of $\mathcal{O}_\bullet$, we have
\begin{equation}
    \langle \operatorname{pull}(o), \operatorname{pull}(\overline{\eta} \otimes v) \rangle_{L^2}=0.
\end{equation}

\s\n
\emph{Step 2.}
Recall that $H/G_p \cong \mu_{d_{\ddot{n}}}$, and we regard it as a subgroup of $H$.
The term 
\begin{equation} \label{eqn: step 2 term}
    \overline{\partial}_{J_{\sigma,  \tau}}\widetilde{u}_{\sigma,\tau,\kappa} + \operatorname{pull}(o)
\end{equation}
is clearly supported in the regions where $r_{gp_\star} \le 2R_{0}$ for some $g \in \mu_{d_{\ddot{n}}}$  and $\star=\bullet,\mathghost$.
By inspecting the formula, one sees that  \eqref{eqn: step 2 term} vanishes for $\star=\bullet$ and for $\star=\mathghost$ and $r_{gp_{\star}}\leq R_0$, since $u_0$ is a ghost on $\Sigma_{\mathghost}$.
Hence \eqref{eqn: step 2 term} is supported on $ R_0\leq  r_{g p_{\mathghost}} \le 2R_{0}$. On the local chart near $x_g=u_{\sigma, \kappa}(gp_{\mathghost})$:
$$\widetilde{u}_{\sigma,\tau,\kappa} = \psi_{\tau} \cdot u_{\sigma,\kappa} \circ \iota_{\tau},$$
where $\psi_\tau$ is given by Equation~\eqref{eqn: psi tau}.
We have:
\begin{align}  \label{eq: sum-for-dbar+pull}
    \overline{\partial}_{J_{\sigma, \tau}} \widetilde{u}_{\sigma,\tau,\kappa} + \op{pull}(o) &= \overline{\partial}_{J_{\sigma, \tau}}(\psi_{\tau} \cdot u_{\sigma,\kappa} \circ \iota_{\tau})\\
     \label{eq: sum-for-dbar+pull2}
    &= \overline{\partial}\psi_{\tau} \cdot u_{\sigma,\kappa} \circ \iota_{\tau}+\tfrac{1}{2}(J_{\sigma, \tau}(\widetilde{u}_{\sigma, \tau, \kappa})-J_{\sigma, \tau}(x_g))d\psi_\tau \circ j_\sigma\cdot u_{\sigma, \kappa} \circ \iota_\tau   \\
    &\quad+\tfrac{1}{2}\psi_{\tau}(J_{\sigma, \tau}(\widetilde{u}_{\sigma, \tau, \kappa})-J_{\sigma, \tau}(u_{\sigma, \kappa})) d(u_{\sigma, \kappa}\circ \iota_{\tau}) \circ j_{\sigma} + \psi_\tau \overline{\partial}_{J_{\sigma, \tau}}(u_{\sigma, \kappa}\circ \iota_\tau). \nonumber
\end{align}

The last term vanishes since $u_{\sigma,\kappa}$ is $J_{\sigma, \tau}$-holomorphic on this region by construction. In view of the $J_{\sigma, \tau}$-holomor\-phicity and the $H$-equivariance of $u_{\sigma,\kappa}$, we can apply Lemma~\ref{lemma: dimension}, which gives a restriction on the Taylor series of $u_{\sigma,\kappa}$ about $gp_{\mathghost}$ and
\begin{equation}\label{eqn: estimates for u and du}
    |u_{\sigma,\kappa} \circ \iota_{\tau}(z)| \le C \frac{\varepsilon^k}{r_{gp_{\mathghost}}^k(z)}, \quad |d(u_{\sigma,\kappa} \circ \iota_{\tau})(z)| \le C \frac{\varepsilon^k}{r_{gp_{\mathghost}}^{k+1}(z)},
\end{equation}
where $z$ is the standard local coordinate on $\Sigma_{\mathghost}$ near $gp_{\mathghost}$ with $z(gp_{\mathghost})=0$. The $G_p$-action on this coordinate has character $\zeta_{m_{\ddot{n}}}^{-1}$.
Hence, the second and third terms in the sum~\eqref{eq: sum-for-dbar+pull2} have $C^0$-norm bounded by $C \varepsilon^{2k}$ (here we use the smoothness of $J_{\sigma, \tau}$ and $r_{gp_{\mathghost}}(z)\geq R_0$). Finally, the first term has an upper bound 
$$|\overline{\partial}\psi_{\tau} \cdot u_{\sigma,\kappa} \circ \iota_{\tau}| \le C  \varepsilon^{k}.$$
Therefore, we have
\begin{gather}\label{eq: leading+e1}
    \overline{\partial}_J \widetilde{u}_{\sigma,\tau,\kappa} + \operatorname{pull}(o) = \overline{\partial}\psi_{\tau} \cdot u_{\sigma,\kappa} \circ \iota_{\tau} +  e_1',  \text{ with}\\
    |\overline{\partial}\psi_{\tau} \cdot u_{\sigma,\kappa} \circ \iota_{\tau}| \le C\varepsilon^k,~~ |e_1'| \le C \varepsilon^{2k},  \nonumber\\
    \label{eq: e1 term}
     |e_1|:= |\langle e_1', \operatorname{pull}(\overline{\eta}\otimes v) \rangle_{L^2}| \le C \varepsilon^{2k}|\eta| |v|.
\end{gather}

\s\n
\emph{Step 3.} We estimate the term
\begin{equation} \label{eq: pairing-with-leading-term}
\langle \overline{\partial}\psi_{\tau} \cdot u_{\sigma,\kappa} \circ \iota_{\tau}, \op{pull}(\overline{\eta}\otimes v) \rangle_{L^2}.
\end{equation}
Again, since $u_{\sigma,\kappa}$ is $J$-holomorphic in the standard coordinate $z$ near the node $p_{\mathghost}$, we have
$$u_{\sigma,\kappa} \circ \iota_{\tau}(z) = u_{\sigma,\kappa}\left(\frac{\tau}{z}\right) = \frac{D^k u_{\sigma,\kappa}(p_{\bullet})[\partial_w^{\otimes k}]}{k!}\left(\frac{\tau}{z}\right)^k + R',$$
where the remainder term $R'$ satisfies
\begin{equation}\label{eq: e2 term}
    |e_2|\coloneqq |\langle \overline{\partial}\psi_\tau\cdot R', \operatorname{pull}(\overline{\eta}\otimes v) \rangle_{L^2}| \le C \varepsilon^{k+1}|\eta| |v|.
\end{equation}
We note that $$\overline{\partial}\psi_\tau(z) = \frac{1}{2}\psi'\left(\frac{|z|}{R_0}\right)\frac{z}{R_0 |z|} d\overline{z}.$$
Also, since $\eta$ is of character $\chi$ with respect to the action of $G_p \cong \mu_{m_{\ddot{n}}}$ (recall that the action on $1$-forms is given by $g\cdot \eta=(g^{-1})^*\eta$, and the action of $\mu_{m_{\ddot{n}}}$ on $z$ is of character $\zeta_{m_{\ddot{n}}}^{-1}$), the Taylor expansion of $\eta$ at $p$ is given by:
\begin{equation} \label{eqn: eta}
    \eta(z) := \tilde{\eta}(z)dz=\left(\frac{D^{k-1}\eta(0)[\partial_z^{\otimes k}]}{(k-1)!}z^{k-1} + O(|z|^{k})\right) dz.
\end{equation}
By the residue theorem:
\begin{align*}
    2\pi \frac{D^{k-1}\eta(0)[\partial_z^{\otimes k}]}{(k-1)!} &= \frac{1}{i} \int_{|z|=r} \frac{\tilde{\eta}(z)}{z^k} \, dz = \frac{1}{i} \int_{0}^{2\pi} \frac{\tilde{\eta}(re^{i\theta})}{(re^{i\theta})^k} (ire^{i\theta}) \, d\theta  =\int_{0}^{2\pi} \frac{\tilde{\eta}(re^{i\theta})}{(re^{i\theta})^{k-1}} \, d\theta.
\end{align*}
We then compute the contribution to~\eqref{eq: pairing-with-leading-term} from the neck region near $p$:
\begin{align*}
    \Big\langle \overline{\partial}\psi_{\tau}\cdot &\frac{D^k u_{\sigma,\kappa}(p_{\bullet})[\partial_w^{\otimes k}]}{k!} \left(\frac{\tau}{z}\right)^k, \overline{\eta} \Big\rangle_{L^2}\\
    &\qquad=C \tau^k D^k u_{\sigma,\kappa}(p_\bullet)[\partial_w^{\otimes k}] \cdot \int_{R_0 \le |z| \le 2R_0} \left( \frac{1}{2} \psi'\left(\frac{|z|}{R_0}\right) \cdot \frac{z}{R_0|z|} \cdot \frac{1}{z^k} \cdot \tilde{\eta}(z) \right) \,d\text{vol} \\
    &\qquad= C \tau^k D^k u_{\sigma,\kappa}(p_\bullet)[\partial_w^{\otimes k}] \cdot \int_{R_0}^{2R_0} \frac{1}{2} \psi'\left(\frac{r}{R_0}\right) \frac{1}{R_0} \cdot \left( \int_0^{2\pi} \frac{\tilde{\eta}(re^{i\theta})}{(re^{i\theta})^{k-1}} d\theta \right) \, dr \\
    &\qquad = C  (D^k u_{\sigma, \kappa}(p_{\bullet})\otimes D^{k-1}\eta(p_{\mathghost}))[\tau\partial_w\otimes\partial_z]^{\otimes k}.
\end{align*}

Summing the contributions from neighborhoods of all $gp$ for $g \in H/G_p$ we obtain:
\begin{align*}
    \langle \overline{\partial}\psi_{\tau} \cdot u_{\sigma,\kappa} \circ &\iota_{\tau}, \op{pull}(\overline{\eta}\otimes v) \rangle  \\
    &= C \textstyle\sum_{g \in \mathbb{Z}/d_{\ddot{n}}\mathbb{Z}} \left\langle (D^k u_{\sigma, \kappa}(p_{\bullet})\otimes D^{k-1}\eta(p_{\mathghost}))[\tau\partial_w \otimes\partial_z]^{\otimes k}, v \right \rangle + e_2 \\
    &= C \cdot d_{\ddot{n}} \left \langle (D^k u_{\sigma, \kappa}(p_{\bullet})\otimes D^{k-1}\eta(p_{\mathghost}))[\tau\partial_w \otimes\partial_z]^{\otimes k}, v \right \rangle + e_2.
\end{align*}
In the last equality we use the fact that $\eta \in H^{1,0}(\Sigma_{\mathghost})_\chi$, which implies $D^{k-1}\eta(gp) = \chi^{-1}(g)D^{k-1}\eta(p)$, while $D^{k}u_{\sigma, \kappa}(gp_{\bullet})=\chi(g)D^{k}u_{\sigma, \kappa}(p_{\bullet})$ (see~\eqref{eq: jet-action} and the proof of Lemma~\ref{lemma: dimension}). That is, under summation over the orbit, these characters ``annihilate" one another. If these characters were not opposite, the sum would vanish.

\s\n
\emph{Step 4.} We estimate the term 
\begin{equation}
    \langle D_{\widetilde{u}_{\sigma,\tau,\kappa}}\hat{\xi} + \mathscr{N}_{\widetilde{u}_{\sigma,\tau,\kappa}}\hat{\xi}, \op{pull}(\overline{\eta}\otimes v) \rangle_{L^2}.
\end{equation}

By~\eqref{eq: leading+e1} and Proposition~\ref{prop: 2nd-step-gluing} we have $$\| \hat{\xi}\|_{W^{1,p}} \le C\varepsilon^k.$$ 
Then the quadratic estimate of \cite[Proposition 3.5.3]{mcduffsalamon2004} implies
\begin{equation}\label{eq: e3 term}
    |e_3| \coloneqq |\langle \mathscr{N}_{\widetilde{u}_{\sigma,\tau,\kappa}}\hat{\xi}, \operatorname{pull}(\overline{\eta}\otimes v) \rangle_{L^2}| \le C\varepsilon^{2k}|\eta|\cdot|v|.
\end{equation}

The term $D_{\widetilde{u}_{\sigma,\tau,\kappa}}\hat{\xi} - \overline{\partial}\hat{\xi}$, defined over $\Sigma_{\mathghost}^{\sigma, \tau} \coloneqq \Sigma_{\sigma, \tau}^{\circ} \cap \Sigma_{\mathghost}$, is supported in the region $\varepsilon^{1/2} \le r_{gp_{\mathghost}} \le 2R_0$ for some $g \in H/G_p$.
Moreover, by \cite[Proposition 5.16]{doan2023} the following holds:

\begin{gather}
    \sum_{g \in \mathbb{Z}/d_{\ddot{n}}\mathbb{Z}} \int_{\varepsilon^{1/2} \le r_{gp_{\mathghost}}(z) \le 2R_0} |D_{\widetilde{u}_{\sigma,\tau,\kappa}}\hat{\xi} - \overline{\partial}\hat{\xi}|  \nonumber\\
     \label{eqn: first term for epsilon 4}
    \le C \cdot \sum_{g \in \mathbb{Z}/d_{\ddot{n}}\mathbb{Z}} \int_{\varepsilon^{1/2} \le r_{gp_{\mathghost}}(z) \le 2R_0} \left( |\widetilde{u}_{\sigma, \tau, \kappa}| \cdot |\nabla \hat{\xi}| + |\hat{\xi}|\cdot|d\widetilde{u}_{\sigma, \tau, \kappa}| \right) 
    \le C \varepsilon^{k+\tfrac{k+1}{2}-1/p}.
\end{gather}
For the latter bound, notice that by the H\"{o}lder inequality we have
$$\int_{\varepsilon^{1/2} \le r_{gp_{\mathghost}}(z) \le 2R_0}|d\widetilde{u}_{\sigma, \kappa, \tau}|\cdot |\hat{\xi}| \le C\|\hat{\xi}\|_{W^{1,p}}\left(\int_{\varepsilon^{1/2} \le r_{gp_{\mathghost}}(z) \le 2R_0}|d\widetilde{u}_{\sigma, \kappa, \tau}|^q\right)^{1/q}, $$
where $\tfrac{1}{p}+\tfrac{1}{q}=1$.
Using~\eqref{eqn: estimates for u and du} a direct computation shows that
$$\int_{\varepsilon^{1/2} \le r_{gp_{\mathghost}}(z) \le 2R_0}|d\widetilde{u}_{\sigma, \kappa, \tau}|^q \le C \varepsilon^{(k-1)q/2+1}.$$
The bound for the other summand in~\eqref{eqn: first term for epsilon 4} is obtained similarly and is omitted (it has higher order due to~\eqref{eqn: estimates for u and du}).
Finally, since $\overline{\eta}$ is antiholomorphic, integration by parts implies:
\begin{equation}
    \label{eqn: second term for epsilon 4}
    |\langle \overline{\partial}\hat{\xi}, \overline{\eta}\otimes v \rangle_{L^2(\Sigma_{\mathghost}^{\sigma, \tau})}| = \Big| \int_{\partial \Sigma_{\mathghost}^{\sigma, \tau}} \langle \hat{\xi}, *(\overline{\eta}\otimes v) \rangle \Big| \le C \varepsilon^{1/2} \|\hat{\xi} \|_{W^{1,p}} \cdot |\eta| \cdot |v| \le C \varepsilon^{k+1/2}|\eta|\cdot|v|.
\end{equation}
Equations~\eqref{eqn: first term for epsilon 4} and \eqref{eqn: second term for epsilon 4} together give:
\begin{equation}\label{eq: e4 term}
    |e_4|\coloneqq |\langle D_{\widetilde{u}_{\sigma,\tau,\kappa}}\hat{\xi}, \operatorname{pull}(\overline{\eta}\otimes v) \rangle_{L^2}| \le C \varepsilon^{k+1/2}|\eta|\cdot|v|.
\end{equation}

\s\n
\emph{Step 5.} Combining Steps 1-4,
\begin{align*}
    \langle \operatorname{pull}(\operatorname{ob}(\sigma,\tau,\kappa)), \operatorname{pull}(\overline{\eta}\otimes v) \rangle_{L^2} &= -\langle \mathscr{F}_{\widetilde{u}_{\sigma,\tau,\kappa}}(\hat{\xi}), \operatorname{pull}(\overline{\eta}\otimes v) \rangle_{L^2} \\
    &= -\langle \overline{\partial}_J \widetilde{u}_{\sigma,\tau,\kappa} + D_{\widetilde{u}_{\sigma,\tau,\kappa}}(\hat{\xi}) + \mathscr{N}_{\widetilde{u}_{\sigma,\tau,\kappa}}(\hat{\xi}), \operatorname{pull}(\overline{\eta}\otimes v) \rangle_{L^2} \\
    &=  C \cdot d_{\ddot{n}} \langle (D^k u_{\sigma, \kappa}(p_{\bullet})\otimes D^{k-1}\eta(p_{\mathghost}))[\tau\partial_w \otimes\partial_z]^{\otimes k}, v \rangle_{L^2} +\textstyle\sum_{i=1}^4 e_i,
\end{align*}
and the error term satisfies the claimed bound by~\eqref{eq: e1 term}, \eqref{eq: e2 term}, \eqref{eq: e3 term}, \eqref{eq: e4 term}.
\end{proof}

\begin{remark} 
Two important ingredients help make the above proof work. 
\be
\item[(i)] Working in the Darboux chart $\mathcal{U}_\chi$ for the ``isotypic direction'' allows us to express the leading term in the Taylor expansion in terms of the character $\chi$.
\item[(ii)] The condition on $\eta$ to be in the $\chi$-subrepresentation guarantees that the leading term for the $L^2$-pairing is dictated by the leading terms of $u_{\sigma, \kappa}$ on the non-ghost side.
\ee
\end{remark}

\n \emph{The interior ghost obstruction.}

\smallskip
We now state the obstruction result for smoothing an interior ghost in the generality sufficient for applications. We return to the setup from the beginning of Section~\ref{section Kuranishi model}.

Let $\{u_{k'}\}_{k' \in \mathbb{Z}_{\ge 0}}$ be a sequence of $G$-equivariant curves in a symplectic manifold $M$, potentially with boundary mapped to prescribed equivariant Lagrangians, converging to a curve $$u_{\infty}: \Sigma_\bullet \cup \Sigma_{\mathghost} \rightarrow M$$ 
with underlying orbicurves $\mathscr{C}_k$ converging in $\overline{\mathscr{H}}_{0, \bm h, d}^G$ to a twisted nodal orbicurve
$$\mathscr{C}_\infty=\mathscr{C}_\bullet \cup_{\ddot{n}} \mathscr{C}_{\mathghost},$$
with $\mathscr{C}_{\mathghost}$ a closed orbicurve attached along a unique node $\ddot{n}$. We assume that the underlying orbicurve
$$v_\infty \colon \mathscr{C}_\bullet \cup_{\ddot{n}} \mathscr{C}_{\mathghost} \to [M/G]$$
sends $\mathscr{C}_{\mathghost}$ to a generic point of a cyclic quotient singularity $[G \cdot \mathcal{Y}/G]$. We denote by $\Sigma_{m_{\mathcal{Y}}} \subset \Sigma_{\mathghost}$ the preimage of $\mathcal{Y}$ with stabilizer $H_{\mathcal{Y}}$. The curves $u_{k'}$ are $J$-holomorphic with respect to a domain-dependent family of $G$-equivariant $\omega$-tame almost complex structures varying continuously over $\overline{\mathscr{H}}_{0, \bm h, d}^G$ within a small $C^k$-neighborhood (for arbitrary $k$) of an $[G \cdot \mathcal{Y}/G]$-adapted almost complex structure $J_0$, such that each $J$ coincides with $J_0$ on a neighborhood of stacky points and on $\Sigma_{\mathghost}$.

\begin{theorem}\label{ref: theorem-isotypic-direction-jet-vanishing}
Let $\{u_{k'}\}_{k' \in \mathbb{Z}_{\ge 0}}$ be a sequence of $G$-equivariant curves as above. Assume there exists a holomorphic 1-form $\eta \in H^{1,0}(\Sigma_{m_{\mathcal{Y}}})_\chi$ with respect to the $H_\mathcal{Y}$-action, such that, at each node of $\Sigma_{m_{\mathcal{Y}}}$ lying over $\ddot{n}$, the form $\eta$ vanishes to order exactly
\[
    k(\chi,\Sigma_{m_{\mathcal{Y}}})-1.
\] Then 
\begin{equation}
    D^{k(\chi, \Sigma_{\mathghost})} (\pi_\chi \circ u_{\infty})(p) = 0
\end{equation}
for any node $p \in \pi^{-1}(\ddot{n})$, where $\pi_{\chi}$ is the local projection onto the auxiliary isotypic submanifold $V_\chi$.
\end{theorem}

\begin{proof}
First, apply the localization trick (Lemma~\ref{lemma: localization trick}) to reduce the given sequence of curves to a sequence of curves as in the setup of Proposition~\ref{proposition: obstruction-on-higher-jets}. Since $J$ varies continuously on a small neighborhood of $J_0$ with respect to any $C^k$-norm, we can apply \cite[Prop. 5.36]{doan2023} and obtain that for $k' \gg 0$ there is a section $\hat{\xi}_{k'}$ such that $u_{k'} = \exp_{\widetilde{u}_{\sigma_{k'}, \tau_{k'}, \kappa_{k'}}}(\hat{\xi}_{k'})$, and therefore
$$\operatorname{ob}(\sigma_{k'}, \tau_{k'}, \kappa_{k'}) = 0.$$
This, together with Proposition~\ref{proposition: obstruction-on-higher-jets}, implies that $D^{k(\chi, \Sigma_{\mathghost})} u_{\sigma_{k'}, \kappa_{k'}}(p) \le C'|\tau_{k'}|^{1/2}$ for $k' \gg 0$ and fixed constant $C' >0$, from which the claim follows.
\end{proof}

\begin{remark}
    The proof of \cite[Proposition 5.36]{doan2023} relies on the fact that the metric used in gluing comes from the initial almost Hermitian structure, but only on the neck regions. It is used to show the containment of the necks within tiny balls. If for pregluing we use a metric which coincides with the initial metric on a ball of fixed size centered at the image of the ghost, these energy arguments, based on \cite[Section 4.7]{mcduffsalamon2004}, go unchanged.
\end{remark}

\subsubsection{Boundary ghosts}

In this subsection, we work in the setup of Section~\ref{subsection: boundary-node-gluing}, where we consider the formation of a ghost attached at a boundary node, and our goal is to obtain restricting conditions on the jet of $u_{\bullet}$ at $n_\bullet$. Since $\ddot{n}$ is a regular, non-stacky boundary node, this condition involves only the first derivative, as equivariance imposes no vanishing constraint on this derivative.

That is to say, the present problem is simpler than the case of the interior ghost formation, which was considered in the previous section. In particular, we do not need to project onto the isotypic directions. 

Let $x=u_{\mathghost}(\Sigma_{m_{\mathcal Y}})\in \mathcal Y$, and let $H=H(\mathcal Y)$ be the stabilizer of $\mathcal Y$. We require that $x\in L \cap \mathcal{Y} \subset M$ for a Lagrangian $L$. Moreover, the orbit $G \cdot L$ is the Lagrangian boundary condition for both $\Sigma_{\mathghost}$ and components of $\partial \Sigma_{\bullet}$ containing $\pi^{-1}(n_{\bullet})$. 

 Recall that the obstruction has the following presentation $$\mathcal{O}_{\mathghost} \cong H^{0,1}_{u_{\mathghost}^*TL}(\Sigma_{m_\mathcal{Y}},u_{\mathghost}^*TM)^{H} \cong (H^{0,1}_{\R}(\Sigma_{m_\mathcal{Y}}) \otimes_{\R} T_xL)^{H}.$$

\begin{proposition}\label{proposition: boundary-ghost-leading-obstruction}
For every $(\sigma, \tau, \kappa) \in \Delta \times K^G$ of sufficiently small norm and every $$\overline{\eta}\otimes v \in( H^{0,1}_{\R}(\Sigma_{\mathghost})\otimes T_x L)^H,$$ 
the following identity holds
\begin{gather}
    \operatorname{Re} \,\langle \operatorname{pull}(\operatorname{ob}(\sigma,\tau,\kappa)), \operatorname{pull}(\overline{\eta}\otimes v) \rangle_{L^{2}} = C \cdot m_\mathcal{Y} \langle (du_{\sigma, \kappa}(p_{\bullet})\otimes \eta(p_{\mathghost}))[\tau\partial_w \otimes\partial_z], v \rangle + e , \nonumber
\end{gather}
where $|e| \le C_0 |\eta|\cdot|v|\varepsilon^{\frac{1}{2}}$ for some fixed constants $C\in \R$ and $C_0\in \R_{>0}$.
\end{proposition}

\begin{proof}[Sketch of proof]
The proof follows the same steps as that of Proposition~\ref{proposition: obstruction-on-higher-jets}. All steps except the third are identical. The error bounds proved there also imply the corresponding bounds after taking real parts.  

It remains to explain the third step, namely, the estimate of 
\begin{equation*} 
\operatorname{Re} \,\langle \overline{\partial}\psi_{\tau} \cdot u_{\sigma,\kappa} \circ \iota_{\tau}, \op{pull}(\overline{\eta}\otimes v) \rangle_{L^2}.
\end{equation*}
We first consider the contribution from a single boundary node
$p \in \pi^{-1}(\ddot{n}) \cap \Sigma_{m_\mathcal{Y}}$. Since $\eta$ is a $1$-form on the upper half-plane, which is real along the boundary, we can locally near the origin extend it to a holomorphic $1$-form, with no residues on a neighborhood of the origin. Let $\Gamma_{r,\delta}$ be the contour consisting of the half-circle $C_r^+$, the intervals $[-r,-\delta]\cup[\delta,r]$, and the small half-circle $C_\delta^-$ around the origin. Taking the limit as $\delta\to 0$, we obtain
\begin{align*}
     2\pi\eta(0) =\lim_{\delta \to 0}\operatorname{Re} \frac{1}{i}\int_{\Gamma_{r, \delta}} \frac{\eta(z)}{z}dz
     &=\operatorname{Re} \left(\frac{1}{i}\int_{C_r^+}\frac{\eta(z)}{z}dz+ \frac{1}{i}\int_{-r}^r\frac{\eta(x)}{x}dx + \lim_{\delta \to 0}\frac{1}{i}\int_{C_\delta^-}\frac{\eta(z)}{z}dz\right) \\
     &=\operatorname{Re} \left(\frac{1}{i}\int_{C_r^+}\frac{\eta(z)}{z}dz\right)+\pi \eta(0).
\end{align*}
Hence
\begin{equation}
    \operatorname{Re}\left(\frac{1}{i}\int_{C_r^+}\frac{\eta(z)}{z}dz\right)=\pi \eta(0).
\end{equation}

The leading term of $u_{\sigma, \kappa} \circ \iota_\tau (z)$ is given by $du_{\sigma, \kappa}(p_{\bullet})[\partial_w] \frac{\tau}{z}$. As in the proof of Proposition~\ref{proposition: obstruction-on-higher-jets}, it is sufficient to estimate the pairing with the leading term:
\begin{align*}
    \operatorname{Re} \Big\langle \overline{\partial}\psi_{\tau}  d u_{\sigma,\kappa}&(p_{\bullet})[\partial_w]  \left(\frac{\tau}{z}\right)^k, \overline{\eta}\Big\rangle_{L^2} \\
    &=  \tau du_{\sigma, \kappa}(p_{\bullet})[\partial_w] \operatorname{Re}\int_{R_0 \le |z| \le 2R_0, \operatorname{Im}(z) \ge 0} \left( \frac{1}{2} \psi'\left(\frac{|z|}{R_0}\right) \frac{\tilde{\eta}(z)}{R_0|z|}  \right) d\text{vol} \\
    &=\tau du_{\sigma, \kappa}(p_{\bullet})[\partial_w] \int_{R_0}^{2R_0} \frac{1}{2} \psi'\left(\frac{r}{R_0}\right) \frac{1}{R_0} \cdot \operatorname{Re} \,\left( \frac{1}{i}\int_{C_r^+}\frac{\eta(z)}{z}dz\right) \, dr\\
    &=Cdu_{\sigma, \kappa}(p_{\bullet})
    \otimes\eta(p_{\mathghost})[\tau \partial_w \otimes \partial_z].
\end{align*}
Since $\overline{\eta}\otimes v$ is $H$-invariant, summing over the nodes in the $H$-orbit of $p$ gives
\begin{equation}
    \operatorname{Re} \,\langle \overline{\partial}\psi_{\tau} \cdot u_{\sigma,\kappa} \circ \iota_{\tau}, \text{pull}(\overline{\eta}\otimes v) \rangle_{L^2}=C\cdot m_{\mathcal{Y}}\langle (du_{\sigma, \kappa}(p_{\bullet})\otimes \eta(p_{\mathghost}))[\tau\partial_w \otimes\partial_z], v \rangle + e'
\end{equation}
for $|e'| \le C' \varepsilon^2 |\eta| |v|$. By the argument from Proposition~\ref{proposition: obstruction-on-higher-jets}, the proof follows.
\end{proof}

\n \emph{The boundary ghost obstruction.}

\smallskip
We now record the corresponding obstruction result for smoothings of boundary ghosts in sufficient generality. To avoid repetition, we assume that we are given a sequence of curves $u_k$ as in the setup prior to Theorem~\ref{ref: theorem-isotypic-direction-jet-vanishing}, with the only change that $\mathscr{C}_{\mathghost}$ is an orbidisk with boundary which is mapped to a point in the intersection of an equivariant Lagrangian $\mathcal{L}=G \cdot L$ and a cyclic quotient singularity $[G \cdot \mathcal{Y}/G]$. We assume that $\Sigma_{m_{\mathcal{Y}}}$ is mapped to a point in $L \cap \mathcal{Y}$. The Lagrangian $L$ is assumed to be \emph{totally geodesic} with respect to the metric induced by an adapted and tame almost complex structure $J_0$.

\begin{theorem}\label{ref: theorem-real-derivatives-vanishing}
Let $\{u_{k'}\}_{k' \in \mathbb{Z}_{\ge 0}}$ be a sequence of curves as above. Suppose there exists a nonzero $1$-form
$$\overline{\eta}\otimes v \in( H^{0,1}_{\R}(\Sigma_{m_{\mathcal{Y}}})\otimes T_x L)^{H_{\mathcal{Y}}}.$$
Let $p$ be a node of $\Sigma_{m_{\mathcal{Y}}}$ in $\pi^{-1}(\ddot{n})$.
If $\eta(p) \neq 0$, then 
\begin{equation}
    \langle du_\infty(p), v \rangle =0.
\end{equation}
\end{theorem}

\begin{proof}
    The proof is analogous to that of Theorem~\ref{ref: theorem-isotypic-direction-jet-vanishing}, and is based on combining the localization trick (Lemma~\ref{lemma: localization trick for boundary}) and Proposition~\ref{proposition: boundary-ghost-leading-obstruction}.
\end{proof}

\section{Orbifold Hecke algebras and the $A_\infty$-algebras of cotangent fibers}\label{section: A_infty and orbifold Hecke}

\subsection{Orbifold Hecke algebras}

In this section, we recall Etingof's {\em orbifold Hecke algebras} for complex global quotient orbifolds, closely following \cite[Section 3]{EtingofPavel}. 

We briefly review $G$-paths and the orbifold fundamental group of a global quotient orbifold $[X/G]$.  For more details, we refer the reader to the classical reference \cite{Haefliger1990orbi}, and to \cite{angelcolman2023pathgroupoids} for a modern exposition, which treats the global quotient case specifically.

Let $\hat{I}=\sqcup_{g \in G} I_g$, where $I_g$ is a copy of $I=[0,1]$. We denote by $g\cdot t$ the point in $I_g$ corresponding to $t \in [0,1]$; this gives rise to a left $G$-action on $\hat{I}$.
A {\em $G$-path on $X$ based at $Gx$ for $x\in X^{\op{reg}}$} is a $G$-equivariant continuous map
$$\hat{\gamma} \colon \hat{I} \to X, \quad \hat{\gamma}(G\cdot0)=\hat{\gamma}(G\cdot 1)=Gx, \quad \hat{\gamma}(e\cdot 0)=x.$$

Let  $\Omega_G(X, Gx)$ be the space of $G$-paths on $X$ based at $Gx$. By equivariance $\hat{\gamma}$ is determined by $\hat{\gamma}_e=\hat{\gamma}|_{I_e}$ and we get
\begin{equation}\label{eq: G-paths-decomposition}
    \Omega_G(X, Gx) \cong \bigsqcup_{g \in G} \Omega(X, x, gx).
\end{equation} 

The orbifold fundamental group $\pi_1^{\operatorname{orb}}([X/G], [x])$ for $x \in X^{\operatorname{reg}}$ can be defined as the group of equivalence classes of $G$-paths based at $Gx$ up to homotopy. The multiplication of $G$-paths is given by concatenation. This definition coincides with the more classical one via the Borel construction 
$$\pi_1^{\operatorname{orb}}([X/G])=\pi_1(EG \times_G X).$$
There is also a short exact sequence
\begin{equation}
    1 \to \pi_1(X,x) \to \pi_1^{\operatorname{orb}}([X/G], [x]) \to G \to 1.
\end{equation}

Let $X$ be a complex manifold and $G \subset \operatorname{Aut}(X)$ a finite subgroup. 
Consider a connected hypersurface $Y \subset X^{\operatorname{sing}}$ with $H_Y$ its stabilizer of order $m_Y$. To each $G$-orbit of such hypersurfaces, we assign formal variables $\tau^Y_0, \dots, \tau_{m_Y-1}^Y$. Let $\bm \tau$ be the collection of the $\tau^Y_j$ variables over all $Y$ and $j$.

We denote by $C_Y$ the conjugacy class in the \emph{braid group} $\pi_1(X^{\operatorname{reg}}/G, [x])$ generated by an arbitrary loop $\gamma=c(\tilde{\gamma},\gamma_Y,\tilde{\gamma}^{-1})$, where $\tilde{\gamma}$ is a path connecting $[x]$ to a point $p$ close to $[Y/G]$, $\gamma_Y$ is a counterclockwise small loop around $[Y/G]$ starting at $p$, and $c$ denotes concatenation of paths. Clearly, any two such loops $\gamma$ are conjugate. 

\begin{definition}
    The \emph{orbifold Hecke algebra} $\mathcal{H}_{\bm \tau}([X/G])$ is the quotient of the $\bm \tau$-adic completion $$\C[\pi_1(X^{\operatorname{reg}}/G, [x])]\llbracket \bm \tau \rrbracket$$ of the group algebra by the Hecke relations
    \begin{equation}
        \prod_{j=0}^{m_Y-1}(\gamma-\zeta_{m_Y}^je^{\tau_j^Y})=0,
    \end{equation}
    for any connected hypersurface $Y$ in $X^{\operatorname{sing}}$ and any $\gamma \in C_Y$.
\end{definition}

Now we introduce a slightly different variant of the orbifold Hecke algebra, which is more suited to our geometric approach, and show that it is equivalent to the above one. As before, for each $G$-orbit of hypersurfaces $Y \subset X^{\operatorname{sing}}$ we introduce variables $\hbar_{0}^Y, \dots, \hbar_{m_Y-1}^Y$. 

\begin{definition}
     The \emph{orbifold Hecke algebra} $\mathcal{H}_{\bm \hbar}([X/G])$ is the quotient of the $\bm \hbar$-adic completion $$\C[\pi_1(X^{\operatorname{reg}}/G, [x])]\llbracket \bm \hbar \rrbracket$$ of the group algebra by the Hecke relations
    \begin{equation}\label{eq: hbar-Hecke-relation}
        \gamma^{m_Y}-\hbar_{m_Y-1}^{Y}\gamma^{m_Y-1}-\dots-\hbar_{1}^{Y}\gamma-(\hbar_0^Y+1)=0,
    \end{equation}
    for any connected hypersurface $Y$ in $X^{\operatorname{sing}}$ and any $\gamma \in C_Y$.
\end{definition}

It is clear from the definitions that both $\mathcal{H}_{\bm \hbar}([X/G])$ and $\mathcal{H}_{\bm \tau}([X/G])$ are deformations of the group algebra of $\pi_1^{\operatorname{orb}}([X/G])$.
There is a natural map
\begin{equation}
    \mathcal{V}_{\bm \tau}^{\bm \hbar} \colon \mathcal{H}_{\bm \tau}([X/G]) \to \mathcal{H}_{\bm \hbar}([X/G]),
\end{equation}
which is given by equating polynomials
\begin{equation}
    P_{m_Y}(T)=\prod_{j=0}^{m_Y-1}(T-\zeta^j_{m_Y}e^{\tau^Y_j})=T^{m_Y}-\hbar_{m_Y-1}^YT^{m_Y-1}-\dots-\hbar_{1}^YT-(\hbar^Y_0+1)
\end{equation}
and expressing the $\bm \hbar$ variables in terms of the $\bm \tau$ variables. Clearly, the natural induced map on completions descends to the map of the quotient since such maps send the $\bm \tau$-ideal of Hecke relations to the corresponding $\bm \hbar$-ideal.  

\begin{lemma}\label{lemma: change-of-variables}
    The map $\mathcal{V}_{\bm \tau}^{\bm \hbar}$ is an isomorphism.
\end{lemma}

\begin{proof}
    The change of variables sends the origin with respect to the $\tau^Y_j$-variables to the origin with respect to the $\hbar^Y_j$-variables. To show that the map upstairs (before quotienting by the Hecke relations) is an isomorphism, it suffices to show that this change of variables is invertible. Since we are working with completions, it suffices to show that the Jacobian of this change of variables at the origin is nonzero. Notice that up to sign $\hbar_j^Y$ (or $(\hbar_0^Y+1)$) is the $j$th symmetric polynomial in the variables $(\zeta_{m_Y}^0e^{\tau_0^Y}, \dots, \zeta_{m_Y}^{m_Y-1}e^{\tau_{m_Y-1}^Y})$. Then up to sign, the Jacobian of this change of variables is the Vandermonde determinant of $P_{m_Y}(T)$ (see e.g. \cite{lascouxpragacz2002}), and since its roots are distinct (at the origin in the $\bm \tau$-variables), it is nonzero at the origin. This shows that the map upstairs is an isomorphism, which immediately implies that $\mathcal{V}_{\bm \tau}^{\bm \hbar}$ is an isomorphism.
\end{proof}

The following fundamental result of \cite{EtingofPavel} regarding the flatness of the deformations provided by orbifold Hecke algebras will be of great importance to us:
\begin{theorem}[{\cite[Theorem 3.7]{EtingofPavel}}]\label{theorem: etingof-flatness}
    If $\pi_2(X) \otimes \Q=0$, then $\mathcal{H}_{\bm \tau}([X/G])$ is a flat formal deformation of $\C[\pi_1^{\operatorname{orb}}([X/G])]$.
\end{theorem}

\begin{remark}\label{remark: hbar-flatness}
    Theorem \ref{theorem: etingof-flatness} implies that $\mathcal{H}_{\bm \tau}([X/G])\simeq \C[\pi_1^{\operatorname{orb}}([X/G])]\llbracket \bm \tau \rrbracket$ as a $\C\llbracket \bm \tau \rrbracket$-module. Combined with Lemma~\ref{lemma: change-of-variables}, it implies that $\mathcal{H}_{\bm \hbar}([X/G])$ is flat under the conditions of the theorem.
\end{remark}
Observe that the group algebra of the braid group $\C[\pi_1(X^{\operatorname{reg}}/G)]\llbracket \bm \hbar \rrbracket$ is isomorphic to the algebra 
\begin{equation}
    H_0(\Omega_G(X^{\operatorname{reg}}, Gx); \C)\llbracket \bm \hbar \rrbracket,
\end{equation} 
with the Pontryagin product, i.e., the concatenation of $G$-paths. Using this perspective, we can define the Hecke algebra as the quotient of
$H_0(\Omega_G(X^{\operatorname{reg}}, Gx); \C)\llbracket \bm \hbar \rrbracket$
by the Hecke relations~\eqref{eq: hbar-Hecke-relation}.

Finally, let us denote by 
\begin{equation}
    \mathcal{H}_{\bm \hbar}|_{\bm \hbar_0=0}([X/G])
\end{equation}
the specialization of the Hecke algebra $ \mathcal{H}_{\bm \hbar}([X/G])$ obtained by setting $\hbar_0^Y=0$ for all hypersurfaces $Y$.

\subsection{Relating the $A_\infty$-algebra of a cotangent fiber and the Hecke algebra}\label{section: evaluation-map-iso}

The goal of this subsection is to prove one of our main results and its corollary:

\begin{theorem}\label{theorem: Hecke-isomorphism}
    There exists a surjective morphism $\mathcal{EV}$ from the algebra $HW^0_{\bm \hbar}(T^*_{[x]}[X/G])$ to the orbifold Hecke algebra $\mathcal{H}_{\bm \hbar}|_{\bm \hbar_0=0}([X/G])$. If $\pi_2(X) \otimes \mathbb{Q}=0$ the morphism $\mathcal{EV}$ is an isomorphism.
\end{theorem}

\begin{corollary}\label{corollary: derived-Nakayama}
    If $X$ is a $K(\pi,1)$ space, then 
    $$HW^*_{\bm \hbar}(T^*_{[x]}[X/G]) \cong \mathcal{H}_{\bm \hbar}|_{\bm \hbar_0=0}([X/G]).$$
\end{corollary}

Let $\overline{\mathscr{T}}^G_{0, \bm h, d}$ be the moduli space of admissible $G$-covers of a \emph{half-disk} $C$ with $(d+2)$ punctures $\psi_{-1}$, $\psi_0$, $\dots$, $\psi_{d}$ in counterclockwise order on the boundary, such that $\psi_{-1}$ and $\psi_0$ correspond to the corner points of the half-disk, and following \cite[Section 4]{abouzaid2012wrapped} we call the portion of the boundary between punctures $\psi_{-1}$ and $\psi_0$ outgoing. The punctures $\psi_0, \dots, \psi_d$ are assumed to be positive, and the puncture $\psi_{-1}$ is assumed to be negative. Observe that $\overline{\mathscr{T}}^G_{0, \bm h, d}$ is the same orbifold as $\overline{\mathscr{H}}^G_{0, \bm h, d+1}$, and the reason for the new notation is that we are about to pick slightly different Floer data on this collection of orbifolds adapting the approach of Abouzaid.

As before, there are universal curves  $\overline{\mathfrak{F}}_{0, \bm h, d}^G\to \overline{\mathscr{T}}^G_{0, \bm h, d}$ and $\overline{\mathfrak{T}}_{0, \bm h, d}^G\to \overline{\mathscr{T}}^G_{0, \bm h, d}$ corresponding to domains and codomains of these admissible covers, and marked point sections $q_1, \dots, q_b \colon \overline{\mathscr{T}}^G_{0, \bm h, d} \to \mathfrak{F}_{0, \bm h, d}^G$. The following definition parallels Definition~\ref{defn: floer data}, and we only highlight the differences:

\begin{definition}\label{defn: floer-data-for-half-covers}
    A {\em Floer data} on $\overline{\mathscr{T}}^{G}_{0, \bm h, d}$ is a tuple 
    $$(\rho, \alpha, H, J)$$ 
    consisting of:
\end{definition}
    \begin{enumerate}
    \item A time-shifting map $\rho \colon \partial^v \overline{\mathfrak{T}}_{0, \bm h, d}^G \to [1, +\infty)$, which is smooth, constant near the boundary punctures of each half-disk $(C, q_1, \ldots, q_b)$, and constant on each component of $\bdry C$ without any boundary punctures. The weight functions additionally satisfy $w_0,w_{-1}\equiv 1$.
     
    \item A basic $1$-form $\alpha \in \Omega^1(\overline{\mathfrak{T}}_{0, \bm h, d}^G)$, which vanishes along the parts of $\bdry^v \overline{\mathfrak{T}}_{0, \bm h, d}^G$ which are not outgoing, on a neighborhood $\mathcal{N}$ of the sections $q_1, \ldots, q_b$, on a neighborhood of punctures $\psi_{-1}$ and $\psi_0$, and also on some auxiliary open subset which is disjoint from $\mathcal{N}$ and intersects each fiber $C$ in an open subset $U_C$ which is disjoint from strip-like ends, finite strips and cylinders, and \emph{an open neighborhood of the outgoing boundary}. The restriction $\alpha|_C$ to any fiber $C$ is required to be sub-closed. 

    \item An equivariant domain-dependent Hamiltonian $H \in \mathcal{H}^T_{\bm h,d}(M)$, where the latter is defined in a similar fashion as $\mathcal{H}_{\bm h, d}(M)$ and the only difference is that $H$ is required to vanish on the product of a neighborhood of the outgoing boundary with a neighborhood of the $0$-section $X \subset M$.

    \item A domain-dependent almost complex structure $J \in \mathcal{J}_{\bm h, d}^T(M)^G$, where the latter space is defined in a similar fashion as $\mathcal{J}_{\bm h, d}(M)^G$ and the only difference is that near the outgoing boundary $J$ is a family of adapted almost complex structures in the sense of Definition~\ref{defn: adapted J}. 
    \end{enumerate}

\s
The boundary of the orbifold $\overline{\mathscr{T}}_{0, \bm h, d}^G$ is naturally covered by the images of 
\begin{gather}
    \overline{\mathscr{T}}_{0, \bm h_1, d}^G \times \overline{\mathscr{H}}_{0, \bm h_2}^G \hookrightarrow \overline{\mathscr{T}}_{0, \bm h, d}^G, \quad \bm h =\bm h_1 \#_{(h_{\ddot{n}})}\bm h_2;\\
    \overline{\mathscr{T}}_{0, \bm h_1, d_1}^G \times \overline{\mathscr{T}}_{0, \bm h_2, d_2}^G \hookrightarrow \overline{\mathscr{T}}_{0, \bm h, d}^G ,\quad d_1+d_2=d,\, \bm h=(\bm h_1, \bm h_2); \\
    \overline{\mathscr{T}}_{0, \bm h_1, d_1}^G \times \overline{\mathscr{H}}_{0, \bm h_2, d_2}^G \hookrightarrow \overline{\mathscr{T}}_{0, \bm h, d}^G, \quad d_1+d_2=d+1, \, \bm h=(\bm h_1, \bm h_2);
\end{gather}
where the first case corresponds to a genus $0$ orbicurve bubbling off of an orbifold half-disk $\mathscr{C}$, 
the second corresponds to the creation of a boundary node along the outgoing boundary, and the last case corresponds to the creation of a boundary node along some other boundary component.
 
\begin{definition}
    A choice of Floer data for all $\overline{\mathscr{T}}^G_{0, \bm h, d}$ is {\em universal and conformally consistent}, if its restriction to a boundary stratum is conformally equivalent to the product of Floer data coming from either the Floer data for the $A_\infty$-operations chosen in Section~\ref{section: orbifold-Fukaya-prelim} or a lower-dimensional moduli space $\overline{\mathscr{T}}^G_{0, \bm h', d'}$, and which near such a boundary stratum agrees to infinite order with the Floer data obtained by gluing.  
\end{definition}

Using such Floer data, we define the moduli space of $G$-equivariant curves covering half-disks as follows: Let $\overrightarrow{\hat{x}}=(\hat{x}_d, \dots, \hat{x}_1)$ be a tuple of orbichords connecting the equivariant Lagrangian $\mathcal{L}_{[x]}=T^*_{[x]}[X/G]$ with itself for some point $x \in X^{\operatorname{reg}}$ and let $\bm m$ be a connected component of $\mathcal{I}[T^*X/G]^b$ for $b\ge 0$.

\begin{definition}\label{defn: moduli-of-half-orbidisks}
    The moduli space $\mathcal{T}(\overrightarrow{\hat{x}}; \bm m)$ is the space of $G$-equivariant maps
    $$u \colon \Sigma \to M$$
    for $(\pi \colon \Sigma \to C; q_1, \dots, q_b) \in \mathscr{T}_{0, \bm h, d}^G$ which satisfy
    \begin{equation}\label{eq: Floer-equation-half}
        (du-\pi^*(X_H \otimes \alpha)|_C)^{0,1}_{J|_C}=0,
    \end{equation}
    \begin{equation}
        \mathcal{E}^{\mathcal{I}}([u]) \in \bm m
    \end{equation}
    and
    \begin{gather}
        u(z) \in \psi^{\rho(z)|_C}\mathcal{L}_i \quad \text{ for }z\in \pi^{-1}(w) \text{ with } w\in \partial C \text{ between }\psi^i \text{ and } \psi^{i+1} \text{ and }i\neq0, \\
        u(z) \in X \quad \text{ for }z\in \pi^{-1}(w) \text{ with } w\in \partial C \text{ between }\psi^{-1} \text{ and } \psi^{0},\\
        \lim_{s\to +\infty}u|_ {\pi^{-1}(\epsilon^i(s, \cdot))}=\psi^{w_k}\circ\hat{x}_i \quad \text{for }i \ge 1, \\
        \lim_{s\to \pm\infty}u|_ {\pi^{-1}(\epsilon^i(s, \cdot))}=G \cdot x \quad \text{for }i = 0, -1.
    \end{gather}
\end{definition}

\begin{remark}
    The conditions on the Hamiltonian $H$ in Definition~\ref{defn: floer-data-for-half-covers} guarantee that the Floer equation~\eqref{eq: Floer-equation-half} near the preimage under $\pi$ of the outgoing boundary agrees with $\overline{\partial}_J=0.$
\end{remark}

For generic data, each $\mathcal{T}(\overrightarrow{\hat{x}}; \bm m)$ is a smooth manifold of dimension
\begin{equation}
    \operatorname{dim} \mathcal{T}(\overrightarrow{\hat{x}}; \bm m)=-|\overrightarrow{\hat{x}}|+2b-2\iota(\bm m),
\end{equation}
and there is a continuous evaluation map
\begin{equation}
    ev \colon \mathcal{T}(\overrightarrow{\hat{x}}; \bm m) \to \Omega_G(X, Gx),
\end{equation}
which uses a parametrization of the arc connecting $\psi^{-1}$ with $\psi^0$ by $[0,1]$; the proof uses methods of Section~\ref{section: transversality-for-jets}.

Let us restrict to the case of a single $\hat{x}$ with $|\hat{x}|=0$ and a tuple 
$$\bm m=\bm m_{\bm Y, \bm h}=((T^*Y_1, (h_1)), \dots, (T^*Y_b, (h_b)))$$ 
of type $A$ inertia for a collection of hypersurfaces $(Y_1, \dots, Y_b)$ as in Section~\ref{section: orbifold-Fukaya-prelim}.  Then 
$$\operatorname{dim}\mathcal{T}(\hat{x}; \bm m_{\bm Y, \bm h})=0.$$
By the results of Section~\ref{section: transversality-for-jets}, the image $ev(\mathcal{T}(\hat{x}; \bm m_{\bm Y, \bm h})) $ is contained in $ \Omega_G(X^{\operatorname{reg}}, Gx)$ and we have a well-defined map
\begin{equation}
    ev \colon \mathcal{T}(\hat{x}; \bm m_{\bm Y, \bm h}) \to \Omega_G(X^{\operatorname{reg}}, Gx). 
\end{equation}

We then define 
\begin{gather}
  \mathcal{EV} \colon CW^0_{\bm \hbar}(T^*_{[x]}[X/G]) \to \mathcal{H}_{\hbar}([X/G]),\\
  [\hat{x}] \mapsto \sum_{\substack{\,b \in \Z_{\ge 0}, \, (\bm Y, \bm h) \in \mathcal{H}^{\mathcal{I}}([X/G])^b\\ [u] \in \mathcal{T}(\hat{x}; \bm m_{\bm Y, \bm h})}} \frac{1}{b!}\prod_{i=1}^b\hbar^{Y_i}_{h_i}ev([u]), \nonumber
\end{gather}
where the sum on the right-hand side is treated as the class in the Hecke algebra of the singular chain in
$$C^0(\Omega_G(X^{\operatorname{reg}}, Gx); \C)\llbracket \bm\hbar \rrbracket$$
given by this sum. Notice that $ev([u])$ is a signed count where the sign is given by comparing the orientation of $\mathcal{T}(\hat{x}; \bm m_{\bm Y, \bm h})$ induced by $[x]$ with $\operatorname{det}(D_u) \cong \R$ via the isomorphism $\operatorname{det}(D_u) \cong o_{\hat{x}}$.

\begin{proposition}\label{prop: algebra-morphism}
    The map $\mathcal{EV}$ descends to a well-defined map
    \begin{equation}
        \mathcal{EV} \colon  HW^0_{\bm \hbar}(T^*_{[x]}[X/G]) \to \mathcal{H}_{\bm \hbar}|_{\bm \hbar_0=0}([X/G]).
    \end{equation}
    Moreover, this map is an algebra morphism.
\end{proposition}

\begin{proof}
   The proof of both assertions is based on the analysis of boundary degenerations for the moduli spaces $\mathcal{T}(\hat{x}; \bm m_{\bm Y, \bm h})$ with $|\hat{x}|=-1$ (for the first assertion) and $\mathcal{T}(\hat{x}_2,\hat{x}_1; \bm m_{\bm Y, \bm h})$ with $|\hat{x}_1|=|\hat{x}_2|=0$ (for the second assertion). It should be regarded as a glimpse of verification that a map of the form $\mathcal{EV}$ provides an $A_\infty$-functor into a full derived Hecke algebra, which we do not define in this paper. 
   
 To prove this proposition, it suffices to establish the relations: 
   \begin{gather}
       \mathcal{EV} \circ \mu_1=\partial_{\operatorname{sing}} \circ \mathcal{EV}, \label{eq: differential-commutes} \\
       \mathcal{EV} \cdot \mathcal{EV}=\mathcal{EV} \circ \mu_2.
   \end{gather}
   The right-hand side of~\eqref{eq: differential-commutes} involves the evaluation map from $CW^1_{\bm \hbar}(T^*_{[x]}[X/G])$ to $C^1(\Omega_G(X^{\operatorname{reg}}, Gx); \C)\llbracket \bm\hbar \rrbracket$, which a priori is not well-defined since in a $1$-parameter family a $G$-path not contained within $X^{\operatorname{reg}}$ can occur. The analysis below shows, in particular, that such contributions cancel each other, making the right-hand side well-defined.
   
   Specifically, we have the following list of boundary degenerations that do not involve ghosts for the moduli spaces $\mathcal{T}(\hat{x}; \bm m_{\bm Y, \bm h})$ and $\mathcal{T}(\hat{x}_2,\hat{x}_1; \bm m_{\bm Y, \bm h})$:
   \begin{gather}
       \mathcal{M}(\hat{x}, \hat{x}'; \bm m_{\bm Y_1, \bm h_1})  \times \mathcal{T}(\hat{x}'; \bm m_{\bm Y_2, \bm h_2}), \quad |\hat{x}|=0, \, (\bm Y, \bm h)=((\bm Y_1, \bm h_1), (\bm Y_2, \bm h_2));\\
       \mathcal{M}((\hat{x}_2, \hat{x}_1), \hat{x}'; \bm m_{\bm Y_1, \bm h_1}) \times \mathcal{T}(\hat{x}'; \bm m_{\bm Y_2, \bm h_2}), \quad |\hat{x}|=0, \, (\bm Y, \bm h)=((\bm Y_1, \bm h_1), (\bm Y_2, \bm h_2));\\
       \mathcal{T}(\hat{x}_2; \bm m_{\bm Y_2, \bm h_2}) \times \mathcal{T}(\hat{x}_1; \bm m_{\bm Y_1, \bm h_1}), \quad (\bm Y, \bm h)=((\bm Y_1, \bm h_1), (\bm Y_2, \bm h_2)).
   \end{gather}
   These contributions are standard, and we omit the detailed analysis of signs and gluing near these breakings as it routinely follows by applying the results of Appendix~\ref{appendix: orientations} and Proposition~\ref{proposition-one-step-gluing}.
   
    The main challenge is to analyze ghost bubbling. Interior ghost bubbles can be ruled out as in the proof of Theorem~\ref{theorem: well-defined-ainfty-cat}. The new difficulty is that some stacky points may approach the outgoing boundary $\partial \mathscr{C}$ since the $0$-section $X$ is a singular Lagrangian, and a limit curve may admit boundary nodal ghosts.
    The first step is devoted to constraining boundary degenerations by showing that in such families only ghost disks with a \emph{unique} stacky point can occur, which relies on our obstructions to smoothing orbicurves. The second step is devoted to showing that the allowed degenerations of a ghost orbidisk with a unique stacky point cancel out using the Hecke relations. We notice that these degenerations behave similarly for both $\mathcal{T}(\hat{x}; \bm m_{\bm Y, \bm h})$ and $\mathcal{T}(\hat{x}_2,\hat{x}_1; \bm m_{\bm Y, \bm h})$ as above, and we focus on the former.

    \s \emph{Step 1: Ruling out obstructed ghosts.}
    
    Let 
    $$u_\infty=(u_{\bullet}, u_{\mathghost}) \colon \Sigma_{\bullet} \cup \Sigma_{\mathghost} \to T^*X$$
    be a limit of a sequence of curves in $\mathcal{T}(\hat{x}, \bm m_{\bm Y, \bm h})$, whose underlying orbicurve $\mathscr{C}_{\mathghost}$ contains at least $2$ stacky points. Let $p_{\bullet}$ be a point on $\Sigma_{\bullet}$ in the preimage of the boundary node, and let $u(p_{\bullet})=y \in Y$ for some fixed hypersurface $Y \subset X$.
    Then Theorem~\ref{ref: theorem-real-derivatives-vanishing} and Lemma~\ref{lemma: index-computation-for-An-singularity} below imply that the derivative $du(p_\bullet)$ has vanishing $\nu_Y$ component. Hence such a curve $u_\bullet$ cannot exist by Proposition~\ref{prop: proposition-jets-transversality-with-submanifolds} and we have eliminated the case where $\mathscr{C}_{\mathghost}$ contains more than $1$ stacky point. 
 
   \s \emph{Step 2: The Hecke relation.} 

   \s\n
   \emph{Intersecting fixed hypersurfaces along the outgoing boundary.} Along a $1$-parameter family of curves $u^{\varepsilon} \in \mathcal{T}(\hat{x}, \bm m_{\bm Y, \bm h})$, we may encounter a curve 
   \begin{equation} \label{eqn: u zero}
       u^0:\Sigma^0\to T^*X \quad \mbox{with} \quad \pi_0: \Sigma^0\to\mathscr{C}^0, 
   \end{equation}
   whose outgoing boundary $\partial_{\operatorname{out}}\Sigma^0=\pi_0^{-1}(\partial_{\operatorname{out}}\mathscr{C}^0)$ 

   transversely intersects $G \cdot Y$ at a single $G$-orbit of boundary points. It follows from Proposition~\ref{prop: proposition-jets-transversality-with-submanifolds} that multiple such intersections and non-transverse intersections cannot occur.  We denote by $q^0_{\times} \in \partial_{\operatorname{out}}\mathscr{C}^0$ and $p^0_{\times}\in \pi_0^{-1}(q^0_{\times})$ points such that $u^0(p^0_{\times})=y\in Y.$

   Choose a holomorphic coordinate $z=(s,t)$ on a neighborhood $\mathcal{N}(q^0_\times)$ of $q^0_{\times}$ in $\mathscr{C}^0$, with $s \ge 0$ and 
   $$\partial_{\operatorname{out}}\mathscr{C}^0=\{s=0\}$$ 
   near the point $q^0_{\times}=(0, t_0)$ with $0<t_0<1$.
   
   This local coordinate lifts via $\pi_0$ to a coordinate near $p^0_{\times}=(0, t_0)$, which we also denote by $z=(s,t)$ by abuse of notation. Since the domains of the curves $u^{\varepsilon}$ can be assumed to differ from the domain of $u^{0}$ by a variation of almost complex structure away from $(\pi_0)^{-1}(q^0_{\times})$, we can use this holomorphic coordinate on the domains of all the $u^{\varepsilon}$ locally. Let us denote by $\gamma^{\varepsilon}(\tau)=u^{\varepsilon}(0,t_0-\tau)$ the curve in $X$ for $\tau \in (-\delta, \delta)$ and $\delta>0$ sufficiently small.

\begin{figure}[ht]
    \centering
    \includegraphics[width=0.75\linewidth]{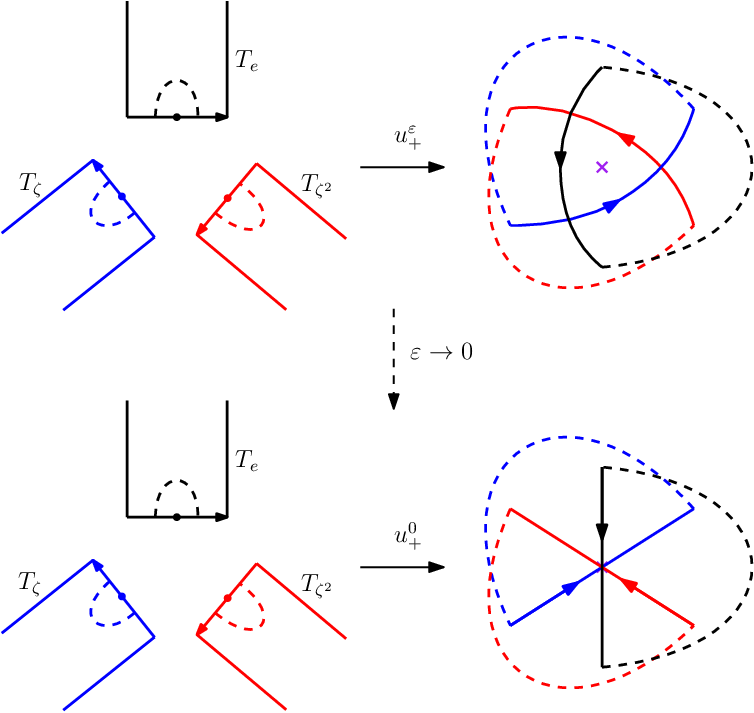}
    \caption{A family of curves $u^{\varepsilon}$ with $\varepsilon <0$ converging to $u^0$, where $H_Y \cong \mu_3$.  The left-hand side gives the portion of the domain $\Sigma^0$ near the orbit $H_Yp_{\times}^0$, which consists of $3$ portions of a half strip $T_e, T_{\zeta_3}, T_{\zeta_3^2}$ related by the $\mu_3$-action such that $\zeta_3T_e=T_{\zeta_3}$. The dotted arcs together with boundary portions trace the contours $C_\delta, \zeta_3C_\delta, \zeta_3^2 C_\delta$. The right-hand side of the figure depicts the images of these contours under the maps $u^{\varepsilon}_+$ inside $\C_+$. The purple marker is the origin.}
    \label{fig:family-with-crossing}
\end{figure}
\begin{figure}[ht]
     \centering
     \includegraphics[width=14.5cm]{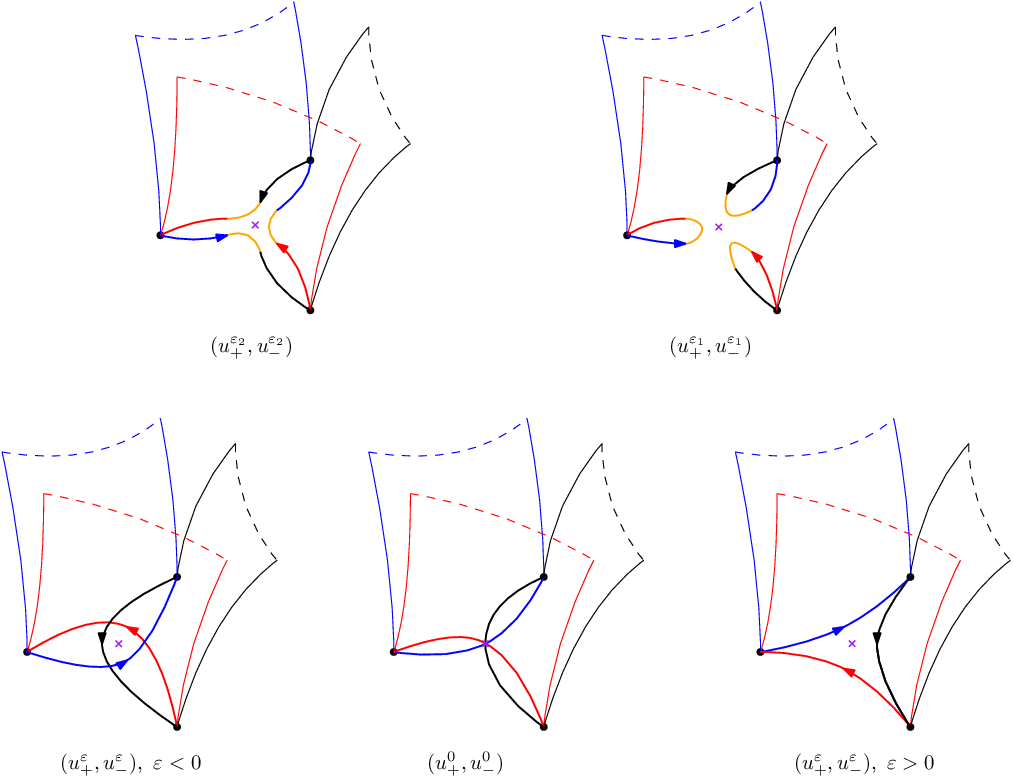}
     \caption{A schematic depiction of the projections of the curves in our families to $\C_+ \times \C_-$, with the boundary mapping to $\R^2 \subset \C_+ \times \C_-$. Dotted points represent the projections of endpoints of the paths contributing to the evaluation $G$-path, under the assumption that the basepoint $x$ is sufficiently close to the hypersurface $Y$.}
    \label{figure: all families in a fountain}
\end{figure}

By our choice of Floer data --- in particular, since near the outgoing boundary the Hamiltonian vanishes and $J$ is adapted to type $A$ inertia --- we can apply \cite[Lemma 6.1.4]{pomerleanoseidel2024} to conclude, without loss of generality, that the curves $u^{\varepsilon}$ with $\varepsilon< 0$ and $|\varepsilon|$ small intersect the divisor $V_{-\chi}$ with multiplicity $1$ at a unique point $p^{\varepsilon}_{\times}$ in the neighborhood $\mathcal{N}(p^0_{\times})$ of $p^0_{\times}$ inherited from $\mathcal{N}(q^0_\times)$. To see this, we locally double $u^{\varepsilon}$ to $u^{\varepsilon,\LEFTcircle}$ using the standard anti-symplectic
involution on $T^*X$ and apply the lemma in loc.\ cit.\ to $u^{\varepsilon,\LEFTcircle}$. Observe that $u^{\varepsilon,\LEFTcircle}$ has algebraic intersection number $2m_Y$ with the divisor $V_{\chi} \cup V_{-\chi}$.

We denote by $\partial_{\varepsilon}$ the vector field on $\mathcal{T}(\hat{x}; \bm m_{\bm Y, \bm h})$ with integral curve $u^\varepsilon$, and assume the following:
\be
\item[(*)] the differential of the boundary evaluation map at $(u^0, p^0_{\times})$ takes $(\partial_\varepsilon, \partial_t)$ to a pair of vectors that project to a positive frame of $(\nu_Y)_y \subset T_yX$.
\ee
The assumption (*) does not affect the final computation. By (*), the curves $\gamma^{\varepsilon}(\tau)$ for small $|\tau|$ and $\varepsilon<0$ move in the counterclockwise direction around $Y$ as $\tau$ increases. See Figure~\ref{fig:family-with-crossing}.

Recall the coordinates 
\begin{equation} \label{eqn: coordinates}
    \C_{+} \times \C_{-} \times T^*\R^{2n-2}
\end{equation}
of $T^*X$ near $y \in Y$ from Remark~\ref{remark: coordinates-for-isotypic-decomposition}. In these coordinates the $0$-section $X$ is given by 
\begin{equation}\label{eq: coordinate-for-zero-section}
    \{(v_+, v_-, x, p) \mid v_-=v_+,  \, p=0\}
\end{equation}
and we write our curves as $u^{\varepsilon}=(u^{\varepsilon}_+, u^{\varepsilon}_-, u^{\varepsilon}_{T^*\R^{2n-2}})$. We pick a contour
$$C_\delta=\{  i (t_0-\delta e^{i\theta}) \mid \theta \in [0, \pi]\} \cup (\{0\} \times [t_0+\delta, t_0-\delta]),$$
consisting of a half-circle and a boundary segment, with $\delta>0$ sufficiently small, so that along this contour the curve 
$$u^{\varepsilon}_+|_{C_\delta} \colon C_\delta \to \C_+ $$
has winding number exactly $1$ for each small $\varepsilon <0$. This is possible, since the number of zeroes of $u^{\varepsilon}_+$ counts the number of intersection points of $u^{\varepsilon}$ with $V_{-\chi}$. For an element $g \in H_Y$, we denote by $g C_{\delta}$ the contour near $g p^0_{\times}$ parametrized similarly (we may assume that the $G$-action identifies these parametrizations), and notice that $u^{\varepsilon}_+|_{gC_\delta}$ also has winding number equal to $1$. 

With respect to coordinates \eqref{eqn: coordinates} the restriction of $u^{\varepsilon}$ to the boundary~\eqref{eq: coordinate-for-zero-section} near $p_{\times}^0$ is given by 
\begin{equation}
    \gamma^{\varepsilon}(\tau)=(\gamma^{\varepsilon}_+(\tau), \gamma^\varepsilon_+(\tau), \gamma^{\varepsilon}_Y(\tau), 0).
\end{equation}
A key observation for later use is that
the vertical component of $\gamma^{\varepsilon}(\tau)$ in the coordinates $\C \times \R^{2n-2}$ for the tubular neighborhood of $Y \subset X$ is given by $2\cdot\gamma_+^{\varepsilon}(\tau)$, 
and therefore $\gamma_+^{\varepsilon}(\tau)$ moves in the counterclockwise direction (as $\tau$ increases) by (*). Indeed, this follows since a point $(v_+, v_-, x, p) \in \C_+ \times \C_- \times T^*\R^{2n-2}$ can be written as
$$(v_++v_-, B_+v_++B_-v_-, x, p)\in T^*\C \times T^*{\R^{2n-2}},$$
where we use the conventions of Remark~\ref{remark: coordinates-for-isotypic-decomposition}.

On the other hand, $\gamma_+^{\varepsilon}(\tau)$ is a portion of the loop $u^{\varepsilon}_+|_{C_\delta}$. Hence, the winding of $u^{\varepsilon}_+|_{C_\delta}$ can be used to analyze the winding around $Y$ of the evaluation of $u^{\varepsilon}$ along the outgoing boundary around $Y$.

Assumption (*) implies the following identity in $\pi_1(X^{\operatorname{reg}}/G)$:
\begin{equation}
    [ev(u^{-\varepsilon})]=\tilde{\gamma}^{m_Y}[ev(u^{\varepsilon})],
\end{equation}
where $\tilde{\gamma} \in C_Y$ is a loop in the conjugacy class of a small counterclockwise loop around $Y$, and $\varepsilon>0$ is sufficiently small. We set $\Gamma=[ev(u^{\varepsilon})]$.
\begin{figure}[ht]
    \centering
    \includegraphics[width=.85\linewidth]{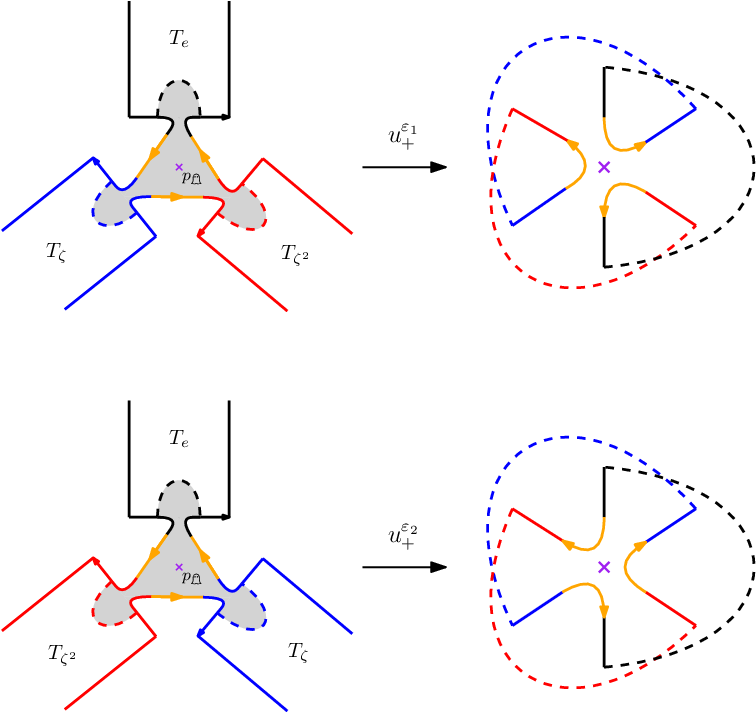}
    \caption{Some curves in the families $u^{\varepsilon_1}_+$ and $u^{\varepsilon_2}_+$ obtained by gluing a nodal ghost orbidisk. The shaded region is $\hat{\Sigma}_{m_Y}^j$. Observe that since the local monodromies at the stacky point are different for $\hat{\Sigma}_{m_Y}^1$ and $\hat{\Sigma}_{m_Y}^2$, the half-strip regions $T_e, T_{\zeta_3}, T_{\zeta_3^2}$ appear in opposite orders. The orange arcs trace the image of the restriction of $u^{\varepsilon_j}_+$ to the portion of the boundary coming from the new disk. It is clear from the cartoon that $u_+^{\varepsilon_1}$ winds once around the origin, while $u_+^{\varepsilon_2}$ winds twice.
    }
    \label{fig: gluing-families-mu3}
\end{figure}
We specify the generator $g_Y \in H_Y$ which acts via multiplication by $\zeta_{m_Y}$ on $\nu_Y$; with respect to this presentation of $H_Y$ the character $\chi$ is exactly $\chi_1^{m_Y}$.

\s\n
\emph{Gluing unobstructed ghost disks.} We now glue an unobstructed ghost orbidisk to the curve $(u^0,\pi_0)$ from \eqref{eqn: u zero}.

Choose $j \in \{1, \dots, m_Y-1\}$ and consider
$$\pi_\infty \colon \Sigma^\infty \to \mathscr{C}^{\infty}, \quad \mathscr{C}^\infty=\mathscr{C}^0\cup \mathscr{C}_{\mathghost},$$
where

$\mathscr{C}_{\mathghost}$ is a disk with a unique stacky point $q_{\mathghost}$ whose local monodromy is given by $(g_Y^j)$ and which is attached to $\mathscr{C}^0$ along a unique boundary node $\ddot{n}=(q^0_{\times}, n_{\mathghost})$. We have a decomposition $\Sigma^{\infty}=\Sigma^0 \cup \Sigma_{\mathghost}$ and 
$$\Sigma_{\mathghost}=\sqcup_{g \in G/H_Y} \Sigma_{m_Y},$$
where $\Sigma_{m_Y}$ is a disjoint union of $\gcd(m_Y, j)$ disks, with $H_Y$ preserving $\Sigma_{m_Y}$. We set
\begin{equation}\label{eqn: j prime}
    j'=\frac{j}{\gcd(m_Y, j)}.
\end{equation}

Next consider the equivariant map
$$u^\infty_j=(u^0, u_j^{\mathghost}) \colon \Sigma^0 \cup \Sigma_{\mathghost} \to T^*X$$
such that $u^\infty_j(\Sigma_{m_Y})=y \in Y$. 
We denote 
\begin{equation}
(\bm Y^j, \bm h^j)=((\bm Y, \bm h), (T^*Y, (g_Y^j))).
\end{equation}

By Proposition~\ref{proposition-one-step-gluing} (or more precisely, by Remark~\ref{remark: continuous-one-step-gluing}), we have a continuous family of curves $$u^{\varepsilon_j} \in \mathcal{T}(\hat{x}; \bm m_{(\bm Y^j, \bm h^j)})$$ with $\varepsilon_j> 0$ converging to $u^{\infty}_j$ (strictly speaking, one should denote this family $u^{\varepsilon_j}_j$, but we will allow ourselves this abuse of notation). We claim that 
\begin{equation}\label{eq: evaluation-gluing-identity}
[ev(u^{\varepsilon_j})]=\tilde{\gamma}^{j}\Gamma.
\end{equation}
From the equivariance it is clear that $[ev(u^{\varepsilon_j})]=\tilde{\gamma}^{j+a\cdot m_Y}\Gamma$ for some integer $a \in \Z$, and hence our goal is to show that $a=0$.

To see this, we first note that in the coordinate $w \in \C$ near each point $p_{\mathghost} \in(\pi_{\infty})^{-1}(q_{\mathghost})$, by Corollary~\ref{cor: moduli-spaces-with-condtions} the Taylor expansion of the projection of $u^{\varepsilon_j}$ to $V_{\chi}$ is of the form
$$\pi_{\chi} \circ u^{\varepsilon_j}(w)=(u_+^{\varepsilon_j},u^{\varepsilon_j}_{T^*\R^{2n}})(w)=c_{\varepsilon_j}(w^{j'}, 0)+O(|w|^{j'+1}),$$
where $c_{\varepsilon_j}$ is a nonzero constant and $j'$ satisfies \eqref{eqn: j prime}. Hence the restriction of the component $u^{\varepsilon_j}_+$ to some small loop around $p_{\mathghost}$ has winding number equal to $j'$. Summing up the contributions over all $p_{\mathghost} \in(\pi_{\infty})^{-1}(q_{\mathghost})$, we obtain that $u^{\varepsilon_j}_+$ has at least $j$ zeroes when restricted to the neighborhood $\widehat{\Sigma}_{m_Y}^{\varepsilon_j}$ of $\Sigma_{m_Y}^{\varepsilon_j}$, the portion of the domain corresponding to $\Sigma_{m_Y}$ under gluing (see the shaded regions of Figure~\ref{fig: gluing-families-mu3}), and moreover we may guarantee that this neighborhood has the boundary comprised of portions of contours $gC_\delta$ for $g \in H_Y$. By the same argument, the curves $u^{\varepsilon_j}_-$ have $m_Y-j$ zeroes.

Now we claim that $(u^{\varepsilon_j}_+, u_-^{\varepsilon_j})^{\LEFTcircle}$ and $(u^0_+, u_-^0)^{\LEFTcircle}$ (here we are using the notation ${\LEFTcircle}$ to indicate the local double) have the same intersection number with $\{v_+=0\} \cup \{v_-=0\}$. 
Indeed, the doubles $(u^{\varepsilon_j}_+, u_-^{\varepsilon_j})^{\LEFTcircle}$ restricted to $\widehat{\Sigma}_{m_Y}^{\varepsilon_j}$ represent the same element in $H_2(D^2_+ \times D^2_-, A)$, where $D^2_{\pm}$ is a disk of a small radius in $\C_{\pm}$, and $A$ is a neighborhood of $\partial D^2_+ \times \partial D^2_-$ (this $A$ is within a small neighborhood of the images of doubles of half-circles used in the contours $gC_\delta$). This class in $H_2(D^2_+ \times D^2_-, A)$ is equal to the class of $(u^0_+, u_-^0)^{\LEFTcircle}$, since we are attaching a ghost bubble, implying the claim. Finally, since $(u^0_+,u^0_-)^{\LEFTcircle}$ has exactly $2m_Y$ zeroes in the neighborhood of the orbit $H_Y \cdot p_{\times}^0$, the intersection number of $(u^{\varepsilon_j}_+, u_-^{\varepsilon_j})^{\LEFTcircle}$ with $\{v_+=0\} \cup \{v_-=0\}$ is exactly $2m_Y$.

\begin{figure}[ht]
    \centering
    \includegraphics[width=.9\linewidth]{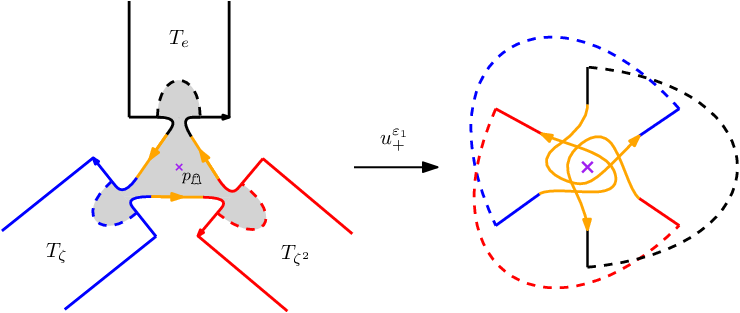}
    \caption{The boundary evaluation of a potential family $u^{\varepsilon_1}_+$, which is eliminated by our argument.}
    \label{fig:wrong-gluing}
\end{figure}

The argument above shows that when restricted to $\widehat{\Sigma}_{m_Y}^{\varepsilon_j}$, the only zeroes of $u^{\varepsilon_j}_+$ are the preimages of $q_{\mathghost}$ and there are exactly $j$ of them counted with multiplicity. Hence, the restriction of $u^{\varepsilon_j}_+$ to the boundary $\partial\widehat{\Sigma}_{m_Y}^{\varepsilon_j}$ winds around the origin exactly $j$ times, and by our earlier observations, this is equal to the count of winding around $Y$ of restriction to the boundary of the curve $u^{\varepsilon_j}$. We refer the reader to Figure~\ref{figure: all families in a fountain} and Figure~\ref{fig: gluing-families-mu3} for illustrations of the local behavior of curves $u^{\varepsilon_j}$ and $u^{\varepsilon}$ near the gluing region, and for their boundary evaluations.

\s\n
\emph{The Hecke relation as a fountain.} We are now in a position to conclude the proof of Proposition~\ref{prop: algebra-morphism}. Indeed, for any generator $[\hat{x}] \in o_{\hat{x}}$ we may glue together moduli spaces $\mathcal{T}(\hat{x}; \bm m_{\bm Y, \bm h})$ and $\mathcal{T}(\hat{x}; \bm m_{\bm Y^j, \bm h^j})$ by identifying the curve $u^0$ with curves $u^\infty_j$ for all $j\in\{1, \dots, m_Y-1\}$. Since in each of these cases we can use the same generator $[x]$ to orient the moduli spaces, we obtain a local relation
\begin{equation}\label{eq: curve-counting-Hecke-relation}
    \tilde{\gamma}^{m_Y}[ev(u^{\varepsilon})]=(\hbar_{{m_Y-1}}^Y\tilde{\gamma}^{m_Y-1}+\dots+\hbar_{1}^Y\tilde{\gamma}^{1}+1)[ev(u^{\varepsilon})],
\end{equation}
by summing up the individual contributions~\eqref{eq: evaluation-gluing-identity} from each curve.
This is precisely the Hecke relation~\eqref{eq: hbar-Hecke-relation}. In essence, we showed that the interval in $\mathcal{T}(\hat{x}, \bm m_{\bm Y, \bm h})$ gave rise to a \emph{fountain} of ${m_Y}-1$ arcs in the glued moduli space; see Figure~\ref{fig: fountain} for an illustration.
Notice that one should account for the difference in the coefficient $\frac{1}{b!}$ and $\frac{1}{(b+1)!}$, by observing that there is a permutation ambiguity in defining $(\bm Y^j, \bm h^j)$. It is also clear that one could start with a family $u^{\varepsilon_j}$ converging to a curve $u^{\infty}_j$ and obtain the curve $u^0$ by deleting the ghost. Hence, such a $1$-parameter family also produces a similar-looking fountain. This concludes the proof of the well-definedness of the map $\mathcal{EV}$, and by our earlier remarks, the proof of the fact that $\mathcal{EV}$ is an algebra morphism is analogous.
\begin{figure}[ht]
    \centering
\includegraphics[width=0.35\linewidth]{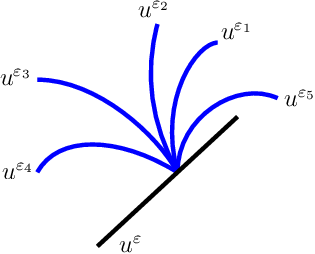}
    \caption{The local topology of the glued moduli space near a fountain of curves in the case $m_Y=6$. }
    \label{fig: fountain}
\end{figure}

\end{proof}

\begin{lemma}\label{lemma: index-computation-for-An-singularity}
Let $\Sigma_m$ be a $\mu_m$-cover of the unit disk $D$ branched at $b$ points. Given a constant map
$$u \colon (\Sigma_m,\bdry \Sigma_m) \to (M=T^*X, X)$$
with $u(\Sigma_m)=x \in X$ and inertia evaluation $\bm m=((T^*Y, (h_1)), \ldots, (T^*Y, (h_b)))$, one has
\begin{enumerate} 
    \item[(a)] $\op{dim}(\op{coker} D_{(u^*TM,u^*TX)} ^{\mu_m})= 2b-2.$
    \item[(b)] Writing $T_xX=T_xY \oplus \nu_Y|_x$, if $b>1$, then for any point $p \in \partial\Sigma_m$ and any $v \in \nu_Y|_x$, there exists 
    $$\overline\eta \otimes v \in (H^{0,1}_{\R}(\Sigma_m) \otimes_{\R} T_xX)^{\mu_m} \cong H^{0,1}_{u^*TX}(\Sigma_m, u^*TM)^{\mu_m} \cong \operatorname{coker} D_{(u^*TM,u^*TX)} ^{\mu_m}$$
    which does not vanish at $p$.
\end{enumerate}
\end{lemma}

\begin{proof}[Proof of (a)]
Observe that 
$$(u^*TM, u^*TX)\cong(\underline{\C}, \underline{\R})^{\oplus(n-2)} \oplus (\underline{V}^{\C}, \underline{V}),$$
where $\C$ is the trivial $\mu_m$-representation and $\underline{\C}$ is the associated trivial bundle over $\Sigma_m$, while $V$ is a complex representation of character $\chi_{k}^m$ and $\underline{V}^{\C}$ is the trivial bundle over $\Sigma_m$ associated with the complexification of the representation $V$. In particular, $V^{\C} \cong \rho_{\chi_k^m} \oplus \rho_{\chi_{m-k}^m}$, and $\underline{V}^{\C}$ is not a trivial $\mu_m$-bundle.
The calculation of (a) reduces to the index computation for the desingularized bundle pair on $D$ corresponding to $(u^*TM, u^*TX)$. (See \cite{Chen2002} and \cite[Section 3]{Cho2014} in particular for a discussion on desingularization in the Lagrangian Floer context.) Since the Maslov index for the desingularization of the trivial bundle is $0$, as the space of holomorphic functions on $\Sigma_m$ with real boundary condition is real $1$-dimensional, it suffices to perform the computation for the component $(\underline{V}^{\C}, \underline{V})$. We denote the desingularized bundle pair by $(|\underline{V}^{\C}|, |\underline{V}|)$.

First we consider the case $b=1$. In this case $\pi \colon \Sigma_m \to D$ is just the map $z \mapsto z^m$ and we may assume that $h_1$ generates $\mu_m$ (if $h_1$ does not generate $\mu_m$, the computation reduces to the analogous computation for $\mu_{\operatorname{ord}(h_1)}$). With respect to the standard $\R$-basis $(e_1, e_2)$ of $V$, the eigenbasis of $V^\C$ is given by 
$$(v_+, v_-)=(e_1-ie_2, e_1+ie_2),$$
and the action of $h_1$ on $\underline{V}^\C=D^2 \times \C^2$ with respect to this basis is given by
$$h_1(z, w_1, w_2)=(\zeta_mz, \zeta_m^\ell w_1, \zeta^{m-\ell}_mw_2),$$
where $0 \le \ell< m$ such that $\chi^m_k(h_1)=\zeta^\ell_m$. Now, with respect to $(w_1,w_2)$, $V$ is the graph 
$$\{(w, \overline{w}) \in \C_{v_+} \oplus \C_{v_-} \mid w \in \C\}$$ 
of the anticomplex involution. Hence the image of $\underline{V}|_{\partial D}$ under the desingularization map is given by
$$(e^{i\theta}, e^{-\frac{i \theta \ell}{m}}w, e^{-\frac{i \theta (m-\ell)}{m}}\overline{w}),  \quad w\in \C, \, \theta \in [0, 2\pi].$$
Setting $w'=e^{-\frac{i\theta \ell}{m}}w$, this Lagrangian subbundle can be rewritten as
$$(e^{i\theta}, w', e^{-i \theta}\overline{w'}), \quad w'\in \C, \, \theta \in [0, 2\pi],$$
and in the $(e_1, e_2)$-basis this is just
\begin{equation} \label{eqn: Lagrangian path}
    (e^{i\theta}, e^{-i \frac{\theta}{2}}\R, e^{-i \frac{\theta}{2}}\R), \quad \theta \in [0, 2\pi].
\end{equation}
The Maslov index of \eqref{eqn: Lagrangian path} is equal to $-2$, and the result for $b=1$ follows since
$$\operatorname{ind}(\overline{\partial}^{\mu_m}_{(\underline{V}^{\C}, \underline{V})})=\operatorname{dim}_{\R} V+\mu(|\underline{V}^{\C}|, |\underline{V}|).$$

The $b\geq 1$ case now follows by summing up the local contributions: choose a trivialization
$$\Phi \colon |\underline{V}^{\C}| \to D \times \C^2$$
of $|\underline{V}^{\C}|$ over the whole disk $D$ which coincides with the trivializations near the stacky points obtained from the desingularization maps. Since $\underline{V}$ is constant, it extends to all of $\Sigma_m$ as a Lagrangian subbundle and induces a Lagrangian subbundle $|\underline{V}|$ of $|\underline{V}^{\C}| =D \times  \C^2$ over $D \setminus \{q_1, \ldots, q_b\}$. Its Maslov index can be computed as the Maslov index of the loop induced by $|\underline{V}|_{\partial D}$ in the Lagrangian Grassmannian of $\C^2$. The class of this loop in $\pi_1(\Lambda Gr_2) \cong H_1(\Lambda Gr_2)$ is equal to the sum of the loops corresponding to restrictions of $|\underline{V}|$ to small loops around $q_1, \ldots q_b$, and we computed their Maslov indices in the $b=1$ case.

\s 
\emph{Proof of (b).}  In view of \cite[Proposition 4.2.2]{Chen2002}, we calculate sections on desingularized bundle pairs.  By Serre duality, it suffices to show that the sheaf $\mathcal{O}(T^*D \otimes_{\C} |\underline{V}^{\C}|^{\vee}, T^*\partial D \otimes_{\R}|\underline{V}|^{\vee})$ has a section which does not vanish at $\pi(p)$. We apply the doubling construction of \cite[Section 3]{KL06} to both bundle pairs 
$$(|\underline{V}^{\C}|, |\underline{V}|)~~~\mbox{and}~~~(T^*D \otimes_{\C} |\underline{V}^{\C}|^{\vee}, T^*\partial D \otimes_{\R}|\underline{V}|^{\vee})$$ 
to obtain vector bundles $\mathcal{V}$ and $\mathcal{W}\cong \mathcal{V}^{\vee} \otimes \mathcal{O}_{\mathbb{P}^1}(-2)$ over $\mathbb{P}^1$ equipped with anticomplex involutions. By \cite[Theorem 3.3.8]{KL06} and the discussion at the beginning of \cite[Section 3.4]{KL06}, for each bundle pair $(E,L)$ on $D$ and the corresponding double $\mathcal{E}$ on $\mathbb{P}^1$ with involution $\tau$, there are isomorphisms
\begin{equation}\label{eq: KL-theorem}
    H^0(\mathcal{O}(E,L))\otimes_{\R} \C \cong H^0(\mathcal{O}(\mathcal{E}))^{\tau} \otimes_{\R} \C \cong H^0(\mathcal{O}(\mathcal{E})).
\end{equation}
 By Grothendieck's theorem
$$\mathcal{V} \cong \mathcal{O}(d_1) \oplus \mathcal{O}(d_2),$$
and by doubling the Maslov index calculation from (a) we have $d_1+d_2=c_1(\mathcal{V})= \mu(|\underline{V}^{\C}|, |\underline{V}|)\leq -2b \leq -4$.

Hence, we can assume $d_1 \le -2$, which implies that for any $v \in \mathcal{O}(-d_1-2)_{\pi(p)}$ there is a global section $s$ of $\mathcal{W}$ with $s(q)=v$. Taking the projector $\frac{s+\tau'(s)}{2}$ we get a $\tau'$-fixed section with value $\frac{v+\tau'_{\pi(p)}(v)}{2}$ at this point. 

Hence, by~\eqref{eq: KL-theorem} we may conclude that for the initial bundle pair $(\underline{V}^{\C}, \underline{V})$ there is a $1$-form 
$$\overline\eta \otimes v \in (H^{0,1}_{\R}(\Sigma_m) \otimes_{\R} T_xX)^{\mu_m}$$ 
that does not vanish at $p$. To conclude the proof, we notice that we also have a $\mu_m$-equivariant form
$$\overline\eta \otimes Iv \in (H^{0,1}_{\R}(\Sigma_m) \otimes_{\R} T_xX)^{\mu_m},$$
where $I$ is the complex structure on $V$, and since $v$ and $Iv$ span $V\cong \nu_Y|_x$, the claim follows.
\end{proof}

\s\n
\emph{Review of standard results on complete algebras.} In the proof of Theorem~\ref{theorem: Hecke-isomorphism}, we use the following two lemmas on complete algebras, both of which are well-known, and we present their proofs for the reader's convenience. Let us introduce the notation $R_{\bm \hbar}=\C\llbracket \bm \hbar \rrbracket$ and $\mathfrak{m}_{\bm \hbar}$ for the maximal ideal of $R_{\bm \hbar}$. 

We start with the following corollary of Nakayama's lemma \cite[\href{https://stacks.math.columbia.edu/tag/00M9}{Lemma 10.96.1(1)}]{stacks-project}, applied to modules over the commutative ring $R_{\bm \hbar}$.
\
\begin{lemma}\label{lemma: Nakayma}
    Let $f \colon B \to A$ be a morphism of $R_{\bm \hbar}$-algebras complete with respect to $\mathfrak{m}_{\bm \hbar}$. If $f|_{\bm \hbar=0} \colon B/\mathfrak{m}_{\bm \hbar}B \to A/\mathfrak{m}_{\bm \hbar}A$ is an isomorphism, then $f$ is a surjection.
\end{lemma}

\begin{proof}
      By completeness, it suffices to show the surjectivity of $f^k \colon B/\mathfrak{m}_{\bm \hbar}^kB \to A/\mathfrak{m}_{\bm \hbar}^kA$ for each $k>0$. We prove this by induction with the $k=1$ case given by the assumption. Let $a \in A/\mathfrak{m}_{\bm \hbar}^{k+1}A$.  Then there exists $b \in B/\mathfrak{m}_{\bm \hbar}^{k+1}B$ such that $a-f(b) \in \mathfrak{m}_{\bm \hbar}^{k}A/\mathfrak{m}_{\bm \hbar}^{k+1}A$. The surjectivity of the induced map $\mathfrak{m}_{\bm \hbar}^{k}B/\mathfrak{m}_{\bm \hbar}^{k+1}B \to \mathfrak{m}_{\bm \hbar}^{k}A/\mathfrak{m}_{\bm \hbar}^{k+1}A$ follows immediately from the assumption.
\end{proof}

\begin{lemma}\label{lemma: flatness-means-injective}
    Let $f \colon B \to A$ be a morphism of $R_{\bm \hbar}$-algebras complete with respect to $\mathfrak{m}_{\bm \hbar}$. If $A$ is flat over $R_{\bm \hbar}$, and $f|_{\bm \hbar=0} \colon B/\mathfrak{m}_{\bm \hbar}B \to A/\mathfrak{m}_{\bm \hbar}A$ is an isomorphism, then $f$ is also injective.
\end{lemma}

\begin{proof}
    The flatness of $A$ implies that there is an exact sequence
    $$0 \to \operatorname{ker}(f)/\mathfrak{m}_{\bm \hbar}\operatorname{ker}(f) \to B/\mathfrak{m}_{\bm \hbar}B \to A/\mathfrak{m}_{\bm \hbar}A,$$
    and since the last arrow is an isomorphism, we conclude that $\operatorname{ker}(f) =\mathfrak{m}_{\bm \hbar}\operatorname{ker}(f)$. On the other hand,  if $\operatorname{ker}(f)$ is not zero, then by the completeness of $B$ there is a minimal $k\ge 0$ such that $\operatorname{ker}(f) \subseteq \mathfrak{m}_{\bm \hbar}^{k}B$ but $\operatorname{ker}(f) \not \subseteq \mathfrak{m}_{\bm \hbar}^{k+1}B$. 
    This is a contradiction since $\operatorname{ker}(f)=\mathfrak{m}_{\bm \hbar}\operatorname{ker}(f) \subseteq \mathfrak{m}_{\bm \hbar}^{k+1}B$.
\end{proof}

\begin{proof}[Proof of Theorem~\ref{theorem: Hecke-isomorphism}]
    By Proposition~\ref{prop: algebra-morphism}, the map $\mathcal{EV}$ is a morphism of algebras. First, we observe that 
    $$\mathcal{EV}|_{\bm \hbar=0} \colon HW^0_{\bm \hbar=0}(T^*_{[x]}[X/G]) \to \mathcal{H}_{\bm \hbar=0}([X/G])$$
    is an isomorphism, since for each $g \in G$ it restricts to counting half-strips with Lagrangian boundaries given by $T^*_{gx}X, T^*_xX$ and $X$, with the outgoing boundary tracing a path in $\Omega(X, x, gx)$, hence it coincides with the evaluation map of Abouzaid in \cite[Section 4]{abouzaid2012wrapped}. This evaluation map is known to induce an isomorphism between wrapped Lagrangian Floer homology $HW^*(T^*_{gx}X, T^*_xX)$ and singular (co)-homology $H_{-*}(\Omega(X, x, gx))$, as it is an inverse of the isomorphism of Abbondandolo-Schwarz \cite{abbondandolo2006floer, abbondandolo2010floer}. In light of Decomposition~\eqref{eq: G-paths-decomposition} and our definition of the orbifold $\pi_1([X/G])$ in terms of $G$-paths, the above implies that $\mathcal{EV}|_{\bm \hbar=0}$ is an isomorphism, and the morphism $\mathcal{EV}$ is surjective by Lemma~\ref{lemma: Nakayma}.
     
    Finally, applying Etingof's theorem (Theorem~\ref{theorem: etingof-flatness}) together with our change of variables (see Remark~\ref{remark: hbar-flatness}) we see that the specialized Hecke algebra $\mathcal{H}_{\bm \hbar}|_{\bm \hbar_0=0}([X/G])$ is flat under the assumption of the vanishing of $\pi_2(X) \otimes \Q$. Hence $\mathcal{EV}$ is injective by Lemma~\ref{lemma: flatness-means-injective}.
\end{proof}

\begin{proof}[Proof of Corollary~\ref{corollary: derived-Nakayama}]
    By the proof of Theorem~\ref{theorem: Hecke-isomorphism}, we see immediately that for a $K(\pi,1)$ space $X$, $HW^*_{\bm \hbar=0}(T^*_{[x]}[X/G])=0$ for $*\neq 0$, as it is isomorphic to $\oplus_{g \in G} H_{-*}(\Omega(X, x, gx))$. For the rest of the proof, we introduce the notation $A_0 =CW^*_{\bm \hbar =0}(T^*_{[x]}[X/G])$ and $A_{\hbar}=CW^*_{\bm \hbar}(T^*_{[x]}[X/G])=A_0\llbracket\bm \hbar \rrbracket.$
    
    The $A_\infty$-algebra $A_{\hbar}$ is a \emph{topologically free} (formal) deformation of the $A_\infty$-algebra $A_0$, and both $A_\hbar$ and $A_0$ are concentrated in non-positive degrees. We claim under these assumptions the vanishing of $H^*(A_0)$ for $* \le -1$ implies the vanishing of $H^*(A_{\hbar})$ for $* \le -1$. This is an instance of the derived Nakayama lemma~\cite[\href{https://stacks.math.columbia.edu/tag/0G1U}{Lemma 0G1U}]{stacks-project}, whose proof we provide below.

    Let $x \in A_{\hbar}^i$ for $i \le -1$ such that $dx=0$. We prove by induction that there exists a sequence $\{y_k\}_{k \ge 1}$ with $y_k \in A_{\hbar}^{i-1}$ such that
    \begin{equation}\label{eqn: terms of sequance}
        x-dy_k \in \mathfrak{m}_{\bm \hbar}^kA_{\bm \hbar}\quad \mbox{and} \quad y_k-y_{k+1} \in \mathfrak{m}_{\bm \hbar}^{k}A_{\bm \hbar}^{i-1}.
    \end{equation}
    The existence of such a sequence $\{y_k\}_{k \ge 1}$ implies that $x$ is a coboundary, since the limit $y$ of such a sequence exists and satisfies $dy=x$.
    
    The first term $y_1$ exists since $d_{\bm \hbar=0}x=0$ and $H^i(A_0)=0$.  Next assume that $y_1,\dots, y_k$ exist for some positive $k \in \Z$. Observe that 
    $$\mathfrak{m}_{\bm \hbar}^kA^i_{\bm \hbar}/\mathfrak{m}_{\bm \hbar}^{k+1}A^i_{\bm \hbar}=A_0^i\otimes \mathfrak{m}_{\bm \hbar}^k/\mathfrak{m}_{\bm \hbar}^{k+1},$$
    since the deformation is topologically free. The vanishing of $H^i(A_0)$, together with the fact that $\mathfrak{m}_{\bm \hbar}^k/\mathfrak{m}_{\bm \hbar}^{k+1}$ is of finite dimension, then implies that there exists an element $z \in \mathfrak{m}_{\bm \hbar}^kA_{\bm \hbar}^{i-1}$ such that
    $$[dz]=[(x-dy_k)] \in \mathfrak{m}_{\bm \hbar}^kA^i_{\bm \hbar}/\mathfrak{m}_{\bm \hbar}^{k+1}A^i_{\bm \hbar}.$$
    Hence $y_{k+1}=y_k+z$ satisfies \eqref{eqn: terms of sequance}.
\end{proof}

\section{Examples and discussion}\label{section: examples} 

In this section, we discuss various examples of Hecke algebras in the literature, and illustrate how Theorem~\ref{theorem: Hecke-isomorphism} allows one to recover these Hecke algebras in terms of the wrapped Fukaya category. The discussion here is complementary to \cite[Section 3.7]{EtingofPavel} and \cite[Section 4]{EtingofPavel}, to which we also refer our reader. 

In Theorem~\ref{theorem: Hecke-isomorphism} we work with cotangent bundles of global quotient orbifolds of the form $[X/G]$, where $X$ is assumed to be \emph{compact}. In Section~\ref{section: compact examples} we list some of the well-known Hecke algebras that arise this way.

In Section~\ref{section: noncompact} we discuss examples which are associated with noncompact $X$. We note that such examples have appeared earlier historically.

\subsection{Compact examples}\label{section: compact examples}

\begin{example}[\emph{Double Affine Hecke algebras}]~\label{example: daha}
    Let $G$ be a simple Lie group, $W$ be its Weyl group, and $V$ be the reflection representation of $W$. Let $Q^{\vee}$ be the dual root lattice of $W$, and set $\Gamma=Q^{\vee} \oplus \eta Q^{\vee}$ for some $\eta \in \mathbb{C}$ with $\operatorname{Im}(\eta)>0$. Then 
    $$X=V/\Gamma$$
    is a compact complex torus. The algebra $\mathcal{H}_{\bm \hbar}([X/W])$ is the \emph{double affine Hecke algebra} (DAHA) $\mathcal{H}_{\bm \hbar}(V, W)$ of Cherednik \cite{cherednik1995}.
    This algebra is known to be flat by the classical results of \cite[Theorem 2.3]{cherednik1995}, \cite[Theorem 3.2]{sahi1999nonsymmetric}. It also follows from Etingof's result since $\pi_2(X)=0$.
 
   Corollary~\ref{corollary: derived-Nakayama} applies in this case since $X$ is a torus. Hence the whole graded algebra $HW^*_{\bm \hbar}(T^*_{[x]}[X/W])$ coincides with the Hecke algebra $\mathcal{H}_{\bm \hbar}|_{\bm \hbar_0=0}([X/W])$.
\end{example}

\begin{example}[\emph{Cherednik algebras of complex  crystallographic reflection groups}]\label{example: crystall-cherednik} 
    Let $G$ be an irreducible complex reflection group and $V$ be its reflection representation. We assume that $G$ is \emph{crystallographic}, i.e., that there exists a cocompact lattice $\Gamma \subset V$ preserved by the action of $G$.  We note that these groups were classified by Popov \cite{popov1982complex}; see also \cite{malle1996presentations} and \cite[Section 3]{puenteshepler2019steinberg}. In particular, they include groups $G(\ell, p, n)$ for $\ell=1,2,3,4$ or $6$ (and only for these values of $\ell \in \Z_{\ge1}$).

    Since $G$ preserves $\Gamma$, it induces an action on the compact complex torus $X=V/\Gamma$. The Hecke algebra $\mathcal{H}_{\bm \hbar}([X/G])$ was studied in \cite[Section 6]{etingofma2008dunkl}, \cite{efmv2011calogeromoser} and \cite{ACL2025sw}. This example is a generalization of Example~\ref{example: daha}, and by the same argument we have
    \begin{equation}\label{eq: isomorphism-for-tori}
        HW_{\bm \hbar}^*(T^*_{[x]}[X/G]) \cong HW^0_{\bm \hbar}(T^*_{[x]}[X/G]) \cong \mathcal{H}_{\bm \hbar}|_{\bm \hbar_0=0}([X/G]).
    \end{equation}
    In the case of $G(\ell, 1, n) \cong S_n \ltimes (\Z/\ell \Z)^n$ for $\ell=2,3,4$ or $6$ this construction recovers the \emph{generalized DAHAs} of types $D_4, E_6, E_7$ or $E_8$, respectively, introduced in \cite{ego2006GDAHA}.
\end{example}

\begin{example}[\emph{Trigonometric analogs of symplectic reflection algebras}]
    The trigonometric analog \cite[Example 3.17]{EtingofPavel} of a symplectic reflection algebra of \cite{etingof2002symplectic} is given as follows: Let $\Lambda$ be a symplectic lattice, and $G \subset Sp(\Lambda)$ be a finite subgroup of its automorphisms. The action of $G$ extends to the action on $V=\R \otimes_{\Z} \Lambda$, which admits a unitary $G$-invariant structure. Then $G$ acts on the compact complex torus  $X=V/\Lambda$, and one obtains the associated Hecke algebra $\mathcal{H}_{\bm \hbar}([X/G])$. 
    
    The isomorphism~\eqref{eq: isomorphism-for-tori} also holds in this case. An interesting family of examples is given by the wreath products $S_n \ltimes \Gamma^n$ acting on $\C^{2n}$, where $\Gamma \subset SL_2(\C)$ is a finite subgroup which preserves a rank $4$ lattice. For instance, one can take $\Gamma$ to be the quaternion group $Q_8$. 
    
    It is also well-known that by ``doubling" a complex reflection group $(V,G)$, which replaces the initial representation with $V \oplus V^*$, one obtains a symplectic reflection group. However, its Hecke algebra has no deformation parameters, as there are no fixed hypersurfaces. 
    
    Wreath products and doubled complex reflection groups constitute most of the indecomposable symplectic reflection groups, and we refer the reader to \cite{cohen1980cassification} for the classification of symplectic reflection groups, and \cite[Section 4]{ballamyschedler2016existence} and \cite{schmitt2023thesis} for a modern exposition.
\end{example}

\begin{example}[\emph{Hecke algebras of orbifold surfaces}]\label{example: hecke-surface}
    Let $H$ be a simply connected Riemann surface, $\Gamma \subset \operatorname{Aut}(H)$ be a cocompact, discrete subgroup acting properly discontinuously with finite stabilizers, and $\Gamma' \subset \Gamma$ be the maximal finite index subgroup acting freely. Let $X=\Sigma=H/\Gamma'$ be the corresponding smooth Riemann surface and $G=\Gamma/\Gamma'$. Then the Hecke algebra of the orbifold $[H/\Gamma]$ is given by $\mathcal{H}_{\bm \hbar}([X/G])$. Let us note that any \emph{developable} (or \emph{good}) closed orbifold Riemann surface arises in this way.
        
    In the case $H= \C$ or $H= \H_+$, this deformation is flat by Theorem~\ref{theorem: etingof-flatness}, and hence Theorem~\ref{theorem: Hecke-isomorphism} provides the computation of the $A_\infty$-algebras of these orbifolds. These Hecke algebras were studied in \cite{etingofrains2005new, eor2007GDAHA}.

    The case $H=X= \P^1$ was studied in \cite{etingofrains2008new2}.  In particular, the corresponding Hecke algebra is not flat, and hence Theorem~\ref{theorem: Hecke-isomorphism} does not apply here. We expect $HW^0(T^*_{[x]}[X/G])$ to be isomorphic to the Hecke algebra, and moreover the differential in $CW^*(T^*_{[x]}[X/G])$ to witness the non-flatness of the Hecke algebra $\mathcal{H}_{\bm \hbar}([X/G])$. It would also be interesting to see the full computation of $HW^*(T^*_{[x]}[X/G])$ for these spherical orbifolds. The Morse-theoretic approach developed in \cite{hkty2025morse} should provide sufficient tools for performing this computation.
\end{example}

\begin{example}[\emph{Hecke algebra of symmetric powers of surfaces}]\label{example: hecke symmetric powers of surfaces}
    Using $H, \Gamma, \Gamma', \Sigma$ from the previous example, we consider $X=\Sigma^n$ and $G= S_n \ltimes (\Gamma/\Gamma')^n$. Then $[X/G]=\operatorname{Sym}^n([H/\Gamma])$, and Theorem~\ref{theorem: Hecke-isomorphism} again applies if $H \neq \P^1$. 
    
    In \cite{HTY2026}, these algebras for $\Gamma$ acting freely were called \emph{braid surface Hecke algebras}. We note that the braid surface Hecke algebras have an additional $c$-variable that we discuss in more detail in the next example.
\end{example}

\begin{example}[\em Full DAHA of type $A_n$]
    There is a caveat to the discussion above. Cherednik's \emph{full DAHA} $\mathcal{H}_{\hbar, c}(\C^n, S_n)$ of $A_n$-type contains an additional variable $c$. As explained in \cite[Section 2.4.3]{simentalnotes}, the full DAHA does not arise from the orbifold perspective alone, since the full DAHA is not a deformation of the group algebra of the braid group. Instead, it is a deformation of its central extension.

    Two groups, \cite{jordanvazirani2021} and \cite{morton2021dahas}, realized that, in order to remedy this, one needs to work with the braid group of the once-punctured torus $T^2\setminus\{\operatorname{pt}\}$ instead.
    In \cite{morton2021dahas}, the additional relation involving the $c$-variable is introduced for isotopies ``passing through" the puncture. It is shown \cite[Theorem 3.5]{morton2021dahas} (see also \cite[Proposition 4.5]{jordanvazirani2021}) that this algebra is isomorphic to the full DAHA. We note that this relation can be implemented at the level of $S_n$-paths in $(T^2)^n$ by adding an additional relation involving $c$ for ``passing through" the $S_n$-orbit of the hypersurface $\operatorname{pt} \times (T^2)^{n-1}$. 
    
    In \cite{HTY2026}, the authors showed that the cylindrical version of the $A_\infty$-algebra of endomorphisms of $T^*_{[x]}\operatorname{Sym}^nT^2$ does recover DAHA. In the cylindrical interpretation, instead of counting orbicurves in $\operatorname{Sym}^n(T^*T^2)$,  one counts curves in $$\R \times [0,1] \times T^*T^2$$ which $n$-fold cover the strip direction. The $c$-variable is obtained by counting the intersection pairing with $\R \times [0,1] \times T^*_{\operatorname{pt}}T^2$. The latter submanifold is not a symplectic divisor, and hence the intersection can be negative and one is forced to work over $\Bbbk\llbracket c^{\pm1}\rrbracket$. From the perspective of the symmetric product, this submanifold corresponds to the $S_n$-orbit of $T^*_{\operatorname{pt}}T^2 \times (T^*T^2)^{n-1}$, and hence this $c$-variable can also be seen through the orbifold lens. Denoting the cohomology of such an $A_\infty$-algebra (we omit the full definition) by $HW^*_{\hbar, c}(T^*_{[x]}\op{Sym}^nT^2)$, we expect that there is an isomorphism $$HW^*_{\hbar, c}(T^*_{[x]}\operatorname{Sym}^nT^2) \cong \mathcal{H}_{\hbar, c}(\C^n, S_n),$$
    which extends our isomorphism from Example~\ref{example: daha}.
    This is a restatement of \cite[Theorem 1.4]{HTY2026} for $T^2$.
\end{example}

\begin{remark}[\emph{Braid surface Hecke algebras}]
    In \cite{HTY2026} the braid surface Hecke algebras $\mathcal{H}_{\hbar, c}([\Sigma^n/S_n])$ were introduced for an arbitrary closed Riemann surface $\Sigma$, extending the definition of \cite{morton2021dahas}. It is claimed that the isomorphism
    $$HW_{\hbar, c}^*(T^*_{[x]}\operatorname{Sym}^n \Sigma) \cong\mathcal{H}_{\hbar, c}([\Sigma^n/S_n])$$
    holds for any closed Riemann surface of genus $\geq 1$.
    We point out that in their proof (see \cite[Section 6]{HTY2026}) it is tacitly assumed that $\mathcal{H}_{\hbar, c}([\Sigma^n/S_n])$ is flat, which is then used to obtain the said isomorphism as in Lemma~\ref{lemma: flatness-means-injective}. However, this flatness is not known for genus $\geq 2$, and does not seem to follow directly from Theorem~\ref{theorem: etingof-flatness}. It would be interesting to see whether methods of \cite[Section 3]{EtingofPavel} allow one to prove flatness of the full $\mathcal{H}_{\hbar, c}([\Sigma^n/S_n])$, and not only of $\mathcal{H}_{\hbar}([\Sigma^n/S_n])$.

    In \cite[Theorem 1.7]{hkty2025morse}, it is shown that the $A_\infty$-algebra $CW_{\hbar}^*(T^*_{[x]}\operatorname{Sym}^nS^2)$ is quasi-isomorphic to a certain explicit differential graded algebra $H_n$, defined in terms of generators and relations. From this computation, it follows that the isomorphism
    $$HW^0_{\hbar}(T^*_{[x]}\operatorname{Sym}^nS^2) \cong \mathcal{H}_{\hbar}(\operatorname{Sym}^nS^2)$$
    holds, and, moreover, the differential in $H_n$ witnesses non-flatness of the Hecke algebra $\mathcal{H}_{\hbar}(\operatorname{Sym}^nS^2)$.

\end{remark}

\begin{example}[\emph{Ball quotients}]
    Consider the $n$-dimensional complex hyperbolic space $H=\mathbb{H}_{\C}^n$ and a finitely-generated discrete subgroup $\Gamma \subset PSU(n,1)$ acting on it. By Selberg's lemma, there exists a finite index subgroup $\Gamma' \subset \Gamma$ acting freely on $H$. Then for $X=H/\Gamma'$ and $G=\Gamma/\Gamma'$ the Hecke algebra is flat, and Theorem~\ref{theorem: Hecke-isomorphism} applies if $\Gamma$ is cocompact.

    For $n=2$ there is a family of \emph{Mostow groups} $\Gamma(p,t) \subset PU(2,1)$ for $p \in \Z_{>0}$ and $t \in \Q$ (some of which were first introduced in \cite{mostow1980remarkable}), which are generated by three complex reflections $R_1, R_2, R_3$ with rotation parameter given by $\zeta_p$, and related via conjugation, i.e., there is a matrix $J$ such that $R_2=JR_1J^{-1}$ and $R_3=J^{-1}R_1J$, and $t \in \Q$ is specialized by requiring
    $$\op{Tr}(R_1J)=e^{(\tfrac{3}{2}+\tfrac{1}{3p}-\tfrac{t}{3})\pi i}.$$
    For a more detailed description, we refer to \cite[Section 3]{DPP2021new}. We also note that these groups are known to be isomorphic to Deligne-Mostow groups, which arise as monodromy groups of hypergeometric functions \cite{delignemostow1986monodromy}.
    
    Only $39$ pairs $(p,t)$ give rise to a lattice, and such lattices are called Deligne-Mostow lattices. We refer to \cite[Section 3]{pasquinelli2016dmthree} and \cite[Appendix A.9]{DPP2021new} for a full list of these lattices. As Deraux shows \cite[Proposition 3.4]{deraux2024subgroups}, $\Gamma(p,t)$ is cocompact if and only if the group does not contain parabolic elements. Using the table \cite[Table 1]{pasquinelli2016dmthree} or the table in \cite[Appendix A.9]{DPP2021new} one sees that only the following $33$ groups among the Deligne-Mostow lattices are cocompact
    \begin{gather*}
        \Gamma(5,7/10), \, \Gamma(7,9/14), \, \Gamma(8,5/8), \, \Gamma(9,11/18), \, \Gamma(10,3/5), \, \Gamma(12,7/12), \, \Gamma(18,5/9),  \\
         \Gamma(4,5/12), \, \Gamma(5,11/30), \, \Gamma(7,13/42), \, \Gamma(8,7/24), \, \Gamma(9,5/  18), \, \Gamma(10,4/15), \, \Gamma(12,1/4), \, \Gamma(18,2/9), \,\\
         \Gamma(3,1/3), \, \Gamma(5,1/5), \, \Gamma(8,1/8), \, \Gamma(12,1/12), \, \Gamma(3,7/30), \, \Gamma(4,3/20), \, \Gamma(5,1/10), \,  \Gamma(10,0), \, \\
         \Gamma(4, 1/12), \, \Gamma(3,5/42), \, \Gamma(3,1/12), \, \Gamma(4,0), \,  \Gamma(3,1/18), \, \Gamma(3,1/30), \, \Gamma(3,0), \, \Gamma(5,1/2), \, \Gamma(7,3/14), \, \Gamma(9,1/18).
    \end{gather*}
\end{example}

\begin{example}[\emph{Global quotients of projective spaces}]
    Given a complex representation $V$ of a finite group $G$, one associates the orbifold Hecke algebra to the pair $(X= \mathbb{P}(V), G)$. It was studied in \cite{BM2014affinity}, and \cite[Lemma 5.4.1]{BM2014affinity} provides a way to compute it in terms of $\mathcal{H}_{\bm \hbar}(V,G)$. Theorem~\ref{theorem: Hecke-isomorphism} does not apply here, and we expect the $A_\infty$-algebra of the cotangent fiber in this case to give rise to an interesting derived orbifold Hecke algebra. 
\end{example}
\subsection{Noncompact examples}\label{section: noncompact}

Before we list the examples of interest, we state the following:

\begin{conjecture}\label{conjecture: noncompact}
    The isomorphism of Theorem~\ref{theorem: Hecke-isomorphism} also holds in the case of a noncompact complex manifold $X$ equipped with the action of the finite subgroup $G \subset \operatorname{Aut}(X)$.
\end{conjecture}

In the case of the trivial group $G$, the above statement follows from \cite[Corollary 6.1]{gps2024microlocal}. However, their proof is not a generalization of the theorems of \cite{abbondandolo2010floer} and \cite{abouzaid2012wrapped}, and is rather a corollary of their main result \cite[Theorem 1.1]{gps2024microlocal}, which gives an isomorphism between the wrapped Fukaya category of $T^*X$ and the category $\operatorname{Sh}^c(X)$ of compact objects in the dg category of sheaves of dg $\Z$-modules on $X$. Currently, it is not known whether such an isomorphism extends to the orbifold case.
\begin{remark} \label{remark: approaches to proving conjecture}
    We comment on the approach to proving Conjecture~\ref{conjecture: noncompact}.  
    The main nontrivial step is the extension of the Abbondandolo-Schwarz map \cite{abbondandolo2010floer} to the case of manifolds with boundary (which one may further extend to the noncompact case by exhaustion), and it is not yet present in the literature: it would require 
    carefully combining their methods with the choices of Floer data used in \cite{gps2020covariant}. After this is done, one also needs to extend Abouzaid's evaluation map from \cite[Section 5]{abouzaid2012wrapped} to the case of manifolds with boundary for which the techniques of \cite[Section 2]{gps2020covariant} should be sufficient. Showing that the extension of Abouzaid's evaluation map is the inverse of the extended Abbondandolo-Schwarz map should be straightforward. We note that using linear Hamiltonians as in \cite{abouzaid2015symplectic} could help overcome some difficulties on this path.

    Once the above is carried out, the proof of the conjecture would follow the steps of the proof of Theorem~\ref{theorem: Hecke-isomorphism}.
\end{remark}

\begin{remark}
    In \cite[Lemma 7.5]{HTY2026} the authors claim that the Abbondan\-dolo-Schwarz map extends to the case of cotangent bundles of punctured surfaces. However, the proof of this is missing in \cite{HTY2026} and should be fixed, potentially using the approach from Remark~\ref{remark: approaches to proving conjecture}.
\end{remark}

\begin{example}[\emph{Hecke algebras of complex reflection groups}]
    Let $G$ be an irreducible complex reflection group, and $X=V$ be its reflection representation. Then $\mathcal{H}_{\bm \hbar}([X/G])$ is the Hecke algebra of \cite{BMR}. This example also includes the classical Hecke algebras of the Weil groups of simple Lie groups.
  
    In this particular case, by taking $X$ to be a unit ball near the origin, one can in fact prove Conjecture~\ref{conjecture: noncompact}, as the required isomorphism for the restriction to $\bm \hbar=0$ is straightforward, since there are no Reeb chords for the standard metric. Hence, the strategy of the proof of Theorem~\ref{theorem: Hecke-isomorphism} extends to this case, once one picks Floer data that guarantees compactness as in \cite{gps2020covariant} (in a different setting, a similar compactness result is shown in \cite[Lemma 7.4]{HTY2026}). 
\end{example}

\begin{example}[\emph{Affine Hecke algebras}]
    Let $G$ be a simple Lie group, $W$ be its Weyl group, and $V$ be the reflection representation of $W$. Let $Q^{\vee}$ be the dual root lattice of $W$. Then 
    $$X=V/Q^{\vee}$$
    is a complex algebraic torus. The algebra $\mathcal{H}_{\bm \hbar}([X/W])$ is the \emph{affine Hecke algebra} $\mathcal{H}_{\bm \hbar}(V, W)$.
\end{example}

\begin{example}[\emph{Hecke algebras of punctured surfaces and their symmetric powers}]
    In Example~\ref{example: hecke-surface}, if we drop the condition for $\Gamma \subset \operatorname{Aut}(H)$ to act cocompactly, one obtains a punctured orbifold surface $[H/\Gamma]$. The Hecke algebras from Example~\ref{example: hecke-surface} extend to this setting.
    
    In the case of $\Gamma$ acting on $H$ freely, the Hecke algebra $\mathcal{H}_\hbar([\Sigma^n/S_n])$ was considered in \cite[Section 7]{HTY2026}, where $\Sigma=H/\Gamma$.
\end{example}

\appendix

\section{Orientation and index theory for orbicurves}\label{appendix: orientations}

The aim of this appendix is to lift the standard result for orientations of determinant lines \cite[Proposition 11.13]{seidel2008fukaya} to the setting of orbidisks with punctures. Let $[u] \colon \mathscr{C} \to [X/G]$ be an orbidisk with $d$ positive and $1$ negative punctures represented by a $G$-equivariant map $u \colon \Sigma \to X$ with boundary conditions given by equivariant Lagrangians $\mathcal{L}_0, \dots, \mathcal{L}_d$ and chords $\hat{x}_0, \dots, \hat{x}_d$. We assume that the evaluation map at the stacky points $q_1, \ldots, q_b$ lands in the connected component $\bm m$ of $\mathcal{I}[X/G]^b$, i.e., $u$ is an element of $\mathcal{M}(\overrightarrow{\hat{x}}, \hat{x}_0; \bm m)$.

\begin{proposition}\leavevmode\par
\begin{enumerate}
    \item[(a)]\label{proposition: orientations}
    The choice of relative spin structures for Lagrangians $\mathcal{L}_i$ and chords $\hat{x}_i$ induces a natural isomorphism
    \begin{equation}
        \operatorname{det}(D_{[u]}) \otimes o_{\hat{x}_d} \otimes\dots \otimes o_{\hat{x}_1} \cong o_{\hat{x}_0}.
    \end{equation}
    \item[(b)]\label{proposition: index} The index of $D_{[u]}$ is given by
\begin{equation}
    \operatorname{ind}(D_{[u]})=|\hat{x}_0|-|\overrightarrow{\hat{x}}|-2\iota(\bm m)
\end{equation}
    \end{enumerate}
\end{proposition}

\begin{proof}[Sketch of proof of (a):] $\mbox{}$

\s\n
\emph{Step 1: Pinching.} We assume that there is at least one stacky point on $\mathscr{C}$, as otherwise the result follows immediately from \cite[Proposition 11.13]{seidel2008fukaya} and $G$-invariance.
Let us introduce notation $(E,F)=(u^*TX, \{u^*T\mathcal{L}_i\}_{i=0}^d)$ for the \emph{bundle pair} on $(\Sigma, \partial \Sigma)$ induced by $u$, and $([E/G], [F/G])$ is the corresponding orbibundle pair on $(\mathscr{C}, \partial \mathscr{C})$. Let $\bm h=((g_1), \dots, (g_b))$ be the Hurwitz datum of $\pi \colon \Sigma \to C$, and set $g=(g_1\dots g_b)^{-1}$. Then we may pinch a curve $\gamma$ parallel to $\partial \mathscr{C}$ to get a nodal orbicurve consisting of an orbidisk $\mathscr{D}$ with Hurwitz datum $\bm h_1=(g^{-1})$ and a closed genus $0$ orbicurve $\mathscr{C}'$ with profile $\bm h_2 =((g_1), \dots, (g_b), (g))$, the node is stacky and balanced. 
This deformation can also be performed at the level of $G$-covers and we get two $G$-covers $\pi_1 \colon \Sigma_1 \to D$ and $\pi_2 \colon \Sigma' \to C'$, with $\Sigma_1$ being a disjoint union of $|G|/|g|$ disks.
        
This deformation carries a deformation of $(E,F)$ to a pair of $G$-equivariant vector bundles $E_1 \to \Sigma_1$ and $E_2 \to \Sigma_2$ with trivial $G$-action on $E_1$. This last claim follows since along $\gamma$ the complex vector bundle $E$ is trivial. Now $D_{E_2}$ is homotopic to a complex linear operator through $G$-equivariant operators, hence $\operatorname{det}(D_{E_2}^G)$ is naturally orientable. By the equivariant version of the gluing result \cite[Theorem 2.4.5]{wehrheimwoodward2015orientations}, whose proof is similar to the equivariant gluing and is omitted here, we have
$$\operatorname{det}(D_{[u]}) \cong \operatorname{det}(D_{E_1, F}^G) \otimes \operatorname{det}(D_{E_2}^G).$$
A similar deformation as above allows one to deform the relative spin structure $E'=u^*E_b$ to a relative spin structure $E_1'$ on $\Sigma_1$.

Now let $D_1$ be one of the discs in $\Sigma_1$ that is preserved by a subgroup of $G$ isomorphic to $\mu_m$ with $m=|g|$. Then, clearly, $\operatorname{det}(D_{(E_1,F)|_{D_1}}^{\mu_m}) \cong \operatorname{det}(D_{E_1}^G)$. Hence, we reduced to the case of one orbidisk with the standard $\mu_m$-action and the trivial action on the fiber bundle $E_1$ over it. We note that for each chord $\hat{x}_i$ the punctures on the boundary of $D_1$ correspond only to an $m$-subtuple of it, but the associated $\mu_m$-invariant Cauchy-Riemann operator is isomorphic to $D_{\hat{x}_i}$ due to $G$-invariance of the choices of relative spin structures in Section~\ref{section: orbifold-Fukaya-prelim}.

\s\n
{\em Step 2: Orbidisk case.}
    First of all, using equivariant gluing again for strip-like ends, we obtain
    \begin{equation}
           \operatorname{det}(D_{E_1, F}^{\mu_m}) \otimes o_{\hat{x}_d}\otimes \dots \otimes o_{\hat{x}_1} \otimes \lambda(\Lambda_{\hat{x}_0(t)}) \otimes o_{\hat{x}_0}^{-1} \cong \operatorname{det}(D_{(\overline{E}_1, \overline{F})}^{\mu_m}).
    \end{equation}
    It remains to obtain a trivialization of the latter orientation line for an operator on the disk $\overline{D}_1$ with closed boundary. Since the action on $\overline{E}_1$ is trivial, we may use standard trivialization results on bundle pairs over Riemann surfaces with boundary to reduce to the one-dimensional case; see \cite[Lemma C.3.8]{mcduffsalamon2004}. We may also assume $D_{E_1}$ to be the standard $\overline{\partial}$ operator on the disk (we may apply the homotopy argument as above).

    Since the pair $(E_1, F)$ is a bundle pair over the disk $D_1$, the relative spin structure on $F$ is equivalent to a spin structure on $F$ by \cite[Lemma A.1]{abouzaid2011cotangent}; see also a more general discussion in \cite[Section 6.6]{chenzinger2024} (needless to say, that the trivial $\mu_m$-action lifts to this spin structure). We denote by $|\overline{E}_1|$ the desingularization complex line bundle over $D$ associated with $E_1$ as in \cite[Section 4.2]{Chen2004}, which is isomorphic to $E_1$ away from the origin via the trivialization
    \begin{gather*}\Psi \colon D_1\setminus \{0\} \times \C \to D \setminus \{0\} \times \C,\\
    (z,w) \mapsto (z^m, w).
    \end{gather*}
    We will use the notation $x=z^m$ to distinguish the coordinate on $D$ from the coordinate on $D_1$. Note that there is a well-defined pushforward real line bundle $\Psi_*F$ equipped with a spin structure on it, since $F$ is a $\mu_m$-invariant subbundle. The claim of the proposition now follows from \cite[Lemma 11.7]{seidel2008fukaya} (or more specifically the arguments in \cite[Appendix A]{abouzaid2011cotangent}) once we establish
    \begin{equation}
           \operatorname{det}(D_{E_1}^{\mu_m}) \cong \operatorname{det}(D_{|E_1|}).
    \end{equation}

    This result then follows from Lemma~\ref{lemma: ker and coker}.
\end{proof}

\begin{proof}[Proof of (b)]  As in the second step above, we may first glue strip-like ends to obtain
    $$\operatorname{ind}(D_{[u]})+|\overrightarrow{\hat{x}}|+n-|\hat{x}_0|= \operatorname{ind}(D_{\overline{\mathscr{C}}}),$$
    where $D_{\overline{\mathscr{C}}}$ is the operator on the closed $\mathscr{C}$ with Lagrangian boundary condition extended using paths of spin Lagrangians associated with orbichords. Now, using \cite[Prop. 6.10]{choshin2016}, we obtain that
    $$\operatorname{ind}(D_{\overline{\mathscr{C}}})=\mu_{CW}(D_{\overline{\mathscr{C}}})-2\iota(\bm m).$$
    To compute $\mu_{CW}(D_{\overline{\mathscr{C}}})$ we recall that it can be computed as $\frac{1}{|G|} \mu(D_{\overline{\Sigma}})$ by \cite[Proposition 6.7]{choshin2016} where $\overline{\Sigma}$ is the compactification of $\Sigma$ covering $\overline{\mathscr{C}}$. Since over the boundary $\partial\Sigma$ the $G$-action is free it consists of $|G|$ circles and applying \cite[Lemma 11.12]{seidel2008fukaya} we conclude that $\mu_{CW}(D_{\overline{\mathscr{C}}})=n$.
\end{proof}

\begin{lemma} \label{lemma: ker and coker}
    Let $\pi \colon \Sigma \to D$ be a cyclic $\mu_m$-cover representing an orbidisk $\mathscr{D}$ and let $(E,F)$ be a bundle pair over $\Sigma$ with a $\mu_m$-action and a choice of trivialization of $E \cong \Sigma \times V$, where $V$ is a representation of $\mu_m$. If $(E^{\operatorname{orb}}, F^{\operatorname{orb}})$ is the orbifold bundle pair on $\mathscr{D}$ obtained via descent, then 
    \begin{gather}
        \operatorname{ker}((\overline{\partial}_{(E,F)})^{\mu_m}) \cong \operatorname{ker}(\overline{\partial}_{(|E^{\operatorname{orb}}|,F^{\operatorname{orb}} )}),\\
         \operatorname{coker}((\overline{\partial}_{(E,F)})^{\mu_m}) \cong \operatorname{coker}(\overline{\partial}_{(|E^{\operatorname{orb}}|,F^{\operatorname{orb}} )}),
    \end{gather}
    where $(|E^{\operatorname{orb}}|,F^{\operatorname{orb}} )$ is the desingularized bundle pair over $D$.
    
    Moreover, these isomorphisms are induced from a natural pullback map
    \begin{equation}
        \Psi^* \colon (|E^{\operatorname{orb}}|, F^{\operatorname{orb}}) \to (E,F).
    \end{equation}
    In particular, given any subspace $\mathcal{O} \subset \Omega^{0,1}_{F^{\operatorname{orb}}}(D, |E^{\operatorname{orb}}|)$ representing $\operatorname{coker}(\overline{\partial}_{(|E^{\operatorname{orb}}|,F^{\operatorname{orb}} )})$, its pullback $\Psi^*(\mathcal{O})$ represents the space $\operatorname{coker}((\overline{\partial}_{(E,F)})^{\mu_m})$.
\end{lemma}

\begin{proof}
    We denote by $\mathcal{O}(E,F)$ the sheaf of holomorphic sections of $E$ with boundary condition given by $F$ (we use the holomorphic structure on $E$ induced from the trivialization). There is also a sheaf $\mathcal{O}(E^{\operatorname{orb}}, F^{\operatorname{orb}})$ on $D$ of holomorphic sections of $E^{\operatorname{orb}}$ with boundary condition given by $F^{\operatorname{orb}}$ (notice that along $\partial \Sigma$ the action of $\mu_m$ is free, hence it is clear how to descend $F$ to $F^{\operatorname{orb}}$), where locally near the stacky points such holomorphic sections are defined as in \cite[Section 4]{Chen2004} via equivariant sections in the local orbifold charts. We note that in our global setup, we have
    \begin{equation}
              \mathcal{O}(E^{\operatorname{orb}}, F^{\operatorname{orb}}) \cong (\pi_*\mathcal{O}(E,F))^{\mu_m},
    \end{equation}
    which follows from the definition. The main observation of \cite[Proposition 4.2.2]{Chen2004} in our case with boundary is the isomorphism of sheaves
    \begin{equation}
              \mathcal{O}(E^{\operatorname{orb}}, F^{\operatorname{orb}}) \cong \mathcal{O}(|E^{\operatorname{orb}}|, F^{\operatorname{orb}}).
    \end{equation}
    The first claim of the lemma follows since these are just isomorphisms of cohomologies of sheaves $(\pi_*\mathcal{O}(E,F))^{\mu_m}$ and $\mathcal{O}(|E^{\operatorname{orb}}|, F^{\operatorname{orb}})$.

          Now, the pullback arises as follows: over a small neighborhood $U \subset D$ away from the stacky points, we have compatible trivializations $|E^{\operatorname{orb}}|_{U} \cong U \times \C^n$ and $E \cong \pi^{-1}(U) \times \C^n$, hence on such $U$ the pullback is defined as usual. 
          
          On a chart $U' \cong D^2$ near a ramification point $p$ where $E$ trivializes as $E|_{U'} \cong D^2 \times V$ with $\mu_{m(p)}$ acting via standard rotations on $D^2$, and $V=\operatorname{Res}^{\mu_m}_{\mu_{m(p)}}V$ is the restricted representation to the stabilizer $\mu_{m(p)}$ of $p$. We pick a basis $V \cong \C^n$ with action
          $$\zeta_{m(p)}(w_1, \dots, w_n)=(\zeta_{m(p)}^{j_1}w_1, \dots, \zeta_{m(p)}^{j_n}w_n).$$
Recall that $|E^{\operatorname{orb}}|_{\pi(U')}$ is trivialized via extension to $\pi(p)$ of the trivialization
\begin{gather}
    \Psi_{U'} \colon D^2 \setminus \{0\} \times \C^n \to D^2 \setminus \{0\} \times \C^n \nonumber \\
    (z, w_1, \dots, w_n) \mapsto (z^m, z^{-j_1}w_1, \dots, z^{-j_n}w_n)
\end{gather}
as defined in \cite[Section 4.2]{Chen2004}. The main observation here is that a smooth section of $|E^{\operatorname{orb}}|_{\pi(U')}$ pulls back to a smooth $\mu_{m(p)}$-equivariant section of $E|_U$. We notice, however, that the converse is not true: take $n=1$ and $j_1=1$ and section $\xi(z)=z^2\overline{z}$, which is equivariant with pushforward $\xi(x)=|x|^{2/m}$ in coordinate $x=z^m$.

Hence, the following commutative diagram for the standard fine resolutions of sheaves $(\pi_*\mathcal{O}(E,F))^{\mu_m}$ and $\mathcal{O}(|E^{\operatorname{orb}}|, F^{\operatorname{orb}})$ holds; see \cite[Section 3.4]{KL06}

\[\begin{tikzcd}
	0& \mathcal{O}(|E^{\operatorname{orb}}|, F^{\operatorname{orb}}) && \mathcal{A}^0(|E^{\operatorname{orb}}|, F^{\operatorname{orb}}) && \mathcal{A}^{0,1}(|E^{\operatorname{orb}}|) & 0 \\
	0 & (\pi_*\mathcal{O}(E,F))^{\mu_m} && (\pi_*\mathcal{A}^0(E, F))^{\mu_m} && (\pi_*\mathcal{A}^{0,1}(E))^{\mu_m}  & 0
	\arrow[from=1-1, to=1-2]
	\arrow[from=1-2, to=1-4]
	\arrow[from=1-2, to=2-2, "\cong"]
	\arrow[from=1-4, to=1-6]
	\arrow[hook, from=1-4, to=2-4, "\Psi^*"]
	\arrow[from=1-6, to=1-7]
	\arrow[hook, from=1-6, to=2-6, "\Psi^*"]
	\arrow[from=2-1, to=2-2]
	\arrow[from=2-2, to=2-4]
	\arrow[from=2-4, to=2-6]
	\arrow[from=2-6, to=2-7]
\end{tikzcd},\]
where each of the maps is naturally induced from the pullback map $\Psi^*$. Since $\Psi^*$ induces the isomorphisms between sheaves of holomorphic sections, it induces the isomorphisms on cohomology, hence the result follows.
       \end{proof}
\printbibliography

@book {Adem-Leida-Ruan,
    AUTHOR = {Adem, Alejandro and Leida, Johann and Ruan, Yongbin},
     TITLE = {Orbifolds and stringy topology},
    SERIES = {Cambridge Tracts in Mathematics},
    VOLUME = {171},
 PUBLISHER = {Cambridge University Press, Cambridge},
      YEAR = {2007},
     PAGES = {xii+149},
      ISBN = {978-0-521-87004-7; 0-521-87004-6},
   MRCLASS = {57R19 (19L64 55N35)},
  MRNUMBER = {2359514},
MRREVIEWER = {Yunfeng\ Jiang},
       DOI = {10.1017/CBO9780511543081},
       URL = {https://doi.org/10.1017/CBO9780511543081},
}

@article {BMR,
    AUTHOR = {Brou\'e, Michel and Malle, Gunter and Rouquier, Rapha\"el},
     TITLE = {Complex reflection groups, braid groups, {H}ecke algebras},
   JOURNAL = {J. Reine Angew. Math.},
  FJOURNAL = {Journal f\"ur die Reine und Angewandte Mathematik. [Crelle's
              Journal]},
    VOLUME = {500},
      YEAR = {1998},
     PAGES = {127--190},
      ISSN = {0075-4102,1435-5345},
   MRCLASS = {20F55 (20C05 20F36)},
  MRNUMBER = {1637497},
MRREVIEWER = {Arjeh\ M.\ Cohen},
}

@book{seidel2008fukaya,
    AUTHOR = {Seidel, Paul},
     TITLE = {Fukaya categories and {P}icard-{L}efschetz theory},
    SERIES = {Zurich Lectures in Advanced Mathematics},
 PUBLISHER = {European Mathematical Society (EMS), Z\"urich},
      YEAR = {2008},
     PAGES = {viii+326},
      ISBN = {978-3-03719-063-0},
   MRCLASS = {53D40 (16E45 32Q65 53D12)},
  MRNUMBER = {2441780},
MRREVIEWER = {Timothy\ Perutz},
       DOI = {10.4171/063},
       URL = {https://doi.org/10.4171/063},
}

@article{lipshitz2006cylindrical,
    AUTHOR = {Lipshitz, Robert},
     TITLE = {A cylindrical reformulation of {H}eegaard {F}loer homology},
   JOURNAL = {Geom. Topol.},
  FJOURNAL = {Geometry and Topology},
    VOLUME = {10},
      YEAR = {2006},
     PAGES = {955--1096},
      ISSN = {1465-3060},
   MRCLASS = {57R58 (57M27)},
  MRNUMBER = {2240908},
MRREVIEWER = {Stanislav Jabuka},
       DOI = {10.2140/gt.2006.10.955},
       URL = {https://doi.org/10.2140/gt.2006.10.955},
}

@article {abbondandolo2010floer,
    AUTHOR = {Abbondandolo, Alberto and Schwarz, Matthias},
     TITLE = {Floer homology of cotangent bundles and the loop product},
   JOURNAL = {Geom. Topol.},
  FJOURNAL = {Geometry \& Topology},
    VOLUME = {14},
      YEAR = {2010},
    NUMBER = {3},
     PAGES = {1569--1722},
      ISSN = {1465-3060},
   MRCLASS = {53D40 (55N45 55P50 57R58)},
  MRNUMBER = {2679580},
MRREVIEWER = {Janko Latschev},
       DOI = {10.2140/gt.2010.14.1569},
       URL = {https://doi.org/10.2140/gt.2010.14.1569},
}

@article {abbondandolo2006floer,
    AUTHOR = {Abbondandolo, Alberto and Schwarz, Matthias},
     TITLE = {On the {F}loer homology of cotangent bundles},
   JOURNAL = {Comm. Pure Appl. Math.},
  FJOURNAL = {Communications on Pure and Applied Mathematics},
    VOLUME = {59},
      YEAR = {2006},
    NUMBER = {2},
     PAGES = {254--316},
      ISSN = {0010-3640},
   MRCLASS = {53D40 (57R58)},
  MRNUMBER = {2190223},
MRREVIEWER = {Michael J. Usher},
       DOI = {10.1002/cpa.20090},
       URL = {https://doi.org/10.1002/cpa.20090},
}

@article {abouzaid2012wrapped,
    AUTHOR = {Abouzaid, Mohammed},
     TITLE = {On the wrapped {F}ukaya category and based loops},
   JOURNAL = {J. Symplectic Geom.},
  FJOURNAL = {The Journal of Symplectic Geometry},
    VOLUME = {10},
      YEAR = {2012},
    NUMBER = {1},
     PAGES = {27--79},
      ISSN = {1527-5256},
   MRCLASS = {53D37 (53D12)},
  MRNUMBER = {2904032},
MRREVIEWER = {Janko Latschev},
       URL = {http://projecteuclid.org/euclid.jsg/1332853049},
}

@article {morton2021dahas,
    AUTHOR = {Morton, Hugh and Samuelson, Peter},
     TITLE = {D{AHA}s and skein theory},
   JOURNAL = {Comm. Math. Phys.},
  FJOURNAL = {Communications in Mathematical Physics},
    VOLUME = {385},
      YEAR = {2021},
    NUMBER = {3},
     PAGES = {1655--1693},
      ISSN = {0010-3616},
   MRCLASS = {57K14 (17B37)},
  MRNUMBER = {4283999},
       DOI = {10.1007/s00220-021-04052-8},
       URL = {https://doi.org/10.1007/s00220-021-04052-8},
}

@article {abouzaid2010geometric,
    AUTHOR = {Abouzaid, Mohammed},
     TITLE = {A geometric criterion for generating the {F}ukaya category},
   JOURNAL = {Publ. Math. Inst. Hautes \'{E}tudes Sci.},
  FJOURNAL = {Publications Math\'{e}matiques. Institut de Hautes \'{E}tudes
              Scientifiques},
    NUMBER = {112},
      YEAR = {2010},
     PAGES = {191--240},
      ISSN = {0073-8301},
   MRCLASS = {53D37},
  MRNUMBER = {2737980},
MRREVIEWER = {Timothy Perutz},
       DOI = {10.1007/s10240-010-0028-5},
       URL = {https://doi.org/10.1007/s10240-010-0028-5},
}

@unpublished{CHT,
    title={Towards a definition of Khovanov homology for links in fibered $3$-manifolds},
    author={Vincent Colin and Ko Honda and Yin Tian},
    year={2025},
    eprint={2510.26164},
    archivePrefix={arXiv},
    primaryClass={math.SG}
}

@misc{ganatra2013thesis,
    title={Symplectic cohomology and duality for the wrapped Fukaya category},
    author={Sheel Ganatra},
    year={2013},
    eprint={1304.7312},
    archivePrefix={arXiv},
    primaryClass={math.SG}
}

@article {etingof2002symplectic,
    AUTHOR = {Etingof, Pavel and Ginzburg, Victor},
     TITLE = {Symplectic reflection algebras, {C}alogero-{M}oser space, and
              deformed {H}arish-{C}handra homomorphism},
   JOURNAL = {Invent. Math.},
  FJOURNAL = {Inventiones Mathematicae},
    VOLUME = {147},
      YEAR = {2002},
    NUMBER = {2},
     PAGES = {243--348},
      ISSN = {0020-9910},
   MRCLASS = {16S10 (14L30 17B66 17B80 20C08)},
  MRNUMBER = {1881922},
MRREVIEWER = {Alexander Braverman},
       DOI = {10.1007/s002220100171},
       URL = {https://doi.org/10.1007/s002220100171},
}

@article {EtingofPavel,
    AUTHOR = {Etingof, Pavel},
     TITLE = {Cherednik and {H}ecke algebras of varieties with a finite
              group action},
   JOURNAL = {Mosc. Math. J.},
  FJOURNAL = {Moscow Mathematical Journal},
    VOLUME = {17},
      YEAR = {2017},
    NUMBER = {4},
     PAGES = {635--666},
      ISSN = {1609-3321},
   MRCLASS = {20C08 (33D80)},
  MRNUMBER = {3734656},
MRREVIEWER = {Gwyn Bellamy},
       DOI = {10.17323/1609-4514-2016-16-4-635-666},
       URL = {https://doi.org/10.17323/1609-4514-2016-16-4-635-666},
}

@misc{abouzaid2015symplectic,
    AUTHOR = {Abouzaid, Mohammed},
     TITLE = {Symplectic cohomology and {V}iterbo's theorem},
 BOOKTITLE = {Free loop spaces in geometry and topology},
    SERIES = {IRMA Lect. Math. Theor. Phys.},
    VOLUME = {24},
     PAGES = {271--485},
 PUBLISHER = {Eur. Math. Soc., Z\"urich},
      YEAR = {2015},
      ISBN = {978-3-03719-153-8},
   MRCLASS = {57R17 (53D40 55P50 57R58 58E05)},
  MRNUMBER = {3444367},
MRREVIEWER = {Darko\ Milinkovi\'c},
}

@article{abouzaid2011cotangent,
    AUTHOR = {Abouzaid, Mohammed},
     TITLE = {A cotangent fibre generates the {F}ukaya category},
   JOURNAL = {Adv. Math.},
  FJOURNAL = {Advances in Mathematics},
    VOLUME = {228},
      YEAR = {2011},
    NUMBER = {2},
     PAGES = {894--939},
      ISSN = {0001-8708,1090-2082},
   MRCLASS = {53D37},
  MRNUMBER = {2822213},
MRREVIEWER = {Michael\ J.\ Usher},
       DOI = {10.1016/j.aim.2011.06.007},
       URL = {https://doi.org/10.1016/j.aim.2011.06.007},
}

@incollection {Chen2002,
    AUTHOR = {Chen, Weimin and Ruan, Yongbin},
     TITLE = {Orbifold {G}romov-{W}itten theory},
 BOOKTITLE = {Orbifolds in mathematics and physics ({M}adison, {WI}, 2001)},
    SERIES = {Contemp. Math.},
    VOLUME = {310},
     PAGES = {25--85},
 PUBLISHER = {Amer. Math. Soc., Providence, RI},
      YEAR = {2002},
      ISBN = {0-8218-2990-4},
   MRCLASS = {53D45 (14N35)},
  MRNUMBER = {1950941},
MRREVIEWER = {Ignasi\ Mundet-Riera},
       DOI = {10.1090/conm/310/05398},
       URL = {https://doi.org/10.1090/conm/310/05398},
}

@article {Chen2004,
    AUTHOR = {Chen, Weimin and Ruan, Yongbin},
     TITLE = {A new cohomology theory of orbifold},
   JOURNAL = {Comm. Math. Phys.},
  FJOURNAL = {Communications in Mathematical Physics},
    VOLUME = {248},
      YEAR = {2004},
    NUMBER = {1},
     PAGES = {1--31},
      ISSN = {0010-3616,1432-0916},
   MRCLASS = {57R19 (53D45)},
  MRNUMBER = {2104605},
MRREVIEWER = {Paolo\ Lisca},
       DOI = {10.1007/s00220-004-1089-4},
       URL = {https://doi.org/10.1007/s00220-004-1089-4},
}

@article {COW2024twisted,
    AUTHOR = {Chen, Bohui and Ono, Kaoru and Wang, Bai-Ling},
     TITLE = {Twisted sectors for {L}agrangian {F}loer theory on symplectic
              orbifolds},
   JOURNAL = {SIGMA Symmetry Integrability Geom. Methods Appl.},
  FJOURNAL = {SIGMA. Symmetry, Integrability and Geometry. Methods and
              Applications},
    VOLUME = {20},
      YEAR = {2024},
     PAGES = {Paper No. 011, 14},
      ISSN = {1815-0659},
   MRCLASS = {53D40 (53D37 57R18)},
  MRNUMBER = {4697924},
MRREVIEWER = {Alexander\ Fel\cprime shtyn},
       DOI = {10.3842/SIGMA.2024.011},
       URL = {https://doi.org/10.3842/SIGMA.2024.011},
}

@article {Cho2014,
    AUTHOR = {Cho, Cheol-Hyun and Poddar, Mainak},
     TITLE = {Holomorphic orbi-discs and {L}agrangian {F}loer cohomology of
              symplectic toric orbifolds},
   JOURNAL = {J. Differential Geom.},
  FJOURNAL = {Journal of Differential Geometry},
    VOLUME = {98},
      YEAR = {2014},
    NUMBER = {1},
     PAGES = {21--116},
      ISSN = {0022-040X,1945-743X},
   MRCLASS = {53D40 (32Q65 57R18)},
  MRNUMBER = {3263515},
MRREVIEWER = {Saibal\ Ganguli},
       URL = {http://projecteuclid.org/euclid.jdg/1406137695},
}

@incollection {ACV2003,
    AUTHOR = {Abramovich, Dan and Corti, Alessio and Vistoli, Angelo},
     TITLE = {Twisted bundles and admissible covers},
      NOTE = {Special issue in honor of Steven L. Kleiman},
   JOURNAL = {Comm. Algebra},
  FJOURNAL = {Communications in Algebra},
    VOLUME = {31},
      YEAR = {2003},
    NUMBER = {8},
     PAGES = {3547--3618},
      ISSN = {0092-7872,1532-4125},
   MRCLASS = {14H10 (14A20 14H30)},
  MRNUMBER = {2007376},
MRREVIEWER = {Andrew\ Kresch},
       DOI = {10.1081/AGB-120022434},
       URL = {https://doi.org/10.1081/AGB-120022434},
}

@book {mcduffsalamon2004,
    AUTHOR = {McDuff, Dusa and Salamon, Dietmar},
     TITLE = {{$J$}-holomorphic curves and symplectic topology},
    SERIES = {American Mathematical Society Colloquium Publications},
    VOLUME = {52},
 PUBLISHER = {American Mathematical Society, Providence, RI},
      YEAR = {2004},
     PAGES = {xii+669},
      ISBN = {0-8218-3485-1},
   MRCLASS = {53D45 (32Q65 53D40 57R17)},
  MRNUMBER = {2045629},
MRREVIEWER = {Ignasi\ Mundet-Riera},
       DOI = {10.1090/coll/052},
       URL = {https://doi.org/10.1090/coll/052},
}

@article {CieliebakMohnke2007,
    AUTHOR = {Cieliebak, Kai and Mohnke, Klaus},
     TITLE = {Symplectic hypersurfaces and transversality in
              {G}romov-{W}itten theory},
   JOURNAL = {J. Symplectic Geom.},
  FJOURNAL = {The Journal of Symplectic Geometry},
    VOLUME = {5},
      YEAR = {2007},
    NUMBER = {3},
     PAGES = {281--356},
      ISSN = {1527-5256,1540-2347},
   MRCLASS = {53D45},
  MRNUMBER = {2399678},
MRREVIEWER = {Michael\ J.\ Usher},
       DOI = {10.4310/jsg.2007.v5.n3.a2},
       URL = {https://doi.org/10.4310/jsg.2007.v5.n3.a2},
}

@book {KMS1993,
    AUTHOR = {Kol\'ar, Ivan and Michor, Peter W. and Slov\'ak, Jan},
     TITLE = {Natural operations in differential geometry},
 PUBLISHER = {Springer-Verlag, Berlin},
      YEAR = {1993},
     PAGES = {vi+434},
      ISBN = {3-540-56235-4},
   MRCLASS = {58A20 (53A55 53B05 53C05)},
  MRNUMBER = {1202431},
MRREVIEWER = {Jaime\ Mu\~noz Masqu\'e},
       DOI = {10.1007/978-3-662-02950-3},
       URL = {https://doi.org/10.1007/978-3-662-02950-3},
}

@article {Zehmisch2015,
    AUTHOR = {Zehmisch, Kai},
     TITLE = {Holomorphic jets in symplectic manifolds},
   JOURNAL = {J. Fixed Point Theory Appl.},
  FJOURNAL = {Journal of Fixed Point Theory and Applications},
    VOLUME = {17},
      YEAR = {2015},
    NUMBER = {2},
     PAGES = {379--402},
      ISSN = {1661-7738,1661-7746},
   MRCLASS = {53D35 (32Q65 53D45 58A20)},
  MRNUMBER = {3397123},
MRREVIEWER = {Josef\ G.\ Dorfmeister},
       DOI = {10.1007/s11784-014-0178-z},
       URL = {https://doi.org/10.1007/s11784-014-0178-z},
}

@article {Galeotti2022,
    AUTHOR = {Galeotti, Mattia},
     TITLE = {Moduli of {$G$}-covers of curves: geometry and singularities},
   JOURNAL = {Ann. Inst. Fourier (Grenoble)},
  FJOURNAL = {Universit\'e{} de Grenoble. Annales de l'Institut Fourier},
    VOLUME = {72},
      YEAR = {2022},
    NUMBER = {6},
     PAGES = {2191--2240},
      ISSN = {0373-0956,1777-5310},
   MRCLASS = {14H10 (14D20 14D23 14E05 14H20 14H60)},
  MRNUMBER = {4500355},
MRREVIEWER = {Luca\ Schaffler},
       DOI = {10.5802/aif.3503},
       URL = {https://doi.org/10.5802/aif.3503},
}

@article {Seidel2000graded,
    AUTHOR = {Seidel, Paul},
     TITLE = {Graded {L}agrangian submanifolds},
   JOURNAL = {Bull. Soc. Math. France},
  FJOURNAL = {Bulletin de la Soci\'et\'e{} Math\'ematique de France},
    VOLUME = {128},
      YEAR = {2000},
    NUMBER = {1},
     PAGES = {103--149},
      ISSN = {0037-9484,2102-622X},
   MRCLASS = {53D40 (53D12 57R58)},
  MRNUMBER = {1765826},
MRREVIEWER = {David\ E.\ Hurtubise},
       URL = {http://www.numdam.org/item?id=BSMF_2000__128_1_103_0},
}

@article {mclaughlin1992orient,
    AUTHOR = {McLaughlin, Dennis A.},
     TITLE = {Orientation and string structures on loop space},
   JOURNAL = {Pacific J. Math.},
  FJOURNAL = {Pacific Journal of Mathematics},
    VOLUME = {155},
      YEAR = {1992},
    NUMBER = {1},
     PAGES = {143--156},
      ISSN = {0030-8730,1945-5844},
   MRCLASS = {57R20 (58B25 81T30)},
  MRNUMBER = {1174481},
MRREVIEWER = {Randy\ A.\ Baadhio},
       URL = {http://projecteuclid.org/euclid.pjm/1102635473},
}

@misc{mak2025orbifoldhamiltonianfloertheory,
      title={Orbifold Hamiltonian Floer theory for global quotients}, 
      author={Cheuk Yu Mak and Sobhan Seyfaddini and Ivan Smith},
      year={2025},
      eprint={2502.11290},
      archivePrefix={arXiv},
      primaryClass={math.SG},
}

@article{cho2013finitegroupactionslagrangian,
    AUTHOR = {Cho, Cheol-Hyun and Hong, Hansol},
     TITLE = {Finite group actions on {L}agrangian {F}loer theory},
   JOURNAL = {J. Symplectic Geom.},
  FJOURNAL = {The Journal of Symplectic Geometry},
    VOLUME = {15},
      YEAR = {2017},
    NUMBER = {2},
     PAGES = {307--420},
      ISSN = {1527-5256,1540-2347},
   MRCLASS = {53D40},
  MRNUMBER = {3683788},
MRREVIEWER = {Christopher\ T.\ Woodward},
       DOI = {10.4310/JSG.2017.v15.n2.a1},
       URL = {https://doi.org/10.4310/JSG.2017.v15.n2.a1},
}

@article{bao2021,
    AUTHOR = {Bao, Erkao and Honda, Ko},
     TITLE = {Equivariant {L}agrangian {F}loer cohomology via semi-global
              {K}uranishi structures},
   JOURNAL = {Algebr. Geom. Topol.},
  FJOURNAL = {Algebraic \& Geometric Topology},
    VOLUME = {21},
      YEAR = {2021},
    NUMBER = {4},
     PAGES = {1677--1722},
      ISSN = {1472-2747,1472-2739},
   MRCLASS = {53D40 (53D10 57M50)},
  MRNUMBER = {4302482},
       DOI = {10.2140/agt.2021.21.1677},
       URL = {https://doi.org/10.2140/agt.2021.21.1677},
}

@article{doan2023,
    AUTHOR = {Doan, Aleksander and Walpuski, Thomas},
     TITLE = {Counting embedded curves in symplectic 6-manifolds},
   JOURNAL = {Comment. Math. Helv.},
  FJOURNAL = {Commentarii Mathematici Helvetici. A Journal of the Swiss
              Mathematical Society},
    VOLUME = {98},
      YEAR = {2023},
    NUMBER = {4},
     PAGES = {693--769},
      ISSN = {0010-2571,1420-8946},
   MRCLASS = {53D05 (14N35 53C15)},
  MRNUMBER = {4680501},
       DOI = {10.4171/cmh/556},
       URL = {https://doi.org/10.4171/cmh/556},
}

@misc{CDW2020,
      title={Fukaya category for Landau-Ginzburg orbifolds}, 
      author={Cho, Cheol-Hyun and Dongwook Choa and Wonbo Jeong},
      year={2020},
      eprint={2010.09198},
      archivePrefix={arXiv},
      primaryClass={math.SG},
}

@article {siegel2021,
    AUTHOR = {Siegel, Kyler},
     TITLE = {Squared {D}ehn twists and deformed symplectic invariants},
   JOURNAL = {J. Symplectic Geom.},
  FJOURNAL = {The Journal of Symplectic Geometry},
    VOLUME = {19},
      YEAR = {2021},
    NUMBER = {5},
     PAGES = {1189--1280},
      ISSN = {1527-5256,1540-2347},
   MRCLASS = {53D12 (53D37)},
  MRNUMBER = {4436690},
       DOI = {10.4310/jsg.2021.v19.n5.a5},
       URL = {https://doi.org/10.4310/jsg.2021.v19.n5.a5},
}

@article {sheridan2015CYhypersurfaces,
    AUTHOR = {Sheridan, Nick},
     TITLE = {Homological mirror symmetry for {C}alabi-{Y}au hypersurfaces
              in projective space},
   JOURNAL = {Invent. Math.},
  FJOURNAL = {Inventiones Mathematicae},
    VOLUME = {199},
      YEAR = {2015},
    NUMBER = {1},
     PAGES = {1--186},
      ISSN = {0020-9910,1432-1297},
   MRCLASS = {14F05 (14J32 14J33 18E30)},
  MRNUMBER = {3294958},
MRREVIEWER = {Sergiy\ Koshkin},
       DOI = {10.1007/s00222-014-0507-2},
       URL = {https://doi.org/10.1007/s00222-014-0507-2},
}

@article {sheridan2020versality,
    AUTHOR = {Sheridan, Nick},
     TITLE = {Versality of the relative {F}ukaya category},
   JOURNAL = {Geom. Topol.},
  FJOURNAL = {Geometry \& Topology},
    VOLUME = {24},
      YEAR = {2020},
    NUMBER = {2},
     PAGES = {747--884},
      ISSN = {1465-3060,1364-0380},
   MRCLASS = {53D37},
  MRNUMBER = {4153652},
MRREVIEWER = {Nathaniel\ Bottman},
       DOI = {10.2140/gt.2020.24.747},
       URL = {https://doi.org/10.2140/gt.2020.24.747},
}

@article {perutzsheridan2023relative,
    AUTHOR = {Perutz, Timothy and Sheridan, Nick},
     TITLE = {Constructing the relative {F}ukaya category},
   JOURNAL = {J. Symplectic Geom.},
  FJOURNAL = {The Journal of Symplectic Geometry},
    VOLUME = {21},
      YEAR = {2023},
    NUMBER = {5},
     PAGES = {997--1076},
      ISSN = {1527-5256,1540-2347},
   MRCLASS = {53D37},
  MRNUMBER = {4754914},
MRREVIEWER = {Johan\ Asplund},
       DOI = {10.4310/jsg.2023.v21.n5.a4},
       URL = {https://doi.org/10.4310/jsg.2023.v21.n5.a4},
}

@misc{sheridan20225bigrelative,
      title={Constructing the big relative Fukaya category, and its open--closed maps}, 
      author={Sheridan, Nick},
      year={2025},
      eprint={2511.01656},
      archivePrefix={arXiv},
      primaryClass={math.SG},
      url={https://arxiv.org/abs/2511.01656}, 
}

@article {tamborini2022shimura,
    AUTHOR = {Tamborini, Carolina},
     TITLE = {Symmetric spaces uniformizing {S}himura varieties in the
              {T}orelli locus},
   JOURNAL = {Ann. Mat. Pura Appl. (4)},
  FJOURNAL = {Annali di Matematica Pura ed Applicata. Series IV},
    VOLUME = {201},
      YEAR = {2022},
    NUMBER = {5},
     PAGES = {2101--2119},
      ISSN = {0373-3114,1618-1891},
   MRCLASS = {14G35 (14H15 14H40)},
  MRNUMBER = {4491459},
MRREVIEWER = {Yifei\ Zhu},
       DOI = {10.1007/s10231-022-01193-y},
       URL = {https://doi.org/10.1007/s10231-022-01193-y},
}

@article {CW1934formula,
    AUTHOR = {Chevalley, C. and Weil, A. and Hecke, E.},
     TITLE = {\"Uber das verhalten der integrale 1. gattung bei
              automorphismen des funktionenk\"orpers},
   JOURNAL = {Abh. Math. Sem. Univ. Hamburg},
  FJOURNAL = {Abhandlungen aus dem Mathematischen Seminar der Universit\"at
              Hamburg},
    VOLUME = {10},
      YEAR = {1934},
    NUMBER = {1},
     PAGES = {358--361},
      ISSN = {0025-5858,1865-8784},
   MRCLASS = {99-04},
  MRNUMBER = {3069638},
       DOI = {10.1007/BF02940687},
       URL = {https://doi.org/10.1007/BF02940687},
}

@article {arapura2022residues,
    AUTHOR = {Arapura, Donu},
     TITLE = {Residues of connections and the {C}hevalley-{W}eil formula for
              curves},
   JOURNAL = {Enseign. Math.},
  FJOURNAL = {L'Enseignement Math\'ematique},
    VOLUME = {68},
      YEAR = {2022},
    NUMBER = {3-4},
     PAGES = {291--306},
      ISSN = {0013-8584,2309-4672},
   MRCLASS = {14H55 (14H37)},
  MRNUMBER = {4452424},
MRREVIEWER = {Jos\'e\ Javier\ Etayo},
       DOI = {10.4171/lem/1030},
       URL = {https://doi.org/10.4171/lem/1030},
}

@incollection {KL06,
    AUTHOR = {Katz, Sheldon and Liu, Chiu-Chu Melissa},
     TITLE = {Enumerative geometry of stable maps with {L}agrangian boundary
              conditions and multiple covers of the disc},
 BOOKTITLE = {The interaction of finite-type and {G}romov-{W}itten
              invariants ({BIRS} 2003)},
    SERIES = {Geom. Topol. Monogr.},
    VOLUME = {8},
     PAGES = {1--47},
 PUBLISHER = {Geom. Topol. Publ., Coventry},
      YEAR = {2006},
   MRCLASS = {53D45 (14N35 32G81)},
  MRNUMBER = {2402818},
       DOI = {10.2140/gtm.2006.8.1},
       URL = {https://doi.org/10.2140/gtm.2006.8.1},
}

@article{seidelkhovanov2002,
    AUTHOR = {Khovanov, Mikhail and Seidel, Paul},
     TITLE = {Quivers, {F}loer cohomology, and braid group actions},
   JOURNAL = {J. Amer. Math. Soc.},
  FJOURNAL = {Journal of the American Mathematical Society},
    VOLUME = {15},
      YEAR = {2002},
    NUMBER = {1},
     PAGES = {203--271},
      ISSN = {0894-0347,1088-6834},
   MRCLASS = {53D40 (18E30 20F36 57R58)},
  MRNUMBER = {1862802},
MRREVIEWER = {Alexander\ E.\ Polishchuk},
       DOI = {10.1090/S0894-0347-01-00374-5},
       URL = {https://doi.org/10.1090/S0894-0347-01-00374-5},
}

@article {pardonvfc,
    AUTHOR = {Pardon, John},
     TITLE = {An algebraic approach to virtual fundamental cycles on moduli
              spaces of pseudo-holomorphic curves},
   JOURNAL = {Geom. Topol.},
  FJOURNAL = {Geometry \& Topology},
    VOLUME = {20},
      YEAR = {2016},
    NUMBER = {2},
     PAGES = {779--1034},
      ISSN = {1465-3060,1364-0380},
   MRCLASS = {53D35 (37J10 53D37 53D40 53D42 53D45 54B40 57R17)},
  MRNUMBER = {3493097},
MRREVIEWER = {Sonja\ Hohloch},
       DOI = {10.2140/gt.2016.20.779},
       URL = {https://doi.org/10.2140/gt.2016.20.779},
}

@article {swaminathan2021relcinfty,
    AUTHOR = {Swaminathan, Mohan},
     TITLE = {Rel-{$C^\infty$} structures on {G}romov-{W}itten moduli
              spaces},
   JOURNAL = {J. Symplectic Geom.},
  FJOURNAL = {The Journal of Symplectic Geometry},
    VOLUME = {19},
      YEAR = {2021},
    NUMBER = {2},
     PAGES = {413--473},
      ISSN = {1527-5256,1540-2347},
   MRCLASS = {53D30 (53D45)},
  MRNUMBER = {4325409},
MRREVIEWER = {Li\ Sheng},
       DOI = {10.4310/JSG.2021.v19.n2.a4},
       URL = {https://doi.org/10.4310/JSG.2021.v19.n2.a4},
}

@misc{cant2022thesis,
      title={A dimension formula for relative symplectic field theory. Stanford University PhD Thesis.}, 
      author={Cant, Dylan},
      year={2022},
      url={https://dylancant.ca/thesis.pdf}, 
}

@misc{schmitt2023thesis,
      title={On $\mathbb{Q}$-factorial terminalizations of symplectic linear quotient singularities. Dissertation, Kaiserslautern-Landau, Rheinland-Pfälzische Technische Universität Kaiserslautern-Landau}, 
      author={Schmitt, Johannes},
      year={2023},
      url={https://kluedo.ub.rptu.de/frontdoor/index/index/docId/7351}, 
}

@incollection {salamon1997,
    AUTHOR = {Salamon, Dietmar},
     TITLE = {Lectures on {F}loer homology},
 BOOKTITLE = {Symplectic geometry and topology ({P}ark {C}ity, {UT}, 1997)},
    SERIES = {IAS/Park City Math. Ser.},
    VOLUME = {7},
     PAGES = {143--229},
 PUBLISHER = {Amer. Math. Soc., Providence, RI},
      YEAR = {1999},
      ISBN = {0-8218-0838-9},
   MRCLASS = {53D40 (37J45 53D45 57R17 57R58)},
  MRNUMBER = {1702944},
MRREVIEWER = {David\ E.\ Hurtubise},
       DOI = {10.1016/S0165-2427(99)00127-0},
       URL = {https://doi.org/10.1016/S0165-2427(99)00127-0},
}

@misc{Wendl2010,
    author = {Chris Wendl},
    title = {Lectures on {Holomorphic} {Curves} in {Symplectic} and {Contact} {Geometry}},
    year={2010},
    eprint={1011.1690},
    archivePrefix={arXiv},
    primaryClass={math.SG},
}

@article {Liu2020moduli,
    AUTHOR = {Liu, C.-C. Melissa},
     TITLE = {Moduli of {$J$}-holomorphic curves with {L}agrangian boundary
              conditions and open {G}romov-{W}itten invariants for an
              {$S^1$}-equivariant pair},
   JOURNAL = {J. Iran. Math. Soc.},
  FJOURNAL = {Journal of the Iranian Mathematical Society},
    VOLUME = {1},
      YEAR = {2020},
    NUMBER = {1},
     PAGES = {5--95},
      ISSN = {2717-1612},
   MRCLASS = {53D45 (14N35)},
  MRNUMBER = {4528678},
       DOI = {10.30504/jims.2020.104185},
       URL = {https://doi.org/10.30504/jims.2020.104185},
}

@misc{henriques2005thesis,
      title={Orbispaces, Ph.D. thesis, Massachusetts Institute of Technology.}, 
      author={Henriques, Andre Gil},
      year={2005},
      pages = {1--130}
}

@article {solomontukachinsky2024forms,
    AUTHOR = {Solomon, Jake P. and Tukachinsky, Sara B.},
     TITLE = {Differential forms on orbifolds with corners},
   JOURNAL = {J. Topol. Anal.},
  FJOURNAL = {Journal of Topology and Analysis},
    VOLUME = {16},
      YEAR = {2024},
    NUMBER = {4},
     PAGES = {561--615},
      ISSN = {1793-5253,1793-7167},
   MRCLASS = {58A10 (18E35 20L05 58A25)},
  MRNUMBER = {4782292},
MRREVIEWER = {Paulo\ Carrillo Rouse},
       DOI = {10.1142/S1793525323500048},
       URL = {https://doi.org/10.1142/S1793525323500048},
}

@article {Lerman2010stacks,
    AUTHOR = {Lerman, Eugene},
     TITLE = {Orbifolds as stacks?},
   JOURNAL = {Enseign. Math. (2)},
  FJOURNAL = {L'Enseignement Math\'ematique. Revue Internationale. 2e
              S\'erie},
    VOLUME = {56},
      YEAR = {2010},
    NUMBER = {3-4},
     PAGES = {315--363},
      ISSN = {0013-8584},
   MRCLASS = {18D05 (22A22)},
  MRNUMBER = {2778793},
MRREVIEWER = {Chenchang\ Zhu},
       DOI = {10.4171/LEM/56-3-4},
       URL = {https://doi.org/10.4171/LEM/56-3-4},
}

@article {pardon2022enough,
    AUTHOR = {Pardon, John},
     TITLE = {Enough vector bundles on orbispaces},
   JOURNAL = {Compos. Math.},
  FJOURNAL = {Compositio Mathematica},
    VOLUME = {158},
      YEAR = {2022},
    NUMBER = {11},
     PAGES = {2046--2081},
      ISSN = {0010-437X,1570-5846},
   MRCLASS = {57R18 (19L10 19L47 55N32 55P91 55R91)},
  MRNUMBER = {4515091},
MRREVIEWER = {Thomas\ O.\ Rot},
       DOI = {10.1112/s0010437x22007783},
       URL = {https://doi.org/10.1112/s0010437x22007783},
}

@book {moerdijkmrcun2003groupoids,
    AUTHOR = {Moerdijk, I. and Mr\v{c}un, J.},
     TITLE = {Introduction to foliations and {L}ie groupoids},
    SERIES = {Cambridge Studies in Advanced Mathematics},
    VOLUME = {91},
 PUBLISHER = {Cambridge University Press, Cambridge},
      YEAR = {2003},
     PAGES = {x+173},
      ISBN = {0-521-83197-0},
   MRCLASS = {58H05 (17B99 57R30)},
  MRNUMBER = {2012261},
MRREVIEWER = {Jan\ Kubarski},
       DOI = {10.1017/CBO9780511615450},
       URL = {https://doi.org/10.1017/CBO9780511615450},
}

@misc{metzler2003smoothstacks,
  title={Topological and Smooth Stacks},
  author={Metzler, David},
  eprint={0306176},
  archivePrefix={arXiv},
  year={2003},
  primaryClass={math.DG},
}

@misc{joyce2017kuranishi1,
  title={Kuranishi spaces and Symplectic Geometry, Volume I},
  author={Joyce, Dominic},
  URL ={https://people.maths.ox.ac.uk/joyce/kubookv1.pdf},
  year={2017}
}

@misc{joyce2017kuranishi2,
  title={Kuranishi spaces and Symplectic Geometry, Volume II},
  author={Joyce, Dominic},
  URL ={https://people.maths.ox.ac.uk/joyce/kubookv2.pdf},
  year={2017}
}

@article {behrendxu2011,
    AUTHOR = {Behrend, Kai and Xu, Ping},
     TITLE = {Differentiable stacks and gerbes},
   JOURNAL = {J. Symplectic Geom.},
  FJOURNAL = {The Journal of Symplectic Geometry},
    VOLUME = {9},
      YEAR = {2011},
    NUMBER = {3},
     PAGES = {285--341},
      ISSN = {1527-5256,1540-2347},
   MRCLASS = {53C08 (19L50 22A22 53D50 58H05)},
  MRNUMBER = {2817778},
MRREVIEWER = {Chenchang\ Zhu},
       DOI = {10.4310/jsg.2011.v9.n3.a2},
       URL = {https://doi.org/10.4310/jsg.2011.v9.n3.a2},
}

@misc{zernik2017moduli,
  title={Moduli of open stable maps to a homogeneous space},
  author={Zernik, Amitai Netser},
  eprint={1709.07402},
  archivePrefix={arXiv},
  year={2017},
  primaryClass={math.SG},
}

@misc{pomerleanoseidel2024,
  title={Symplectic cohomology relative to a smooth anticanonical divisor},
  author={Pomerleano, Daniel and Paul Seidel},
  eprint={2408.09039},
  archivePrefix={arXiv},
  year={2024},
  primaryClass={math.SG},
}

@misc{gironellazhou2021exact,
  title={Exact orbifold fillings of contact manifolds.},
  author={Gironella, Fabio and Zhengyi Zhou},
  eprint={2108.12247},
  archivePrefix={arXiv},
  year={2021},
  primaryClass={math.SG},
}

@book{taams2025,
    AUTHOR = {Taams, Lisanne},
     TITLE = {Sheaves on stacky curves},
    SERIES = {Radboud Dissertation Series},
 PUBLISHER = {Radboud University Press},
   ADDRESS = {Nijmegen},
      YEAR = {2025},
      ISBN = {978-94-6515-168-7},
       DOI = {10.54195/9789465151687},
       URL = {https://doi.org/10.54195/9789465151687},
}

@incollection {Haefliger1990orbi,
    AUTHOR = {Haefliger, Andr\'e},
     TITLE = {Orbi-espaces},
 BOOKTITLE = {Sur les groupes hyperboliques d'apr\`es {M}ikhael {G}romov
              ({B}ern, 1988)},
    SERIES = {Progr. Math.},
    VOLUME = {83},
     PAGES = {203--213},
 PUBLISHER = {Birkh\"auser Boston, Boston, MA},
      YEAR = {1990},
      ISBN = {0-8176-3508-4},
   MRCLASS = {57M10 (57M05)},
  MRNUMBER = {1086659},
       DOI = {10.1007/978-1-4684-9167-8\_11},
       URL = {https://doi.org/10.1007/978-1-4684-9167-8_11},
}

@article {angelcolman2023pathgroupoids,
    AUTHOR = {\'Angel, Andr\'es and Colman, Hellen},
     TITLE = {Free and based path groupoids},
   JOURNAL = {Algebr. Geom. Topol.},
  FJOURNAL = {Algebraic \& Geometric Topology},
    VOLUME = {23},
      YEAR = {2023},
    NUMBER = {5},
     PAGES = {1959--2008},
      ISSN = {1472-2747,1472-2739},
   MRCLASS = {55P35 (18B40 55R91 58D19 58E40)},
  MRNUMBER = {4621422},
MRREVIEWER = {Timothy\ Porter},
       DOI = {10.2140/agt.2023.23.1959},
       URL = {https://doi.org/10.2140/agt.2023.23.1959},
}

@misc{baixu2022transversality,
  title={A new transversality condition on orbifolds and integer-valued Gromov-Witten type invariants.},
  author={Bai, Shaoyun and Guangbo Xu},
  eprint={2201.02688v3},
  archivePrefix={arXiv},
  year={2022},
  primaryClass={math.SG},
}

@article {ekholmshende2025ghost,
    AUTHOR = {Ekholm, Tobias and Shende, Vivek},
     TITLE = {Ghost bubble censorship},
   JOURNAL = {Comm. Anal. Geom.},
  FJOURNAL = {Communications in Analysis and Geometry},
    VOLUME = {33},
      YEAR = {2025},
    NUMBER = {8},
     PAGES = {1815--1863},
      ISSN = {1019-8385,1944-9992},
   MRCLASS = {32Q65 (53D30)},
  MRNUMBER = {5039054},
       DOI = {10.4310/cag.260215235638},
       URL = {https://doi.org/10.4310/cag.260215235638},
}

@article{rezaeeswaminathan2025construction,
    AUTHOR = {Rezaee, Fatemeh and Swaminathan, Mohan},
     TITLE = {Constructing smoothings of stable maps},
   JOURNAL = {Adv. Math.},
  FJOURNAL = {Advances in Mathematics},
    VOLUME = {467},
      YEAR = {2025},
     PAGES = {Paper No. 110188, 66},
      ISSN = {0001-8708,1090-2082},
   MRCLASS = {14D20 (14H10 14N35 32Q65 53D45)},
  MRNUMBER = {4875370},
MRREVIEWER = {Amin\ Gholampour},
       DOI = {10.1016/j.aim.2025.110188},
       URL = {https://doi.org/10.1016/j.aim.2025.110188},
}

@article {rezaeeswaminathan2024obstruction,
    AUTHOR = {Rezaee, Fatemeh and Swaminathan, Mohan},
     TITLE = {An obstruction to smoothing stable maps},
   JOURNAL = {Selecta Math. (N.S.)},
  FJOURNAL = {Selecta Mathematica. New Series},
    VOLUME = {32},
      YEAR = {2026},
    NUMBER = {3},
     PAGES = {Paper No. 53},
      ISSN = {1022-1824,1420-9020},
   MRCLASS = {14D20 (14N35 32G15)},
  MRNUMBER = {5083451},
       DOI = {10.1007/s00029-026-01160-y},
       URL = {https://doi.org/10.1007/s00029-026-01160-y},
}

@article {ionel1998genus1,
    AUTHOR = {Ionel, Eleny},
     TITLE = {Genus {$1$} enumerative invariants in {$\bold P^n$} with fixed
              {$j$} invariant},
   JOURNAL = {Duke Math. J.},
  FJOURNAL = {Duke Mathematical Journal},
    VOLUME = {94},
      YEAR = {1998},
    NUMBER = {2},
     PAGES = {279--324},
      ISSN = {0012-7094,1547-7398},
   MRCLASS = {53D45 (14N10 14N35)},
  MRNUMBER = {1638587},
MRREVIEWER = {Andreas\ Gathmann},
       DOI = {10.1215/S0012-7094-98-09414-5},
       URL = {https://doi.org/10.1215/S0012-7094-98-09414-5},
}

@article {zinger2009sharp,
    AUTHOR = {Zinger, Aleksey},
     TITLE = {A sharp compactness theorem for genus-one pseudo-holomorphic
              maps},
   JOURNAL = {Geom. Topol.},
  FJOURNAL = {Geometry \& Topology},
    VOLUME = {13},
      YEAR = {2009},
    NUMBER = {5},
     PAGES = {2427--2522},
      ISSN = {1465-3060,1364-0380},
   MRCLASS = {53D45 (32Q65)},
  MRNUMBER = {2529940},
MRREVIEWER = {Hsian-Hua\ Tseng},
       DOI = {10.2140/gt.2009.13.2427},
       URL = {https://doi.org/10.2140/gt.2009.13.2427},
}

@misc{niu2016thesis,
      title={Refined Convergence for Genus-Two Pseudo-Holomorphic Maps. ProQuest LLC, Ann Arbor,
MI, 2016. Thesis (Ph.D.)–State University of New York at Stony Brook.}, 
      author={Niu, Jingchen},
      year={2016},
}

@article {zinger2009reduced,
    AUTHOR = {Zinger, Aleksey},
     TITLE = {The reduced genus 1 {G}romov-{W}itten invariants of
              {C}alabi-{Y}au hypersurfaces},
   JOURNAL = {J. Amer. Math. Soc.},
  FJOURNAL = {Journal of the American Mathematical Society},
    VOLUME = {22},
      YEAR = {2009},
    NUMBER = {3},
     PAGES = {691--737},
      ISSN = {0894-0347,1088-6834},
   MRCLASS = {14N35 (14J32 14J33 33C90 53D45)},
  MRNUMBER = {2505298},
MRREVIEWER = {Hsian-Hua\ Tseng},
       DOI = {10.1090/S0894-0347-08-00625-5},
       URL = {https://doi.org/10.1090/S0894-0347-08-00625-5},
}

@article {vakil2000enumerative,
    AUTHOR = {Vakil, Ravi},
     TITLE = {The enumerative geometry of rational and elliptic curves in
              projective space},
   JOURNAL = {J. Reine Angew. Math.},
  FJOURNAL = {Journal f\"ur die Reine und Angewandte Mathematik. [Crelle's
              Journal]},
    VOLUME = {529},
      YEAR = {2000},
     PAGES = {101--153},
      ISSN = {0075-4102,1435-5345},
   MRCLASS = {14N10 (14H10 14N35)},
  MRNUMBER = {1799935},
MRREVIEWER = {Gilberto\ Bini},
       DOI = {10.1515/crll.2000.094},
       URL = {https://doi.org/10.1515/crll.2000.094},
}

@incollection {vakil2001tool,
    AUTHOR = {Vakil, Ravi},
     TITLE = {A tool for stable reduction of curves on surfaces},
 BOOKTITLE = {Advances in algebraic geometry motivated by physics ({L}owell,
              {MA}, 2000)},
    SERIES = {Contemp. Math.},
    VOLUME = {276},
     PAGES = {145--154},
 PUBLISHER = {Amer. Math. Soc., Providence, RI},
      YEAR = {2001},
      ISBN = {0-8218-2810-X},
   MRCLASS = {14H10 (14D06)},
  MRNUMBER = {1837115},
MRREVIEWER = {Susan\ J.\ Colley},
       DOI = {10.1090/conm/276/04517},
       URL = {https://doi.org/10.1090/conm/276/04517},
}

@misc{mss2025orbifoldspectral,
      title={Orbifold Floer spectral invariants, symmetric product links and Weyl laws}, 
      author={Cheuk Yu Mak and Sobhan Seyfaddini and Ivan Smith},
      year={2025},
      eprint={2512.00454},
      archivePrefix={arXiv},
      primaryClass={math.SG},
      url={https://arxiv.org/abs/2512.00454}, 
}

@article {olsson2007log,
    AUTHOR = {Olsson, Martin C.},
     TITLE = {({L}og) twisted curves},
   JOURNAL = {Compos. Math.},
  FJOURNAL = {Compositio Mathematica},
    VOLUME = {143},
      YEAR = {2007},
    NUMBER = {2},
     PAGES = {476--494},
      ISSN = {0010-437X,1570-5846},
   MRCLASS = {14D22 (14A20 14N35)},
  MRNUMBER = {2309994},
MRREVIEWER = {Charles\ D.\ Cadman},
       DOI = {10.1112/S0010437X06002442},
       URL = {https://doi.org/10.1112/S0010437X06002442},
}

@article {abramovichvistoli2002,
    AUTHOR = {Abramovich, Dan and Vistoli, Angelo},
     TITLE = {Compactifying the space of stable maps},
   JOURNAL = {J. Amer. Math. Soc.},
  FJOURNAL = {Journal of the American Mathematical Society},
    VOLUME = {15},
      YEAR = {2002},
    NUMBER = {1},
     PAGES = {27--75},
      ISSN = {0894-0347,1088-6834},
   MRCLASS = {14H10 (14A20 14D20 14N35)},
  MRNUMBER = {1862797},
MRREVIEWER = {Tyler\ J.\ Jarvis},
       DOI = {10.1090/S0894-0347-01-00380-0},
       URL = {https://doi.org/10.1090/S0894-0347-01-00380-0},
}

@misc{ekholmshende2024counting,
      title={Counting bare curves}, 
      author={Ekholm, Tobias and Vivek Shende},
      year={2024},
      eprint={2406.00890},
      archivePrefix={arXiv},
      primaryClass={math.SG},
      url={https://arxiv.org/abs/2406.00890}, 
}

@article {HTY2026,
    AUTHOR = {Honda, Ko and Tian, Yin and Yuan, Tianyu},
     TITLE = {Higher-dimensional {H}eegaard {F}loer homology and {H}ecke
              algebras},
   JOURNAL = {J. Eur. Math. Soc. (JEMS)},
  FJOURNAL = {Journal of the European Mathematical Society (JEMS)},
    VOLUME = {28},
      YEAR = {2026},
    NUMBER = {4},
     PAGES = {1661--1694},
      ISSN = {1435-9855,1435-9863},
   MRCLASS = {53D10 (20C08 53D40 57R58)},
  MRNUMBER = {5032986},
       DOI = {10.4171/jems/1525},
       URL = {https://doi.org/10.4171/jems/1525},
}

@misc{hkty2025morse,
    title={Morse theory of loop spaces and Hecke algebras.},
    author={Honda, Ko and Krutowski, Roman and Tian, Yin and Yuan, Tianyu},
    eprint={2503.07543},
    archivePrefix={arXiv},
    primaryClass={math.SG},
    year={2025}
}

@misc{hirschihugtenburg2025,
    title={An open-closed Deligne-Mumford field theory associated to a Lagrangian submanifold.},
    author={Hirschi, Amanda and Kai Hugtenburg},
    eprint={2501.04687},
    archivePrefix={arXiv},
    primaryClass={math.SG},
    year={2025}
}

@misc{rabah2024,
    title={Fukaya Algebra over $\mathbb {Z}$.},
    author={Rabah, Mohamad},
    eprint={2411.14657},
    archivePrefix={arXiv},
    primaryClass={math.SG},
    year={2024}
}

@misc{xu2026reduced,
    title={Reduced Gromov-Witten invariants without ghost bubble censorship.},
    author={Xu, Guangbo},
    eprint={2604.13209},
    archivePrefix={arXiv},
    primaryClass={math.SG},
    year={2026}
}

@article {lascouxpragacz2002,
    AUTHOR = {Lascoux, Alain and Pragacz, Piotr},
     TITLE = {Jacobians of symmetric polynomials},
   JOURNAL = {Ann. Comb.},
  FJOURNAL = {Annals of Combinatorics},
    VOLUME = {6},
      YEAR = {2002},
    NUMBER = {2},
     PAGES = {169--172},
      ISSN = {0218-0006,0219-3094},
   MRCLASS = {05E05 (26B10)},
  MRNUMBER = {1955517},
       DOI = {10.1007/PL00012583},
       URL = {https://doi.org/10.1007/PL00012583},
}

@misc{oh2021reduced,
    title={Geometry of Liouville sectors and the maximum principle.},
    author={Oh, Yong-Geun},
    eprint={2111.06112},
    archivePrefix={arXiv},
    primaryClass={math.SG},
    year={2021}
}

@article {abouzaidseidel2010,
    AUTHOR = {Abouzaid, Mohammed and Seidel, Paul},
     TITLE = {An open string analogue of {V}iterbo functoriality},
   JOURNAL = {Geom. Topol.},
  FJOURNAL = {Geometry \& Topology},
    VOLUME = {14},
      YEAR = {2010},
    NUMBER = {2},
     PAGES = {627--718},
      ISSN = {1465-3060,1364-0380},
   MRCLASS = {53D40 (53D12 53D37)},
  MRNUMBER = {2602848},
MRREVIEWER = {Timothy\ Perutz},
       DOI = {10.2140/gt.2010.14.627},
       URL = {https://doi.org/10.2140/gt.2010.14.627},
}

@misc{wendl2020sft,
      title={Lectures on Symplectic Field Theory.}, 
      author={Wendl, Chris },
      year={2020},
      url={https://www.mathematik. hu-berlin.de/~wendl/Sommer2020/SFT/lecturenotes.pdf}, 
}

@article {ego2006GDAHA,
    AUTHOR = {Etingof, Pavel and Gan, Wee Liang and Oblomkov, Alexei},
     TITLE = {Generalized double affine {H}ecke algebras of higher rank},
   JOURNAL = {J. Reine Angew. Math.},
  FJOURNAL = {Journal f\"ur die Reine und Angewandte Mathematik. [Crelle's
              Journal]},
    VOLUME = {600},
      YEAR = {2006},
     PAGES = {177--201},
      ISSN = {0075-4102,1435-5345},
   MRCLASS = {20C08 (17B67 37K20)},
  MRNUMBER = {2283803},
MRREVIEWER = {Bogdan\ Ion},
       DOI = {10.1515/CRELLE.2006.091},
       URL = {https://doi.org/10.1515/CRELLE.2006.091},
}

@article {eor2007GDAHA,
    AUTHOR = {Etingof, Pavel and Oblomkov, Alexei and Rains, Eric},
     TITLE = {Generalized double affine {H}ecke algebras of rank 1 and
              quantized del {P}ezzo surfaces},
   JOURNAL = {Adv. Math.},
  FJOURNAL = {Advances in Mathematics},
    VOLUME = {212},
      YEAR = {2007},
    NUMBER = {2},
     PAGES = {749--796},
      ISSN = {0001-8708,1090-2082},
   MRCLASS = {20C08 (14J26)},
  MRNUMBER = {2329319},
MRREVIEWER = {James\ W.\ Parkinson},
       DOI = {10.1016/j.aim.2006.11.008},
       URL = {https://doi.org/10.1016/j.aim.2006.11.008},
}

@misc{wehrheimwoodward2015orientations,
    title={Orientations for pseudoholomorphic quilts.},
    author={Wehrheim, Katrin and Chris Woodward},
    eprint={1503.07803},
    archivePrefix={arXiv},
    primaryClass={math.SG},
    year={2015}
}

@book {chenzinger2024,
    AUTHOR = {Chen, Xujia and Zinger, Aleksey},
     TITLE = {Spin/pin-structures and real enumerative geometry},
 PUBLISHER = {World Scientific Publishing Co. Pte. Ltd., Hackensack, NJ},
      YEAR = {[2024] \copyright 2024},
     PAGES = {xiv+452},
      ISBN = {[9789811278532]; [9789811278549]; [9789811278556]},
   MRCLASS = {53Dxx (14N35 14P25 53C27)},
  MRNUMBER = {4704417},
       DOI = {10.1142/13476},
       URL = {https://doi.org/10.1142/13476},
}

@article {choshin2016,
    AUTHOR = {Cho, Cheol-Hyun and Shin, Hyung-Seok},
     TITLE = {Chern-{W}eil {M}aslov index and its orbifold analogue},
   JOURNAL = {Asian J. Math.},
  FJOURNAL = {Asian Journal of Mathematics},
    VOLUME = {20},
      YEAR = {2016},
    NUMBER = {1},
     PAGES = {1--19},
      ISSN = {1093-6106,1945-0036},
   MRCLASS = {53D12 (57R18)},
  MRNUMBER = {3460756},
MRREVIEWER = {Dmitri\u i\ Vladimir\ Alekseevsky},
       DOI = {10.4310/AJM.2016.v20.n1.a1},
       URL = {https://doi.org/10.4310/AJM.2016.v20.n1.a1},
}

@misc{stacks-project,
  author       = {The {Stacks project authors}},
  title        = {The Stacks project},
  howpublished = {\url{https://stacks.math.columbia.edu}},
  year         = {2026},
}

@article {LCHL2024,
    AUTHOR = {Amorim, Lino and Cho, Cheol-Hyun and Hong, Hansol and Lau,
              Siu-Cheong},
     TITLE = {Big quantum cohomology of orbifold spheres},
   JOURNAL = {Comm. Anal. Geom.},
  FJOURNAL = {Communications in Analysis and Geometry},
    VOLUME = {32},
      YEAR = {2024},
    NUMBER = {6},
     PAGES = {1509--1592},
      ISSN = {1019-8385,1944-9992},
   MRCLASS = {14N35},
  MRNUMBER = {4836043},
       DOI = {10.4310/cag.241203233544},
       URL = {https://doi.org/10.4310/cag.241203233544},
}

@article {polterovichshelukhin2023,
    AUTHOR = {Polterovich, Leonid and Shelukhin, Egor},
     TITLE = {Lagrangian configurations and {H}amiltonian maps},
   JOURNAL = {Compos. Math.},
  FJOURNAL = {Compositio Mathematica},
    VOLUME = {159},
      YEAR = {2023},
    NUMBER = {12},
     PAGES = {2483--2520},
      ISSN = {0010-437X,1570-5846},
   MRCLASS = {53D12 (37J06 53D05)},
  MRNUMBER = {4651502},
       DOI = {10.1112/s0010437x23007455},
       URL = {https://doi.org/10.1112/s0010437x23007455},
}

@misc{FLYZ,
    title={Remodeling Conjecture with Descendants.},
    author={Fang, Bohan and Liu, Chiu-Chu Melissa and Yu, Song and Zong, Zhengyu},
    eprint={2504.15696},
    archivePrefix={arXiv},
    primaryClass={math.SG},
    year={2025}
}

@article {maksmith2021,
    AUTHOR = {Mak, Cheuk Yu and Smith, Ivan},
     TITLE = {Non-displaceable {L}agrangian links in four-manifolds},
   JOURNAL = {Geom. Funct. Anal.},
  FJOURNAL = {Geometric and Functional Analysis},
    VOLUME = {31},
      YEAR = {2021},
    NUMBER = {2},
     PAGES = {438--481},
      ISSN = {1016-443X,1420-8970},
   MRCLASS = {53D12 (57K10)},
  MRNUMBER = {4268306},
MRREVIEWER = {Roman\ Golovko},
       DOI = {10.1007/s00039-021-00562-8},
       URL = {https://doi.org/10.1007/s00039-021-00562-8},
}

@article {cherednik1995,
    AUTHOR = {Cherednik, Ivan},
     TITLE = {Double affine {H}ecke algebras and {M}acdonald's conjectures},
   JOURNAL = {Ann. of Math. (2)},
  FJOURNAL = {Annals of Mathematics. Second Series},
    VOLUME = {141},
      YEAR = {1995},
    NUMBER = {1},
     PAGES = {191--216},
      ISSN = {0003-486X,1939-8980},
   MRCLASS = {33D80 (05E15 05E35 17B20 39A70)},
  MRNUMBER = {1314036},
MRREVIEWER = {Eric\ M.\ Opdam},
       DOI = {10.2307/2118632},
       URL = {https://doi.org/10.2307/2118632},
}

@article {sahi1999nonsymmetric,
    AUTHOR = {Sahi, Siddhartha},
     TITLE = {Nonsymmetric {K}oornwinder polynomials and duality},
   JOURNAL = {Ann. of Math. (2)},
  FJOURNAL = {Annals of Mathematics. Second Series},
    VOLUME = {150},
      YEAR = {1999},
    NUMBER = {1},
     PAGES = {267--282},
      ISSN = {0003-486X,1939-8980},
   MRCLASS = {33D52 (20C08 33D80)},
  MRNUMBER = {1715325},
MRREVIEWER = {Kenji\ Taniguchi},
       DOI = {10.2307/121102},
       URL = {https://doi.org/10.2307/121102},
}

@incollection {etingofma2008dunkl,
    AUTHOR = {Etingof, Pavel and Ma, Xiaoguang},
     TITLE = {On elliptic {D}unkl operators},
      NOTE = {Special volume in honor of Melvin Hochster},
   JOURNAL = {Michigan Math. J.},
  FJOURNAL = {Michigan Mathematical Journal},
    VOLUME = {57},
      YEAR = {2008},
     PAGES = {293--304},
      ISSN = {0026-2285,1945-2365},
   MRCLASS = {14J60 (16S99 20C08 20F55 32L05)},
  MRNUMBER = {2492454},
MRREVIEWER = {Maurizio\ Martino},
       DOI = {10.1307/mmj/1220879410},
       URL = {https://doi.org/10.1307/mmj/1220879410},
}

@article {efmv2011calogeromoser,
    AUTHOR = {Etingof, Pavel and Felder, Giovanni and Ma, Xiaoguang and
              Veselov, Alexander},
     TITLE = {On elliptic {C}alogero-{M}oser systems for complex
              crystallographic reflection groups},
   JOURNAL = {J. Algebra},
  FJOURNAL = {Journal of Algebra},
    VOLUME = {329},
      YEAR = {2011},
     PAGES = {107--129},
      ISSN = {0021-8693,1090-266X},
   MRCLASS = {37J35 (20F55 20H15 81R12)},
  MRNUMBER = {2769318},
MRREVIEWER = {Maurizio\ Martino},
       DOI = {10.1016/j.jalgebra.2010.04.011},
       URL = {https://doi.org/10.1016/j.jalgebra.2010.04.011},
}

@book {popov1982complex,
    AUTHOR = {Popov, V. L.},
     TITLE = {Discrete complex reflection groups},
    SERIES = {Communications of the Mathematical Institute,
              Rijksuniversiteit Utrecht},
    VOLUME = {15},
 PUBLISHER = {Rijksuniversiteit Utrecht, Mathematical Institute, Utrecht},
      YEAR = {1982},
     PAGES = {89},
   MRCLASS = {20H15 (14D25 20F38 22E40)},
  MRNUMBER = {645542},
MRREVIEWER = {Klaus\ Pommerening},
}

@article {malle1996presentations,
    AUTHOR = {Malle, Gunter},
     TITLE = {Presentations for crystallographic complex reflection groups},
   JOURNAL = {Transform. Groups},
  FJOURNAL = {Transformation Groups},
    VOLUME = {1},
      YEAR = {1996},
    NUMBER = {3},
     PAGES = {259--277},
      ISSN = {1083-4362,1531-586X},
   MRCLASS = {20H15 (20F05 20F34)},
  MRNUMBER = {1417713},
MRREVIEWER = {Arjeh\ M.\ Cohen},
       DOI = {10.1007/BF02549209},
       URL = {https://doi.org/10.1007/BF02549209},
}

@article {puenteshepler2019steinberg,
    AUTHOR = {Puente, Philip and Shepler, Anne V.},
     TITLE = {Steinberg's theorem for crystallographic complex reflection
              groups},
   JOURNAL = {J. Algebra},
  FJOURNAL = {Journal of Algebra},
    VOLUME = {522},
      YEAR = {2019},
     PAGES = {332--350},
      ISSN = {0021-8693,1090-266X},
   MRCLASS = {20F55 (20H15)},
  MRNUMBER = {3899045},
MRREVIEWER = {Gordon\ Ian\ Williams},
       DOI = {10.1016/j.jalgebra.2018.11.032},
       URL = {https://doi.org/10.1016/j.jalgebra.2018.11.032},
}

@article {jordanvazirani2021,
    AUTHOR = {Jordan, David and Vazirani, Monica},
     TITLE = {The rectangular representation of the double affine {H}ecke
              algebra via elliptic {S}chur-{W}eyl duality},
   JOURNAL = {Int. Math. Res. Not. IMRN},
  FJOURNAL = {International Mathematics Research Notices. IMRN},
      YEAR = {2021},
    NUMBER = {8},
     PAGES = {5968--6019},
      ISSN = {1073-7928,1687-0247},
   MRCLASS = {20C08 (17B37 20C30)},
  MRNUMBER = {4251269},
MRREVIEWER = {Anton\ Cox},
       DOI = {10.1093/imrn/rnz030},
       URL = {https://doi.org/10.1093/imrn/rnz030},
}

@misc{simentalnotes,
      title={Notes for a talk on Double Affine Hecke Algebras.}, 
      author={Simental, Jos\'{e} },
      year={2017},
      url={https://ivanloseu.github.io/DAHAEHA_Jose2.pdf}, 
}

@article {etingofrains2005new,
    AUTHOR = {Etingof, Pavel and Rains, Eric},
     TITLE = {New deformations of group algebras of {C}oxeter groups},
   JOURNAL = {Int. Math. Res. Not.},
  FJOURNAL = {International Mathematics Research Notices},
      YEAR = {2005},
    NUMBER = {10},
     PAGES = {635--646},
      ISSN = {1073-7928,1687-0247},
   MRCLASS = {20F55 (16S34 16S80 20C05 20C08)},
  MRNUMBER = {2147005},
MRREVIEWER = {James\ E.\ Humphreys},
       DOI = {10.1155/IMRN.2005.635},
       URL = {https://doi.org/10.1155/IMRN.2005.635},
}

@article {etingofrains2008new2,
    AUTHOR = {Etingof, Pavel and Rains, Eric},
     TITLE = {New deformations of group algebras of {C}oxeter groups. {II}},
   JOURNAL = {Geom. Funct. Anal.},
  FJOURNAL = {Geometric and Functional Analysis},
    VOLUME = {17},
      YEAR = {2008},
    NUMBER = {6},
     PAGES = {1851--1871},
      ISSN = {1016-443X,1420-8970},
   MRCLASS = {20F55 (16S34 20C05 20C08)},
  MRNUMBER = {2399085},
       DOI = {10.1007/s00039-007-0642-7},
       URL = {https://doi.org/10.1007/s00039-007-0642-7},
}

@article {gps2024microlocal,
    AUTHOR = {Ganatra, Sheel and Pardon, John and Shende, Vivek},
     TITLE = {Microlocal {M}orse theory of wrapped {F}ukaya categories},
   JOURNAL = {Ann. of Math. (2)},
  FJOURNAL = {Annals of Mathematics. Second Series},
    VOLUME = {199},
      YEAR = {2024},
    NUMBER = {3},
     PAGES = {943--1042},
      ISSN = {0003-486X,1939-8980},
   MRCLASS = {53D37 (14J33 53D40)},
  MRNUMBER = {4740209},
MRREVIEWER = {Christopher\ T.\ Woodward},
       DOI = {10.4007/annals.2024.199.3.1},
       URL = {https://doi.org/10.4007/annals.2024.199.3.1},
}

@article {gps2020covariant,
    AUTHOR = {Ganatra, Sheel and Pardon, John and Shende, Vivek},
     TITLE = {Covariantly functorial wrapped {F}loer theory on {L}iouville
              sectors},
   JOURNAL = {Publ. Math. Inst. Hautes \'Etudes Sci.},
  FJOURNAL = {Publications Math\'ematiques. Institut de Hautes \'Etudes
              Scientifiques},
    VOLUME = {131},
      YEAR = {2020},
     PAGES = {73--200},
      ISSN = {0073-8301,1618-1913},
   MRCLASS = {53D40},
  MRNUMBER = {4106794},
MRREVIEWER = {Alexander\ Fel\cprime shtyn},
       DOI = {10.1007/s10240-019-00112-x},
       URL = {https://doi.org/10.1007/s10240-019-00112-x},
}

@book {cherednik2005,
    AUTHOR = {Cherednik, Ivan},
     TITLE = {Double affine {H}ecke algebras},
    SERIES = {London Mathematical Society Lecture Note Series},
    VOLUME = {319},
 PUBLISHER = {Cambridge University Press, Cambridge},
      YEAR = {2005},
     PAGES = {xii+434},
      ISBN = {0-521-60918-6},
   MRCLASS = {32G34 (05E15 11L05 17B20 20C08 33D52 33D80)},
  MRNUMBER = {2133033},
       DOI = {10.1017/CBO9780511546501},
       URL = {https://doi.org/10.1017/CBO9780511546501},
}

@article {ACL2025sw,
    AUTHOR = {Argyres, Philip C. and Chalykh, Oleg and L\"u, Yongchao},
     TITLE = {Complex crystallographic reflection groups and
              {S}eiberg-{W}itten integrable systems: rank 1 case},
   JOURNAL = {J. High Energy Phys.},
  FJOURNAL = {Journal of High Energy Physics},
      YEAR = {2025},
    NUMBER = {12},
     PAGES = {Paper No. 84, 49},
      ISSN = {1126-6708,1029-8479},
   MRCLASS = {81T60 (81R12)},
  MRNUMBER = {5004580},
       DOI = {10.1007/jhep12(2025)084},
       URL = {https://doi.org/10.1007/jhep12(2025)084},
}

@article {ballamyschedler2016existence,
    AUTHOR = {Bellamy, Gwyn and Schedler, Travis},
     TITLE = {On the (non)existence of symplectic resolutions of linear
              quotients},
   JOURNAL = {Math. Res. Lett.},
  FJOURNAL = {Mathematical Research Letters},
    VOLUME = {23},
      YEAR = {2016},
    NUMBER = {6},
     PAGES = {1537--1564},
      ISSN = {1073-2780,1945-001X},
   MRCLASS = {14B05 (14L35 20H20)},
  MRNUMBER = {3621098},
MRREVIEWER = {D.\ A.\ Stepanov},
       DOI = {10.4310/MRL.2016.v23.n6.a1},
       URL = {https://doi.org/10.4310/MRL.2016.v23.n6.a1},
}

@article {cohen1980cassification,
    AUTHOR = {Cohen, Arjeh M.},
     TITLE = {Finite quaternionic reflection groups},
   JOURNAL = {J. Algebra},
  FJOURNAL = {Journal of Algebra},
    VOLUME = {64},
      YEAR = {1980},
    NUMBER = {2},
     PAGES = {293--324},
      ISSN = {0021-8693},
   MRCLASS = {20H15},
  MRNUMBER = {579063},
MRREVIEWER = {J.\ D.\ Dixon},
       DOI = {10.1016/0021-8693(80)90148-9},
       URL = {https://doi.org/10.1016/0021-8693(80)90148-9},
}

@misc{KY2023,
    title={Heegaard Floer symplectic homology and {V}iterbo's isomorphism theorem in the context of multiple identical particles},
    author={Roman Krutowski and Tianyu Yuan},
    year={2023},
    eprint={2311.17031},
    archivePrefix={arXiv},
    primaryClass={math.SG}
    }

@misc{MMMR2019,
    title={Floer homology for global quotient orbifolds, Master’s thesis, Técnico Lisboa},
    author={Moreira, Miguel Miranda Ribeiro},
    year={2019},
    url={https://math.mit.edu/~miguel73/master_thesis.pdf}
    }

@misc{rothgang2024,
    title={Equivariant transversality for closed holomorphic curves. Dissertation, Humboldt Universitaet zu Berlin (Germany)},
    author={Rothgang, Michael Benjamin},
    year={2019},
    url={https://doi.org/10.18452/33812}
    }

@article {pasquinelli2016dmthree,
    AUTHOR = {Pasquinelli, Irene},
     TITLE = {Deligne-{M}ostow lattices with three fold symmetry and cone
              metrics on the sphere},
   JOURNAL = {Conform. Geom. Dyn.},
  FJOURNAL = {Conformal Geometry and Dynamics. An Electronic Journal of the
              American Mathematical Society},
    VOLUME = {20},
      YEAR = {2016},
     PAGES = {235--281},
      ISSN = {1088-4173},
   MRCLASS = {22E40 (20M10 30F40 32M05 51M10 57M50)},
  MRNUMBER = {3522983},
MRREVIEWER = {Wensheng\ Cao},
       DOI = {10.1090/ecgd/299},
       URL = {https://doi.org/10.1090/ecgd/299},
}

@article {DPP2021new,
    AUTHOR = {Deraux, Martin and Parker, John R. and Paupert, Julien},
     TITLE = {New nonarithmetic complex hyperbolic lattices {II}},
   JOURNAL = {Michigan Math. J.},
  FJOURNAL = {Michigan Mathematical Journal},
    VOLUME = {70},
      YEAR = {2021},
    NUMBER = {1},
     PAGES = {135--205},
      ISSN = {0026-2285,1945-2365},
   MRCLASS = {20H10 (22E40 32M18)},
  MRNUMBER = {4255091},
MRREVIEWER = {Grzegorz\ Gromadzki},
       DOI = {10.1307/mmj/1592532044},
       URL = {https://doi.org/10.1307/mmj/1592532044},
}

@article {deraux2024subgroups,
    AUTHOR = {Deraux, Martin},
     TITLE = {On subgroups finite index in complex hyperbolic lattice
              triangle groups},
   JOURNAL = {Exp. Math.},
  FJOURNAL = {Experimental Mathematics},
    VOLUME = {33},
      YEAR = {2024},
    NUMBER = {3},
     PAGES = {456--481},
      ISSN = {1058-6458,1944-950X},
   MRCLASS = {22E40 (20H10)},
  MRNUMBER = {4790979},
MRREVIEWER = {O.\ V.\ Shvartsman},
       DOI = {10.1080/10586458.2022.2158969},
       URL = {https://doi.org/10.1080/10586458.2022.2158969},
}

@article {BM2014affinity,
    AUTHOR = {Bellamy, Gwyn and Martino, Maurizio},
     TITLE = {Affinity of {C}herednik algebras on projective space},
   JOURNAL = {Algebra Number Theory},
  FJOURNAL = {Algebra \& Number Theory},
    VOLUME = {8},
      YEAR = {2014},
    NUMBER = {5},
     PAGES = {1151--1177},
      ISSN = {1937-0652,1944-7833},
   MRCLASS = {14F05 (14F10 16S80 20C08)},
  MRNUMBER = {3263139},
MRREVIEWER = {Andreas\ H\"oring},
       DOI = {10.2140/ant.2014.8.1151},
       URL = {https://doi.org/10.2140/ant.2014.8.1151},
}

@article {mostow1980remarkable,
    AUTHOR = {Mostow, G. D.},
     TITLE = {On a remarkable class of polyhedra in complex hyperbolic
              space},
   JOURNAL = {Pacific J. Math.},
  FJOURNAL = {Pacific Journal of Mathematics},
    VOLUME = {86},
      YEAR = {1980},
    NUMBER = {1},
     PAGES = {171--276},
      ISSN = {0030-8730,1945-5844},
   MRCLASS = {22E40 (30F35 51M10)},
  MRNUMBER = {586876},
MRREVIEWER = {\`E.\ Vinberg},
       URL = {http://projecteuclid.org/euclid.pjm/1102780622},
}

@article {delignemostow1986monodromy,
    AUTHOR = {Deligne, P. and Mostow, G. D.},
     TITLE = {Monodromy of hypergeometric functions and nonlattice integral
              monodromy},
   JOURNAL = {Inst. Hautes \'Etudes Sci. Publ. Math.},
  FJOURNAL = {Institut des Hautes \'Etudes Scientifiques. Publications
              Math\'ematiques},
    NUMBER = {63},
      YEAR = {1986},
     PAGES = {5--89},
      ISSN = {0073-8301,1618-1913},
   MRCLASS = {22E40 (32G20 33A30)},
  MRNUMBER = {849651},
MRREVIEWER = {Shigeaki\ Tsuyumine},
       URL = {http://www.numdam.org/item?id=PMIHES_1986__63__5_0},
}

\end{document}